\documentclass[10pt,reqno]{amsart}

\usepackage{latexsym}
\usepackage{amsfonts,amssymb,amsmath,amscd,amsthm}
\usepackage[mathscr]{eucal}
\usepackage{mathrsfs}
\usepackage{mathbbol}
\usepackage{mathabx}
\usepackage{oldgerm,units}
\usepackage{centernot}
\usepackage{ifthen}
\usepackage{epsfig,graphics,epic}
\usepackage{bbold}
\usepackage{wrapfig,epsfig}

\usepackage{graphicx}
\usepackage{pst-node}
\usepackage{tikz}

\definecolor{lgray}{gray}{0.90}

\def\xdtl{\dottedline{0.15}}

\def\grid{\thinlines \xdtl(1,1)(1,9) \xdtl(2,1)(2,9)
\xdtl(3,1)(3,9) \xdtl(4,1)(4,9) \xdtl(5,1)(5,9) \xdtl(6,1)(6,9)
\xdtl(7,1)(7,9) \xdtl(8,1)(8,9) \xdtl(9,1)(9,9)
\xdtl(1,1)(9,1) \xdtl(1,2)(9,2) \xdtl(1,3)(9,3)
\xdtl(1,4)(9,4) \xdtl(1,5)(9,5) \xdtl(1,6)(9,6)
\xdtl(1,7)(9,7)  \xdtl(1,8)(9,8)  \xdtl(1,9)(9,9) }

\def\hSkip{\hskip -10pt}

\input xy
\xyoption{all}

\newcommand{\ifdef}[3]{\ifthenelse{\equal{#1}{true}}{#2}{#3}}

\newcommand{\withFigures}{true}

\newcommand{\longleft}[1]{\;{\leftarrow%
\count255=0 \loop \mathrel{\mkern-6mu}%
    \relbar\advance\count255 by1\ifnum\count255<#1\repeat}\;}
\newcommand{\longright}[1]{\;{\count255=0 \loop \relbar\mathrel{\mkern-6mu}%
    \advance\count255 by1\ifnum\count255<#1\repeat\rightarrow}\;}
\newcommand{\Right}[2]{\overset{#2}{\longright{#1}}}
\newcommand{\RIGHT}[3]{\mathrel{\mathop{\kern0pt\longright{#1}}
    \limits^{#2}_{#3}}}

\newcommand{\LEFT}[3]{\mathrel{\mathop{\kern0pt\longleft{#1}}\limits^{#2}_{#3}}
}
\newcommand{\longleftright}[1]{\;{\leftarrow\mathrel{\mkern-6mu}%
    \count255=0\loop\relbar\mathrel{\mkern-6mu}%
    \advance\count255 by1\ifnum\count255<#1\repeat\rightarrow}\;}

\newcommand{\onto}[1]{\;{\count255=0 \loop \relbar\joinrel
    \advance\count255 by1
    \ifnum\count255<#1 \repeat \twoheadrightarrow}\;}

\newcommand{\RLEFT}[3]{\mathrel{%
   \mathop{\vcenter{\baselineskip=0pt\hbox{$\kern0pt\longright{#1}$}%
   \hbox{$\kern0pt\longleft{#1}$}}}\limits^{#2}_{#3}}}

\newcommand{\To}{\longrightarrow }
\newcommand{\TO}{\longright{3} }
\newcommand{\ONTO}{\onto{3} }
\newcommand{\ISOTO}{\Right{3}{\,\smash{\raisebox{-0.45ex}{\ensuremath{\scriptstyle\sim}}}\,}}
\newcommand{\Isoto}{\Right{1}{\,\smash{\raisebox{-0.45ex}{\ensuremath{\scriptstyle\sim}}}\,}}
\newcommand{\xto}[1]{\xrightarrow{#1}}

\newcommand{\hooklongrightarrow}{\lhook\joinrel\longrightarrow}
\newcommand{\Hooklongrightarrow}{\lhook\joinrel\joinrel\joinrel\joinrel\TO}

\newcommand{\Into}{\hooklongrightarrow}
\newcommand{\Onto}{\onto{1}}
\newcommand{\INTO}{\Hooklongrightarrow}
\newcommand{\Mto}{\longmapsto}
\newcommand{\mTo}{\longmapsto }

\newcommand{\plim}[1]{\lim_{\substack{\longleftarrow\\#1}}}
\newcommand{\ilim}[1]{\lim_{\substack{\longrightarrow\\#1}}}

\def\dispace{\setlength{\itemsep}{2pt}}
\newcommand{\etype}[1]{\renewcommand{\labelenumi}{(#1{enumi})}}

\def\eroman{\etype{\roman} \dispace}
\def\ealph{\etype{\alph} \dispace}

\newcommand\semph[1]{\emph{\textbf{#1}}}

\newcommand{\ds}[1]{\ {#1} \ }

\newtheorem{theorem}{Theorem}[section]
\newtheorem{proposition}[theorem]{Proposition}

\newtheorem{definition}[theorem]{Definition}

\newtheorem{lemma}[theorem]{Lemma}

\newtheorem{notation}[theorem]{Notation}
\newtheorem{corollary}[theorem]{Corollary}

\newtheorem{example}[theorem]{Example}
\newtheorem{remark}[theorem]{Remark}

\newtheorem{lem}[theorem]{Lemma}
\newtheorem{rem}[theorem]{Remark}
\newtheorem{comm}[theorem]{Comment}
\newtheorem*{comm*}{Comment}
\newtheorem{prop*}{Proposition}

\newtheorem{proper}[theorem]{Properties}

\newtheorem{prop}[theorem]{Proposition}
\newtheorem{defn}[theorem]{Definition}
\newtheorem*{examp*}{Example}
\newtheorem*{examples*}{Examples}
\newtheorem*{remark*}{Remark}
\newtheorem*{defn*}{Definition}
\newtheorem{construction}[theorem]{Construction}

\newcommand{\Trop}{\mathbb T}
\newcommand{\Comp}{\mathbb C}
\newcommand{\Real}{\mathbb R}
\newcommand{\Rati}{\mathbb Q}
\newcommand{\Int}{\mathbb Z}
\newcommand{\Net}{\mathbb N}
\newcommand{\Proj}{\mathbb P}
\newcommand{\Aff}{\mathbb A}
\newcommand{\Fld}{\mathbb K}
\newcommand{\bbF}{\mathbb F}
\newcommand{\NetZ}{\Net_0}

\def\sbool{{\mathbb{ B}}}
\def\bool{{\mathbb{ B}}}

\def\zReal{\Real_{-\infty}}
\def\zRealn{\zReal^{(n)}}

\newcommand{\Dir}{\dss{ \Rightarrow }}

\def\sset{\subset}
\def\ssetq{\subseteq}

\def\lsset{\prec}

\newcommand{\one}{\mathbb{1}}
\newcommand{\zero}{\mathbb{0}}

\newcommand{\trop}[1]{\mathcal{#1}}

\newcommand{\tA}{\trop{A}}
\newcommand{\tB}{\trop{B}}
\newcommand{\tC}{\trop{C}}
\newcommand{\tD}{\trop{D}}
\newcommand{\tE}{\trop{E}}
\newcommand{\tF}{\trop{F}}
\newcommand{\tG}{\trop{G}}
\newcommand{\tH}{\trop{H}}

\newcommand{\tJ}{\trop{J}}
\newcommand{\tK}{\trop{K}}

\newcommand{\tM}{\trop{M}}
\newcommand{\tN}{\trop{N}}

\newcommand{\tP}{\trop{P}}

\newcommand{\tS}{\trop{S}}
\newcommand{\tT}{\trop{T}}

\newcommand{\tU}{\trop{U}}
\newcommand{\tV}{\trop{V}}

\newcommand{\tZ}{\trop{Z}}

\newcommand{\Om}{\Omega}
\newcommand{\io}{\iota}

\newcommand{\eps}{\epsilon}
\newcommand{\al}{\alpha}
\newcommand{\bt}{\beta}

\newcommand{\Gm}{\Gamma}

\newcommand{\lm}{\lambda}
\newcommand{\Lm}{\Lambda}

\def\tlOm{\widetilde{\Om}}
\def\NDa{MD1}
\def\NDb{MD2}
\def\NMa{NM1}
\def\NMb{NM2}
\def\NMc{NM3}
\def\PSRa{NS1}
\def\PSRb{NS2}

\def\cnd{\ds \vert}
\def\hh{\psi}
\def\ihh{\hh^\inv}
\def\ff{\phi}
\def\tlff{\widetilde{\ff}}
\def\iff{{\ff^\inv}}
\def\sff{\ff^\#}
\def\dff{{\ff_*}}

\def\pff{(\ff,\sff)}
\def\ffx{{\ff(x)}}

\def\sffx{{\sff_x}}

\def\Iff{\ds{\Leftrightarrow}}
\def\VV{\tV}

\def\VVg{\VV(g)}

\def\EE{\tE}

\def\CC{\tC}
\def\CCf{\CC(f)}

\def\MD{\mfD}
\def\MDx{\MD_x}
\def\Ux{U_x}

\newcommand\fcl[1]{\widecheck{#1}}

\def\fcU{\fcl{U}}
\def\fcV{\fcl{V}}

\def\DD{\tD}
\def\fcDD{\fcl{\DD}}
\def\fcDDf{\fcDD(f)}
\def\fcDDg{\fcDD(g)}
\def\fcDDh{\fcDD(h)}
\def\DDf{\DD(f)}
\def\DDg{\DD(g)}
\def\DDfg{\DD(fg)}
\def\DDh{\DD(h)}

\def\EEf{\EE(f)}
\def\EEg{\EE(g)}
\def\EEh{\EE(h)}

\def\II{\tK}

\def\GG{\operatorname{G}}

\def\muo{{\mu_\otimes}}

\def\ST{\tS}

\def\SST{\tS}

\def\ST{S}
\def\STs{\ST^*}

\def\STf{{\ST(f)}}
\def\STsf{{\STs(f)}}
\def\STg{{\ST(g)}}
\def\STh{{\ST(h)}}
\def\STfg{{\ST(fg)}}

\def\hgt{\operatorname{ht}}
\def\val{\operatorname{val}}

\def\Cor{\operatorname{Cor}}

\def\sval{\operatorname{sval}}
\def\STR{\operatorname{STR}}

\def\reF{f_{_\Real}}

\def\trivial{isolated}

\def\res{\rho}
\def\scn{\sig}

\def\impl{\ds{ \Rightarrow}}
\def\Impl{\dss{ \Rightarrow}}
\def\Spec{\operatorname{Spec}}
\def\MSpec{\operatorname{Spm}_\ell}
\def\NSpec{\operatorname{Spn}_\ell}

\def\tNSpec{\NSpec}

\def\RSpec{\operatorname{Spr}}
\def\TSpec{\operatorname{Spt}}
\def\DSpec{\operatorname{Spd}}
\def\SpecR{\Spec(R)}
\def\SpecA{\Spec(\sA)}

\def\SpecB{\Spec(\sB)}
\def\OSpecA{\scO_{\SpecA}}
\def\OSpecB{\scO_{\SpecB}}

\def\scG{\mathscr G}
\def\scF{\mathscr F}

\def\scO{\mathscr O}

\def\rsmf{K}

\def\O{\scO}
\def\OX{\scO_X}
\def\OXX{\scO(X)}
\def\cOX{\overline{\scO}_X}
\def\OXf{\scO_{X_f}}
\def\OXx{\scO_{X,x}}

\def\OY{\scO_Y}
\def\OYY{\scO(Y)}

\def\OYfx{\scO_{Y,\ffx}}

\def\XOX{(X,\OX)}
\def\YOY{(Y,\OY)}

\def\OXU{\scO_{X|U}}

\def\XxSY{X \times_S Y}
\def\XxTY{X \times_T Y}

\def\Pr{\bullet}
\def\PrS{\circ}
\def\reduced{essential}
\def\tTX{\tT^\times}
\def\tTP{\tT^\Pr}
\def\tTPS{\tT^\PrS}

\def\mone{\one_\tM}
\def\sone{\one_\tS}
\def\mzero{\zero_\tM}

\def\udscr{\,\underline{\phantom{w}}\,}

\def\TT{T}
\def\MS{C}

\def\MSf{{\MS(f)}}

\def\RMS{R_\MS}
\def\iMS{\MS^{-1}}

\def\chU{\widecheck{U}}

\def\RxR{R \times R}
\def\AxA{A \times A}

\def\RX{R^\times}

\def\AX{A^\times}
\def\tTX{\tT^\times}

\def\MxN{{M\times N}}
\def\RMxN{{R(\MxN)}}
\def\MoN{{M\otimes_R N}}
\def\MoRN{{M\otimes_R N}}
\def\AoRB{{A\otimes_R B}}

\def\qq{/\hskip -0.65ex/ \hskip -1pt}

\def\sm{\setminus}

\def\sA{A}

\def\sAx{\sA_x}

\def\sAptng{\sA|_\tng^\Pr}

\def\sAptngs{\sA|_\tng^\PrS}
\def\sAxptngs{\sA_x|_\tng^\PrS}
\def\sAptngsD
{\sAptngs \sm \gDivA}

\def\sAtng{\sA|_\tng}

\def\sAual{\sA|_\ual}
\def\sAptng{\sA|_\tng^\Pr}
\def\sAptngM{\sA|_\tng^\ast}
\def\sAghs{\sA|_\ghs}
\def\sB{B}
\def\sAX{\sA^\times}

\def\sC{C}

\def\httM{{\widehat{\tM}}}
\def\httG{{\widehat{\tG}}}
\def\htnu{{\widehat{\nu}}}

\def\htU{{\widehat{U}}}

\def\AAn{ A^{(n)}}
\def\Rn{ R^{(n)}}

\def\tGn{ \tG^{(n)}}

\def\tTzn{ \tTz^{(n)}}

\newcommand{\gen}[1]{\langle #1 \rangle}

\def\cmp{{\operatorname{c}}}
\def\int{{\operatorname{int}}}

\def\tlR{\widetilde R}
\def\tlF{\widetilde F}

\def\olvrp{\overline \varphi}

\def\tlvrp{\widetilde \varphi}

\def\tlmu{{\widetilde \mu}}

\def\tltau{\widetilde \tau}

\def\nufib{\operatorname{fib}_{\nu}}

\def\nuequivalent{$\nu$-equivalent}
\def\nuequivalence{$\nu$-equivalence}
\def\nufiber{$\nu$-fiber}

\def\congradical{\ccong\ radical}
\def\cradical{$\mfc$-radical}

\def\sradical{$\mfs$-radical}
\def\psradical{principal $\mfs$-radical}
\def\sradicalc{\sradical\ congruence}
\def\gprad{\tN}

\def\vrp{\varphi}
\def\ivrp{{\vrp^\inv}}
\def\iphi{{\phi^\inv}}

\def\vvrp{{\vrp^\vee}}
\def\htvrp{\widehat{\vrp}}
\def\avrp{{^a\hskip -.35ex\vrp}}
\def\atlvrp{{^a\hskip -.35ex\tlvrp}}
\def\iavrp{{(\avrp)^\inv}}

\def\apsi{{^a\hskip -.25ex\psi}}

\def\api{{^a \hskip -.25ex \pi}}
\def\atau{{^a \hskip -.25ex \tau}}
\def\btau{{_a \hskip -.25ex \tau}}
\def\iatau{(\atau)^\inv}

\def\ipi{\pi^\inv}
\def\vpi{\pi^\vee}

\def\gdiv{$\mfg$-divisor}
\def\gsdiv{ghost divisor}

\def\tngres{tangible part}
\def\tngpres{tangibly \prs\ part}

\def\tng{{\operatorname{tng}}}

\def\ual{{\operatorname{ual}}}

\def\ghs{{\operatorname{gh}}}

\def\gc{\mathsf{G}_\cls}
\def\tc{\mathsf{T}_\cls}
\def\tp{\mathsf{P}_\Pr}

\newcommand{\iPcl}{\tp^\inv}

\newcommand{\Tcl}[1]{\tc(#1)}
\newcommand{\cTcl}[1]{\tc^\cmp(#1)}
\newcommand{\iTcl}{\tc^\inv}
\newcommand{\ciTcl}{\tc^\cinv}

\newcommand{\Gcl}[1]{\gc(#1)}
\newcommand{\cGcl}[1]{\gc^\cmp(#1)}
\newcommand{\iGcl}{\gc^\inv}
\newcommand{\ciGcl}{\gc^\cinv}

\def\bfnu{\mbox{\boldmath$\nu$}}

\def\bfDD{\mbox{\boldmath$\DD$}}
\def\NSMR{\mbox{\boldmath$\operatorname{\nu Smr}$}}

\def\NMOD{\mbox{\boldmath$\operatorname{\nu Mod}$}_R}
\def\NMON{\mbox{\boldmath$\operatorname{\nu Mon}$}}
\def\Sch{\mbox{\boldmath$\operatorname{\nu Sch}$}}
\def\ASch{\mbox{\boldmath$\operatorname{\nu ASch}$}}

\def\mfA{\mathfrak A}

\def\mfC{\mathfrak C}
\def\mfD{\mathfrak D}

\def\mfF{\mathfrak F}
\def\mfG{\mathfrak G}

\def\mfI{\mathfrak I}

\def\mfK{\mathfrak K}

\def\mfM{\mathfrak M}
\def\mfN{\mathfrak N}

\def\mfP{\mathfrak P}
\def\mfO{\mathfrak O}

\def\mfR{\mathfrak R}

\def\mfT{\mathfrak T}

\def\mfa{\mathfrak a}

\def\mfc{\mathfrak c}
\def\mfd{\mathfrak d}

\def\mfg{\mathfrak g}

\def\mfi{\mathfrak i}

\def\mfm{\mathfrak m}

\def\mfp{\mathfrak p}
\def\mfo{\mathfrak o}
\def\mfq{\mathfrak q}
\def\mfr{\mathfrak r}
\def\mfs{\mathfrak s}
\def\mft{\mathfrak t}

\def\({\left(}
\def\){\right)}
\def\Z{{\mathbb Z}}
\def\Q{{\mathbb Q}}
\def\im{{\operatorname{im}\,}}
\def\nufib{{\operatorname{fib}_\nu\,}}
\def\gker{{\operatorname{gker}\,}}

\def\gcoker{{\operatorname{gcoker}\,}}
\def\tcor{{\operatorname{tcor}\,}}

\def\stcor{{\operatorname{tcor}^\PrS}}
\def\cker{{\operatorname{corn-ker}\,}}

\def\pipe{{\underset{{\ \, }}{\mid}}}

\def\nmod{\models}

\def\pipe1{{\underset{{1}}{\mid}}}
\def\lmod1{\mathrel  \pipe1  \joinrel \joinrel =}

\def\CFunFF1{\operatorname{CFun} (F,F)}

\def\semiring0{semiring$^{\dagger}$}
\def\semidomain0{semidomain$^{\dagger}$}
\def\semifield0{semifield$^{\dagger}$}
\def\semifields0{semifields$^{\dagger}$}
\def\domain0{domain}
\def\domains0{domains$^{\dagger}$}

\def\numodh{\numod\ homomorphism}
\def\nusemiring0{$\nu$-\semiring0}
\def\nusemirings0{$\nu$-\semirings0}

\def\nusemifield0{$\nu$-\semifield0}

\def\nuvar{$\nu$-variety}
\def\nuvars{$\nu$-varieties}

\def\expr{reduced}
\def\uexpr{meaningless}
\def\uexpr{redundant}

\def\tsset{tangible set}
\def\tmon{tangible monoid}

\def\tame{tame} 
\def\tamenusmr{\tame\ \nusmr} 
\def\tamesmr{\tame\ \nusmr} 

\def\tcls{tangibly closed}
\def\Tcls{$\tT$-closed}
\def\tclsnusmr{\tcls\ \nusmr}

\def\prs{persistent}
\def\prsfl{\prs\ full}
\def\prscls{\prs\ closed}

\def\tcore{tangible core}
\def\ptcore{\prs\ \tcore}

\def\tprsmon{\tprs\ monoid}
\def\tprsset{\tprs\ set}
\def\tprsset{\tprs\ set}
\def\tprs{$\mft$-\prs}
\def\tanprs{tangibly \prs}

\def\ualt{unalterable}
\def\tualt{$\mft$-\ualt}

\def\smr{semiring}
\def\smf{semifield}
\def\rsset{restricted subset}
\def\strict{strict}

\def\nualg{$\nu$-algebra}
\def\numon{$\nu$-monoid}
\def\nusmon{$\nu$-submonoid}
\def\nusmr{$\nu$-semiring}
\def\nussmr{$\nu$-subsemiring}
\def\lnusmr{local \nusmr}

\def\nusmrhom{\nusmr\ \qhom}

\def\nudom{$\nu$-domain}
\def\inudom{integral \nudom}
\def\nuddom{definite \nudom}
\def\nusmf{$\nu$-semifield}
\def\numod{$\nu$-module}
\def\nusmod{$\nu$-submodule}
\def\nutop{$\nu$-topology}

\def\ssmf{supertropical semifield}

\def\fzone{focal zone}
\def\Fzone{Focal zone}

\def\nustalk{$\nu$-stalk}

\def\nusch{$\nu$-scheme}
\def\nushf{$\nu$-sheaf}
\def\nushv{$\nu$-sheaves}

\def\nussch{$\nu$-subscheme}

\def\nushf{$\nu$-sheaf}
\def\nusmrsp{\nusmr ed space}
\def\lnusmrsp{locally \nusmrsp}
\def\ccong{congruence}
\def\gstcong{ghost congruence}

\def\lcong{$\ell$-congruence}
\def\gcong{$\mfg$-congruence}
\def\qcong{$\mfq$-congruence}

\def\cqcong{congruence}
\def\scong{$\mfi$-congruence}

\def\gprime{$\mfg$-prime}
\def\gprimec{$\mfg$-prime congruence}
\def\cprime{$\mfc$-prime}
\def\gradical{$\mfg$-radical}
\def\gradicalc{\gradical\ congruence}
\def\ggradical{ghostpotent}
\def\ggradicalc{\ggradical \ radical \ccong}
\def\gpradical{$\mfg \mfp$-radical}
\def\gpradicalc{\gpradical}

\def\maximalc{maximal \lcong}

\def\tminimal{$\mft$-minimal}
\def\tminimalc{\tminimal\ \lcong}
\def\ctminimalc{central \tminimalc}

\def\lmaximalc{maximal\ \lcong}

\def\qnoetherian{$\mfq$-noetherian}
\def\qartinian{$\mfq$-artinian}

\def\nuhom{$\nu$-homomorphism}
\def\qhom{$\mfq$-homomorphism}
\def\lqhom{local \qhom}
\def\hom{homomorphism}

\def\intr{interweaving}
\def\intrc{interweaving congruence}
\def\Intr{Interweaving}
\def\dtrm{deterministic}
\def\Dtrm{Deterministic}
\def\dtrmc{\dtrm\ \lcong}
\def\dtrmlc{\dtrm\ \lcong}
\def\dtrmqc{\dtrm\ \qcong}
\def\Dtrmc{\Dtrm\ \qcong}
\def\dfnt{definite}

\def\domains0{domains$^{\dagger}$}

\def\domains0{domains$^\dagger$}

\def\nucong{\cong_\nu}

\def\nug{>_\nu}
\def\nul{<_\nu}
\def\mug{>_\mu}

\def\ker{\operatorname{ker}}
\def\cker{\ker_{\operatorname{c}}}

\def\nuge{\ge_\nu}
\def\nule{\le_\nu}
\def\nule{\le_\nu}

\def\pipe{{\underset{{\tG}}{\mid}}}

\def\lmod{\mathrel  \pipe \joinrel \joinrel =}
\def\pipe{{\underset{{\tG}}{\mid}}}

\def\ghost{\operatorname{ghost}}

\def\sig{\sigma}
\def\gpt{q}

\def\pSkip{\vskip 1.5mm \noindent}
\def\sSkip{$ $ \vskip 1.5mm}

\def\inv{{\operatorname{-1}}}
\def\cinv{{\operatorname{-c}}}
\def\Ann{{\operatorname{Ann}}}
\def\Eql{{\operatorname{Eq}}}

\def\srad{\operatorname{rad_s}}
\def\crad{\operatorname{rad_c}}
\def\grad{\operatorname{rad_g}}
\def\rcl{\operatorname{rcl}}
\def\RSet{\operatorname{RSet}}
\def\jac{\operatorname{jac}}

\def\vrp{\varphi}

\def\diag{\Delta}
\def\idiag{\diag^{\inv}}

\newtheorem{thm}[theorem]{Theorem}

\newtheorem*{thm*}{Theorem}

\newtheorem{cor}[theorem]{Corollary}

\def\Ann{\operatorname{Ann}}
\def\Hom{{\operatorname{Hom}}}
\def\Ten{\operatorname{Ten}}
\def\ext{\operatorname{Ext}}

\def\gDiv{\operatorname{gdiv}}

\def\gDivA{\gDiv(\sA)}

\def\R{\Real}

\def\La{\Lambda}
\def\tT{\mathcal T}
\def\Fun{\operatorname{Fun}
}
\def\tTz{\tT_\zero}

\numberwithin{equation}{section}

\def\M0{M_{\zero}}

\def\id{\operatorname{id}}

\def\SR{R}

\def\tGz{\mathcal G}

\def\PS{P}

\def\iSet{S}

\def\rzero{\zero_\SR}

\def\sone{\one_\iSet}
\def\rone{{\one_\SR}}

\def\rzero{\zero_\SR}

\def\semirings0{semirings}

\newcommand{\nPS}[1]{\PS_{(!#1)}}
\newcommand{\nPSo}[1]{\nPS{\one}}

\newcommand{\dss}[1]{\quad {#1} \quad }

\def\srHom{\varphi}

\def\tlf{\widetilde f}
\def\tlg{\widetilde g}

\def\olf{\overline{f}}

\def\bfa{ \textbf{a}}
\def\bfi{ \textbf{i}}

\def\Pol{\operatorname{Pol}}

\def\stak(f){\got{P}}
\def\stak{\got{S}}

\def\cls{{\operatorname{cls}}}

\def\Rees{{\operatorname{Rees}}}

\def\bfa{ \textbf{a}}
\def\bfb{ \textbf{b}}

\def\bfi{ \textbf{i}}

\def\bfx{ \textbf{x}}

\def\brf{ {\bar{f}}}

\def\brY{ {\overline{Y}}}

\def\cng{\equiv}

\def\ccng{\equiv_\mfc}
\def\pcng{\equiv_\mfp}

\def\gcng{\equiv_\mfg}

\def\scng{\equiv_\mfi}
\def\tsim{\sim_\mft}

\def\rcng{\equiv_\mfr}

\def\dcng{\equiv_\mfd}
\def\zcng{\equiv_\mfo}

\def\Cong{\mfA}
\def\olCong{\overline{\Cong}}

\def\bCong{\Cong_b}

\def\pCong{{\mfP}}
\def\gCong{{\mfG}}

\def\tCong{{\mfT}}
\def\sCong{{\mfI}}

\def\xpCong{\pCong_x}
\def\ypCong{\pCong_y}

\def\pypCong{\pCong_{y'}}

\def\cCong{{\mfC}}
\def\mCong{{\mfM}}
\def\nCong{{\mfN}}

\def\rCong{{\mfR}}
\def\dCong{{\mfD}}
\def\zCong{{\mfO}}

\def\xnCong{\nCong_x}

\def\Cng{\operatorname{Cong}}
\def\CngA{\Cng(A)}

\def\GCng{\Cng_\mfg}

\def\LCng{\Cng_\ell}
\def\SCng{\Cng_\mfi}
\def\QCng{\Cng_\mfq}

\def\olCong{\overline{\Cong}}

\def\cmp{{\operatorname{c}}}
\def\int{{\operatorname{int}}}

\def\fcong{\cong}

\def\fCong{\mfF}

\def\pcspec{\Spec}

\makeindex
\begin{document}


\title[Supertropical Algebraic Geometry]
{  Commutative $\bfnu$-Algebra
\\ and \\
Supertropical Algebraic Geometry
 }

\author[Z. Izhakian]{Zur Izhakian $^\dag$}
\address{Institute  of Mathematics,
 University of Aberdeen, AB24 3UE,
Aberdeen,  UK. }
\email{zzur@abdn.ac.uk}


\date{October 17, 2018}


%
%
\thanks{\pSkip $^\dag$ Institute  of Mathematics,
 University of Aberdeen, AB24 3UE,
Aberdeen,  UK. \\ $ $ \ Email: zzur@abdn.ac.uk.}


\begin{abstract}
This paper lays out a foundation for a theory of supertropical algebraic geometry, relying on  commutative \nualg.
To this end, the paper introduces \qcong s, carried over \nusmr s, whose distinguished ghost and tangible  clusters  allow both quotienting and localization. Utilizing  these clusters,  \gprime, \gradical, and maximal  \qcong s are naturally  defined, satisfying   the classical   relations among analogous ideals. Thus, a foundation of systematic  theory of  commutative \nualg\ is laid.
In this framework, the underlying spaces  for a theoretic construction of schemes  are spectra of \gprimec s, over which the correspondences between  \qcong s and varieties emerge directly. Thereby,
scheme theory within supertropical algebraic geometry follows the Grothendieck approach,  and is applicable
to  polyhedral geometry.
\end{abstract}

\maketitle

\setcounter{tocdepth}{2}
{\small \tableofcontents}

\section{Introduction}
\numberwithin{equation}{section}

The evolution of supertropical mathematics having  been initiated in \cite{zur05TropicalAlgebra} by introducing a semiring structure whose arithmetic intrinsically formulates combinatorial properties that address  the lack of subtraction in semirings.
This mathematics is carried over  \nusmr s whose  structure permits a systematic development of algebraic theory, analogous  to the  theory over rings \cite{IzhakianRowen2007SuperTropical},  in which  fundamental notions can be  interpreted  combinatorially \cite{IzhakianRowen2008Matrices,IzhakianRowen2009Equations,IzhakianRowen2010MatricesIII,IzhakianRowen2008Resultants}.

   The introduction of supertropical theory was  motivated by the aim of capturing  tropical varieties in a purely algebraic sense by extending  the max-plus semiring (cf. Example \ref{exmp:extended}).   The ultimate goal has been to establish a profound theory of polyhedral algebraic geometry in the spirit of  A. Grothendieck, whose foundations are built upon  commutative algebra. The present paper
  introduces an algebraic framework for such a theory,
   evolving further supertropical mathematics.

The aspiration of this theory has been  to provide an intuitive algebraic language, clean and closer as passible to classical theory, but at the same time  abstract enough to frame objects having a discrete nature. In this theory, the mathematical formalism involves no complicated combinatorial formulas  and  enables a direct implementation of familiar algebro-geometric approaches.
The conceptual ideas and main principles of the theory are summarized  below.

\subsection{Supertropical structures}
\sSkip \noindent
The underlaying additive structure of a \nualg\  is a \textbf{\numon} $(\tM, \tG,  \nu )$ -- a monoid~ $\tM$ with a distinguished (partially ordered) \textbf{ghost  submonoid} $\tG$ and a projection $\nu: \tM \To \tG$ -- which satisfies the key property
$$ a + b = \nu(a) \qquad \text{whenever } \quad \nu(a) = \nu(b).$$
In particular, $a + a = \nu(a)$ for any $a \in \tM$. Thus,    an   element $\nu(a)$ of $\tG$ can be thought of as carrying an ``additive multiplicity'' higher than 1, i.e., a sum $a+ \cdots + a$ with $a $ repeated at least twice.  When the monoid operation $+$ is  maximum, this arithmetic intrinsically formulates the common tropical phrase  ``the maximum is attained at least twice'', replacing  the ``vanishing condition'' in classical theory.

   Equipping an additive  \numon\ with a multiplicative operation that respects the relevant axioms, a \textbf{\nusmr} $(R, \tT, \tG, \nu)$ is obtained.  In \nusmr s the ghost submonoid $\tG$ becomes an ideal -- a ``ghost absorbing'' submonoid -- which plays the customary role of  zero element in commutative algebra.  In this view,   a traditional equation $\lozenge \cdots \lozenge = 0$  is   reformulated
as $\lozenge \cdots \lozenge \in \tG$, replacing the familiar   ``vanishing condition'' (which is often  meaningless for  semiring structures) by possessing  ghost.
 On the other edge,  to obtain a suitable multiplicative substructure, the    subset ~$\tT$ of \textbf{tangible elements} in the complement of ~$\tG$ is distinguished, satisfying  the conditions:
 if $a + b \in \tT$ then   $a,b \notin \tG$. As well,   $\tT$ contains the well behaved subset~ $\tTPS$ of \textbf{\tprs\ elements}, closed for taking powers.

Together, the ghost ideal $\tG$ and the \tsset\ $\tT$ of a \nusmr\
 enable a reformulation of basic algebraic notions. For instance, a \textbf{\qhom }  is a  homomorphism of \nusmr s that preserves the components $\tT$ and $\tG$.
The \textbf{tangible core} and the \textbf{ghost kernel} are,  respectively,  the preimages of ~$\tT$ and $\tG$  which characterize   \qhom s.
 With these objects in place, the category of \nusmr s whose morphisms are \qhom s  is
 established.

\nusmr s generalize the max-plus semiring $(\Real, \max, +)$ and the  boolean algebra $(\bool, \vee, \wedge)$, enriching them with extra algebraic properties, as detailed in Examples 
\ref{exmp:extended} and \ref{exmp:superboolean}. Furthermore, any  ordered monoid $(\tM, \cdot \,)$ gives rise to a semiring structure by setting its addition to be maximum, and therefore, supertropical theory is carried over transparently  to ordered monoids.
Former works have frequently assumed   multiplicative
 cancellativity  (i.e., $ca = cb \Rightarrow a =b$ for any~ $a,b,c$), to compensate  the luck of inverses in semirings. This condition  is too restrictive   for   \nusmr s  (see Example ~ \ref{examp:cancellative.gdiv}). Therefore,  we avoid  any cancellativity conditions. As a consequent, pathological elements, as ghost divisors and ghostpotents,  emerge in this setting and are treated by the use of  congruences.

\subsection{Congruences versus ideals}
\sSkip \noindent
Quotienting and localization are central notions in algebra.
In commutative ring  theory and in classical algebraic geometry these notions are delivered by  ideals.
A ring ideal defines an equivalence relation, and thus a quotient, while the complement of a prime ideal is a multiplicative system, used for localizing.
Since substraction is absent in semirings, ideals do not determine equivalence relations and, therefore, are not applicable for quotienting. Consequently,  one has to work directly with congruences, i.e., equivalence relations that respect the semiring operations. However,  by itself, this approach does not address  localization.  But, with extra properties, quotienting and localization are accessible via congruences on \nusmr s.

To introduce \textbf{clustering} on congruences -- a coarser decomposition of classes -- we enhance  the classification of  elements as tangibles or ghosts to equivalence classes.  In particular, an equivalence class ~$[a]$ is  tangible if it consists only of tangibles, $[a]$ is  ghost if it contains some ghost, and $[a]$ is neither tangible  nor ghost otherwise.  Thus,  a congruence $\Cong$ on a \nusmr\ ~ $R$ is endowed with two disjoint clusters, consisting  of  equivalence  classes:
\begin{description}\dispace
  \item[$\maltese$] \textbf{tangible cluster} whose classes  are preserved as tangibles in $R/\Cong$,
  \item[$\maltese$] \textbf{ghost cluster}   whose classes are identified as ghosts in $R/\Cong$.
\end{description}
The former serves for localization and the latter serves for quotienting.
These clusters are not necessarily the complement of each other; so one has to cope with an extra degree of freedom. This divergency is addressed  by the tangible and ghost projections of $\Cong$ on~ $R$, which are respectively determined by the diagonals of classes within the  clusters of $\Cong$.

A \textbf{\qcong} $\Cong$ on a \nusmr \ $R$ is a congruence whose tangible projection contains the group of units of $R$, and thus a submonoid of \tprs\ (tangible) elements.  In the special case, when the  tangible projection is a monoid by itself,   $\Cong$ is an \textbf{\lcong}. In our theory,
  \qcong s are elementary entity, providing the building stones for commutative algebra. Quotienting by a \qcong\ is done in the standard way, while the monoid structure of tangible projections of \lcong s allows  for  executing {localization}.   The tangible projection and the ghost cluster enable the utilization of  familiar methods in commutative algebra.
  Moreover,  \qcong s preserve the \nusmr\ structure in the transition to  quotient structures and, at the same time, perfectly coincide with \qhom s. Therefore, concerning  \nusmr s,  \qcong s play the traditional role of ideals.

 \qcong s support executing ``\textbf{ghostification}'' of elements -- a redeclaration of elements as ghosts in quotient structures. This redeclaration is the supertropical analogy to quotienting by an ideal, whose elements are ``identified'' with zero. Formally, an element $a \in R $ is ghostified by a congruence $\Cong$, if~ $\Cong$ includes the equivalence  $a \cng \nu(a)$. This means that, despite the absence of additive inverses, we are capable  to  quotient out  a \nusmr\ by its substructure (even by a subset) in a meaningful sense. Hence,   cokernels of \qhom s can be defined naturally.

\subsection{Congruences and spectra}
\sSkip
\noindent
Classical theory employs spectra of prime ideals, endowed with Zariski topology,  as  underlying topological spaces.
In  supertropical theory,  ideals are replaced by  \qcong s, whose special structure permits the key definitions:
\begin{itemize}\dispace
  \item[$\maltese$] a \qcong\ is \textbf{\gradical}  if  $a^k \rcng \nu(a^k)$ implies   $a \rcng \nu(a)$,
  \item[$\maltese$] an \lcong\ is \textbf{\gprime}  if $ab \pcng \nu(ab)$   implies $a \pcng \nu(a)$ or $b \pcng \nu(b)$.
\end{itemize}
Note that the conditions in these definitions are determined solely by equivalence to ghosts, while  localization by \gprimec s is performed through  their tangible projections.
This setup enables to formulate many tropical relations, analogous to well known relations among radical, prime, and maximal ideals. With these relations,  manifest  over \qcong s,  the notions of \textbf{noetherian \nusmr s} and   \textbf{Krull dimension} are obtained.

To cope simultaneously  with congruences  and ``ghost absorbing'' subsets of \nusmr s, we study several types of radicals, which are shown to coincide. Moreover, they provide a version of abstract Nullstellensatz (Theorem \ref{thm:Null}). In comparison to ring ideals, the hierarchy of \qcong s contains unique types of congruences, including  \dtrmc\  and \intrc s.
However, maximal congruences are not applicable for defining locality, since they do not necessarily coincide with maximality of clusters. Instead,  \textbf{\tminimalc s},   determined  by  maximality of non-tangible projections, are used to define locality.  A \nusmr\ $R$ is   \textbf{local}, if all \tminimalc s on ~$R$ share the same tangible projection.

The  well behaved interplay between localization and  various types of \qcong s allows for introducing  the  \textbf{spectrum of \gprimec s}, endowed  with a Zariski type  topology. With this setting, our study   follows the standard methodology  of exploring the correspondences among  closed sets,  open sets, \qcong s, and  \nusmr \ structures, resulting eventually in a construction of sheaves.

\subsection{Sheaves and schemes}
\sSkip
\noindent
Upon the  spectrum $X= \SpecA$ of all \gprimec s on a \nusmr\ $\sA$, realized as a topological space,
a closed set $\VV(f)$ in $X$ is the set of all \gprimec s $\pCong$ that ghostify a given $f \in \sA$.  That is, ~$f$ belongs to the ghost projection of $\pCong$, i.e., $f \pcng \nu(f)$ in ~$\pCong$. Therefore, since $f \pcng \nu(f)$ for any ghost~ $f$,   we focus on non-ghost elements $f \in \sA$ to avoid triviality.  Yet,  this focusing does not imply that a non-ghost element $f$ belongs to the tangible projection of all \gprimec s composing the complement $\DD(f)$ of $\VV(f)$ in $X$. Thus, $f$ does not necessity possess tangible values over the entire open set $\DD(f)$, in particular when $f$ is not tangible.  This bearing requires a special care in the construction of sheaves, as explained below.

A \textbf{variety} in $X$ is a closed set which can be determined either as the set of \gprimec s containing a  \qcong\ $\Cong$ on $\sA$, or as the set of \gprimec s that ghostify  a subset $E$ of~ $\sA$. These equivalent definitions establish the correspondences between properties of \qcong s  on $\sA$ and subsets of ~$X$. For example, the one-to-one correspondence between \gprimec s and \textbf{irreducible varieties} (Theorem \ref{thm:16}). In this setting,  closed immersions $\Spec(\sA/\Cong) \Isoto \VV(\Cong)$ of spectra appear naturally (Corollary~ \ref{cor:irreducible}). Open immersions are more subtle and require a further adaptation, that is, a restriction $\DD(\MS,f)$ of $\DDf$ to  \gprimec s whose tangible projection contains a given tangible monoid $\MS$ of  $\sA$.  This restriction yields the bijection $\Spec(\sA_\MS) \Isoto \bigcap_{f \in \MS} \DD(\MS,f)$.

To preserve the \nusmr\ structure,
localization  is feasible only by \tprs\ elements. Nevertheless, to construct a sheaf, each open set $\DD(f)$ has to be assigned with a \nusmr, allocated by the means of localization, even when $f$ is not tangible. To overcome this drawback, we consider the submonoid  $\STf$ of particular \tprs\ elements $h$ such that $\DD(f)\subseteq \DD(h)$, which delivers  the well defined map $\DD(f) \mTo \sA_\STf$ for every  $f \in \sA$.
To ensure that this map coincides with the map $\pCong \mTo \sA_\pCong$, sending a point $\pCong$ in $\DDf$ to the localization of $\sA$ by $\pCong$, sections over $\DDf$ are customized to \gprimec s in $\DDf$ whose tangible projection contains the submonoid  $\STf$. The subset of these \gprimec s  assembles   the \textbf{\fzone} $\fcDDf$ of the  open set~ $\DDf$.
Accordingly, \nustalk s are determined as inductive limits taken with respect to \fzone s determined  by \tprs \ elements,
whereas the naive definition of morphism applies and respects
\fzone s (Lemma ~ \ref{lem:fc.map}).
The building of structure \nushv\ and \lnusmrsp s is then  standard, admitting  functoriality as well. This construction provides a scheme structure $\XOX$ of  \nusmr s, called \textbf{\nusch}.

While most of our  theory follows classical scheme theory,
a \nusmr\ $\OX(\DDf)$ of sections in  an affine
\nusch\ $\XOX$, with $X = \SpecA$,  might not be isomorphic to the localized  \nusmr\ $\sA_\STf$, since sections are specialized to
\fzone s. However, this excludes \textbf{\strict\ elements} $f \in \sA$  (and in particular units) for which the equality $\fcDDf = \DDf$ holds. Therefore, in the theory of \nusch s, global sections $\Gm(X,\OX)$ are isomorphic to the underlying  \nusmr\ $\sA$ (Theorem \ref{prop:stalks}). Moreover,  there is  a one-to-one correspondence between \qhom s of \nusmr s and
morphisms of \nusch s (Theorem~ \ref{prop:6.6.9}).
Fiber products of \nusch s also exist (Proposition ~\ref{lem:g.5.4.7}).

The advantage of this framework appears in the analysis of the local \nusmr \ at a point, which includes the  notions of \textbf{tangent space}, \textbf{local dimension}, and \textbf{singularity}.  These notions are not so evident in standard tropical geometry,  since ideals are not well applicable in semirings.

\subsection{Varieties towards polyhedral geometry }
\sSkip
\noindent
Traditional tropical geometry is a  geometry over the max-plus semiring $(\Real, \max, + )$.  Features of this geometry  are balanced  polyhedral complexes of pure dimension \cite{IMS,MIK1,MS}, called tropical varieties.  These varieties are determined as the so-called ``corner loci'' of tropical polynomials and are obtained as valuation images of toric varieties, linking tropical geometry to classical theory.
Although $(\Real, \max, + )$ suffices for describing tropical varieties, from the perspective of polyhedral geometry this family is rather restrictive; for example, it does not include  ploytopes.

Supertropical algebraic sets are defined directly as ghost loci of polynomial equations \cite{IzhakianRowen2007SuperTropical,IzhakianRowen2008Resultants,IzhakianRowenIdeals};  that is,  $\tZ(f) = \{ \bfx \in \sA^{(n)} \cnd f(\bfx) \in \tG \}$. Thereby,  tropical varieties are captured as algebraic sets of tangible polynomials over  the supertropical extension of $(\Real, \max, +)$, cf. Example \ref{exmp:extended}.  This approach does not rely on the so-called ``tropicalzation'' of toric  varieties and dismisses the use of the balancing condition ~ \cite{Coordinates}. Furthermore, it yields formulation of additional polyhedral features, previously inaccessible by tropical
geometry, such as polytopes, and more generally subvarieties of the same dimension as their ambient variety (Example~ \ref{examp:varieties}).

Supertropical structures provide a sufficiently general framework to deal with finite and infinite underlying semirings, as well as with bounded semirings. Our theory allows for approaching  convex geometry and discrete geometry, utilizing similar principles as in this paper. We leave the study of these geometries for future work.

\subsection{A brief overview of   related  theories } \sSkip
\noindent
Tropical semirings $(R, \max, +)$, where  $R = \bool, \Net, \Int$,  are linked to number theory \cite{CC1,CC2,CC3} and arithmetic geometry via the  Banach
semifield theory  of characteristic one \cite{Lic}. Features of traditional tropical geometry  are received as the Euclidean closures of  ``tropicalization'' of subvarieties of a torus $(\Fld^\times)^n$, where $\Fld$ is
a non-archimedean algebraically closed valued field, complete with respect to the valuation \cite{Gat,IMS,MS}.  A generalization to subvarieties of a toric variety is given by Payne \cite{Pay09}, using stratification by torus orbits. With this geometry,
   a translation of algebro-geometric questions into combinatorial problems is  obtained,  where  varieties are replaced by polyhedral complexes,  helping to solve problems in enumerative geometry~ \cite{MIK1}.
   The  translations of classical problems into combinatorial framework  have been motivating the development  of profound  algebraic foundations for geometry over semirings.

Over the past decade, many  works have dealt with scheme theory  over semirings and monoids, aiming mainly to develop   a characteristic one arithmetic geometry over the field $\bbF_1$ with one element, e.g., Connes-Consani \cite{CC1}, Deitmar \cite{Dei08}, Durov \cite{Dur07}, To\"{e}n-Vaqui\'{e} \cite{TV}.
See  \cite{LL} for survey.
 Berkovich  uses abstract skeletons  \cite{Ber}, while   Lorscheid \cite{Lor1,Lor} uses   blueprints.
 Giansiracusa-Giansiracusa \cite{GG} and Maclagan-Rin\'{c}on \cite{MR} specialize  $\bbF_1$   to propose tropical schemes, subject to ``bend relations''.
 The  point of departure of these works is an underlying spectrum whose atoms are prime ideals.
Alternatively, other works employ congruences as atoms of spectra. However, the use of congruences raises  the issue of defining primeness that exhibits the desired attributes.
This approach is taken by Jo\'{o}-Mincheva  \cite{JoM} and Rowen \cite{RowAN}, who use twisted products,  by Bertram-Easton ~ \cite{BE13}, and by Lescot \cite{Les1,Les2,Les3}. Primeness in \cite{Les1,Les2,Les3}  depends only on equivalence to zero, ignoring other relations that are determined by a congruence.
 In~ \cite{CHWW}, primes are defined in terms of cancellative quotient monoids.

Although the above approaches address the needs of   other theories, for
polyhedral algebraic geometry they seem to  have their own deficiencies.  Some approaches are too restrictive to capture the wide range of polyhedral objects, while others are  much too abstract and difficult to be implemented.
The concept in the current paper is different in the sense that it uses congruences,  enriched by
 supertropical attributes,  as atoms of spectra  that  fulfill the required  correspondences between topological spaces and algebras. The algebraic theory then follows the classical commutative algebra, allowing for the development of scheme theory within algebraic geometry along the track of  Grothendieck.
This concept provides a clear-cut  algebraic framework for a straightforward
study of polyhedral varieties and schemes, in particular tropical varieties,   without the need to constantly referring back to their classical valuation  preimages.

\subsection{Paper outline}
\sSkip
\noindent
 For the reader's convenience, the paper is designed as stand-alone.
  We include all the relevant definitions, detailed proofs, and basic examples. We follow standard structural theory of scheme within algebraic geometry, involving  a categorical viewpoint, e.g., \cite{Bosch,Eisenbud,Gortz,HAR77}. The paper consists of two parts.  The first part (Sections \ref{sec:1}--\ref{sec:4}) is devoted to   commutative \nualg\ and  the second part (Sections \ref{sec:5}--\ref{sec:7}) develops  scheme theory over \nusmr s.

  Section~ \ref{sec:1} opens by recalling the setup of known algebraic structures, to be used in the paper.
Section~\ref{sec:2} provides the definitions of various $\nu$-structures,
as well as their connections to familiar semiring structures.
Section~\ref{sec:3}  intensively   studies  congruences on  \nusmr s,  including the introduction of
\qcong s and \lcong s.
Section~\ref{sec:4} introduces  \numod s and tensor products.
Section~\ref{sec:5} discusses varieties  through the correspondences of their characteristic properties to \qcong s. Section~\ref{sec:6} constructs \nushv, combining categorical and algebraic perspectives. Finally, Section~\ref{sec:7} develops  the notions of  \nusch s and locally \nusmrsp s.

\bigskip
\part*{Part I: Commutative $\bfnu$-Algebra}
\bigskip

\section{Algebraic structures}\label{sec:1}
In this section we recall  definitions of traditional  algebraic structures and their morphisms, to be used  in this paper.
As customary, $\Net$ denotes the positive natural numbers, while $\NetZ$ stands for  $\Net \cup \{ 0 \}$;  $\Q$ and~ $\R$ denote respectively   the rational and real numbers. We write $\AAn$ for the cartesian product $A \times \cdots \times A$ of a structure $A$, with $A$ repeated~ $n$ times.
\pSkip
\textbf{\emph{In this paper addition and multiplication are always assumed to be associative operations.}}

\subsection{Congruences in universal algebra}\label{ssec:cong.univ.alg}  \sSkip

 A~\textbf{congruence} \index{congruence} $\Cong $ on
an   algebra $A$ -- a carrier algebra --  is an equivalence relation $\equiv$ that preserves
 all the relevant operations and relations of $A$. That is,   if
$a _i \equiv b_i$, $i = 1,2$,  then
\begin{enumerate}\eroman
\item
 $a_1+a_2 \equiv b_1+b_2$,
\item $a_1a_2 \equiv
b_1b_2$.\end{enumerate}
Note that to prove (ii) it is enough to show that $a_1a \equiv b_1a$ and $a a_1 \equiv a b_1$ for all $a \in A$, since then
$$a_1b_1  \equiv a_2b_1  \equiv a_2b_2  .$$
We call  $\equiv$ the \textbf{underlying  equivalence} of $\Cong  $ and write $[a]$ for the  equivalence class of an element $a \in A$. (These are called the \textbf{homomorphic
equivalences} in~\cite{IKR2}.)  We write $A/\Cong$ for the factor algebra, whose elements are equivalence classes $[a]$  determined by $\Cong$, with operations
\begin{equation}\label{eq:cong.oper}
[a] \cdot [b] = [ab], \qquad  [a]  + [b ]  = [a+ b],
\end{equation}
for $a,b \in A$.

A congruence $\Cong$ on algebra $A$ may be
viewed either as a subalgebra of $A \times A$ containing the
diagonal and satisfying two additional conditions corresponding to
symmetry and transitivity (in which case ~$\Cong$ is  described as
the appropriate set of ordered pairs), or otherwise $\Cong$ may be
viewed as an equivalence relation $\equiv$ satisfying certain
algebraic conditions. To make the exposition clearer we utilize both views, relying on the
context.

We denote the set of all congruences on an algebra $A$ by
$$\Cng(A) := \{ \ \Cong \text{ is a congruence on } A \ \}, $$
which is closed for intersection.
We write $\Cong_1 \subseteq \Cong_2$ if $a \equiv_1 b$ implies  $a \equiv_2 b$, in other terms
$(a,b) \in \Cong_1$ implies $(a,b) \in \Cong_2$. Thus,  $\Cng(A)$ is endowed with a partial order $\leq$ induced from  $\subseteq$, i.e.,
$\Cong_1 \leq \Cong_2$ if and only if  $\Cong_1 \subseteq \Cong_2$.

The \textbf{diagonal} $\diag(A)$ of $A \times A$ is a congruence by itself, contained  in any  congruence~ $\Cong$ on $A$, and is minimal with respect to inclusion.  It  provides the bijection   \begin{equation}\label{eq:con.diag.1}
\iota: A \ISOTO \diag(A) \ds\subseteq \Cong,\qquad a \mTo (a,a),\end{equation}
with the   inverse map  $\idiag((a,a)) = a$ for any $(a,a) \in \diag(A).$ On the other hand, $\diag(A)$ is also obtained as
\begin{equation}\label{eq:con.diag.2}
\diag(A) =  \bigcap_{\Cong \in \Cng(A)} \Cong ,
\end{equation}
since an intersection of congruences is again a congruence.

\begin{rem}\label{rem:cong.gen} For  any collection $\{ (a_i, b_i) \}_{i\in I}$ of pairs $(a_i, b_i) \in A \times A$, there is a unique minimal congruence in which  $a_i \cng b_i$ for every $i \in I$. It  is termed the congruence \textbf{determined}
by the pairs $(a_i , b_i )$.  This congruence  can be obtained in two equivalent ways, either by taking the transitive closure of all the equivalences $a_i \cng b_i$ with respect to the operations and relations of $A$, or by considering  the intersection of all congruences on $A$ that include   the equivalences  $a_i \cng b_i$ for all  $i \in I$.
\end{rem}
In this paper we follow the latter approach, as it better fits our purpose, allowing a direct restriction to congruences subject to particular properties.

\begin{defn}\label{def:cong.types} A congruence $\Cong$ on an algebra $A$ is said to be
\begin{enumerate} \eroman
\item \textbf{proper}, if it has more than one equivalence class;
\item  \textbf{trivial}, if each of its equivalence classes is a
singleton (i.e., $\Cong = \diag(A)$);
  \item  \textbf{maximal},  if there is no other congruence properly containing $\Cong$;

  \item  \textbf{irreducible}, if it cannot be written as the intersection of two congruences properly containing~$\Cong$;

  \item   \textbf{cancellative}, if $ca \equiv cb$ implies $ a
  \equiv b$.


\end{enumerate}
An element  $a \in A $ is called \textbf{\trivial} with respect to
$\Cong$ if it is congruent only to itself.
 \end{defn}
\noindent In this context, an algebra $A$ is called \textbf{simple}, if the trivial congruence is  the only proper congruence on~$A$.
One may  also consider a specialization of the above congruences to a restricted family of congruences on~ $A$, determined by particular attributes. For example, maximality with respect to a given property.

\begin{rem}\label{rem:cong1} We recall some key
results from ~\cite[\S 2]{Jac2}.

 \begin{enumerate} \eroman  \item Given a congruence  $\Cong$ on an algebra
  $A$, one can endow the set $$A/ \Cong := \big \{ [a] \cnd  a \in A \big \}$$ of equivalence
 classes  with the same (well-defined) algebraic structure via \eqref{eq:cong.oper}.  The  \textbf{canonical surjective  homomorphism}
 \begin{equation}\label{eq:canonical.qu}
 \pi_\Cong: A\ONTO A/\Cong, \qquad  a \mTo [a],
\end{equation}
 is defined trivially.

\item  In the opposite direction, for any homomorphism
$\vrp:A_1 \To A_2$  one can obtain a congruence $\Cong_\vrp $ on $A_1$,
defined by  $$a \equiv_\vrp b \dss{\text{ iff }} \vrp(a) = \vrp
(b).$$
 $\Cong_\vrp$ is termed the \textbf{congruence-kernel} (written \textbf{c-kernel}) of $\vrp$, and is also denoted  by $\cker(\vrp)$.
Then, $\vrp$ induces a one-to-one homomorphism $\tlvrp
:A_1/\Cong_\vrp   \To A_2,$ via $\tlvrp ([a]) = \vrp(a)$, where
$\vrp$ factors through
$$A_1 \ONTO A_1/\Cong _\vrp \ISOTO A_2,$$ as indicated in
\cite[p.~62]{Jac2}. Therefore, the homomorphic images of $A_1$ correspond to
the congruences defined on~$A_1$. \pSkip

\item A  homomorphism $\vrp:A_1 \To A_2$ induces a map of
congruences $\vvrp: \Cong_2 \mTo \Cong_1$, sending a congruence
$\Cong_2$ on $A_2$ to the congruence $\Cong_1 = \Cong_{\pi_{2} \circ \vrp}$ on $A_1$, where
$ \pi_{2} : A_2 \Onto A_2/\Cong_2$.
That is $$a \cng_1 b \dss{\text{ iff }}  \vrp(a) \cng_2 \vrp(b).$$
When $\vrp = \pi_1 : A_1 \To  A_2 = A_1/\Cong_1$ is the canonical surjection \eqref{eq:canonical.qu}, we denote the map $\ipi_2$ also  by~ $\vpi_2$.

\end{enumerate}
 \end{rem}

Given  a homomorphism  $\vrp : A_1 \To A_2$,  there is the induced  map
\begin{equation*}
  \htvrp: A_1\times A_1 \TO A_2 \times A_2, \qquad (a,b) \Mto (\vrp(a), \vrp(b))
\end{equation*} that restricts to a map  of  congruences $\Cong_1 \subset A_1 \times A_1$.
Note that  $\htvrp(\Cong_1)$ need not be a congruence on $A_2$, as transitivity and reflexivity may fail, but it induces a congruence-inclusion relation:
\begin{equation}\label{eq:cong.inclus.map}
  \htvrp(\Cong_1) \subset \Cong_2 \dss{\text{ iff }}  (\vrp(a), \vrp(b)) \in \Cong_2 \ \text{ for all } \ (a,b) \in \Cong_1.
\end{equation}

To approach restrictive relations on subsets  of $A$ we  use the following terminology.
\begin{defn}\label{def:partial.cong}  The \textbf{restriction} of a congruence $\Cong$ on $A$ to a subset $B$ of  $A$ is
$$ \Cong|_B := \{ (a,b) \in \Cong \ds | a,b \in B\}.$$
A \textbf{partial congruence} on $A$ is a congruence on a subalgebra  $B$  of  $A$.
\end{defn}
Set theoretically, a congruence $\Cong'$ on a subalgebra  $B \subset A$ is a subset of $B \times B \subset \AxA$ containing the diagonal of $B$, but it is not a subalgebra of $\AxA$ that contains the
diagonal of $A$. For this reason we call it ``partial''. However, any partial congruence extends naturally  to a congruence on the whole  $A$ by
joining  the diagonal of $A \sm B$, subject to the transitive closure of $\Cong'$ over  $A$. A restriction $\Cong|_B$ of $\Cong$ is a congruence on $B$,  if $B$ is a subalgebra of $A$.

\begin{defn}  Let  $\Cong_1$ and $\Cong_2$ be  congruences on  algebras $A_1$ and $A_2$, respectively.
A \textbf{homomorphism  of congruences} $\Psi: \Cong_1 \To \Cong_2$
  is a \hom\ $\Psi = (\psi, \psi)$, where $\psi: A_1 \To A_2$ is a \hom\ for which
  $$ a \equiv_1 b \Dir  \psi(a) \equiv_2 \psi(b),$$
  i.e., $\Psi((a,b)) = (\psi(a), \psi(b)) \in \Cong_2$ for all $(a,b) \in \Cong_1$.
\end{defn}

We recall some standard  relations on congruences.
The product of two  congruences $\Cong_1, \Cong_2 \in \Cng(A)$ is defined  as
$$ \Cong_1 \Cong_2 := \{ (ac,bd) \cnd (a,b) \in \Cong_1, (c,d) \in \Cong_2 \},$$
for which the relation $\Cong_1  \Cong_2 \subseteq  \Cong_1 \cap \Cong_2$ holds. Then, for a given congruence $\Cong$ on $A$,  we have the chain
\begin{equation}\label{eq:cong.chain}
  \Cong \ds\supseteq \Cong^2 \ds \supseteq \cdots  \ds \supseteq \Cong^n \ds \supseteq \cdots \quad,
\end{equation}
which induces quotients on $\Cong$.

\begin{defn}\label{def:cong.fractor} The quotient $\Cong_1 / \Cong_2$ of congruences $\Cong_1 \subset \Cong_2$ on $A$ is defined
as
$$ \Cong_1 / \Cong_2 : = \{ ([a]_{2}, [b]_{2}) \cnd (a,b) \in \Cong_1 \},$$
for which there exists the  homomorphism
$$ \Psi: \Cong_1 \TO \Cong_1/\Cong_2   $$
of congruences.
  \end{defn}
Composing this definition with \eqref{eq:cong.chain}, one obtains a sequence of congruence homomorphisms
\begin{equation}\label{eq:cong.chain.2}
  \Cong \TO  \Cong / \Cong^2 \TO  \cdots  \TO \Cong^{n-1}/ \Cong^n  \TO \cdots .
\end{equation}
Later, we mainly refer to the initial step of this sequence.
\begin{rem}\label{rem:2cong} If $\Cong_1 \subset \Cong_2$ are congruences on $A$, then $\Cong_1 / \Cong_2$ is a congruence on $A/\Cong_1$, for which  $(A/\Cong_1)/(\Cong_1/ \Cong_2) \cong A/\Cong_2$. This also gives a factorization of the \hom\ $A \To  A/\Cong_2$ as $A \To  A/\Cong_1 \To  A/\Cong_2.$

\end{rem}

We defined the \textbf{congruence closure} of $\Cong_1 \cup \Cong_2 $ and $\Cong_1 + \Cong_2 $, for  $\Cong_1, \Cong_2 \in \Cng(A)$,  to be respectively the intersections, when are  nonempty   \footnote{The intersection of all congruences containing $ \Cong_1 \ast \Cong_2$ for a given operation $\ast$ , possibly subject to certain properties, implicitly produces the minimal transitive closure of $ \Cong_1 \ast \Cong_2$.},
\begin{equation}\label{eq:cong.closure}
  \overline{\Cong_1 \cup  \Cong_2} := \hskip -2mm   \bigcap_{\scriptsize \begin{array}{c}
 \Cong \in \Cng(A)\\
\Cong_1 \cup \Cong_2 \subset \Cong
\end{array}} \hskip -2mm  \Cong \ ,
\qquad
\overline{\Cong_1 +  \Cong_2} := \hskip -2mm  \bigcap_{\scriptsize \begin{array}{c}
 \Cong \in \Cng(A)\\
\Cong_1 + \Cong_2 \subset \Cong
\end{array}} \hskip -2mm  \Cong \ ,
\end{equation}
where $\Cong_1 \cup \Cong_2$ is the set theoretic union of $\Cong_1$ and $\Cong_2$ and $\Cong_1 + \Cong_2$ is induced  by the addition of $\AxA$.

\subsection{Semigroups, monoids, and semirings}  \sSkip

 A (multiplicative) \textbf{semigroup} $\tS:=(\tS, \cdot \, )$ is a set of elements  $\tS$, closed with respect to an associative binary operation ~$(\, \cdot \, )$.
A \textbf{monoid} $\tM:=(\tM, \cdot \, )$ is a semigroup with an identity element $\mone$. Formally,  any semigroup $\tS $  can be  attached with
identity element $\sone$ by declaring that $\sone \cdot a = a \cdot \sone = a$ for
all $a\in \tS$. So, when dealing with multiplication, we work with
monoids. As usual, when $(\, \cdot \, )$ is clear from the context, $a \cdot b$ is written as $ab$.
An \textbf{abelian} monoid is a commutative monoid, i.e., $ab = ba$ for all $a,b \in \tM$.
Analogously, we use additive notation for monoids, written $\tM := (\tM, +)$, whose identity is denoted by ~ $\zero_\tM$.

\begin{rem}\label{rem:mon.intersection} The intersection $\tA \cap \tB$ of two submonoids $\tA, \tB \subset \tM$ is again a submonoid. Indeed, if $a,b \in \tA \cap \tB$, then $ab \in \tA$ and $ab \in \tB$, and hence $ab \in \tA \cap \tB$.

\end{rem}

 An abelian monoid
$\tM := (\tM, \cdot \, )$ is \textbf{cancellative} with respect to
a subset $\TT \subseteq \tM$, if $ac = bc$ implies $a=b$ whenever
$a,b \in \tM$ and $c \in \TT.$ In this case, we  say that
$\TT$ is a \textbf{cancellative subset} of $\tM$. Clearly, when $T$ is cancellative,  the monoid generated by $\TT$ in $\tM$ also is cancellative,
so one can assume that ~$\TT$ is a submonoid. A monoid $\tM$ is
\textbf{strictly cancellative},  if $\tM$ is cancellative with
respect to itself.

An
element $a$ of $\tM$ is  \textbf{absorbing}, if $ab = ba = a$ for
all $b\in \tM$. Usually, it is identified as
$\zero_\tM$. A monoid $\tM$ is called \textbf{pointed monoid} if it has an
absorbing  element $\zero_\tM$.
An element $a\in \tM$ is a \textbf{unit} (or \textbf{invertible}), if there exists $b \in \tM$ such that $ab = ba = \one_\tM$.  The subgroup of all units in  $\tM$ is denoted by $\tM^\times$.

\begin{defn}\label{def:orderedMonoid}
 A \textbf{partially ordered monoid}
is a monoid~$\tM $ with a partial order $\leq$ that respects the monoid operation:
\begin{equation}\label{ogr1} a \le b \quad \text{implies}\quad ca
\le cb, \ ac
\le bc,  \end{equation} for all  $a,b,c \in \tM$. A monoid
$\tM $ is   \textbf{ordered} if the ordering $\leq$ is a total order.
\end{defn}

When working with pointed ordered monoid we usually assume that $\zero_\tM$ is a minimal (or maximal) element in $\tM$.

\begin{definition}\label{def:mon.hom}
A \textbf{homomorphism} of
monoids is a map
$\srHom : \tM \To \tN $  that respects the monoid operation:
\begin{enumerate} \eroman
    \item $\srHom(a \cdot b) = \srHom(a) \cdot \srHom(b)$;
    \item  $\srHom(\one_\tM) = \one_{\tN}$.
\end{enumerate}
$\srHom$ is called \textbf{local monoid homomorphism},
if $\srHom^\inv(\tN^\times) = \tM^\times$ (every $\srHom$ satisfies `` $\supset$'').
\end{definition}

 A standard general reference for structural theory of
semirings is~\cite{golan92}.
\begin{defn}\label{def:semiring} \eroman
A (unital \footnote{The given definition is for a unital semiring. But, as in this paper deals only with unital semirings, we call it semiring, for short.}) \textbf{semiring} $R := (R,+ \, ,\cdot \,)$ is a set $R$ equipped with two (associative)
binary operations $(+)$ and~$(\, \cdot \,)$,  addition and
multiplication respectively, such that:
\begin{enumerate} \eroman
    \item $(\SR, +)$ is an abelian monoid with identity $\rzero$;
    \item $(\SR, \; \cdot \; )$ is a monoid with  element
    $\rone$;
    \item $\rzero$ is an absorbing element, i.e., $a \cdot \rzero = \rzero \cdot a = \rzero$  for every $a \in R$;
    \item multiplication distributes over addition.

\end{enumerate}
$R$ is a  \textbf{commutative semiring}, if $a b= ba $ for all $a,b \in R$.
\end{defn}

A semiring $R$ is said to be \textbf{idempotent semiring}, if $a + a = a$ for every $a \in R$.
It is called \textbf{bipotent} (sometimes called \textbf{selective}) if $a + b \in \{ a,b\} $ for any $a,b$.
For example, the max-plus semiring $\zReal := (\Real \cup \{ -\infty\}, \max, +)$ is a bipotent semiring with $\one_{\zReal} = 0$ and  $\zero_{\zReal} = - \infty$, see e.g. \cite{pin98}.

\begin{rem}\label{makesemi} Any  (totally) ordered monoid $(\tM, \cdot\,)$
  gives rise to
a bipotent semiring. We formally add the zero element $\zero$ as the smallest element,   employed  also as absorbing element, and define  the idempotent addition $a+b$  to be $\max\{ a, b\}$.
  Indeed, this is a semiring.
Associativity is clear, and distributivity  follows from~
\eqref{ogr1}.
\end{rem}

To designate meaningful sums in semirings, we exclude inessential terms in  the following sense.
\begin{defn}\label{def:reduced.sum}
 A sum $\sum_i {a_i}$ of elements is said to be \textbf{\uexpr} if $\sum_i {a_i} = \sum_{j \neq i} {a_j}$ for some $i$; in this case $a_i$ is said to be \textbf{inessential}.   Otherwise,  $\sum_i {a_i}$ is called \textbf{\expr}, in which each $a_i$ is \textbf{essential}.
\end{defn}
For example, in an idempotent semiring every sum of the form  $ a + a$ is a \uexpr\ sum, while $\zero$ is always inessential.

\begin{defn}\label{def:ideal.smr}
  An \textbf{ideal} $\mfa$ of a semiring
$R := (R, +, \cdot \,)$, written $ \mfa \lhd R$, is an additive submonoid of
$(R,+)$ such that $ab \in \mfa $ and $ ba \in \mfa$ for all $a \in
\mfa$ and $b \in R$.
\end{defn}

 An ideal $\mfa$ of $R$ determines a ``Rees type'' congruence
$\Cong_\Rees(\mfa)$   on $R$ by letting $(a,b) \in
\Cong_\Rees(\mfa)$ iff $a,b\in \mfa$. Then, every element
of $R \setminus \mfa$  is \trivial\ with respect to $
\Cong_\Rees(\mfa)$, so $\Cong_\Rees(\mfa) $  includes no nontrivial relations on $R \setminus \mfa$, which makes it very limited.

Semiring ideals do not have the extensive role as ideals have in ring theory, since,  due to the lack of negation, they do not determine  congruences naturally.
Furthermore,  their correspondence to kernels of semiring homomorphisms is not so obvious.
Yet, they are useful for classifying  special substructures of semirings, and we employ them only for this purpose.
The passage to factor semirings is done by semiring congruences, which are a particular case of congruences on algebras,  cf.~ \S\ref{ssec:cong.univ.alg}.
\begin{defn}\label{def:equaliser}
The \textbf{equaliser} of two elements $a,b$ in a semiring $R$ is the set
$$  \Eql(a,b)=\{s \in R\mid sa = sb, \ as = bs  \}, $$
which might  be empty.
\end{defn} When $R$ is a commutative semiring, the equaliser $\Eql(a,b)$ is a semiring ideal of $R$, often  playing the role of annihilators  in ring theory.

\begin{definition}\label{def:smr.hom}
A \textbf{homomorphism} of \semirings0 is  a map
$\srHom : R \To S $  that preserves addition and multiplication.
To wit, $\srHom  $   satisfies the following properties for all
$a, b  \in R$:
\begin{enumerate} \eroman
    \item $\srHom(a + b) = \srHom(a) + \srHom(b)$;
    \item $\srHom(a \cdot b) = \srHom(a) \cdot \srHom(b)$;
    \item $\srHom(\rzero) =
\zero_{S}$;
\end{enumerate}
A \textbf{unital} semiring homomorphism is a semiring homomorphism that
preserves  the multiplicative identity, i.e., $\srHom(\rone) = \one_{S}$.
\end{definition}
In the sequel, unless otherwise is specified, semiring homomorphisms are all assumed to be unital.

\subsection{Monoid localization}\label{sec:monLocal} \sSkip

 Recall the well-known construction of the
\textbf{localization} $\iMS\tM$ of an
abelian monoid $\tM := (\tM, \cdot \,)$ by a
cancellative submonoid $\MS$, cf.  Bourbaki \cite{B}.   The elements of $\iMS\tM$  are the fractions
$\frac ac$ with  $a\in \tM$ and  $c\in \MS,$ where $$\frac ac = \frac
{a'}{c'} \dss{\text{ iff }} ac' = a'c,$$ and multiplication given
by
$$\frac {a_1}{c_1}   \frac{a_2}{c_2} = \frac{a_1
  a_2}{c_1c_2}.$$
(Although many texts treat localization for algebras, precisely
the same constructions and proofs work for monoids.) This
construction is easily generalized to non-cancellative submonoids,
where now
\begin{equation}\label{eq:nonc.loc}
\frac ac = \frac {a'}{c'} \dss{\text{ iff} } ac'c'' =
a'cc'' \ \text{ for some } c'' \in \MS.
\end{equation}
(In the case of a pointed monoid $\tM$, we assume that $\zero_\tM \notin C$.)
If the monoid $\tM$ is ordered, then $\iMS \tM$ is also ordered, by letting
$\frac a c  \leq \frac {a'}{c'}$ iff $a {c'} c'' \leq  {a'} c c''$ for some $c'' \in C$.

Any monoid homomorphism $\varphi: \tM \To \tN$, for which
$\varphi(c)$ is invertible for every  $c\in \MS$, extends naturally
to a unique monoid homomorphism $\htvrp: \iMS \tM \To \tN$,
given by
$$\htvrp \bigg(\frac ac \bigg) =  \varphi(a){\varphi(c)}^{-1}.$$
If $\tM$ is endowed also with addition $(+)$, then $(+)$ extends to
$\iMS \tM$ via the rule: $$\frac {a_1}{c_1} + \frac{a_2}{c_2} =
\frac{c_2a_1 + c_1 a_2}{c_1c_2}.$$ Each of the basic properties
(associativity of $(+)$ and distributivity) can be extended straightforwardly  from
$\tM$ to $\iMS\tM$.

For a strictly cancellative monoid $\tM$ (and thus without
$\mzero$), or when $C = \tM$,  we write $Q(\tM)$ for the localization $\iMS \tM$. Therefore, $Q(\tM)$ is an abelian group, since
$(\frac ac)^{-1} = \frac ca .$

\subsection{Modules over semirings} \sSkip

Modules (called also
{semi-module} in the literature) over semirings are defined similarly to modules over rings \cite{Katsov}.

 \begin{defn}\label{def:modules}
 A (left) $R$-\textbf{module}  over a semiring $R$ is an abelian
monoid $M := (M,+)$ with an operation $R\times M \To M$ -- a left action of $R$ -- which is
associative and distributive over addition, and which satisfies
$a\zero_M = \zero_M = \zero_R v$, $\rone v = v$ for all $a$ in $R$, $v \in M$.

An \textbf{$R$-module homomorphism} is a map $\vrp:M \To N$ that satisfies the conditions:
\begin{enumerate}\eroman
  \item $\vrp(v+w) = \vrp(v) + \vrp(w)$ for all $v,w \in M$;
  \item $\vrp(av) = a \vrp(v)$ for all $a \in R$ and $v \in M$.
\end{enumerate}
$\Hom_R(M,N)$ denotes the set of $R$-module homomorphisms
from $M$ to $N$.
\end{defn}

 Right modules are defined dually.
If $R$ and $S$ are semirings and $M$ is a left $R$-module
and a right $S$-module, then $M$ is called \textbf{$(R,S)$-bimodule}, if $(av)b = a(vb)$ for
all $a \in R$, $v \in   M$, $b  \in S$.

\begin{example} Let $R$ be a semiring.

\begin{enumerate}\eroman
\item
 A semiring ideal  $ \mfa \lhd R$ is an
$R$-module, cf. Definition \ref{def:ideal.smr}.

\item The direct sum of $R$-modules is clearly an $R$-module. In
particular, we define $R^{(n)}$ to be the direct sum of $n$ copies
of $R$.

\end{enumerate}
\end{example}

Congruences of monoids, cf. \S\ref{ssec:cong.univ.alg}, extend directory to modules.

 \begin{defn} An $R$-\textbf{module
congruence} is a monoid congruence on $M$ that  satisfies the additional
property that, if $v \equiv w$, then $av \equiv a w$ for all
$ v,w \in M$ and  $a \in R.$

\end{defn}

\section{Supertropical structures }\label{sec:2}

To develop  a solid  algebraic theory, we introduce a new monoid structure which leads to existing definitions of supertropical structures \cite{zur05TropicalAlgebra,IKR1,IzhakianRowen2007SuperTropical,Coordinates},  these are later generalized by applying additional  modifications.

\subsection{Additive \numon s}\label{ssec:numon} \sSkip

We start with our basic underlying additive  structure.

\begin{defn}\label{def:numonoid} An  \textbf{additive \numon} is a triplet
$\tM := (\tM, \tG, \nu)$, where  $\tM = (\tM, + )$ is an  additive abelian monoid \footnote{Supertropical monoids in \cite{IKR3} regard with a multiplicative monoid structure, which does not involve the monoid ordering, and are of a different nature.}, $\tG$ is a distinguished partially ordered  submonoid of $\tM$ with $\zero_\tM < a$ for all $a \in \tG$,  and $\nu: \tM \To \tG$ is an idempotent monoid homomorphism  (i.e., $\nu^2 = \nu$)---a projection on $\tG$---satisfying for every $a, b \in \tM$ the conditions:
\begin{description} \dispace
  \item[\NMa:]  $a+b =  a$   whenever  $\nu(a) > \nu(b)$,
  \item[\NMb:] $a+b =  \nu(a)$   whenever  $\nu(a) = \nu(b)$,
  \item[\NMc:] If $a+b \notin \tG$ and  $a +  \nu(b) \in \tG$, then    $a+b = \nu(a) + b$.
\end{description}
The element $\zero_\tM$ is  the monoid identity: $\zero_\tM + a =a + \zero_\tM = a$ for every $a \in \tM$.
\end{defn}
\noindent  By definition  $\nu(a) =a$ for every $a \in \tG$, and  $\zero_\tM \in \tG$, since $\tG$ is an (additive) submonoid of ~$\tM$. As $\nu$ is a monoid homomorphism, we have
$\nu(a+b) = \nu(a) + \nu(b)$ for any $a,b \in \tM$.  We often term ~$\tM$ a \textbf{\numon}, for short.

In the extreme case that the partial ordering of $\tG$ is degenerate, i.e., only trivial relations $a = a$ occur,    Axiom \NMa\ is dismissed.
Replacing Axiom \NMb\ by the (weaker) axiom
\begin{description} \dispace
  \item[\NMb':] $a+a =  \nu(a)+a = \nu(a)$ for all $a \in M,$
\end{description}
Axiom \NMc\ implies \NMb.  Indeed, suppose $\nu(a) = \nu(b) $ but $a+ b \notin \tG$, then $a + \nu(b) = a + \nu(a) \in \tG$, and hence $a+b = \nu(a) + b = \nu(b) + b \in \tG$ -- a contradiction. Literally,
Axiom \NMc\ gives a condition whether the contribution of $a$ to a sum $a+b$ is ``idempotent'', namely $a +b  = a+ a + b = \nu(a) + b$.
This includes that case that $a+ b $ is \uexpr\ (Definition \ref{def:reduced.sum}), i.e., $a+b = \nu(a) + b =  b$.

The
elements of $\tG$ are termed \textbf{ghost elements} and  $\tG$ is called the \textbf{ghost submonoid} of $\tM$. An element $a \notin \tGz$ is said to be \textbf{non-ghost}.
Although, $\mzero \in \tG$, as an exceptional case,  it is sometimes considered as a non-ghost element.
The projection map $\nu: \tM
\To \tG$ is called  the \textbf{ghost map} of $\tM$.
 We write~$a^\nu$ for the image $\nu(a)$ of $a$ in $\tG$. ($\tM$ might have elements $a \in \tG$ without  $\nu$-preimage in $\tM \sm \tG$.)
 A ~\numon\ $\tM$ is said to be \textbf{ghost monoid}, if $\tM = \tG$, i.e., if all its elements are ghost. In this case, the ghost $\nu = \id_\tM$ is the identity map.

  We say that elements $a$ and $b$ in $\tM$ are  \textbf{\nuequivalent}, written  $ a \nucong b$, if
$a^\nu = b^\nu$. That is
\begin{equation}\label{eq:nu.eqv}
 a \nucong b \dss \Iff \nu(a) = \nu(b).
\end{equation}
The \textbf{\nufiber} of an element $a \in \tM$ is the set
\begin{equation}\label{eq:nu.fib}
\nufib(a) := \{ b \in \tM \cnd b \nucong a \},
\end{equation}
 usually considered for $a \in \tG$. Accordingly, for any $b,b' \in  \nufib(a)$ we have
$ b + b' = a^\nu$,  while  $\nufib(a) \cap \tG = \{ a^\nu \}$ for every $a \in \tM$.
\begin{rem}\label{rem:nmon.zero}
One can easily verify the following properties, obtained from Axioms \emph{\NMa}, \emph{\NMb} and the ordering of $\tG$.

\begin{enumerate}\ealph
  \item $a + b = \zero_\tM \impl a = b = \zero_\tM .$\footnote{This property is called ``lack zero sums'' in  \cite{decomposition}.}
  \item $\nu(a) = \zero_\tM \impl a = \zero_\tM .$
  \item $a+a=a+a^\nu =a^\nu$ for every $a \in \tM$.
\end{enumerate}
\end{rem}

We define the \textbf{ghost surpassing} relation:
\begin{equation}\label{eq:gs.relation}
a \nmod b \quad \text{ if } \ a = b + c \ \text{ for some } \ c \in \tG.
\end{equation}
Accordingly, since $\tG$ is a submonoid,  if $b \in \tG$ and $a \nmod b$, then $a \in \tG$.
Also, we always have $a^\nu \nmod a$, while $a \nmod a^\nu$ only when  $a \in \tG.$

Given a \numon\ $\tM := (\tM, \tG, \nu)$, one can recover the  partial ordering of $\tG$ from the additive structure of $\tM$. That is,  $a > b$ iff $a+b = a$, for $a,b \in \tG$.
The ordering of $\tG \subseteq \tM$  induces a (partial) $\nu$-ordering $\nuge$ on the whole  $\tM$ via the ghost map,   defined as
\begin{equation}\label{eq:nu.order}
\begin{array}{rcl}
  a \nuge b  & \ds {\Leftrightarrow}&  a^\nu \ge  b^\nu,  \\
   a \nug b & \ds {\Leftrightarrow} & a^\nu > b^\nu,
   \end{array}
\end{equation}
   in which $a \nuge \zero_\tM$ for every $a \in \tM$.

\begin{example}\label{exmp:nu.mon} In the degenerate case that $\tM$ is a \numon\ with $\tG = \{\zero_\tM, a \}$ and ghost map $\nu: b \mTo a$ for all $b \neq \zero_\tM$ in $\tM$, by Axiom \NMb,  $b+c = a$  for all $b,c \neq \zero_\tM $. When  $\tG = \{ \zero_\tM \}$, we have  $\nu: b \mTo \zero_\tM$ for all $b$, implying that $b = b + \zero_\tM = \nu(b) = \zero_\tM$, and thus $\tM = \{ \zero_\tM \} $ is trivial.
\end{example}
   For \expr\ sums 
   $x = \sum_{i \in I} a_i $ and $y = \sum_{j \in J} a_j $  with nonempty finite sets of indices $I,J$,  we define the relation  $\lsset$ as
   \begin{equation}\label{eq:lorder}
     x \lsset y  \dss {\Leftrightarrow} I  \subset J ,
   \end{equation}
   which introduces an additional  partial ordering  on $\tM$.
\begin{defn}\label{def:numon.hom} A \textbf{homomorphism} of additive \numon s
is a monoid homomorphism
\begin{equation}\label{eq:nu.mon.hom} \vrp: (\tM,  \tG, \nu)
\TO (\tM', \tG', \nu')\end{equation} of  \numon s, i.e., $\vrp(a+b) = \vrp(a)+ \vrp(b)$ for all $a,b \in \tM$, where
 $\srHom(\zero_\tM) = \zero_{\tM'}$.
 A \textbf{gs-morphism} (a short for ghost surpassing morphism) is a map $\phi: \tM \To \tM'$ that satisfies $\phi(a)+ \phi(b) \nmod \phi(a+b)$ for all $a,b \in \tM$.

 The \textbf{image} of $\vrp$, denoted by $\im(\vrp)$, is the set  $\vrp(\tM) \subset  \tM'$.
 The \textbf{ghost-kernel}  of $\vrp$, abbreviated \textbf{g-kernel},    is defined as
$$\gker(\vrp) : = \{ a \in \tM \cnd \vrp(a) \in \tG'\} \subset \tM. $$
We say that $\vrp$ is \textbf{ghost injective}, if $\gker(\vrp) = \tG.$\footnote{Note that ghost injectivity does not imply that the restriction $\vrp|_\tG: \tG \To \tG' $ is an injective map. } A homomorphism $\vrp$ is a \textbf{ghost homomorphism}, if $\gker(\vrp) = \tM.$
\end{defn}

Clearly any gs-morphism is a \hom.
It easy to verify that $\gker(\vrp)$ is a \nusmon\  of $\tM$ containing the ghost submonoid $\tG$.
A \numon\ homomorphism \eqref{eq:nu.mon.hom} respects the ghost map $\nu: \tM
\To \tG$,  as well as the induced $\nu$-ordering ~\eqref{eq:nu.order}:

\begin{lem}\label{lem:nu.hom} Any \hom\ $\vrp: \tM \To \tM'$ of additive \numon s  satisfies the following
properties for all $a,b \in \tM$:
\begin{enumerate}\eroman
    \item   $\vrp (a^{\nu})=\vrp (a)^{\nu'};$
 \item   $\vrp (\tG) \subseteq \tG'$, and thus $\tG \subseteq \gker(\vrp)$;

   \item    If  $a >_\nu b$,
 then  $\vrp(a)  > _ {\nu'} \vrp(b); $

   \item    If  $a \cong_ {\nu}b$,
 then  $\vrp(a)  \cong _ {\nu'} \vrp(b). $

\end{enumerate}
\end{lem}
\begin{proof} (i): Write $\vrp (a^{\nu})= \vrp (a+a) =  \vrp (a)+\vrp(a) =\vrp (a)^{\nu'}$.
  \pSkip
  (ii):  Immediate from (i), since any $a \in \tG$ satisfies
   $\vrp (a)=\vrp (a^{\nu}) = \vrp (a)^{\nu'}\in \tG'$.
\pSkip
   (iii): If  $a >_ {\nu}b$, then $a+b = a$,
 implying  $ \vrp(a) = \vrp( a +  b) =  \vrp(a) +  \vrp(b) , $
i.e.,  $  \vrp(a)  > _ {\nu'} \vrp( b) $ by \eqref{eq:nu.order}.
\pSkip
   (iv): If  $a \cong
_ {\nu}b$, then $a+b = a^\nu = b^\nu$,
 implying  $ \vrp(a) ^{\nu'} = \vrp( a +  b) =  \vrp(b)^{\nu'}, $
i.e.,  $  \vrp(a)  \cong _ {\nu'} \vrp( b). $
\end{proof}

\begin{lem}\label{lem:injct.hom}
  Let $\vrp: \tM \To \tM'$ be a \hom\ of additive \numon s.
  \begin{enumerate}\eroman
    \item If $\vrp$ is  injective,
  then $\gker(\vrp) = \tG.$

  \item If $\gker(\vrp) = \tG$, then $a+b \in \tG$ for any $a,b$ with $\vrp(a) = \vrp(b)$.

  \end{enumerate}

\end{lem}
\begin{proof} (i): $\tG \subseteq \gker(\vrp)$ by Lemma \ref{lem:nu.hom}.(ii). If $a\in \gker(\vrp) \sm \tG$, then $a \neq a^\nu$, but $\vrp (a)^{\nu'}  = \vrp(a^\nu) = \vrp(a)$ by Lemma \ref{lem:nu.hom}.(ii), since $a \in \gker(\vrp)$, which contradicts the injectivity of $\vrp$.
  \pSkip
  (ii): $\vrp(a) = \vrp(b) \impl \vrp(a) + \vrp(b) \in \tG' \impl \vrp(a+b) \in \tG' \impl a+b \in  \gker(\vrp) = \tG$.
  \end{proof}

 Given additive \numon s $\tM'$ and $\tM$, as customarily,   $\Hom(\tM, \tM')$ denotes the set of all \hom s $\vrp: \tM \To~ \tM' $.  $\Hom(\tM, \tM')$ is equipped  with  a partial ordering, determined  for $\vrp, \psi \in \Hom(\tM, \tM')$  by
 $$  \vrp >_{\nu_{\Hom}} \psi  \dss \Iff \vrp(a) >_{\nu'} \psi(a) \text{ for all } a \in \tM,  $$
and further with the ghost map $\nu_{\Hom}: \Hom(\tM, \tM') \To \Hom(\tM, \tG')$ defined by $\vrp \Mto \nu' \circ \vrp$.

\begin{prop}\label{prop:Hom.is.numon}
  $\Hom(\tM, \tM')$ is a \numon\ (Definition~\ref{def:numonoid}).
\end{prop}
\begin{proof}
  The operation $\vrp + \psi $ is well defined, as $\vrp $ and $\psi$ are \hom.  It is associative   and commutative, since $\tM'$ is an abelian monoid. The function  $0_\Hom : a \mTo 0_{\tM'}$ is the neutral element for $(+)$. So $\Hom(\tM, \tM')$ is an abelian monoid.
    The ghost map $\nu'$ of $\tM'$ is idempotent,  therefore  $\nu' \circ \nu' \circ \vrp =   \nu' \circ \vrp$, and hence $\nu_\Hom = \nu_\Hom\circ \nu_\Hom $ is also idempotent. The ghost submonoid of $\Hom(\tM,\tM')$  is the image of $\nu_\Hom$.
  Axioms \NMa--\NMc\  are obtained by point-wise verification, provided that they hold in $\tM'$.
\end{proof}

\subsection{Colimts, pullbacks, and pushouts }\label{ssec:pullback} \sSkip

We denote the category of additive \numon s  by \NMON, whose objects are \numon s\ and its morphisms are \numon\ homomorphisms.

\begin{prop}\label{prop:numon.colimt}
 Colimits exist for additive \numon s.
\end{prop}

\begin{proof}
Let $I$ be a directed index set, and let
$(\tM _i, \vrp_{ij})_{i,j \in I}$ be a direct system of \numon s, where  $\vrp_{ij}: \tM_i \To \tM_j$ are homomorphisms of \numon s.
Take $\sim$ to be the equivalence determined by
  $x \sim  y$ if $\vrp_{ik}(x) = \vrp_{jk}(y)$ for some  $k \geq i, j$ in $I$, where   $x \in  \tM_i$ and $y \in \tM_j$.
The direct limit $\httM:= \underrightarrow{\lim}_{i \in I } \tM_i = \big(\coprod_{i\in I}  \tM_i \big)/_\sim$ exists in the category of monoids,  together
with the associated maps $\vrp_i : \tM_i \To \httM$. Similarly, $\httG := \underrightarrow{\lim}_{i \in I } \tG_i = \big(\coprod_{i\in I}  \tG_i \big)/_\sim$ is the direct limit of the system $(\tG _i, \psi_{ij})_{i,j \in I}$ of ghost submonoids $\tG_i \subset \tM_i$ such that  $\psi_{ij} = \vrp_{ij}|_{\tG_i}$.  The ghost maps $\nu_i: \tM_i  \To \tG_i  $ satisfy $\psi_{ij} = \nu_j \circ \vrp_{ij}$, and are preserved under
$\vrp_{ij}$, by Lemma \ref{lem:nu.hom}. Therefore,
$(\httM, \httG, \htnu)$ is a \numon\ with ghost map $\htnu : \httM \To \httG$.
\end{proof}

\begin{defn}\label{def:pullback}
Let $\tA, \tB, \tC \in \NMON$ be \numon s  with morphisms $\phi_1 : \tA \To \tC$ and $\phi_2 : \tB \To ~\tC$. A \numon\ $\tP \in \NMON$, with morphisms $\pi_1: \tP \To \tA$ and $\pi_2:\tP \To \tB$,   is a \textbf{pullback} along  $\phi_1$ and ~$ \phi_2$ if it renders the following diagram commute and  universal,
$$ \xymatrix{ \tU  \ar@{..>}[dr]^{\xi}  \ar@/^1pc/@{->}[rrd]^{\psi_2}  \ar@/_1pc/@{->}[rdd]_{\psi_1} & & \\
& \tP \ar@{->}[d]^{\pi_1}  \ar@{->}[r]_{\pi_2} & \tB \ar@{->}[d]^{\phi_2} \\
& \tA \ar@{->}[r]_{\phi_1} & \tC } $$
i.e., for every other \numon\ $\tU$ with morphisms $\psi_1 : \tU \To \tA$ and $\psi_2 : \tU \To \tB$ there is a unique morphism $\xi : \tU \To \tP$ which makes the diagram commute.
\end{defn}

In the standard way we define the  \numon\ $\tP $ as the subset of the direct product $\tA \times \tB$:
\[
\tA \times_\tC \tB := \{(a,b) \in \tA \times \tB \cnd  \phi_1(a)=\phi_2(b)\},
\]
 endowed with component-wise addition  $(+)$ and identity element  $(\zero_\tA, \zero_\tB)$; its ghost map is  induced from~ $\tA \times \tB$.

\begin{prop}\label{prop:pullback}
  $\tA \times_\tC \tB$ is a pullback and it is universal.
\end{prop}
\begin{proof}

First, to see that $\tA \times_\tC \tB$ is  a \numon, notice that $\phi_1(\zero_\tA)=\zero_\tC=\phi_2(\zero_\tB)$, since $\phi_1$ and $\phi_2$ are \numon\ homomorphisms, thus   $(\zero_\tA,\zero_\tB) \in \tA \times_\tC \tB $. Clearly, $\tA \times_\tC \tB$  is closed for addition in  $\tA \times \tB$, and is also closed under $\nu$. Indeed, $\phi_1(a)=\phi_2(b) \ds\Rightarrow  \phi_1 (a^\nu)=\phi_1(a)^{\nu}= \phi_2(b)^{\nu}=\phi_2(b^\nu)$.

Observe that $\tA \times_\tC \tB$ is  set-theoreic pullback.
Given $\tU \in \NMON$ with $\psi_1: \tU \To \tA$, $\psi_2: \tU \To \tB$ such that
$\phi_1 \circ \psi_1 = \phi_2 \circ \psi_2$,  we have
$$ \xi: \tU \TO \tA \times \tB, \qquad u \Mto (\psi_1(u), \psi_2(u)), $$ hence
$\phi_1 \circ \pi_1 \circ \xi = \phi_2 \circ \pi_2 \circ \xi$.
Then,  a routine check shows that there is a one-to-one correspondence between pairs of homomorphisms
$(\psi_1, \psi_2): \tU  \To \tA \times_\tC \tB$  such that  $\phi_1 \circ \psi_1 = \phi_2 \circ \psi_2$
and homomorphisms $\xi \colon \tU \To\tA \times_\tC \tB.$
\end{proof}

Pushout of \numon s as usual is the dual notion of a pullback, determined by the diagram
$$ \xymatrix{ \tV      & & \\
& \tP \ar@{..>}[ul]_{\xi}  &  \ar@{->}[l]_{\pi_2}  \ar@/_1pc/@{->}[llu]_{\psi_1}  \tB  \\
& \ar@/^1pc/@{->}[luu]^{\psi_2}  \ar@{->}[u]^{\pi_1}    \tA  & \ar@{->}[u]^{\phi_2} \ar@{->}[l]_{\phi_1} \tC } $$
It is universal, where the proof follows similar arguments as in the proof of Proposition \ref{prop:pullback}.

\subsection{\nusmr s}\label{ssec:nusmr}
\sSkip

Equipping an additive  \numon\ (Definition \ref{def:numonoid}) by a multiplicative operation with additional traits,  we obtain the following semiring structure, which is the central  structure of the  study in this paper.

\begin{defn}\label{def:nusemiring} A \textbf{\nusmr} is a quadruple
$R := (R, \tT, \tG, \nu)$, where  $R$ is a \smr,  $(R,\tG, \nu)$ is an additive  \numon, and    $\tT \subseteq  R \sm \tG
$ is a distinguished subset, containing a designated subset  ~$\tTPS$ with $\RX \subseteq   \tTPS$  which satisfies    the conditions:
\begin{description} \dispace
  \item[\PSRa:] \ $a \in \tTPS$ implies $a^n \in \tTPS$ for any $n \in \Net$,
  \item[\PSRb:] \ $a+b \in \tT$ implies $a+b^\nu \notin \tG$, unless $a+b$ is \uexpr.

\end{description} The  additive identity $\rzero$ is an absorbing element, i.e., $\rzero a = a \rzero = \rzero$ for every  $a \in R$.
\pSkip

 A \nusmr\ $R$ is said to be
 \begin{itemize} \dispace
   \item \textbf{faithful} if $\nu|_\tT: \tT \To \tG$ is injective;

  \item \textbf{commutative} if $ab = ba$ for all $a,b \in R$;
   \item \textbf{definite} if $\tT= R \sm \tG;  $

  \item \textbf{\tame}  if any $a \in R \sm (\tT \cup \tG)$ can be written as $a = c + d^\nu$, where $c,d \in \tT$;

   \item \textbf{\prsfl} if $\tT = \tTPS$;

   \item \textbf{\prscls} if $\tTPS $ is a (multiplicative) monoid;

\item \textbf{\Tcls} if $\tT$ is a monoid.
 \end{itemize}
\end{defn}

We set $e_R := \rone+\rone = (\rone)^\nu$,   hence     $e = e^2 = e + e = e^\nu$ for  $e = e_R$. In particular $e \in \tG$, and $a^\nu = a +a = a (\rone+\rone) =(\rone+\rone) a$;  thus $a^\nu = e a = ae$ for every $a \in R$. Therefore,  in \nusmr s the ghost submonoid $\tG$ becomes a semiring ideal (Definition \ref{def:ideal.smr}),  called the \textbf{ghost ideal} (which  can be thought of as a ``ghost absorbing'' subset).  Indeed,
$$ab^\nu = a e b =  a(\rone + \rone) b = a(b+b)= ab + ab = (ab)^\nu \in \tGz,$$
furthermore $a^\nu b^\nu = eaeb = e(ab) = (ab)^\nu$. Note also that, for any $a_1, \dots, a_n$ in $R$,
\begin{equation}\label{eq:ghost.sum}
  \big(\sum a_i\big)^\nu  = e \big(\sum a_i\big) = \sum e a_i = \sum a_i^\nu,
\end{equation}
which implies that  if $x$ is a ghost, then all \expr \ sums $y \lsset x$ are also ghosts, cf. \eqref{eq:lorder}.
The \textbf{\nuequivalence} $\nucong$ on $R$ is induced from the \nuequivalence\ \eqref{eq:nu.eqv} of its additive \numon\ structure (cf.~ \S\ref{ssec:numon}), respecting  multiplication as well, i.e.,   $a \nucong b$ implies $ca \nucong c b$ for every $c \in R$. On \nusmr s  $\nucong$ can be stated as $a \nucong b$ iff $ea = eb$.

The elements of the distinguished  subset  $\tT$ are  called   \textbf{tangible elements} and $\tT$ is termed  the \textbf{\tsset} of $R$. \footnote{Intuitively, the tangible
elements in a \nusmr\ correspond to the original max-plus algebra in Example \ref{exmp:extended} below, although now
$a+a = a^\nu$ instead of $a+a = a.$ Therefore, as  $a + \cdots + a = a^\nu$, this formulation encodes an additive multiplicity $> 1$, or equivalently the phrase  ``the sum attends by at  least two different terms''. Taking the sum to be maximum, this is the tropical analogy of the vanishing condition in ring theory.} When $R$ is \Tcls,  $\tT$ is called the  \textbf{\tmon} of $R$ and we say that $R$ is \textbf{\tcls}, for emphasis.
We write $\tTz$ for $\tT \cup \{ \rzero\} $, which is a pointed monoid when $R$ is a  \tclsnusmr. The element $\rzero$ can be considered either as ghost or tangible.
An element $a \notin \tT$ is termed \textbf{non-tangible}.

A tangible element $a \in \tTPS$  is  said to be  \textbf{\tanprs}, written \textbf{\tprs},  i.e., if  $a^n \in  \tTPS \subseteq \tT$ for any  $n \in \Net$.
In particular, every tangible idempotent is \tprs.
The designated subset
$\tTPS \subseteq \tT$ is called the \textbf{\tprsset} of $R$, and is never empty, since $\rone \in \RX \subseteq \tTPS$. $\tTPS$ may contain a larger monoid than $\{\rone \}$, which we denote by $\tTP$ and call it a \textbf{\tprsmon}, or a \textbf{tangible monoid},   in which any product of \tprs\ elements is \tprs.
Accordingly, for a \prscls\ \nusmr\ we have  $\tTP = \tTPS$, while  $\tTP = \tTPS = \tT$ in a \tclsnusmr.

\begin{example} A tangible element $a = b+c $ that can  be written as a \expr\ sum of tangibles $b,c \in \tT$ in a commutative \nusmr\ is not \tprs. Indeed, $a^2 = (b+c)^2 = b^2 + ebc + b^2 $, which is not tangible, unless the term $ebc$ is inessential.
\end{example}

Literally, Axiom \PSRa\ means that any power of a \tprs\ element is \tprs, but a product of \tprs\ elements need not be \tprs.
On the other hand, $ab \in \tT$ (reps. $ab \in \tTPS$) does not imply that $a, b \in \tT$ (resp. $a, b \in \tTPS$). 

  Axiom \PSRb\ strengthens Axiom \NMc\ of \numon s (Definition \ref{def:numonoid}) to tangible sums  in \nusmr s. Practically,  it follows from  Axiom \PSRb\ that  a  \expr\ tangible sum cannot  involve ghosts:
\begin{equation}\label{eq:tan.sum}
 a + b \in \tT \Dir a,b \notin \tG.
\end{equation}
Indeed, otherwise,  if $a + b \in \tT$, say with $b \in \tG$, then both $a^\nu$ and $b$ are ghost and hence  $a^\nu + b \in \tG$, since~ $\tG$ is an ideal, which contradicts  Axiom \PSRb. This implies that if $a \nmod b$ where $a \in \tT$, cf. \eqref{eq:gs.relation}, then $a = b$.

\begin{remark} In a \tamesmr\  $a b \notin \tT$ for any $b \notin \tT$, since $b = c + d^\nu$ and  $ab = ac + (cd)^\nu \notin \tT$ by~ \eqref{eq:tan.sum}.
\end{remark}

Recall that $\RX$ denotes the subgroup of  units in $R$, where an element $a \in R$ is a \textbf{unit}, and said to be  \textbf{invertible},  if there exits $ b \in R $ such that $ab= ba = \rone$.   In this case, $b$ is  called the  \textbf{inverse of $a$} and is denoted by $a^\inv$.
 A unit in \nusmr s    must be \tprs, to stress this we sometimes write  $\tTX$ for $\RX $.
\begin{rem}\label{rem:prud.unit}
In a commutative \nusmr\ $R$,  $ab \in \RX $ implies $a,b \in \RX$, since $c(ab) = \rone = (ca)b$ for some $c \in \tT$, and hence $b \in \RX$.  Commutativity implies  that also  $a \in \RX$.
\end{rem}

To summarize our setting,  for a \nusmr\ $R$ we have the (multiplicative) components
$$
\rone \in \tTX = \underset{\text{group}}{\RX} \dss\subseteq
\underset{\text{set}}{\tTPS} \dss \subseteq
\underset{\text{set}}{\tT} \dss \subseteq \underset{\text{set}}{R \sm \tG}, \qquad e_R \in \underset{\text{ideal}}{\tG} \lhd R,$$
each of which is nonempty.

We say that two non-invertible elements $a,b \notin \RX$ are \textbf{associates} if $a = bc$ for some $c \in \RX$.
An element $a \notin \RX$  is \textbf{irreducible}, if $a = bc$ implies $b \in \RX$  or $c \in \RX$. Accordingly, a ghost element $a\in \tG$ is never irreducible, as $a = a^\nu = e a$, where $e$ is not a unit.

Since $R$ is a semiring, where multiplication distributes over addition, and $(R,\tG,\nu)$ is a partially $\nu$-ordered monoid \eqref{eq:nu.order}, then
\begin{equation}\label{eq:smr.order}
  a \nug b \Dir ca \nug cb, \ ac \nug bc \qquad \text{for any }    a,b,c \in R, \ c \neq \rzero. 
\end{equation}
So, the multiplication of $R$ respects the  partial $\nu$-ordering of $R$, and  in this sense all nonzero elements may be realized as  ``positive'' elements.

When $R$ is \tcls, we have the strong property
\begin{equation*}\label{AX:SR}
    a \in \tT \text{ and } b \in \tT \dss \Rightarrow ab \in \tT,    
\end{equation*}
which provides $\tT$ as a monoid.

\begin{rem}\label{rem:strc.relations}
  The following structural relations hold:
\begin{enumerate} \ealph
  \item definite $\Rightarrow$ \tame;
  \item \tcls\ $\Rightarrow$   \prsfl\ and \prscls.

\end{enumerate}
\end{rem}

One can  construct a \nusmr\  (more precisely a supertropical semiring, as defined below in \S\ref{ssec:nudomain}) from any ordered monoid.

\begin{construction}
  \label{exmp:mon2smr} Given an ordered monoid $\tM:= (\tM, \cdot \, )$, cf. Definition \ref{def:orderedMonoid}, we duplicate  $\tM$ to have a second copy $\tM^\nu$ of $\tM$ and  create the set $\STR(\tM): = \tM  \cup \{ \zero \} \cup \tM^\nu $, where $\zero$ is formally added as the smallest element.
 $\STR(\tM)$ becomes an ordered  set by  declaring that  $$a < a^\nu < b<  b^\nu$$  for any $a < b$ in ~ $\tM$.  We define the addition of  $x,y \in \STR(\tM)$ as
$$x+y := \max\{ x, y\},$$  and take the multiplication induced by the  monoid operation, i.e.,
$ab = a \cdot b$ for $a,b \in \tM$ and  $ab^\nu = a^\nu b  = a^\nu b^\nu = (a \cdot b)^\nu$, where $a^\nu, b^\nu \in \tM^\nu$. Accordingly, $\tM$ is assigned as the \tmon\ $\tT$, while~ $\tM^\nu \cup \{ \zero \}  $ is allocated as the ghost ideal  $\tG$.
Then, $\STR(\tM)$ is endowed with the structure  of  \tclsnusmr, which is also faithful and \dfnt,  with ghost map $\nu|_\tT = \id$, in which  $\zero a = \zero a = \zero$ is an absorbing element.

\end{construction}

   The elements of the Cartesian product $\Rn = R \times \cdots \times R$ are $n$-tuples $\bfa =  (a_1, \dots, a_n)$ with $a_i \in R$. A point $\bfa \in R^{(n)}$ is \textbf{tangible} if $a_i \in \tTz$, for all $i$, and is ghost if $a_i \in \tG$ for every  $i$. We write $(\rzero)$ for the point $(\rzero, \dots, \rzero)$, which belongs to both $\tTz^{(n)}$ and $\tG^{(n)}$.
Given a subset $E \subset \Rn$,  we write $E|_\tng^\PrS$ for   $E \cap
(\tTPS)^{(n)}$ --   called the \textbf{\tngpres} of $E$. We write $E|_\tng$ for   $E \cap
\tTzn$ --  called the \textbf{\tngres } of $E$,   $E|_\ghs := E
\cap \tGn$ is called the \textbf{ghost part} of $E$. The latter parts, $E|_\tng$ and $E|_\ghs$, need not be the complement of each other.

\begin{example}\label{examp:pre.smr} Suppose $R$ is a commutative, \tame,  \tclsnusmr.
\begin{enumerate}\eroman
  \item The semiring  $R[\lm_1, \dots, \lm_n]$
  of polynomials over $R$ is a commutative \nusmr;  its \tsset\ consists of the polynomials with tangible coefficients.
        This \nusmr\ is \tame\ but not \prsfl,  as the tangible polynomials do not form a monoid. As explained below in~ \S\ref{ssec:ploynomials}, $R[\lm_1, \dots, \lm_n]$  is canonically associated to  the \prsfl \ \nusmr \ of polynomial functions on $\Rn$. 

  \item The semiring of $n \times n $ matrices over $R$  is a noncommutative \nusmr,
  which is \tame\ but not \tcls.
   It contains the subgroup of invertible matrices (i.e., the group of generalized permutation matrices) and the subgroup of diagonal tangible  matrices, which give rise to (\tame) \tcls\ \nussmr\ structures.

  \item The set $\Rn$ of all $n$-tuples over $R$, with entry-wise addition and multiplication,  is a commutative  \tcls\  \nusmr. The identity element $\one_{\Rn}$ of this \nusmr\ can be generated additively and multiplicatively by other elements of $\Rn$.
\end{enumerate}
Parts (ii) and (iii) provide examples of \nusmr\ which are \tame, but this property can be dependent on the way that their tangible elements are defined, e.g., if zero entries are allowed in tangible elements.
\end{example}

The present paper deals with commutative structures arising from commutative \nusmr s, for example, from  polynomial functions that establish a \prsfl\ \nusmr, as described below in \S\ref{ssec:ploynomials}. Therefore, 
in the sequel, \textbf{\emph{our underlying \nusmr s are always assumed to be commutative}}. (A similar but more involved theory can be developed for noncommutative\ \nusmr s, to cope also with noncommutative structures, e.g., with matrices.)
\pSkip

Although semiring ideals (Definition \ref{def:ideal.smr}) in the present paper are  mostly employed to classify special subsets of \nusmr s, rather than for factoring out substructures, for a matter of completeness and for future use, we introduce their special types.
\begin{defn}\label{def:nuprime.ideal}
An ideal $\mfa \lhd R$ of a \nusmr\ $R$  is called:
 \begin{enumerate} \eroman
   \item \textbf{ghost radical}, if for any $a^n \in \mfa$, with $n \in \Net$, the following condition holds
 $$a^n \in \mfa|_\ghs  \Dir a \in \mfa|_\ghs,  $$

   \item \textbf{ghost prime}, if for any $ab \in \mfa$ the following condition holds  $$ab \in \mfa|_\ghs \Dir a \in \mfa|_\ghs \ds{ \text{ or }} b \in \mfa|_\ghs;$$

   \item \textbf{ghost primary}, if for any $ab \in \mfa$ the following condition holds  $$ab \in \mfa|_\ghs \Dir a \in \mfa|_\ghs \ds {\text{ or }} b^n \in \mfa|_\ghs \text{ for some } n \in \Net;$$

\item \textbf{maximal}, if $\mfa$ is proper and maximal with respect to inclusion.
 \end{enumerate}
 \end{defn}
\noindent As one sees, the structural condition on these ideals applies only to ghost products; this is the curtail merit of these ideals.\footnote{These types of ideals here are different from the ideals studied in \cite{IzhakianRowen2008Resultants,IzhakianRowenIdeals}.} As will be seen later, this is a conceptual idea in our theory.

\subsection{\nudom s and \nusmf s}\label{ssec:nudomain}

\sSkip

To avoid   pathological cases, e.g., as in Example \ref{examp:pre.smr},  one can turn to more rigid $\nu$-structures which  manifest a better behavior. Henceforth, when it is clear from the context, we write $\one$, $e$,  $\zero$, for $\rone$, $e_R$, $\rzero$, respectively.

\begin{defn}\label{def:nudomain}
An element $a \notin \tG$ in a \nusmr\ $R  := (R, \tT, \tG, \nu)$ is said to be  a  \textbf{ghost divisor} if there exists an element $b \notin \tG$ such that $ab \in \tG$ or $ba  \in \tG$.
It is a \textbf{zero divisor} if $ab = ba = \zero$.
We denote the set of all ghost divisors in $R$ by $\gDiv(R)$.

A \tclsnusmr\ $R$ is called \textbf{\nudom}, if it contains no ghost divisors.\footnote{This definition changes the definition provided in \cite{Coordinates}, which is based on cancellativity of multiplication. The latter definition does not suite here, as seen in Examples \ref{examp:cancellative.gdiv} below.}
A commutative \nudom\ is called \textbf{\inudom}.
If furthermore $\tT$ is an abelian group, then $R$ is
called \textbf{\nusmf}.

\end{defn}
\noindent In the sequel, we restrict to commutative structures, and write \nudom\ for \inudom, for short.
\pSkip

Clearly any zero divisor is a ghost divisor.
Suppose that $a$  is a ghost divisor with $ab \in \tG$, where $b \notin \tG$,  then any product $ca$ with $ca \notin \tG$ is also a ghostdivisor.  This shows that $\gDiv(R) \cup \tG$ is a monoid (with $R$ commutative).
Moreover, $a$ cannot be a unit, since otherwise $a^\inv a b =b \in \tG$ -- a contradiction.
If $a$ is  \tprs\ in a \tprsmon\ $\tTP$, then $b \notin \tTP$, as~ $\tTP$ is a tangible  monoid.

\begin{rem}\label{rem:ghost.div}
     Suppose that $a,b \notin \gDiv(R)$, then $ab \notin \gDiv(R)$. Indeed, otherwise $(ab)c \in \tG$ with $c \notin \tG$ implies $a(bc)\in \tG$, where $bc \notin \tG$ since $b \notin \gDiv(R)$,
   and thus $a \in \gDiv(R)$ -- a contradiction.  
%
\end{rem}

We have the following trivial example.
\begin{example}\label{exmp:nudom}
  A definite \tclsnusmr\ has no ghost divisors, cf. Remark \ref{rem:strc.relations}, and therefore it is a \nudom.
\end{example}

In comparison to zero divisors in rings, ghost divisors in \nusmr s appear much  often, both in commutative and noncommutative \nusmr s.

\begin{example}\label{examp:cancellative.gdiv} Let $a = b + e c$ be a \expr\ sum, where $ b,c \in \tT$ are tangibles and $R$ is commutative. Then
$$(b + e c)(eb + c) = eb^2 + bc + ebc + ec^2=eb^2 + ebc + ec^2.$$
and thus $a$ is ghost divisor. Furthermore, writing
$$(b + e c)(eb + c) = (eb + e c)(eb + c) =(b + e c)(eb + ec)$$
shows that $R$ is not  cancellative with respect to multiplication, and also that unique factorization fails.

Similarly, even when  $ a= b+c$ is tangible we have
$ a^2 = (b+c)^2 = b^2 + ebc +c^2$, where
$$ (b^2 + ebc +c^2) (eb^2+bc+ec^2) = e(b^4 + b^3c + b^2 c^2 + bc^3 +   c^4).$$
Thus $a^2$ is ghost divisor (unless $ebc$ is inessential),  implying that $a$ is a ghost divisor.
\end{example}

\begin{lem}\label{lem:gdiv.in.tame}
Let $R$ be a commutative \tamesmr.
\begin{enumerate}\eroman
  \item Every $a \in R \sm (\tTPS \cup \tG)$ is a \gsdiv.
  \item $\tTPS \sm \gDiv(R)$ is a multiplicative monoid.
  \item If a product $ab$ is \tprs, then $a$ and $b$ are  both \tprs.
\end{enumerate}

\end{lem}

\begin{proof} (i): Assume first that $a \in R \sm (\tT\cup \tG)$.
  Since $R$ is \tame, we can write  $a = c + ed$, where $c,d  \in \tT$, which is a \gsdiv, by Example ~ \ref{examp:cancellative.gdiv}.
  If $a \in R \sm (\tTPS \cup \tG)$, then for some $m$ either $a^m \in \tG$ or $a^m \in R \sm (\tT \cup \tG)$, and  hence $a$ is a \gsdiv. Indeed, for the latter take the minimal $m$ so that $a^m$ is a \gsdiv\ and there is $b \notin \tG $ such that $a^m b \in \tG$. Write $a^m b = a(a^{m-1} b)$, where $a^{m-1} b \notin \tG$, since otherwise we would contradict to the minimality of $m$.
  \pSkip
  (ii): Let $a,b \in \tTPS \sm \gDiv(R)$, and assume that $ab \notin \tTPS$. The product $ab$ is not ghost since
  $a,b \in \tTPS \sm \gDiv(R)$. Then, also $bab$ is not ghost, since $b \notin \gDiv(R)$ and $ab \notin \tG$, implying that $(ab)^2 \notin \tG$, and iteratively that $(ab)^n \notin \tG$.
    So, $(ab)^n \in R \sm (\tT \cup \tG)$ for some $n$,  and thus $(ab)^n = c +ed$, with  $c,d \in \tT$. But then $(ab)^n$ is a ghost divisor by part (i), and there is $q \notin \tG$ such that $(ab)^n q \in \tG$. Now $bq \notin \tG$ since $b \notin  \gDiv(R)$, so if $a(bq) \in \tG$, then $a \in  \gDiv(R)$ -- a contradiction. Iteratively, we get that $(ab)^n q \in \tG$, where $b(ab)^{n+1} \notin \tG$, contradictively implying  $a \in \gDiv(R)$.
    \pSkip
    (iii): Suppose $a \notin \tTPS$, then there exists $m$ such that $a^m \notin \tT$, where $a^m \notin \tG$, since otherwise $(ab)^m \in \tG$.
    Hence $a^m = c + ed $, with $c, d \in \tT$. But, then $(ab)^m=a^mb^m=(c+ed)b^m=cb^m +edb^m$ -- contradicting~ \eqref{eq:tan.sum}.
%
  \end{proof}

\subsection{Supertropical semirings}\label{ssec:supertrpical}

\sSkip

So far all our $\nu$-structures have been considered to have an (additive) ghost submonoid which is partially  ordered (Definition \ref{def:nusemiring}). The next supertropical structures strengthen this property to totally ordered submonoids~ \cite{IzhakianRowen2007SuperTropical}.\footnote{The present paper develops the theory in the more abstract setting of \nusmr s and supertropical structures are brought as particular examples. Supertropical semirings were employed for an extensive development of linear and  matrix algebra \cite{IKR-LinAlg1,IKR-LinAlg2,IzhakianRowen2008Matrices, IzhakianRowen2009Equations, IzhakianRowen2010MatricesIII}.}

\begin{defn}\label{def:supertropical}
A \textbf{supertropical semiring} $R := (R, \tT, \tG, \nu)$ is a \nusmr\ whose ghost ideal is totally ordered.
 A \textbf{supertropical (integral)
domain} is a (commutative) \dfnt \ \tclsnusmr \ $R$. i.e.,   $\tTPS = \tT = R \sm \tG$ is a  abelian monoid,  in which the restriction
$\nu|_\tT: \tT \To \tG$ is onto. A \textbf{supertropical semifield} is a supertropical integral domain whose tangible monoid is  an abelian group, i.e., $\RX = \tT.$
 \end{defn}

In a supertropical domain $R$ the \tsset\ $\tTPS = \tT = R \sm \tG$ is a monoid, which  directly implies that~ $R$ has no ghost divisors (Example \ref{exmp:nudom}), and thus it is a \nudom.  Furthermore, $R$ is \tame, as it is \dfnt\ (cf. Remark \ref{rem:strc.relations}).
In addition, the \nufiber\ \eqref{eq:nu.fib} of each ghost $b\in \tG$ contains a  tangible element $a \in \tT$, not necessarily unique.


\begin{example}\label{exmp:grp.fld} Letting the ordered monoid $\tM$ in Construction
\ref{exmp:mon2smr} be an ordered abelian group, the \nusmr\ $\STR(\tM)$ is then a supertropical semifield.

Particular examples for $\STR(\tM)$ are obtained by taking $\tM$ to be $(\NetZ,+)$, $(\Z,+)$, and  $(\Q, +)$, where $+$ stands for the standard summation. The \nusmr s $\STR(\Z)$ and $\STR(\Q)$ are supertropical  semifields, while $\STR(\NetZ)$ is not a supertropical semifield, since  $\STR(\NetZ)|_\tng = \NetZ$ where $\NetZ^\times = 0$. But, it is a supertropical domain (Definition~ \ref{def:supertropical}).
\end{example}

The next  example establishes  the natural extension of the familiar (tropical) max-plus semiring $(\Real \cup \{ -\infty \} , \max, +)$  and its connection to standard tropical geometry \cite{IMS}.

\begin{example}\label{exmp:extended}
Our main supertropical example is the \textbf{extended tropical
semifield} \cite{zur05TropicalAlgebra}, that is $\Trop := \STR(\tM)$ with $\tM = (\Real, +)$ ordered traditionally, cf.  Construction \ref{exmp:mon2smr}.
Explicitly, $$\Trop := {\Real}
\cup \{- \infty \} \cup {\Real}^\nu,$$ with $\tT = {\Real}$, $\tG =
{\Real}^\nu \cup \{ -\infty  \}$, where the restriction of the ghost map $\nu|_\tT: \R \To
\R^\nu $ is the identity map  and addition and multiplication are
induced respectively by the maximum and standard summation of the
real numbers~\cite{zur05TropicalAlgebra}. The supertropical
semifield $\Trop$ extends the familiar max-plus semifield,  which is the underlying structure of  tropical geometry. It serves as a main numerical examples, which we traditionally  call \textbf{logarithmic notation} (in particular $\one = 0$ and $\zero = - \infty$).

\end{example}

The following  example presents the smallest finite supertropical semifield, extending the well-known  boolean algebra.
\begin{example}\label{exmp:superboolean}
The \textbf{superboolean semifield}  $\sbool : = (\{ \one, \zero, \one^\nu
\},  \{ \one\}, \{ \zero, \one^\nu\}, \nu
 \;)$  is a three element supertropical semifield, extending the boolean semiring  $(\{ \zero,\one\},
\vee , \wedge \;)$, endowed with the following addition and multiplication:
$$ \begin{array}{l|lll}
   + & \zero & \one & \one^\nu \\ \hline
     \zero & \zero & \one & \one^\nu \\
     \one & \one & \one^\nu & \one^\nu \\
     \one^\nu &\one^\nu & \one^\nu & \one^\nu \\
   \end{array} \qquad
\begin{array}{l|lll}
   \cdot & \zero & \one & \one^\nu \\ \hline
     \zero & \zero & \zero & \zero \\
     \one &  \zero & \one & \one^\nu \\
     \one^\nu & \zero & \one^\nu & \one^\nu \\
   \end{array}
   $$
 $\sbool$ is totally  ordered as  $ \one^\nu
>  \one
>  \zero .$
The tangible element of $\sbool$ is $\one$, while $e = \one  + \one =  \one^\nu$ is its
 {ghost} element;  $\tG := \{\zero, \one^\nu\}$
is the {ghost ideal} of $\sbool$ with the obvious  ghost map $\nu:\one \To e $.
\end{example}
In other words, by Construction \ref{exmp:mon2smr},  the superboolean semifield $\sbool$ is $\STR(\tM)$ with $\tM$ being the trivial group. This semifield suffices  for realizations of matroids, and more generally of  finite abstract simplicial complexes, as semimodules \cite{MatI,MatII}.

\begin{prop}[Frobenius Property {\cite[Proposition~3.9]{IzhakianRowen2007SuperTropical}}]\label{prop:Frobenius}
If $R$ is a supertropical semiring, then
$$ (a + b)^n = a^n + b^n, \qquad n \in \Net,$$
for any $a,b \in R$.
\end{prop}

\subsection{\nutop} \sSkip

Let $R:= (R, \tT, \tG, \nu)$ be a \nusmr. Given a nonempty subset $U \subseteq R$, using the \nufiber s \eqref{eq:nu.fib} of its members we define the subset
\begin{equation}\label{eq:nu.Fib}
 \nufib(U) := \{ b \in \nufib(a) \ds | a \in U\}.
 \end{equation}
  In particular, if $U \subseteq \tG$, then $U \subseteq \nufib(U)$. Employing these subsets, a given topology $\Om$ on the ghost ideal $\tG$ induces a topology~ $\tlOm$ on~ $R$ whose open sets are
$$ V \subseteq  \nufib(U) \text{ such that } \nu(V) = U,  \qquad  U \text{ is open  in }  \Om.$$
 Clearly,
$\tlOm$ admits continuity of semiring operations on $V$, whenever $\Om$ preserves continuity of operations on $U$.  We  call this topology the \textbf{\nutop} induced by $\Om$. A \nutop\ $\tlOm$ on $R $ is extended to  $R^{(n)}$ by taking the product \nutop\ of $\tlOm$.

\begin{example}
  In the case of the extended tropical
semifield $\Trop$ in Example \ref{exmp:extended}, the Euclidian tropology on  $\Real^\nu$ induces a  \nutop\ on $\Trop$ which admits continuity of addition. However, this induced \nutop\ is not Hausdorff, since $a$ and $a^\nu $  cannot be separated by neighbourhoods.
\end{example}

In general,  when $\tG$ is totally ordered, we can define a \nutop\ directly.
\begin{example}
   A supertropical semiring $R$, i.e., $\tG$ is totally ordered, is endowed with a \nutop\ on $R$ having the intervals
$$V_{a,b} = \{ x \in R \ds | a \nul x  \nul b \},
\qquad  a,b \in \tG,$$
as a base of its open sets.
\end{example}

\subsection{Homomorphisms of \nusemirings0}\label{ssec:smr.hom}  \sSkip

Recall from Lemma \ref{lem:nu.hom} that $\vrp(\tG) \subseteq \tG'$ for any homomorphism of \numon s $\vrp: \tM \To \tM'$ (Definition ~\ref{def:numon.hom}).

\begin{defn}\label{def:nusmr.hom} A \textbf{homomorphism} of \nusmr s
is a  semiring homomorphism
 (Definition \ref{def:smr.hom})
\begin{equation}\label{eq:nu.smr.hom} \vrp: (R, \tT, \tG, \nu)
\TO (R', \tT', \tG', \nu')\end{equation} which is also a \numon\ homomorphism. $\vrp$ is \textbf{unital}, if $\srHom(\rone) = \one_{R'}$.

A \hom\ $\vrp$ is a  \textbf{\qhom}, abbreviation for quotient  homomorphism, if $\vrp(a) \in ~\tT'$ implies $a \in \tT$ for all $\vrp(a) \in \tT'$, that is  $\ivrp(\tT') \subseteq \tT$.


The \textbf{\tcore}  and the \textbf{\ptcore} of a \qhom\ $\vrp$ are respectively  the subsets
$$\tcor(\vrp) : = \{ a \in \tT \cnd  \vrp(a) \in \tT'\}, \qquad \stcor(\vrp) : = \{ a \in \tT \cnd  \vrp(a) \in (\tT')^\PrS \}.  $$
We say that $\vrp$ is \textbf{tangibly injective}, if $\vrp(\tT) \subseteq  \tT'.$  We call $\vrp$ a
\textbf{tangibly local homomorphism}, if $\ivrp(\tT') = \tT$, namely if  $\tcor(\vrp) = \tT$.

\end{defn}

In other words, a \qhom\  $\vrp: R \To R'$ maps only tangible elements of $R$ to the \tsset\ ~ $\tT'$ of~$R'$.\footnote{This condition prevents a mapping of non-tangibles of $R$ which are not ghosts to tangibles in $R'$. } When  $\vrp$ is unital, it cannot be a ghost homomorphism (Definition \ref{def:numon.hom}) and
$\vrp(e_R) = \vrp(\one_R) + \vrp(\one_R)  = \one_{R'}+ \one_{R'} = e_{R'}$, which shows again that $\vrp(\tG) \subseteq \tG'$, since
$\vrp(a^\nu) = \vrp(e_R a ) = \vrp(e_R) \vrp(a) = e_{R'} \vrp(a)$.
 In addition,
$\vrp(\RX) \subseteq (R')^\times$, since $\vrp(\rone) = \vrp(a a^\inv) = \vrp(a) \vrp(a^\inv) = \one_{R'}$.

 By definition, we  see that $\stcor(\vrp) \subseteq \tcor(\vrp)$, where $\stcor(\vrp)$ of any \qhom\ $\vrp$ is nonempty, since $\vrp$ is unital and thus $\rone \in \stcor(\vrp)$.
\begin{lem}\label{lem:stcore}
  $\stcor(\vrp) \subset \tTPS$ for any \qhom\ $\vrp: R \To R'$.
\end{lem}
\begin{proof}
  Assume  $ a \in \stcor(\vrp) $ is not \tprs, then $a^n \notin \tT$ for some $n$, and thus $a^n \notin \ivrp(\tT')$, as $\vrp$ is a \qhom. On the other hand, $\vrp(a^n) = \vrp(a)^n \in (\tT')^\PrS$, and therefore $a^n \in  \ivrp(\tT')$ -- a contradiction.
\end{proof}

 \begin{prop}\label{prop:rpim} Let $\vrp: R \To R'$ be a \qhom.   Then,
  $\ivrp((\tT')^\Pr) \subseteq \tTPS$ is a monoid for any \tmon\ $(\tT')^\Pr \subset (\tT')^\PrS$.
\end{prop}
\begin{proof}
  $\ivrp((\tT')^\PrS) \subseteq \tTPS$ by the Lemma \ref{lem:stcore}.  In particular, $\rone \in \stcor(\vrp)$,
 since $\vrp$ is a \qhom.  Assume that  $a,b \in \ivrp((\tT')^\Pr)$, i.e., $\vrp(a), \vrp(b)  \in (\tT')^\Pr$.  Then,
$  \vrp(ab) = \vrp(a)\vrp(b)  \in (\tT')^\Pr$, hence  $ab \in \ivrp((\tT')^\Pr)$, implying that $\ivrp((\tT')^\Pr)$ is a monoid.
\end{proof}

\begin{cor}\label{cor:rpim}
If $\vrp: R \To R' $ is a \qhom, where $R'$ is a \tclsnusmr, then  $\tcor(\vrp) = \stcor(\vrp)$ is a monoid.
\end{cor}

\begin{example}
The superboolean semifield $\sbool$ (Example \ref{exmp:superboolean}) embeds naturally in any \nusmr \ $R$ via
$$ \iota: \sbool \TO R, \qquad \one \longmapsto  \rone, \ \zero \longmapsto \zero_R, \  \one^\nu \longmapsto e_R. $$
Note that the surjective map $\psi:  R \To  \sbool,  $
given by $a \Mto  \one $ for any $a\in \tT$,  $\rzero  \Mto  \zero $, and $a \Mto  \one^\nu $ for every $a \in R \sm \tTz$, is not a homomorphism of \nusmr s.
\end{example}

\begin{example}
If $\vrp: R \To R'$ is a ghost injective \hom\ (Definition \ref{def:numon.hom}) of supertropical domains,
then $\vrp|_\tT: \tT \To \tT'$ is injective.
 Indeed, if $a \neq b$, say with tangibles $a \nug b$,   such that $\vrp(a) = \vrp(b)$, then
$$ \vrp(a^\nu) = \vrp(a)^{\nu'} = \vrp(a) + \vrp(b) = \vrp (a+b) = \vrp(a),$$
so $\vrp(a) \in \tG'$-- a contradiction. 
\end{example}

It is easy to verify that the ghost kernel $\gker(\vrp)$ (Definition \ref{def:numon.hom})
of a \nusmr\ \hom\  $\vrp: R \To R'$   is a semiring ideal containing the ghost ideal $\tG$ of $R$.
Clearly,  $\gker(\vrp) \cap \tcor(\vrp) =  \emptyset$ for any \qhom\ $\vrp$ of \nusmr s.

\begin{definition}\label{def:nusmr.cag}
The category $\NSMR$ of \nusmr s, is defined to be the category whose objects are \nusmr s (Definition \ref{def:nusemiring}) and whose morphisms are \qhom s (Definition \ref{def:nusmr.hom}).
\end{definition}

In what follows,  all our objects are taken from the category $\NSMR$.

\subsection{Localization of \nusmr s}\label{ssec:tngLocal} \sSkip

Let $R := (R,\tT,\tG,\nu)$ be a commutative \nusmr, and let $\MS \subset R$ be a  multiplicative
  submonoid  with $\rone \in \MS$. 
     We always assume that $\MS$ is not pointed, i.e., that $\rzero \notin \MS$.
  When it is clear from the context, we write~ $\one$  and $\zero$ for $\rone$ and~ $\rzero$, for short.

    We define the \textbf{localization} $\iMS R$ of $R$ by $\MS$ as the monoid localization by a non-cancellative submonoid, as described in  \S\ref{sec:monLocal}. To wit,  this localization is
determined by 
the equivalence  $\sim_\MS$ on $R \times \MS $ given as
\begin{equation}\label{eq:ms:1}
   (a,c) \sim_\MS  (a',c') \dss{\text{ iff} } ac'c'' =
a'cc'' \quad  \text{for some } c'' \in \MS,
\end{equation}
 written $\frac{a}{c} = \frac{a'}{c'}$.  The addition and multiplication of $\iMS R$ are defined  respectively via
$$\frac {a_1}{c_1} + \frac{a_2}{c_2} = \frac{c_2a_1
+ c_1 a_2}{c_1c_2}, \qquad \frac {a_1}{c_1}  \frac{a_2}{c_2} = \frac{a_1 a_2}{c_1c_2},$$  for $a_1,a_2 \in R,$ $c_1,c_2 \in \MS$.
Then,
$\iMS R$ could become a \nusmr\ by defining  $\frac {a_i}
{c_j} $ to be tangible if and only if both $a_i$  and ~ $c_j$ are tangibles in~$R$, and letting $\frac {a_i}
{c_j} $ be ghost if  $a_i$  or $c_j$ is ghost in~$R$. But then, the elements $\frac a c $ are not necessarily  invertible, as $c\in \MS$ could be non-tangible.
Moreover, taking an arbitrary monoid $\MS$ does not suit here, as seen by the next remark.

\begin{remark} The idea of localization by non-tangible elements has a major defect.
For example, suppose that $a = p +q e $ is a non-tangible element in  $\MS$, where $p, q \in \tT$. In this case we would have
$$ \one = \frac {a}{a} = \frac{p +eq } {p +eq } =  \frac{p +eq +eq} {p +eq } =  \frac{p +eq } {p +eq } +  \frac{eq } {p +eq } = \one +  e \bigg( \frac{q } {p +eq } \bigg),$$
     which for a large $q$ contradicts \emph{\PSRb}, cf. \eqref{eq:tan.sum}. (In addition, $\one$ is then a \gdiv\ by Example \ref{examp:cancellative.gdiv}, implying contradictively  that $b = \one b  = \ghost$ for some non-ghost $b \in R \sm \tG$.)
     Similarly, for $a = eq$ we would have  $\one = \frac aa =  \frac{eq} {eq} = \frac{eeq} {eq} = e \frac{eq} {eq} = e$, introducing a contradiction again.
\end{remark}

With this nature, to ensure that a localized \nusmr\ is well defined, where the elements of $\MS$ become units in the  localization $R_\MS$, initially,  all the members of $\MS$ must be tangibles. More precisely, they must be \tprs s, i.e., $\MS \subseteq \tTPS$,   since $\MS$ should be a tangible monoid.
\begin{defn}\label{def:tangible.localization}
When  $\MS \subseteq \tT$ is a tangible submonoid, we say that $\iMS R$ is a \textbf{tangible localization} of ~$R$.
If $\MS = \tT$ (and hence $\tT = \tTP$), then
$\iMS R$ is called the \textbf{\nusmr\ of fractions} of $R$ and is
denoted by $Q(R).$ When $R$ is a \nudom, $\iMS R$ is called the \textbf{\nusmf\ of fractions} of $R$.
\end{defn}
\noindent
 Note that $\MS$ may contain \gdiv s, as far as it is a tangible monoid. \emph{Henceforth, we assume that $\MS \subseteq \tTPS$ is a tangible monoid.}


The canonical
\qhom\ (Definition \ref{def:nusmr.hom}) is given by
\begin{equation}\label{eq:localbyMS}
\tau_\MS : R \TO \iMS R,  \qquad a \longmapsto \frac a
\one,
\end{equation}
and is an injection. 
We identify $\iMS R$ with the \nusmr\
$(\iMS R,\; \iMS \tT,\;  \iMS \tG, \;  \nu')$, whose ghost map is given by $\nu' : \frac
{a} {c} \Mto \frac {a^\nu} {c}$, and write
$$R_\MS:= \iMS R$$
for the localization of $R$ by $\MS$.
When the multiplicative submonoid $\MS$ is generated by a single  element $a\in R$, i.e., $ \MS = \{\one, a, a^2, a^3, \dots  \}$, we sometimes write $R_a$ for $R_{\MS}$.

\begin{rem}\label{rem:R2RC} Let $\tau_\MS: R \To R_\MS$ be the (canonical) injective \qhom\  \eqref{eq:localbyMS}.
\begin{enumerate} \eroman
  \item The ghost kernel of $\tau_\MS$ is  determined as $$ \gker(\tau_\MS) = \{ a \in R \ds | ac \in \tGz \text{ for some } c \in \MS \}. $$ Indeed, $\frac{a}{\one} = \frac{b^\nu}{\one}$ is ghost in $R_\MS$ iff  $ac = b^\nu c $ for some $c \in C$,
but $b^\nu c = (bc)^\nu \in \tGz$ is ghost, implying that  $ac \in \tGz$.

  \item For all $c \in \MS$, $\tau_\MS(c)$ is a tangible unit in $R_\MS$, i.e., $\tau_\MS(c) \in (R_\MS)^\times$.

\item $\tau_\MS$ is bijective, if $\MS$ consists of tangible units in $R$.
\end{enumerate}
\end{rem}

For a \qhom\ $\vrp:R \To R'$ of \nusmr s, where $R'$ is a \tclsnusmr,  we define the tangible localization $R_\vrp$ of $R$ by $\vrp$ to be
\begin{equation}\label{eq:homLocal}
R_\vrp := R_{\stcor(\vrp)}= (\stcor(\vrp))^\inv R.
\end{equation}
It is well defined, since $\stcor(\vrp) \subset R$ is a multiplicative tangible submonoid by
Corollary \ref{prop:rpim}.

\begin{prop}[Universal property of tangible localization]\label{prop:local.univ}
  Let $R$ be a \nusmr, and let $R_\MS$ be its tangible localization by a (multiplicative) submonoid $\MS \subseteq \tT$.
The canonical \qhom\ $\tau_\MS : R \To~ R_\MS$ satisfies $\tau_\MS (\MS) \subseteq  (R_\MS)^\times$  and it is
universal: For any \nusmr\ \qhom\ $\vrp: R \To S$ with $\vrp(\MS) \subseteq S^\times$  there is a unique \nusmr\ \qhom\ $\htvrp : R_\MS \To  S$
such that the diagram
$$\xymatrix{
R \ar@{->}[rrd]_{\vrp } \ar@{->}[rr]^{\tau_\MS} && R_\MS \ar@{..>}[d]^{\htvrp} \\
 && S \\
}
$$
 commutates. Furthermore, if $\vrp : R \To S$ satisfies  the same universal property as $\tau_\MS$
does, then the \qhom\ $\htvrp  : R_\MS \To S$ is an isomorphism.
\end{prop}

\begin{proof} \emph{\underline{Uniqueness}}:
For $a \in R$ and $c \in \MS$, we have
$\vrp(a)   = \htvrp\big( \frac{a}{\one}\big) = \htvrp\big( \frac{a}{c} \frac{c}{\one}\big) = \htvrp\big( \frac{a}{c}\big) \vrp(c), $
hence $ \htvrp\big( \frac{a}{c}\big) = \vrp(a)\vrp^\inv(c)$ since $ \vrp(\MS) \subseteq S^\times$,  implying that $\htvrp$ is uniquely determined by $\vrp$. \pSkip
\emph{\underline{Existence of a \qhom}} $\htvrp : R_\MS \To S$: Set
$ \htvrp \big( \frac{a}{c}\big) =\vrp(a)\vrp(c)^\inv$, with $a \in R$ and $c \in \MS$. 
Suppose $ \frac a c = \frac{a'}{c'}$, that is  $ac'c'' =  a'c c''$ for some $c'' \in  \MS$. Then
$ \vrp (a)\vrp(c') \vrp(c'') =  \vrp(a')\vrp(c)  \vrp (c'') $, which implies   $ \vrp (a)\vrp(c')  =  \vrp(a')\vrp(c) $, since $\vrp(c'')$ is a unit in $S$, and this is equivalent to
$\vrp(a)\vrp(c)^\inv  =  \vrp(a')\vrp(c')^\inv .$
Hence, $\htvrp  : R_\MS \To S$ is  well-defined, and it is easily checked that $\htvrp$ satisfies  $\vrp  = \htvrp \circ  \tau_\MS $.

Assume that both $\tau_\MS$ and $\vrp$ are universal in the above  sense.
Then, $\htvrp : R_\MS \To S $ satisfies  $\vrp = \htvrp \circ \tau_\MS$, and there is a \qhom\
$\phi  : S  \To R_\MS$ such that $\tau_\MS = \phi \circ \vrp$. Applying the uniqueness part to
$$\id_{S} \circ \vrp  = \htvrp \circ \tau_\MS = (\htvrp \circ \phi) \circ \vrp, \qquad  \id_{R_\MS}
\circ \tau_\MS = \phi \circ \vrp = (\phi \circ \htvrp) \circ \tau_\MS,$$
we conclude that  $\htvrp  \circ \phi = \id_{S}$ and  $\phi \circ \htvrp  = \id_{R_\MS}$. Consequently, $\htvrp  : R_\MS \To S$
is an isomorphism.
\end{proof}

\begin{cor}
  Suppose $\MS_2 \subset \MS_1 $ are multiplicative submonoids of $R$, then $R_{\MS_1}$ is isomorphic to $(R_{\MS_2})_{\tau_{\MS_2}(\MS_1)}$.
\end{cor}

\begin{proof} Once we have the universal property of tangible localization in Proposition \ref{prop:local.univ}, the proof is similar to the  case of rings, e.g., \cite[Propoisition 7.4]{Jac2}.
\end{proof}

\begin{lem}\label{lem:irr.powers}
  Suppose $a \in \tTPS \sm \gDiv(R)$ is irreducible \tprs, where $R$ is a \tame\ \nusmr.
  If $a^n = bc$, then $b= u a^s$ and  $c= v a^t$ with $u,v$~ units.
\end{lem}
\begin{proof} As $a$ is irreducible, we may assume that $a$ cannot be written as a product of two elements,  otherwise we can divide by one terms which is a unit. Since $a$ is \tprs \ and $R$ is \tame,  $b$ and $c$ are also  \tprs\ by  Lemma~ \ref{lem:gdiv.in.tame}.(iii). Localizing by $b$, we have $\frac{a^n} {b} = \frac{c}{\one}$, implying  that $\frac{a^s}{b} = \one$ for some~ $s$, since $a$ is irreducible; hence $a^s = b$.
\end{proof}

\subsection{Functions and polynomials}\label{ssec:ploynomials}
\sSkip

We repeat some basic definitions from \cite{Coordinates}, for the
reader's convenience. Given a semiring $R$, in the usual way via point-wise addition and
multiplication, we define  the semiring  $\Fun (X, R)$ of set-theoretic
functions from a set $X$ to~$R$. As customarily, we write $f|_Y$ for the restriction of a function ~$f$ to a
nonempty subset $Y$ of~$X$.

For a \nusmr\ $R:= (R, \tT, \tG, \nu)$,  $\Fun (X, R)$ is a
\nusmr\ whose ghost elements are functions defined as $f^\nu(x) = (f(x))^\nu$, for all $x \in X$.
Defining tangible functions is more subtle, and includes several possibilities \cite{Coordinates}. We elaborate this issue   later.

\begin{defn}\label{def:ghost.locus}
The \textbf{ghost locus} of a function $f \in \Fun (X, R)$, with $R$ a \nusmr,  is defined as
$$ \tZ(f) := \{ x \in X \cnd  f(x) \in \tG \}. $$
 When $f$ is determined by a polynomial, $\tZ(f) $ is  called an \textbf{algebraic set}.
\end{defn}
\noindent
An equivalent way to define the ghost locus of $f$ is by
\begin{equation}\label{eq:Zf.2}
\tZ(f) := \{ x \in X \cnd  f(x) = f^\nu(x) \}.
\end{equation}
By this definition we see that $X = \tZ(f^\nu)$ for every  ghost function $f^\nu$ over any set $X$.
When $X \subset \Rn$, the distinguishing of tangible subsets and ghost subsets in $X$ can be made by  a point-wise classification.

\pSkip

 Polynomials   in $n$ indeterminates  $\Lm := \{ \lm_1,
\dots, \lm_n\}$  over a \nusmr \ $R$ are  defined as customarily by formulas
$$ \sum_{\bfi \in I} \al_\bfi \Lm^\bfi, \qquad \al_\bfi \in R,$$
where $I \subset \Net^{n}$ is a finite subset and $\bfi = (i_1, \dots, i_n)$ is a multi-index.
 $R[\La] := R[\lm_1, \dots, \lm_n]$ denotes the \nusmr \ of all polynomials over  $R$, whose addition and multiplication are induced from  $R$ in the standard way.
 A polynomial $f \in R[\Lm]$ is  a  \textbf{tangible polynomial}, if $\al_{\bfi} \in  \tT$ for all $\bfi \in I$. $f$ is a \textbf{ghost  polynomial}, if~ $\al_{\bfi} \in \tG$ for all $\bfi \in I$.
(Note that, as a function, a tangible polynomial does not necessarily take tangible values everywhere).

\begin{rem}\label{rem:tame.poly} Clearly, if $R$ is a \tame\ \nusmr, then $R[\Lm]$ is also \tame, cf. Example \ref{examp:pre.smr}.(i).
\end{rem}

The polynomial \nusmr\ $R[\Lm]$ is not a \tclsnusmr\ (resp. \dfnt \ \nusmr),  even if $R$ is \tcls\ (resp. \dfnt),  as a product (or even powers) of  tangible polynomials can be non-tangible (e.g. $(\lm +a)^2 = \lm^2 + a^\nu \lm + a^2$, and  $\lm + a$ is not \tprs). Therefore, tangible polynomials do not constitute  a monoid.
 To resolve this drawback, we view polynomials as functions under the natural map $$\phi: R[\Lm] \TO \Fun(X,R),$$ defined
by sending a polynomial $f$ to the function $\tlf: \bfa \Mto f(\bfa)$, where $\bfa = (a_1, \dots, a_n)\in X \subset \Rn.$

We denote the image of $R[\Lm]$ in $\Fun(X,R)$ by $\Pol (X,R)$, which is a \nussmr\ of $\Fun(X,R)$.
The map~$\phi$ induces a natural congruence  $\fCong_X$ on $X$, whose underlying equivalence  $\fcong_X$ is determined by
$$ f \fcong_X g \ds\Iff \tlf|_X = \tlg|_X.$$
Accordingly, the \nusmr\ $\Pol (X,R)$ is isomorphic to  $R[\Lm] / \fCong_X$, whose elements are termed  \textbf{polynomial functions}.\footnote{Polynomial functions have been studied in \cite{IzhakianRowen2007SuperTropical}, and yielded a supertropical  version of Hilbert Nullstellensatz. Their restriction to an algebraic set forms a coordinate semiring \cite{Coordinates}. This view provides an  algebraic formulation, given in terms of congruences, for the balancing condition used in tropical geometry \cite{IMS}.
}

In this setting, an element  $[f ]\in R[\Lm] / \fCong_X$ is
 \begin{itemize}\dispace
   \item  a \textbf{tangible polynomial function} if and only if
$f \fcong_X h$ only to tangible polynomials $h \in R[\Lm]$,
 \item a \textbf{ghost polynomial function} if and only if  $f \fcong_X g$ for some ghost  polynomial $g \in R[\Lm]$,

   \item a \textbf{zero function} if and only if $f \fcong_X \zero $.
 \end{itemize}
We write $f|_X \nucong  g|_X$, if $f(\bfa) \nucong g(\bfa)$ for all $\bfa \in X.$
\pSkip

As a consequence of these definitions, if $R$ is a \tclsnusmr, then $R[\Lm] / \fCong_X$ is also \tcls. 
In particular, $F[\Lm] / \fCong_X$  is \tcls\ for any \nusmf\ $F$.

\begin{rem} Let $f \in R[\Lm] / \fCong_X$ be a polynomial function.
 \begin{enumerate}\eroman
   \item  The property of being tangible for $f$ depends on the domain $X$. A polynomial function $f$ can be tangible over $X$ but non-tangible (or even ghost) over a  subset $Y \subset X$, for example take $Y = \tZ(f)$. In particular,  a tangible polynomial can be a ghost polynomial function over some  subsets  $X \subset \Rn$.
   \item If $f$ is a ghost function over $X$, then it is ghost over any subset $Y \subseteq X$.
   \item If a product $fg$ is ghost, where $f$ is tangible, then $g$ is ghost. However, we may have a ghost product of two non-ghost functions, i.e., a ghost divisor  as in Example \ref{examp:cancellative.gdiv}.
   \item $f \fcong_X \zero$ iff $f = \zero$, independently on the subset $X$, as long as  $X \neq \{ \zero \}$.
 \end{enumerate}

\end{rem}

When $X = \Rn$, we denote by $\tlR[\Lm]$ the \nusmr\ $ \Pol(X,R) =  R[\Lm] / \fCong_X$  of polynomial functions. Similarly, we write  $\tlF[\Lm]$ for the \tclsnusmr\ $F[\Lm] / \fCong_X$ with $F$ a \nusmf.  (In this case, every tangible polynomial is assigned with a tangible polynomial function.) The restriction of $\tlF[\Lm]$ to a subset $Y \subset X$, denoted as $\tlF[Y]$, is  considered as the \textbf{coordinate \nusmr}  of $Y$, cf.  \cite{Coordinates}.

\begin{rem}\label{rmk:pointFun} Given a subset $Y  \subseteq  X$ and  a point $\bfa \in
Y$, we  define the ``evaluation mapping'' $$\eps_\bfa : \tlR[Y] \ds \TO R$$ by $\eps_\bfa (f) = f (\bfa)$.
If $\eps_\bfa = \eps_\bfb$ for $\bfa, \bfb \in X$, then $\bfa =
\bfb.$ (Indeed, if $\eps_\bfa = \eps_\bfb$, then $f(\bfa) =
f(\bfb)$ for every $f \in \tlR[X]$, and in particular for $f(\bfx) =
\bfx$.)
The map $\eps_\bfa$ is a \qhom\ of \nusmr s, for which an element $\bfa \in \tZ(f)$ iff  $f \in \gker(\eps_\bfa).$

Any \qhom\ $\vrp : \tlR[Y] \To \tlR[Z]$ defines a morphism $\phi :
Z \To Y$. Indeed, the composition   $\eps_\bfa \circ \vrp $ is clearly a
 function $\tlR[Y] \To R$, and thus equals $\eps_\bfb$ for some $\bfb \in
 Y$.   Then, the  identification $\phi$ sending $\bfa \Mto \bfb$ defines a
 morphism.

\end{rem}

The view of polynomial functions over supertropical semirings provides the \textbf{Frobenius Property}, written  as
\begin{equation}\label{eq:Frobenius}
    f^m = \bigg(\sum_\bfi a_\bfi \Lm^\bfi \bigg)^m = \sum_\bfi (a_\bfi \Lm^\bfi)^m,
\end{equation}
which is a direct consequence of the point-wise computation in Proposition \ref{prop:Frobenius}. 

\begin{thm}[{\cite[Theorem 8.35]{IzhakianRowen2007SuperTropical}}]\label{thm:poly.factorization}
Let $F$ be divisibly closed supertropical
semifield. Then, any polynomial function $f \in \tlF[\lm]$ factorizes uniquely to a product of linear and quadratic terms of the form
$$ \lm + a, \quad e \lm + a, \quad \lm + a^\nu, \quad  \lm^2 + b^\nu\lm +a, \qquad \text{ with } b > \sqrt{a}.$$
These terms are called  primitive elements.
\end{thm}
The extended tropical semiring $\Trop$ in Example \ref{exmp:extended} \ is a supertropical semifield that admits the conditions of the theorem.
However, in general, unique factorization fails for \nusmr s (as seen in  Example~ \ref{examp:cancellative.gdiv}), especially for supertropical polynomial functions, even for tangible functions.

\begin{example}\label{examp:non-unique-factorization} Consider the \nusmr\ of  polynomial functions $\tlF[\Lm]$ over a \nusmf\ $F$.
\begin{enumerate}\eroman
  \item
 The  tangible polynomial function \footnote{This example holds  also for polynomials  over the standard tropical semiring.}  $$\begin{array}{lll}
f  = g_1 g_2 & = (\lm_1 + \lm_2 + \one)(\lm_1 + \lm_2 + \lm_1 \lm_2 ) \\[1mm]                                                                         & = \lm_1^2 \lm_2 + \lm_1 \lm_2^2 + \lm_1^2 + \lm_2^2  + e \lm_1 \lm_2  + \lm_1 + \lm_2 \\[1mm]
                                                & = (\lm_1 + \one) (\lm_2 + \one)(\lm_1 + \lm_2 )   = h_1 h_2 h_3                                                                      \end{array}$$
has two different factorizations $g_1 g_2$ and $h_1  h_2 h_3$  \cite[Example 5.22]{Coordinates}. Note that $h_i \lsset g_1$ for $i = 1,2,3$, cf. \eqref{eq:lorder}, but $\tZ(h_i) \nsubseteq \tZ(g_1)$.
  \item The square of the linear function $f = e\lm_1 + e\lm_2 +\one$ can be written as
  $$(e\lm_1 + e\lm_2 +\one)^2 = (e\lm_1 + \lm_2 +\one)(\lm_1 + e\lm_2 +\one)= g_1g_2,$$
  although $f$ is irreducible. But still $f \nucong g_i$, for $i =1,2$.
\end{enumerate}
\end{example}

By Lemma \ref{lem:irr.powers}, for powers of tangible polynomial functions in $ \tlF[\Lm]$, with $\tlF$ a \tame\ \nusmf, we have the following property.
\begin{rem}\label{rem:power.of.tng.irreducible}
  Assume that $f\in \tlF[\Lm]$ is a tangible irreducible which is not a ghost divisor;  thus $f$ is \tprs\ by Lemma \ref{lem:gdiv.in.tame}. If $f^m = gh$,  then $g= a f^s$ and  $h= bf^t$ for some units   $a,b \in \RX$, cf. Lemma ~\ref{lem:irr.powers}. 
  More generally, this holds for  $f \in \tlR[\Lm]$, where $R$ is a \tame\ \tclsnusmr. 
\end{rem}

Considering the ghost locus $\tZ(f)$ in Definition \ref{def:ghost.locus} as the root set of a polynomial function $f \in \tlR[\lm]$, a very easy analog of the \emph{Fundamental Theorem of Algebra} is obtained.
\begin{prop} Over a \nusmr\ $R$,  every $f \in \tlR[\lm]$ which  is not a tangible constant has a root.
\end{prop}
\begin{proof}
  As a function, $f(a) \in \tG$ for a  large enough $a \in \tG$,  unless $f$ is a tangible constant.
\end{proof}

Restricting to tangible roots, we recall Proposition 5.8 from  \cite{IzhakianRowen2007SuperTropical}:
\begin{prop}Over a divisibly closed supertropical semifield $F$, every  $f \in \tlF[\lm]$
which is not a tangible monomial has a  tangible root.
\end{prop}

The present paper develops the theory in the general framework of \nusmr s and the above results serve later as particular examples.

\subsection{Hyperfields and valuations}\label{ssec:valuation}\sSkip

The link of classical theory to  \nusmr s to  is established by  valuations.
Recall that a \textbf{valuation}  of a valued field $\Fld$ is a map $$\val : \Fld \TO \tTz := \tT \cup \{ \zero \}, \qquad (\zero:= - \infty),  $$ where $\tT := (\tT, + )$ is a totally ordered abelian group,  that   satisfies
\begin{enumerate} \ealph
  \item $\val(f \cdot g)  =  \val(f) + \val(g)$;
  \item $\val(f + g)  \leq  \max \{ \val(f), \val(g)\}$  with equality if  $\val(f)  \neq \val(g)$.
\end{enumerate}
By Example \ref{exmp:grp.fld},  the ordered group $\tT$  extends to the supertropical semifield $\STR(\tT) = (\tTz \cup \tT^\nu, \tT, \tTz^\nu, \nu )$, where   $\val: \Fld \To  \tTz$ gives the supervaluation
$\sval: \Fld \To \STR(\tT)$, sending $\Fld$ to the tangible submonoid of $\STR(\tT)$, as studied in \cite{IKR1}.

A valuation in general is not a homomorphism, as it does not
 preserve associativity. Yet, one wants to realize this map at least as a ``homomorphic relation''. To receive  such realization, we  view $\STR(\tT)$, usually with $\tT = \Real$, as a hyperfield \cite{Baker,Viro}. This is done by assigning every tangible   $a \in \tTz$ with the singleton $P _{a} := \{ a \} \subset \tTz$, while  each ghost   $a^\nu \in \tT^\nu$ is associated to the subset $P_{a^\nu} := \{ b \in \tTz  \cnd b \leq a\}  \subset \tTz$. The hyperfield operations are induced from the operations of $\STR(\tT)$:
 $$ P_{x} + P_y := P_{x+y}\, , \qquad P_{x} \cdot P_y := P_{xy} \, . $$
  For $x \neq y$, this construction provides the inclusions
\begin{equation}\label{eq:pContainements}
  P_{x}  \subseteq P_{y}  \dss{\text{iff}} x \nule y,
  \qquad \text{for all } x \in \STR(\tT), y \in \tT^\nu  ,
\end{equation} while
 the (non-unique) inclusion $\val(f) \in
P_x$  gives the binary relation
\begin{equation}\label{eq:homomorphicRelation}
  \val(f) \in P_x \dss{\text{ or}}  \val(f) \notin P_x,
\end{equation}
for every $x \in \STR(\tT)$ and $f \in \Fld$.
\begin{prop}\label{thm:homomorphicRelation}
The relation \eqref{eq:homomorphicRelation} is homomorphic in the sense that $$\val(f g) \in P_{x  y} \dss{\text{ and}} \val(f
+ g) \in P_{x + y},$$  for $\val(f) \in P_x$, $\val(g)
\in P_y$.
\end{prop}
\begin{proof}
Let  $ a = \val(f) $, $b = \val(g) $. Then,
$
  \val(f  g)  =  \val(f) + \val(g) =
   a  b  \in  P_{a  b} ,$ and
$$
\begin{array}{lrrl}
  \val(f + g) \leq   \max \{ a, b\}  =
 \left\{
\begin{array}{lllllll}
  a \in P_{a}  = P_{a + b},  &  a >  b,  \\
  a \in  P_{a}  \subset  P_{a +  a} = P_{a^\nu},&  a  = b,\\
  b \in P_{b}  = P_{a + b},  &  b > a.
\end{array}
\right.
\end{array}
$$ \vskip -5mm
\end{proof}

\begin{example}
Consider the supertropical semifield  $\Trop = \STR((\Real,+))$ in  Example \ref{exmp:extended}. To realize $\Trop$ as a hyperfield over $\Real$, each $a \in \Real \cup \{ - \infty \} $ is assigned with the one-element set $\{ a \} \subset \Real \cup \{ - \infty \} $, while $b \in \Real^\nu$ is assigned with the closed  ray $[-\infty, a] \subset \Real \cup \{ - \infty \}  $.
\end{example}

\subsection{A view to polyhedral geometry}\label{ssec:polyhedral.geomtry}
\sSkip

Traditional tropical varieties may be obtained  by  taking different viewpoints, as outlined below, see e.g. \cite{Gat}. We write
$\zReal$ for the max-plus semiring  $\zReal := \{\Real \cup \{ -\infty \}, \max, + )$, where $\zRealn$ stands for the cartesian product of $n$ copies of  $\zReal$.

 The \textbf{amoeba} of a complex affine variety $Y = \{ ( z_1, \dots, z_n) \ds| z_i \in
\Comp\}$  is  defined as
$$\tA_t(Y) = \{ (\log_t| z_1|, \dots, \log_t|  z_n|) \ds |
( z_1, \dots, z_n) \in Y \} \subset \zReal^{(n)},$$ which  by taking limit $t \to 0$
degenerates to a \textbf{non-Archimedean
amoeba} $\tA_0$  in the $n$-space over the max-plus semiring  $\zReal$, cf. \cite{Gat}.
$\tA_0$ is a finite polyhedral complex of pure dimension, i.e., all its maximal faces (termed facets) have the same dimension.
This symplectic  viewpoint leads to the topological definition \cite{IMS}: a \textbf{tropical
variety} $X \subset \Real^{(n)}$ is a finite rational polyhedral
complex of pure dimension whose weighted facets  $\delta$  carry positive integral values  $m(\delta)$
such that for each face $\sigma$  of codimension $1$ in $X$
the following \textbf{balancing
condition} holds
\begin{equation}\label{eq:balanceCond} \sum_{\sigma\subset\delta}m(\delta)n_\sigma(\delta)=0 ,\end{equation}
where $\delta$ runs over all  facets of $X$
containing $\sigma$, and $n_\sigma(\delta)$ is the primitive unit
 vector normal to $\sigma$ lying in the cone centered at $\sigma$
and directed by $\delta$. Thereby, a tropical hypersurface must have (topological) dimension $n-1$.

The ``combinatorial--algebraic'' approach to tropical varieties starts with a polynomial $\reF$
over the max-plus semiring $\zReal$, which determines a piecewise linear convex function
$\reF : \Real^{(n)} \To \Real$. Its domain of
non-differentiability $\Cor({\reF})$, called \textbf{corner locus}, defines a tropical hypersurface.
In combinatorial sense, $\Cor({\reF})$ is the set of
points in $\Real^{(n)}$ on which the evaluation of $\reF$\ is attained by at least
two of its monomials. Yet, this formalism is not purely algebraic.

 Valuations as described in \S\ref{ssec:valuation} give a direct passage
from classical algebraic varieties  to tropical varieties  \cite{MS}.
For example, take  $\tT = \Real$  to be the valued group of  the
field $\Fld$ of Puiseux series  $ p(t) =
\sum_{q \in Q} c_{q} t ^{q},$ with $c_q \in \Comp$
and  $Q \subseteq \Rati$  bounded from below, where $\val: \Fld \To \zReal$ is given by
\begin{equation*}\label{eq:valPowerSeries}
  \val(p(t))\  := \
\left\{%
\begin{array}{ll}
    - \min \{q \in Q \cnd c_{q} \neq 0 \}, &  p(t) \in
    \Fld^\times,
    \\[1mm]
    -\infty, & p(t) = 0. \\
\end{array}%
\right.
\end{equation*}
A tropical variety is now defined as the closure $\overline{\val(Y)}$ of  a subvariety $Y$  of the torus $(\Fld^\times)^{(n)}$, where  $\val$ is applied  coordinate-wise to $Y$. A parallel way  is to tropicalize the generating elements of the ideal that determines $Y$ and then to consider the intersection of their corner loci.

Supertropical theory provides a purely algebraic  way to capture tropical varieties as ghost loci of systems of polynomials (Definition \ref{def:ghost.locus}). In this setting, standard tropical varieties are a subclass of ghost loci obtained as the tangible domains of tangible polynomial functions \cite{IzhakianRowen2007SuperTropical}. Furthermore,  ghost loci  allow to frame a larger family  of polyhedral objects, including objects whose dimension  equals to that of their  ambient space.

\ifdef{\withFigures}{
\begin{figure}[h]
\setlength{\unitlength}{0.5cm}
\begin{picture}(10,9)(0,0)
\grid
\thicklines
\put(6,6){\line(1,1){2.5}}
\put(4,6){\line(-2,1){2.5}}
\put(6,4){\line(1,-2){1.2}}
\pspolygon[fillstyle=none,fillcolor=lgray](2,3)(3,2)(3,3)
\put(4,0){(a)}
\end{picture}
\begin{picture}(10,7)(0,0)
\grid
\thicklines
\put(6,6){\line(1,1){2.5}}
\put(4,6){\line(-2,1){2.5}}
\put(6,4){\line(1,-2){1.2}}
\pspolygon[fillstyle=solid,fillcolor=lgray](2,3)(3,2)(3,3)
\put(4,0){(b)}
\end{picture}
\begin{picture}(10,9)(0,0)
\grid
\thicklines
\pspolygon[fillstyle=solid,fillcolor=lgray](2,3)(3,2)(3,3)

\put(4,0){(c)}
\end{picture}

\caption{Tangible parts of supertropical algebraic sets.}\label{fig:2}
\end{figure}
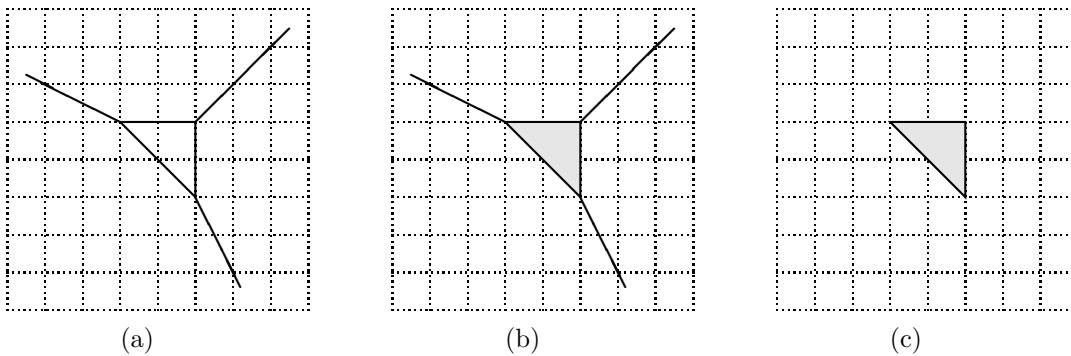
}

\begin{example}\label{examp:varieties}
Let $f= \lm_1^2\lm_2 + \lm_1 \lm_2^2+ \al \lm_1 \lm_2 + 0$ be a polynomial in $\Trop[\lm_1, \lm_2]$, where  $\Trop$ is the extended tropical semiring of Example \ref{exmp:extended}. Assume $\al \nul 0$,  and let $X = \tZ(f)$ be the  ghost locus of $f$.
\begin{enumerate} \eroman
  \item For a tangible $\al \in \tT = \Real$, the restriction $X|_\tng$ is a standard tropical elliptic curve, as described in Figure \ref{fig:2}.(a).

  \item When $\al \in \tG = \Real^\nu$ is a ghost, we obtain the supertropical curve $X|_\tng$ illustrated in Figure \ref{fig:2}.(b). This type of objects is not accessible by traditional tropical geometry.

\item Consider  the set of polynomials $E = \{ \al^\nu + \lm_1, \al^\nu + \lm_2 , 0 + (-\al)^\nu   \lm_1 \lm_2, \} \subset \Trop[\lm_1, \lm_2]$ with $\al > 0$, and let $Y  = \tZ(E) = \bigcap_{f \in E} \tZ(f)$  be  ghost locus of $E$. Then, restricting to tangibles,  $Y|_\tng$ is the filled triangle given in Figure \ref{fig:2}.(c).

\end{enumerate}

\end{example}

In fact, by the same way as in (iii),  any convex polytope with $m$ facets having rational slopes can be described as the tangible restriction of the ghost locus of a set of $m$ binomials over $\Trop$.
More specifically, to capture the geometric features   with tropical geometry one can simply use the extension $\Trop$ of the max-plus semiring $\zReal$, as demonstrated Example \ref{examp:varieties}.(i).

\section{Congruences on supertropical structures}\label{sec:3}

In this section we employ congruences on \nusmr s, starting with their   underlying additive \numon\  structure, and later concern their multiplicative structure as well.
 To enable a meaningful passage to quotient structures, congruences in our theory  play the customary role of ideals in ring theory. In this view,  we study main types of congruences, analogous  to types of  ideals, and explore their meaning in supertropical theory. Types of these congruences are of different nature, due to the special structure of \nusmr s.

\subsection{Congruences on additive \numon s}\label{ssec:cong.mon} \sSkip

 We denote by $\Cng(\tM)$ the set of all congruences on a given additive \numon\ $\tM := (\tM, \tG, \nu)$, cf.~ Definition \ref{def:numonoid}. A congruence $\Cong \in \Cng(\tM)$ is an equivalenc relation that respects the \numon\ operation and all relevant  relations (especially associativity), cf. \S\ref{ssec:cong.univ.alg}. Its underlying  equivalence is denoted by $\cng$, unless otherwise is specified.
 Recall from \eqref{eq:canonical.qu} that for any congruence~ $\Cong$ on $\tM$ there exists the canonical surjection
 $ \pi_\Cong: \tM \Onto \tM/ \Cong $, see Remark \ref{rem:cong1}.

 \pSkip

 Given a congruence  $ \Cong \in \Cng(\tM)$ with underlying  equivalence $\cng$,
we write $$a \cng \ghost \dss{\text{if}} a \cng b  \text{ for some } b \in \tGz. $$
(This notation includes the case that $a \cng \zero$, in which also $a^\nu \cng \zero$.)
\begin{lem}\label{lem:ghs.eqv}  Suppose $a  \cng \ghost$ in a congruence $\Cong\in \Cng(\tM)$,   then $a \cng a^\nu$.
\end{lem}

\begin{proof}  By assumption $a \cng b$ for some $b \in \tG$.  As $\Cong$ respects the \numon\ operation $+$, we have
$ a + a \cng b + b = b$, since $b \in \tGz$.   Hence  $a^\nu \cng b $, and   the transitivity of $\Cong$ implies  $a \cng a^\nu$.
\end{proof}

We have the obvious characterization of ghosts in terms of congruences:
\begin{cor}
 An element  $a \in  \tM$ is ghost if and only if $a \cng a^\nu$ in  all congruences $\Cong$ on $\tM$.
\end{cor}

Viewing a  congruence $\Cong $ on $\tM$ as subalgebra of $ \tM \times \tM$,
 we define its   \textbf{ghost cluster} of $\Cong $ to be
\begin{equation}\label{eq:ghost.cls}
\begin{array}{lcl}
  \Gcl {\Cong}  :=   \{ (a,b) \in \Cong \cnd a \cng a^\nu \}  & \subseteq \tM \times \tM.  \end{array}
\end{equation}
This ghost  cluster provides a coarse classification of the ghost equivalence classes of $\Cong$.
In particular, for every $\Cong$ we always have $\tGz \times \tGz \subseteq \Gcl {\Cong}$. Moreover, by Lemma \ref{lem:ghs.eqv}, an inclusion $(a,b) \in \Gcl{\Cong}$ implies  that $(a,a) \in \Gcl{\Cong} $ and $(b,b) \in \Gcl{\Cong}$, whereas  $a \cng a^\nu$ and $b \cng b^\nu$.

The  set-theoretic  complement of $\Gcl {\Cong}$ in $\Cong$ is denoted by
$$\cGcl{\Cong} := \Cong \setminus \Gcl{\Cong}. $$
A congruence $\Cong$ on $\tM$  is said to be a \textbf{ghost congruence}, if $\Cong = \Gcl{\Cong}$, namely   $\cGcl{\Cong} = \emptyset$.

\begin{rem}\label{rem:zero.cls}
  If the ghost cluster $\Gcl{\Cong}$  of $\Cong \in \Cng(\tM)$  consists of elements from a single (ghost) equivalence class, then $\Cong$ is not a proper congruence (Definition \ref{def:cong.types}). Indeed, in this case $a^\nu \cng \zero$ for every $a^\nu \in \tGz$, and then
  $$a = a + \zero \cng a + a^\nu  =  a + (a + a + a) = (a + a) + (a + a) = a^\nu + a^\nu \cng \zero + \zero =  \zero.$$
  Therefore, each  $a \in \tM$ is congruent to $\zero$. 
\end{rem}
Formally, in special cases,  we use the \textbf{zero congruence} $\zCong$, whose underlying equivalence $\zcng$ is given~ by
\begin{equation}\label{eq:zero.cong}
  a \zcng \zero \qquad \text{ for all } a \in \tM.
\end{equation}
$\zCong$ consists of the single equivalence class $[\zero]$.

 We define the \textbf{ghost projection}  of the ghost cluster $ \Gcl{\Cong}$  of a congruence $\Cong$ on $\tM$ to be the subset
  \begin{equation}\label{eq:invghost}
 \iGcl(\Cong) :=  \{ a \in \tM \cnd a \cng a^\nu \} \ds \subseteq \tM .
 \end{equation}
 In other words, $\iGcl(\Cong)$ is the preimage
 of the  diagonal of $ \Gcl{\Cong}$ under the map  $\iota: \tM \To \diag(\Cong)$, cf.~ \eqref{eq:con.diag.1}.
 In this setup, $\Gcl{\Cong}$ is the restriction of $\Cong$ to $\iGcl(\Cong) \subseteq \tM$, and by itself is a congruence on $\iGcl(\Cong)$.
 Note that an element $a \in \iGcl(\Cong)$  need not  be a ghost belonging to $\tGz$, and that $a \in \iGcl(\Cong)$  if and only if   $(a,a) \in \Gcl{\Cong}$. Clearly, the inclusion $\tGz \subseteq \iGcl(\Cong)$ holds for any congruence $\Cong$ on $\tM$.
  For short, we write
  \begin{equation}\label{eq:cinvghost}
 \ciGcl(\Cong) :=  (\iGcl(\Cong))^\cmp
 \end{equation}
 for the set-theoretic complement of $\iGcl(\Cong)$ in $\tM$.

The quotient of a \numon \ $\tM:= (\tM, \tG, \nu)$ by a congruence $\Cong$  is   defined  as
\begin{equation}\label{eq:qnumon}
\tM/\Cong := (\tM/\Cong,  \
 \iGcl(\Cong)/\Gcl{\Cong}, [\nu]),
\end{equation}
 where the  ghost map  $\nu:\tM \To \tGz$ of $\tM$ induces the  ghost map $[\nu]$ of $ \tM/\Cong$ via $[a]^{[\nu]} = [a^\nu].$
  A class ~$[a]$ of   $\tM / \Cong$ is a ghost class, if it contains a  ghost element of $\tM$, where  Lemma ~\ref{lem:ghs.eqv}  implies that $a \cng a^\nu$, and hence also $a^\nu \in [a]$.
 The partial ordering of the ghost submonoid of $\tM/\Cong$  is induced by addition, i.e., $[a]> [b]$ if $[a]+ [b] = [a]$, and thus $[a] + [b] = [a]^{[\nu]}$ whenever $[a]^{[\nu]} = [b]^{[\nu]}$. Hence, $\tM/\Cong$ is an additive  \numon\ (Definition \ref{def:numonoid}).


 We see that in fact the ghost projection $\iGcl(\Cong)$ is the preimage of the ghost submonoid of $\tM/\Cong$ under the  canonical surjection
$ \pi_\Cong: \tM \Onto \tM/ \Cong $, cf. \eqref{eq:canonical.qu}.  Namely, it is the g-kernel of $\pi_\Cong$ (Definition \ref{eq:nu.mon.hom}).

\begin{rem}\label{rem:cong.order} For any congruence $\Cong$ on a \numon\ $\tM$ we have the following properties.
\begin{enumerate}\eroman
\item If $a^\nu \cng b^\nu$, then
$$[a + b] = [a] + [b] = [a]^\nu \Dir a +b \cng a^\nu.$$
This equivalence is compatible with the canonical surjection
$ \pi_\Cong: \tM \Onto \tM/ \Cong $, cf. Remark ~\ref{rem:cong1}.(i).

\item Suppose $a + b = c$. If $a \cng b$,  then $a^\nu = a + a \cng a+ b$, and thus
$$ a^\nu \cng a+b = c, \qquad a^\nu = a^\nu + a^\nu \cng (a+b) + (a + b) = c + c = c^\nu,  $$
implying that $ c \cng c^\nu$. In the case that $a^\nu + b \in \tG$, the same holds under the equivalence $a \cng a ^\nu$, since $a^\nu + b \cng a + b = c$.

  \item If $a + b = a$ and $a \cng b$,  then
  $ a \cng  a^\nu$ by (i), and  hence $b \cng b^\nu$, since $$b \cng a \cng a^\nu = a + a \cng b + b =  b^\nu.$$ In particular, this implies that if $a  \cng b$ where $b \nule a $, then  $a \cng a^\nu$ and $b \cng b^\nu$.

  \item If $a + b =a$ and $a \cng b$, then $a \cng c$ for every $c $ such that $a + c = a$ and $c + b= c$ (especially when  $a \nug c \nug b$). Indeed
  $$ a = a + c \cng  b +c = c.$$
Furthermore,  $c \cng c ^\nu$ for each such $c$, since  $a \cng a^\nu$ by (iii).

\end{enumerate}
\end{rem}

Clearly, an inclusion of \ccong s implies the  inclusion of their ghost clusters:
\begin{equation}\label{eq:inclusion.gcls}
  \Cong_1 \subset \Cong_2 \Dir \Gcl{\Cong_1} \subset \Gcl{\Cong_2}.
\end{equation}
Also, by transitivity of congruences,  the intersection $\Cong_1 \cap \Cong_2$ respects intersection of ghost clusters:    $\Gcl{\Cong_1 \cap \Cong_2} = \Gcl{\Cong_1} \cap \Gcl{\Cong_2}$.

Having the above insights, we next observe congruences that arise from suitable ``ghost relations'', taking place on  subsets of \numon s.
For a nonempty subset  $E $ of $\tM$ we define the  set of congruences
\begin{equation}\label{eq:numon.var}
  \GG(E) := \{ \Cong \in \Cng(\tM) \cnd E \subseteq \iGcl(\Cong) \}  \ds\subseteq \Cng(\tM),
\end{equation}
and consider the congruence determined as the intersection of all its members:
  \begin{equation}\label{eq:G.E}
  \mfG_E := \bigcap_{\Cong \in \GG(E)} \Cong.
\end{equation}
 This construction provides $\gCong_E$ as the minimal congruence in which the entire subset $E$ is declared  as ghost, that is $a \cng a ^\nu$ for every $a \in E$, cf. Remark \ref{rem:cong.gen} and Lemma \ref{lem:ghs.eqv}.

 The congruence $\mfG_E$  respects the monoid operation of $\tM$, as it is the intersection of congruences, and hence it is transitively closed. We call
 $\gCong_E$ the \textbf{ghostifiying congruence} of $E$, while $\tM / \mfG_E $ is said to be   the  \textbf{ghostification} of $E$, for short. We also say that $K$ is \textbf{ghostfied} by $\mfG_E$, when $K \subseteq E$.
\begin{lem}\label{lem:ghostification}
  The underlying   equivalence $\cng$ of $\mfG_E$ can be formulated as
  \begin{equation}\label{eq:ghost.addt}
  a + b \cng a + b^\nu \qquad \text{for all   } b \in E.
\end{equation}
In particular, $b \cng b^\nu$ for all $ b \in E$.
\end{lem}
\begin{proof} By Lemma \ref{lem:ghs.eqv}, the inclusion $E \subseteq \iGcl(\Cong)$ is equivalent to having the relation $b \cng b^\nu$ satisfied for all  $b \in E$. Every  $\Cong \in \GG(E)$ satisfies the equivalence  \eqref{eq:ghost.addt}, since $\Cong$  respects addition, and therefore their intersection  $\mfG_E$ also admits condition \eqref{eq:ghost.addt}. Then, the minimality of $\mfG_E$ completes the proof.
\end{proof}

Accordingly, the ghsotification of a subset $E \subseteq \tM$ is provided by the minimal congruence whose underlying  equivalence $\cng$ admits the relation
$$ a \cng a^\nu \quad \text{ for all } a \in E \cup \tG. $$
 In the degenerated case that  $E \subseteq \tGz$, we simply have $\mfG_E = \diag(\tM)$. 
 On the other hand, if  $\one \in E$, then $\mfG_E$ is a ghost congruence.

\begin{rem}\label{rem:eqv.cls} Let $[a], [b]$ be classes of $R/ \mfG_E$.  From \eqref{eq:ghost.addt} it follows that $[a ] = [b]$ only if $a \nucong b$, cf. ~\eqref{eq:nu.eqv}.
(The converse does not hold in general, take for example $a,b \notin \tG$ such that $a \neq  b$ and $a \nucong b$.)
\end{rem}

We say that the congruence $\gCong_E$ is \textbf{determined} by $E \subseteq \tM$, and define the set
\begin{equation}\label{eq:g.Cong}
\GCng(\tM) := \{ \gCong_E \cnd E \subseteq \tM\}
\end{equation}
of all ghostifying congruences on $\tM$. We call these congruences  \textbf{\gcong s}, for short.
$\GCng(\tM)$ is a nonempty set as it contains $\diag(\tM) $.

\begin{rem}\label{rem:ghostification.intersection} For any subsets $E, E' \subseteq \tM$ we have the following properties:
\begin{enumerate}\eroman
  \item $\mfG_{E \cap E '} = \mfG_E  \cap \mfG_{E'} $;
  \item $\mfG_{E \cup E '} =  \overline{\mfG_E  \cup \mfG_{E'}} $ (cf. \eqref{eq:cong.closure});
  \item $E \subset E' \Rightarrow \mfG_E \subset  \mfG_{E'} $.
\end{enumerate}
Namely, $\GCng(\tM)$ is closed for intersection (and respects inclusion),  in which $\mfG_\tM$ is  maximal  congruence. Therefore,  $\GCng(\tM)$ has the structure of a semilattice.   \end{rem}
\noindent (By Remark \ref{rem:nmon.zero}, $a \gcng \zero$ in $\mfG_E$ implies that $a = \zero$, and thus formally  the zero congruence $\mfO$ does not belong to $\GCng(\tM)$.)
\pSkip

\begin{lem}\label{lem:gker-smon.0} Let $\vrp:\tM \To \tM'$ be a  \numon\ \hom\ (Definition \ref{def:numon.hom}), then
  $ \gker(\vrp) = \iGcl(\gCong_{\gker(\vrp)})$.
\end{lem}

\begin{proof}  Let $E = \gker(\vrp)$.
  The inclusion $ E \subseteq \iGcl(\gCong_{E})$ is obvious. Conversely, suppose $\iGcl(\gCong_{E}) \sm E$ is nonempty and take
  $a \in \iGcl(\gCong_{E}) \sm E$. Thus,  $a$  is not a ghost, since $E = \gker(\vrp)$,  and $a$ can be written as
\begin{equation*}\label{eq:a.in.mon.igcl}
 a = b + d + \sum e_i   \quad \text{ for some } \ b \in \tG, \ d \in \tM \sm \tG, \  e_i \in E \sm \tG,
\end{equation*}
  where $b$ and $e_i$ are  possibly all $\zero$. (In fact $b + \sum e_i \in E$ as $E$ is a monoid.)
  By Lemma \ref{lem:ghostification} we obtain
 $$a^\nu = d^\nu + b + \sum  e_i^\nu  =  d + b + \sum e_i^\nu  ,$$
   since  $a \in  \iGcl(\gCong_E)$,
 which implies by Axiom \NMc\ in Definition \ref{def:numonoid} that $a = d^\nu  + b + \sum  e_i $.  Thus
\begin{equation*}\label{eq:a.in.mon.igcl}
 a = b' + \sum e_i   \quad \text{ for some }  e_i \in E, \ b' \in \tG.
\end{equation*}
But $\tG \subset E$, and hence $a \in E$.
\end{proof}

Using  ghostifying  congruences, one can define quotients of \numon s.
\begin{defn}\label{def:ghostification}
  The \textbf{quotient \numon} of a \numon \ $\tM$ by a  subset $E \subseteq \tM$ is defined to be the \numon\ $\tM/\mfG_E$. We write $\tM\qq E$ for $\tM/\mfG_E$.
\end{defn}
In other words, all the elements of $E \subseteq \tM$ are identified as ghosts in $\tM \qq E$. In particular, for $E \subseteq \tGz$, we have $\mfG_E = \diag(\tM)$, and therefore, in this case,  $\tM \qq E = \tM$. On the other hand,  $\tM \qq E \cong \tGz$ when $E = \tM$.
The process of quotienting   by the means of ghsotification is canonically defined and allows a  factoring oyt by substructures, as well as taking closures.
\begin{defn}\label{def:coker}
  The \textbf{g-cokernel} of a \hom\ $\vrp:\tM \To \tM'$ of \numon s is defined as
  \begin{equation}\label{eq:gcoker}
   \gcoker(\vrp):= \tM' \qq \, \im(\vrp) \, .
\end{equation}
\end{defn}

Let $E $ be a subset of $\tM$, and consider the surjective \hom\
$$ \vrp: \tM \TO \tM \qq E.$$
By Lemma,  \ref{lem:gker-smon.0} $\gker(\vrp) = \iGcl(\gCong_{\gker(\vrp)})$ is a   \nusmon\ of $\tM$ containing $E \cup \tG$.

\begin{lemma}\label{lem:gker-smon}
Let $\tN \subset \tM$ be a  \nusmon\  containing  $\tG$, and let $\phi: \tM \To \tM \qq \tN$ be the canonical \hom.
Then $\gker(\phi) = \tN$.
\end{lemma}

\begin{proof} Follows from Lemma \ref{lem:gker-smon.0}.
\end{proof}

\begin{theorem}\label{thm:numon.first.thm}
  Let $\vrp:\tM \To \tM'$ be a  \numon\ \hom. There exists a unique \nuhom\
  $$ \olvrp: \tM \qq \gker(\vrp) \TO \tM', \qquad \olvrp:[a] \Mto \vrp(a), $$
  such that $\vrp = \olvrp \circ \pi$, where $\pi : \tM \To \tM \qq \gker(\vrp)$ is the canonical \hom.
Namely, $\vrp$ factors uniquely through $\pi$ and the diagram
$$\xymatrix{
\tM  \ar@{->}[rrd]_{\vrp = \olvrp \circ \pi} \ar@{->}[rr]^{\pi} && \tM \qq \gker(\vrp) \ar@{..>}[d]^{\olvrp} \\
 && \tM'  \\
}
$$
commutes.
\end{theorem}
\begin{proof}  Clearly, $\olvrp$ is a \hom:
  $\olvrp([a+b]) = \vrp(a+b) = \vrp(a) + \vrp(b) = \olvrp([a]) + \olvrp([b])$.
  Let $\Cong_\vrp $ be the c-kernel of $\vrp $ (Remark \ref{rem:cong1}.(ii)) and  write $E = \gker(\vrp)$, which is a monoid.
  Then, $E = \iGcl(\gCong_E)$  by Lemma  \ref{lem:gker-smon.0}, and $\gCong_E \subset \Cong_\vrp$ since  $\Cong_\vrp$ includes the  relations $a \cng a^\nu$ for all $a \in E$, and their transitive closure, implying that  $\olvrp$ is well defined by   Remark \ref{rem:2cong}. Also,
$(\olvrp \circ \pi)(a)={\olvrp}(\pi(a))={\olvrp}([a])=\vrp(a),$ proving that ${\olvrp}\circ \pi = \varphi.$
 Suppose there exists another $\psi$ such that $\psi \colon \tM/\gker(\vrp) \To \tM'$ with $\psi \circ \pi = \vrp $, then
$\psi([a])=\psi(\pi(a))=(\psi\circ\pi)(a)=\varphi(a)$ so that ${\olvrp}=\psi$, and $\olvrp$ is unique.
\end{proof}

\subsection{Congruences on \nusemirings0}  \sSkip

 Henceforth, our \nusmr\ $R:= (R, \tT, \tG, \nu)$   is always assumed to be commutative,  cf.~ \S\ref{ssec:nusmr}.  
   Recall that ~$\equiv$ denotes the  underlying  equivalence of a congruence $\Cong$.
  We begin with the set $\Cng(R)$ of all  congruences on $R$, and later restrict to
  better behaved  congruences which include families of  \gcong s.

 \begin{rem}\label{rem:cong.unit}
Any congruence $\Cong \in \Cng(R)$  satisfies the condition
that if  $a \equiv b$ then $$a^\nu = e a \cng e b =
b^\nu.$$
Conversely, when $\Cong$ is cancellative (Definition \ref{def:cong.types}), if $a^\nu \cng b^\nu$ then
$a \cng b$, and moreover $a \equiv a^\nu$ for every $a \in R$. (Indeed,
$ e a \cng  e a^\nu$, and  $a \cng a^\nu$ by cancellativity.)
Therefore, one  sees that  cancellativity is much  too restrictive for congruences on  \nusmr s.

\end{rem}

 We specialize congruences on \numon s  to \nusmr s, involving their multiplicative
structure.  While the additive structure induced by a congruence $\Cong$ on $R$ has been addressed through  ghost elements, characterized by the ghost cluster $\Gcl{\Cong}$, the multiplicative structure induced by $\Cong$ is approached  via tangible elements. \pSkip

The \textbf{tangible cluster} $\Tcl {\Cong}$ of a congruence $\Cong \in \Cng(R)$ is defined as
\begin{equation}\label{eq:tng.cls}
\begin{array}{lcl}
 \Tcl {\Cong}  := &  \{ (a,b) \in \Cong \cnd a \in \tT, \  (a , t) \in \Cong   \text{ only for } t \in \tT  \} \ds \subseteq \tT \times \tT  ,  \\
\end{array}
\end{equation}
By transitively, $b $ must also be congregant only to tangibles.
The set-theoretic  complement of $ \Tcl {\Cong}$ in~ $\Cong$ is denoted  by
$$ \cTcl{\Cong} := \Cong \setminus \Tcl{\Cong}\; .$$
In analogy to $\iGcl(\Cong)$ in \eqref{eq:invghost}, the \textbf{tangible projection} of $\Cong$ is the  projection of the diagonal of $\Tcl{\Cong}$ on ~ $R$,   defined  as
  \begin{equation}\label{eq:invghost}
 \iTcl(\Cong) :=  \{ a \in \tT \cnd  a \cng t \text{ only for  } t \in \tT \} \ds \subseteq \tT,
 \end{equation}
 that is  $\iTcl(\Cong) = \idiag(\Tcl{\Cong}), $ cf. \eqref{eq:con.diag.1}. Accordingly, the tangible projection  $\iTcl(\Cong)$ is the preimage of the tangible subset of $R/\Cong$ under the  canonical surjection
$ \pi_\Cong: R \Onto R/ \Cong $.
By  definition, we immediately see that
 \begin{equation}\label{eq:a.in.tcl}
  a \in \iTcl(\Cong) \Dir a \in \tT.
\end{equation}
 For short, we write
  \begin{equation}\label{eq:cinvghost}
 \ciTcl(\Cong) :=  (\iTcl(\Cong))^\cmp
 \end{equation}
 for the set-theoretic complement of $\iTcl(\Cong)$ in $R$.
\begin{rem}\label{rem:tng.cluster}
Let $\Cong $ be a  congruence  on $R$, let  $a,b \in \tT \sm \RX$ be tangibles,  and let $u,v \in \RX$ be units.
\begin{enumerate}\eroman
  \item If $u \cng v$ then $u^\inv \cng v^\inv$, since $v^\inv = u u^\inv v^ \inv \cng v u^\inv v^\inv = u^\inv v v^\inv = u^\inv$.

  \item The equivalence $u \cng a$ implies that $[u^\inv] = [a]^\inv$ in $R / \Cong$, since $[u^\inv] [a] = [u^\inv] [u] = [u^\inv u] = [\one]$. Thus, $[a]$ is invertible in $R/\Cong$.

  \item If  $u \in \iGcl(\Cong)$, then $\iTcl(\Cong) = \emptyset$. Indeed, $u$ is then congruent to a ghost element in $R$, and since $\Cong$ respects multiplication,  also $\one = u u^\inv \in \iGcl(\Cong)$ and  $\one \cng e$. Consequently,  every  $a \in R$ is congruent to some ghost element, and thus $\iTcl(\Cong)$ is empty.

      \item By Remark \ref{rem:cong.order}.(iv), if $a \cng b$ where $a+ u = a$ and $b+u = u$ (in particular when $a \nug u \nug b$), then $u \cng u^\nu$ and
      $\iTcl(\Cong) = \emptyset$.

      \item If $a + b = u$, where  $a \cng a'$ and $b \cng b'$ such that $a' + b'$ is ghost, then $u$ is congregate to a ghost and again $\iTcl(\Cong) = \emptyset$.
\end{enumerate}
These properties will be of much use in the analysis  of maximal congruences.
 \end{rem}

The ghost cluster $\Gcl {\Cong}$ in \eqref{eq:ghost.cls} and the tangible cluster $\Tcl{\Cong}$ in \eqref{eq:tng.cls} of a congruence $\Cong$  are  set theoretically disjoint.  Together they  induce a classification of equivalent classes  as tangible, ghost, or neither tangible nor ghost. Accordingly, we sometimes  refer to a \ccong~ $\Cong$ as the
triplet $$ {\Cong} :=  (\Cong, \Tcl{\Cong}, \Gcl{\Cong}).$$
   By  definition, $(a,b) \in \Tcl{\Cong}$ implies $(a,a) \in \Tcl{\Cong} $ and $(b,b) \in \Tcl{\Cong}$. The same holds for $\Gcl{\Cong}$, where it may contain a  pair of tangibles $(a,b) \in \tT \times \tT$, if $a \equiv a^\nu $.

Set-theoretically,  $\Tcl{\Cong}$ is not necessarily the complement of $\Gcl{\Cong}$ in $\Cong$, since  $\Cong$ may elements which are  neither in $\Tcl{\Cong}$ nor in $\Gcl {\Cong}$.
This  means that $\cTcl{\Cong}$ and $\cGcl{\Cong}$ need not be disjoint; this happens only in certain cases.  For example, when $R$ a supertropical domain (Definition ~\ref{def:supertropical}) we have ${\Cong} =  \Gcl {\Cong} \cup \Tcl {\Cong}.$

%
  \begin{rem}\label{rem:cong.ideal} As $\Cong$ respects the multiplication of the carrier  \nusmr\ $R$,
 $a \in \iGcl(\Cong)$ implies that  $ab \in \iGcl(\Cong) $ for any $b \in R$, since  the product of any element in $R$ with a ghost element is ghost. Hence,  the ghost projection $\iGcl(\Cong)$ of $\Cong$ is a semiring  ideal of ~$R$ (Definition \ref{def:ideal.smr}) --  a ghost absorbing subset containing $\tG$.

 If $\iGcl(\Cong)$ contains a unit $u \in \RX$, then $\Cong $ is a ghost congruence.
 Indeed, $\one = u^\inv u \in \iGcl(\Cong)$, and thus  $b = b \one \in  \iGcl(\Cong)$ for every  $b \in R$. In other words,    $\iGcl(\Cong) = R$, which is the case of
 Remark~ \ref{rem:tng.cluster}.(iii).
  \end{rem}

Categorically, the projections  $\iTcl(\udscr)$ and $\iGcl(\udscr)$ are viewed as maps.

\begin{rem}\label{rem:classfor}The subsets
 $\iTcl(\Cong) $ and $\iGcl(\Cong)$ of $R$ define class forgetful maps $\iTcl(\udscr): \Cng(R) \To \tT$ and $\iGcl(\udscr): \Cng(R) \To R$ that preserve only clusters' decomposition. That is, the property of being tangible or ghost under the canonical surjection $\pi_\Cong: R \Onto R/\Cong$. This data is fully recorded by  restricting  classes to subsets of the diagonal of $\Cong$, where the diagram
$$\xymatrix{
R \ar@{->>}[rrd]_{\pi_\Cong } \ar@{^{(}->}[rr]^{\diag} && \Cong \ar[d]^{\psi} \\
 && R/\Cong \\
}
$$
respects this clustering.
\end{rem}

In oppose to \eqref{eq:inclusion.gcls}, an inclusion of \ccong s implies an inverse inclusion of tangible clusters:
\begin{equation}\label{eq:inclusion.tcls}
  \Cong_1 \subset \Cong_2 \Dir \Tcl{\Cong_1} \supset \Tcl{\Cong_2}.
\end{equation}

\begin{rem}\label{rem:intersection.clusters}
The clusters  $\Tcl{\Cong_i}$ and $\Gcl{\Cong_j}$ admit the following relations for intersection:
$$
  \Tcl{\Cong_1 \cap \Cong_2} \supseteq  \Tcl{\Cong_1} \cap \Tcl{\Cong_2}, \qquad
  \Gcl{\Cong_1 \cap \Cong_2} = \Gcl{\Cong_1} \cap \Gcl{\Cong_2}.$$
Intersections of clusters of different congruences  need not be empty, and include the following cases:
\begin{enumerate}\eroman
  \item $\Tcl{\Cong_1} \cap \Gcl{\Cong_2} \subseteq \Tcl{\Cong_1};$
  \item $\Tcl{\Cong_1} \cap \cTcl{\Cong_2} \subset \Tcl{\Cong_1};$
  \item $\Tcl{\Cong_1} \cap \Gcl{\Cong_2} \nsubseteq \Gcl{\Cong_2},$ unless it is empty;
  \item $\cTcl{\Cong_1} \cap \cTcl{\Cong_2}$ can be in $\Tcl{\Cong_1}$, $\Gcl{\Cong_1}$, or in the complement of their union.
\end{enumerate}
    The intersection of the projections $\iTcl(\udscr)$ and $\iGcl(\udscr)$ are induced by these cases.

\end{rem}

The ghostification of subsets of \numon s, i.e.,  identifying subsets as ghosts,  extends   naturally  to \nusmr s.

\begin{rem}\label{rem:ghostification.smr}
  The underlying   equivalence $\cng$ of the ghostifiying congruence $\mfG_E$ of a subset $E \subseteq R$  is formulated as
  \begin{equation}\label{eq:ghost.smr}
  a + b \cng a + b^\nu, \quad a  b \cng (a  b)^\nu,  \qquad \text{for all   } b \in E,
\end{equation}
cf. \eqref{eq:G.E} and \eqref{eq:ghost.addt}.
Indeed,
 Lemma \ref{lem:ghostification} gives the additive relation;  the multiplicative relation is obtained from the \nusmr\ multiplication and the role of the ghost ideal, since $a^\nu = e a$ for every $a \in R$.

 From relations \eqref{eq:ghost.smr} it follows  that, for any  $a \in R$,
\begin{equation}\label{eq:a.in.igcl}
a \in \iGcl(\gCong_E) \Dir a = q + \sum c_i  e_i  \quad \text{ for some } e_i \in E,\  c_i \in R, \ q \in \tG,
\end{equation}
 since  $a$ is ghostified as a consequence of  ghostifiying  $E$. Indeed, assume that $a = d + q + \sum c_i e_i $  for some \reduced\ term  $d \notin  \iGcl(\gCong_E)$,   then
 $$a^\nu = d^\nu + q + \sum  c_i e_i^\nu  =  d + q + \sum c_i e_i^\nu  ,$$
 implying by Axiom \emph{\NMc}\ in Definition \ref{def:numonoid} that $a = d^\nu  + q + \sum  c_i e_i $.
  Thus, $a$ admits \eqref{eq:a.in.igcl}, since  $d^\nu + q \in \tG$. (See also Lemma \ref{lem:gker-smon}.)
\end{rem}
\noindent  In this view, since $\mfG_E$ is a congruence,  the  passage to  quotient  structures by subsets $E \subset R$ is natural. A ghostifying  congruence $\gCong_{E}$  with $E = \{ a \}$, $a \notin \tG$,  is called a \textbf{principal congruence}.
\pSkip

One sees that ghostifiying a subset $E$ of a \nusmr\ makes $E$ a ``ghost absorbing'' subset, in analogy to ideals generated by subsets in ring theory.
Note that the ghsotifcation of a \expr\ sum ~ $x$ (Definition ~\ref{def:reduced.sum}) does not necessarily ghsotifies sums ~ $y \lsset x$, cf. \eqref{eq:lorder}.
Using \eqref{eq:a.in.igcl}, the ghostification of~ $E$ by \eqref{eq:ghost.smr} determines a ``ghost dependence'',   in the sense  that any combination of elements in $E$ becomes ghost.
Although this dependence  is weaker than spanning, it often suffices to simulate the role of the latter.

\begin{lem}\label{lem:cong.alg.oper}  Given $a,b \in R$,  for $\GG(a)$ as defined in  \eqref{eq:numon.var} we have:
\begin{enumerate}\eroman
  \item $\GG(a) \cap \GG(b) \subseteq \GG(a+b) $;
  \item $ \GG(a) \cap \GG(b) \subseteq \GG(ab)$;
  \item $\GG(a) \cup \GG(b) \subseteq \GG(ab)$.
\end{enumerate}

\end{lem}

\begin{proof} (i)-(ii): $\Cong \in \GG(a) \cap \GG(b)$ means that both $a \in \iGcl(\Cong)$ and
$b \in \iGcl(\Cong)$, thus $a+b \in \iGcl(\Cong)$ and  $ab \in \iGcl(\Cong)$, since $\iGcl(\Cong)$ is an ideal of $R$ by Remark \ref{rem:cong.ideal}.

\pSkip
  (iii): $\Cong \in \GG(a)$ means that $a \in \iGcl(\Cong)$ and thus also $ab \in \iGcl(\Cong)$, since $\iGcl(\Cong)$ is an ideal of $R$. By symmetry, the same holds for $\Cong \in \GG(b)$.
\end{proof}

\begin{cor}\label{cor:cong.alg.oper} Let $\tA = (a_i)_{i \in I}$ and
$\tB = (b_j)_{j \in J}$ be families of elements $a_i, b_j \in R$.
Suppose  $d = \sum a_i b_j$ is  a finite sum, then $\mfG_{\{ d\}} \subseteq \mfG_\tA \cap \mfG_\tB,$
cf. \eqref{eq:G.E}.
\end{cor}
\begin{proof} Lemma \ref{lem:cong.alg.oper} gives
  $ \GG(\tA) \cap  \GG(\tB) \subseteq \GG(\{ d\}) $, implying that $\mfG_{\{ d\}} \subseteq \mfG_\tA \cap \mfG_\tB. $
\end{proof}


\begin{lemma}\label{lem:gker-ideal}
Let $\mfa \lhd  R$ be a  \nusmr\ ideal (Definition \ref{def:ideal.smr}) containing  $\tG$, and let $\vrp: \tM \To R \qq \mfa$ be the canonical \hom.
Then $\gker(\vrp) = \mfa$.
\end{lemma}
\begin{proof} In \numon\ view,  $\gker(\vrp) = \mfa$, by   Lemma \ref{lem:gker-smon.0}, while   $\gker(\vrp)$ is a \nusmr\ ideal by \S\ref{ssec:smr.hom}.
\end{proof}

\subsection{\qcong s and \lcong s} \sSkip
Not all congruences $\Cong$  on a \nusmr \ $R$ possess  \qhom s (Definition \ref{eq:nu.smr.hom}); furthermore, a quotient $R/\Cong$  does not necessarily preserve a tangible component of $R$. To retain the category $\NSMR$ of \nusmr s,  we restrict to congruences which endow $R/\Cong$ with \nusmr\ structure and are also applicable for localization. Moreover,
 to obtain the desired supertropical analogy of ideals in classical commutative algebra, together with their correspondences to varieties,  such congruences should coincide  with the notion of ghostification. Let us first address some guiding pathological cases.

\begin{example}\label{exmp:cong}
Let $\Cong$ be a congruence on a \nusmr\ $R$. 
\begin{enumerate}\eroman
                 \item   If $a + b = u $ is a unit and $a \cng b$ in $\Cong$, then $u \cng u ^\nu$ by         Remark \ref{rem:cong.order}.(iv), implying by Remark \ref{rem:cong.ideal} that  $\Cong$ is a \gstcong. Hence, $R/ \Cong$ has no tangible component.

\item
  If $a \cng b$ for any pair of tangible elements $a,b$ in $R$, and $\tG$ is ordered,  then for $a \nug b$ we have
   $$ a = a+ b \cng a + a = a^\nu, $$
   implying that $\Tcl{\Cong}$ is empty, cf. Remark \ref{rem:cong.order}.(iii).
\item
   Similarly, if $ a \cng u$ for some unit $u \in \RX$, where $a$ is non-tangible, then $R/ \Cong $ does not necessarily have tangible elements and $\Tcl{\Cong}$ could be empty. It might also not be a \nusmr, since $[a]$ must be unit, but is non-tangible.

\end{enumerate}
Consequently, in all these cases $R/\Cong$ need not be a \nusmr\ and the  canonical surjection $\pi_\Cong : R \Onto R/\Cong$ is  not necessarily a \qhom\ (Definition \ref{def:nusmr.hom}), whereas  $\tcor(\pi_\Cong) = \emptyset$.
\end{example}

To avoid the above drawbacks, we first distinguish those elements in $R$ which must be preserved as tangible in the passage to a quotient $R/ \Cong$.

\begin{defn}\label{def:t.strict} An element $a \in \tT$ is called \textbf{tangible  \ualt}, written \textbf{\tualt}, if
\begin{equation}\label{eq:t.strict}
a \cng b \text{ for some } b \notin \tT  \dss \Rightarrow  \RX \nsubseteq \iTcl(\Cong),
\end{equation}
in all $\Cong \in \Cng(R)$. We denote the set of all \tualt\ elements by $\SST$.
\end{defn}
\noindent Clearly, $\one \in \SST$ and thus $\SST$ is nonempty. Also, $\SST \subset \tTPS$, since congruences  respect the \nusmr\ multiplication. However, \tualt\ elements need not be units, while $  \RX  \subseteq \SST $.

For instance, every tangible element in a \dfnt\ \nusmf\ $F$ is \tualt, as well as in the polynomial \nusmr\ $F[\Lm]$. For example, consider the polynomial \nusmr\ $\Trop[\lm_1,  \lm_2]$  from Example ~ \ref{examp:varieties};  its \tualt\ elements are the tangible elements of $\Trop$. In this case, each \tualt\ element is a unit. On the other hand, $(0, -1) \in \Trop^{(2)}$ is \tualt, since  $(0, -1) =  (-1, 0) = (1,1) = \one_{\Trop^{(2)}}$,  but is not a unit.

\pSkip

To ensure that a quotient semiring $R/\Cong$ is a proper \nusmr\ (Definition \ref{def:nusemiring}) and that the canonical surjection $\pi_\Cong : R \Onto R/\Cong$ is a (unital) \qhom\ of \nusmr s with a nonempty \tcore,  we  exclude all the \gstcong s and  restrict our intention to the following types of  congruences. As will be seen later, the characteristic of these particular congruences,  concerning localization as well,  is curial for our forthcoming results.

\begin{defn}\label{def:nucong}
A congruence $\Cong$  on a \nusmr\ $R$ is a \textbf{\qcong } (abbreviation for  quotienting congruence), if $\RX  \subseteq  \iTcl(\Cong)$. Hence,  its  tangible projection  $ \iTcl(\Cong)$  contains a nonempty tangible submonoid  $\iPcl(\Cong) \supseteq \RX$.    
 The set of  all \qcong s on $R$ is denoted by $\QCng(R)$.

A \qcong\ $\Cong$ is an \textbf{\lcong } (abbreviation for localizing congruence), if $ \iTcl(\Cong)$ by itself is a multiplicative tangible submonoid of $\tT$, written $ \iTcl(\Cong) = \iPcl(\Cong)$,  and
hence
\begin{equation}\label{eq:q.cong}
  ab \notin \iTcl(\Cong) \Dir a \notin \iTcl(\Cong) \text{ or } b \notin \iTcl(\Cong).
\end{equation}
The set of  all \lcong s  is denoted by ~$\LCng(R)$.

\end{defn}
\noindent By this definition, $\iTcl(\Cong) \subseteq \tTPS$ for every \lcong\ $\Cong$, while
$\SST \subseteq \iTcl(\Cong)$ for any \qcong\ $\Cong$, since otherwise $\RX \nsubseteq \iTcl(\Cong).$
\pSkip

\qcong s and \lcong s are defined solely  by the structure of their tangible projections $\iTcl(\udscr)$, cf. \eqref{eq:invghost}.  They are proper congruences (Definition
\ref{def:cong.types}) whose tangible clusters and ghost clusters are disjoint and nonempty.
For example, the trivial congruence $\diag(R)$  on a  \nusmr\ $R$  is a \qcong. $\diag(R)$ is an \lcong, if $R$ is \tcls, since then $\iTcl(\diag) = \tT$ is a multiplicative monoid.

\begin{lemma}\label{lem:qsmr} Let $\Cong$  be a \qcong\ on a \nusmr\ $R$, then  $R/\Cong$ is a \nusmr\ (Definition \ref{def:nusemiring}).
\end{lemma}

\begin{proof} 
The quotient $R/\Cong$  is a \numon\ by \S\ref{ssec:numon}. $\iTcl(\Cong)/\Tcl {\Cong}$ is the \tsset\ of $R/\Cong$, containing ~$[u]$, since $u \in \iTcl(\Cong)$, for every $u \in \RX$, and thus each  $[u]$ is \tprs\ in $R/\Cong$. The \tprsset\ of $R/ \Cong$ is formally defined to be all $[a] \in \iTcl(\Cong)/\Tcl{\Cong}$ such that $[a]^n \in \iTcl(\Cong)/\Tcl{\Cong}$ for every $n$, so that   Axiom \PSRa \ holds.

Suppose $[a] + [b]$ is tangible,  where $[b]$ is ghost in $R/\Cong$,  that is $b \cng b^\nu$ by Lemma \ref{lem:ghs.eqv}. Then $a^\nu + b \cng a^\nu + b^\nu$ is ghost, and thus  $[a + b]$ is not tangible -- a contradiction. Hence   Axiom \PSRb\ holds.
\end{proof}

\begin{cor}\label{cor:lcong.vs.tcls}
The quotient   $R/\Cong$ of a \nusmr\ $R$ by an \lcong\ $\Cong$ is a \tclsnusmr.
\end{cor}
\begin{proof}
  Indeed, $R/\Cong$ is a \nusmr\ by Lemma \ref{lem:qsmr}, where $\iTcl(\Cong)/\Tcl{\Cong}$ is its \tprsmon.
\end{proof}

It follows from Definition \ref{def:nucong}  that in a \qcong \ ~$\Cong$  none of the units of $R$ is congruent to a non-tangible, especially not to a ghost or $\zero$. Also, we have the following properties.

\begin{rem}\label{rem:tcng.units}  For  a \qcong\ $\Cong$,   Remarks  \ref{rem:tng.cluster} and \ref{rem:cong.ideal} can be strengthened.
\begin{enumerate}
 \eroman

\item If $u$ is a unit, then $[u]$ is a unit of $R/ \Cong$, by Remark \ref{rem:tng.cluster}.(ii),  and  thus
$\pi_\Cong(\RX) = (R/\Cong)^\times$, implying that $\pi_\Cong: R \Onto R/\Cong$ is a local \hom\ (Definition \ref{def:mon.hom}).

\item 
Any  \qcong \ is a proper congruence having at least three equivalence classes (Definition \ref{def:cong.types}): a tangible class, a ghost class, and the zero class.

\item If $ u = a +b$, then  $a$ and  $b$ cannot  be ghostified simultaneously  by a \qcong\ $\Cong$, since otherwise $u$ would be congruent to a ghost, implying that  $\iTcl(\Cong)  = \emptyset$, cf. Remark \ref{rem:tng.cluster}.(v).

\item From Remark \ref{rem:tng.cluster}.(iii) we learn that if a subset $E \subset R$ contains a unit, then the  ghostifying  congruence $\gCong_E$, cf.  \eqref{eq:G.E},  is not a \qcong.

\item $\Cong$ is not  a \qcong\ whenever $\ciTcl(\Cong)$ contains a \tualt\ element.
%
 \end{enumerate}
\end{rem}

We observe that for  \qcong s the pathologies in Example \ref{exmp:cong} are dismissed, where
 the compatibility of \qcong s  with \qhom s follows obviously.
The quotient $R/\Cong$ of a \nusmr \ $R:= (R, \tT, \tG, \nu)$ by a \qcong\ $\Cong$  is  again  a \nusmr\ (Lemma \ref{lem:qsmr}), given   as
\begin{equation}\label{eq:qnusmr}
R/\Cong = (R/\Cong, \ \iTcl(\Cong)/\Tcl {\Cong} , \
 \iGcl(\Cong)/\Gcl{\Cong}, [\nu]),
\end{equation}
 where  the  ghost map $[\nu]$ of $ R/\Cong$ is induced from the  ghost map  $\nu:R \To \tGz$ of $R$
via $[a]^{[\nu]} = [a^\nu].$
Moreover, the canonical surjection (Remark \ref{rem:cong1}) $$\pi_\Cong: R \ONTO
R/\Cong, \qquad a \longmapsto [a],$$
   is a \qhom
   \ (Definition \ref{def:nusmr.hom}), in particular $\pi_\Cong(\rone) = \one_{R/\Cong}$, with
\begin{equation}\label{eq:gker2}
 \iTcl(\Cong) = \tcor(\pi_\Cong)
 \dss{\text{and}} \iGcl(\Cong) = \gker(\pi_\Cong) .
\end{equation}
(See respectively Proposition \ref{prop:rpim} and Lemma \ref{lem:nu.hom}.)

\begin{rem}\label{rem:qcong.intersection}
The intersection of two \lcong s  $\Cong_1$ and $\Cong_2$ need not be an \lcong, since $\iTcl(\Cong_1  \cap \Cong_2)$ is not necessarily closed for multiplication of tangibles. For example, take \lcong s  $\Cong_1$ and $\Cong_2$ such that $a_i \in \iTcl(\Cong_i)$ and $a_i \notin \iTcl(\Cong_j)$ for  $i \neq j$, where   $i,j = 1,2$. Then, both
$a_1, \, a_2 \in \iTcl(\Cong_1 \cap \Cong_2)$, but $a_1 a_2 \notin \iTcl(\Cong_1  \cap \Cong_2)$.
However,  $\iTcl(\Cong_1  \cap \Cong_2)$ is nonempty, it contains the group ~$\RX$, and thus  $\Cong_1  \cap \Cong_2$ is a \qcong.
On the other hand,  the intersection of \qcong s  is a \qcong, since~ $\RX$ is contained the intersection of their tangible clusters.
 Therefore, $\QCng(R)$ is closed for  intersection.
\end{rem}

Suppose  $\Cong'$ is a congruence contained in a \qcong\ $\Cong$, then $\Cong'$ is a \qcong.
Indeed, $\Cong' \subset \Cong$ implies by \eqref{eq:inclusion.tcls} that
  $\Tcl{\Cong'} \supseteq \Tcl{\Cong}$, and hence $\RX \subseteq  \Tcl{\Cong'}$.

We specialize the congruence closures from \eqref{eq:cong.closure}
to \qcong s  $\Cong_1, \Cong_2$  by setting
\begin{equation}\label{eq:q.cong.closure}
  \overline{\Cong_1 \cup  \Cong_2} :=  \hskip -2mm  \bigcap_{\scriptsize \begin{array}{c}
 \Cong \in \QCng(R)\\
\Cong_1 \cup \Cong_2 \subseteq \Cong
\end{array}} \hskip -2mm  \Cong \ ,
\qquad
\overline{\Cong_1 +  \Cong_2} := \hskip -2mm  \bigcap_{\scriptsize \begin{array}{c}
 \Cong \in \QCng(R)\\
\Cong_1 + \Cong_2 \subseteq \Cong
\end{array}} \hskip -2mm  \Cong \ ,
\end{equation}
to obtain  these closures  in $\QCng(R).$

\begin{rem}\label{rem:induced.cong}
Let $R$ and $R'$ be \nusmr s.
\begin{enumerate} \eroman
  \item For any \qcong\  $\Cong$ on $R$,  the canonical
  \qhom\ $\pi_\Cong: R\Onto
R/\Cong$   induces a one-to-one order preserving correspondence between the \qcong s (rep. \lcong s) on  $R$
which contain ~$\Cong$ and the \qcong s (resp. \lcong s) on   $R/\Cong$.

\item Given a \qhom\ $\vrp: R \To R'$, the  pullback $\vvrp(\Cong')$ of a \qcong\ $\Cong'$ on $R'$ (Remark~ \ref{rem:cong1}) is a \qcong\ on $R$. Indeed, the \ccong\
    $\vvrp(\Cong')$  on $R$ is defined via   $\vrp(a) \cng' \vrp(b)$, where  $\vrp(a) \in \iTcl(\Cong') $ 
     implies $a \in \tT$ and $a \in \iTcl(\vvrp(\Cong'))$,  since $\vrp$ is a \qhom, in particular $\RX \subseteq \iTcl(\vvrp(\Cong'))$. By the same argument, if $\Cong'$ is an \lcong, then $\vvrp(\Cong')$ is an \lcong.
\end{enumerate}

\end{rem}

\pSkip

In the sequel of this paper we extensively rely on \qcong s, especially to define radical, prime, and maximal \ccong s. To preserve our objects in the category $\NSMR$ of \nusmr s, we comply  the principles:
\begin{itemize}\dispace
  \item[$\maltese$]\emph{quotienting is done only by \qcong s,}
  \item[$\maltese$]\emph{localization is  performed only by \lcong s.}
\end{itemize}
As seen later,  \qcong s and \lcong s appear naturally in various ways.
Note that a \gcong\ $\gCong_E$ need not be a \qcong , e.g., see Remark \ref{rem:tcng.units}.(iv).

\begin{lemma}\label{lem:a.vs.qcong}
     Given $a \in R$ where $a^k \notin \tG$ for every $k$,   there exists a \qcong \ $\Cong$ such that $\ciGcl(\Cong)$ is a multiplicative monoid which contains $a$, i.e.,  $a \notin \iGcl(\Cong)$.
\end{lemma}
\begin{proof} Note that $a$ by itself is not ghost.
Let $\tJ$ be the set of all \qcong s $\Cong$ on $R$ such that no power of  $a$ is in $\iGcl(\Cong)$.  First, $\tJ$ is not empty since it contains the trivial congruence $\diag(R)$, as $a^k \notin \tG$ for every $k$.  Second, as ghost projections  are semiring ideals (Remark \ref{rem:cong.ideal}), from Zorn's lemma it follows that $\tJ$ has at least one \qcong\
$\olCong$ with maximal ghost projection  $E = \iGcl(\olCong)$. Moreover,  $ E \neq R$, since $\olCong$ is a \qcong.

 Suppose $b,c \notin E$, and let $E' := E \cup \{b\}$, $E'' := E \cup \{c\}$. Then, both \gcong s  $\gCong_{E'}$ and $\gCong_{E''}$ do not belong to $\tJ$, and
thus there exist powers $m,n$ such that $a^m \in \iGcl (\gCong_{E'})$ and $a^n \in \iGcl (\gCong_{E''})$. Hence, by \eqref{eq:a.in.igcl}, we have  $a^m = g' + \sum e'_i s'_i + b t' $  and $a^n = g'' + \sum e''_j s''_j + c t'' $  for some  $e'_i, e''_j \in E$,  $s'_i, s''_j, t', t'' \in R$, $g',g'' \in \tG$. Computing the product  $a^m a^n$, we get
$$\begin{array}{rll}
a^{m+n} & = \big(g' + \sum e'_i s'_i + b t'\big) \big(g'' + \sum e''_j s''_j + c t'' \big)\\[1mm]
& = \big(g'g''+  g'\sum e''_js''_j + g'c t'' + g'' \sum e'_i s'_j + g''b t'\big) +
\big (\sum e'_i s'_i  e''_js''_j +ct'' \sum e'_i s'_i + bt'\sum  e''_js''_j \big ) + bct't''\\[1mm]
& =g +  \big( \sum e'_i s'_i  e''_js''_j +ct'' \sum e'_i s'_i + bt'\sum  e''_js''_j \big) + bct't''  ,
\end{array}
$$
which belongs to $\iGcl(\gCong_{K})$, where $K=  E \cup \{ bc\} $.
 Consequently, $\gCong_{K} \notin  \tJ$ and   $bc \notin E$, since $\gCong_E \subseteq \olCong$. Therefore,
 $\ciGcl(\olCong)$ is a multiplicative monoid, where  $a \notin  \iGcl(\olCong)$.
\end{proof}

\subsection{\Intr\ congruences}\label{ssec:intr.cong} \sSkip

Remark \ref{rem:cong.order} deals with equivalence relations  which are derived from ghostifiying elements through   the underlying \numon\ addition (Lemma \ref{lem:ghostification}).  However, ghostification does not capture ``direct'' relations on non-ghost elements. For example, it does not encompass  cases  where
  $a' \cng a''$ for \nuequivalent\ elements $a', a'' \notin \tG$, i.e., $a'\nucong a''$, namely when $a'$ and $a''$ are  contained in a \nufiber \ $\nufib(a)$, cf. \eqref{eq:nu.fib}.

\begin{defn} A \textbf{(tangible) \intr\ congruence} $\sCong_a$, written \textbf{\scong}, of an element $a \in \tT$ is a \cqcong\ whose underlying  equivalence $\scng$ is determined by
 \begin{equation}\label{eq:str.elm}
 a \scng a' \quad \text{ for all }  a' \in \tT \text{ such that }  a' \nucong a,
\end{equation}
 i.e.,  $a + b  \scng a' + b$ and $a  b  \scng a'  b$ for every $b \in R$. The \intr\ congruence $\sCong_E$ of a subset $E \subseteq \tT$ is the \cqcong\ with the equivalence  $\scng$ applies  for all $a \in E$.
The set of  all \scong s on $R$ is denoted by ~ $\SCng(R)$.

  \end{defn}
The \intr\ congruence $\sCong_a$ of an element  $a \in \tT$ effects  directly  the  tangible members of the \nufiber\ $\nufib(a)$ and unite all of them to a single tangible element.
\begin{lemma}
  $\sCong_E$ is  a \qcong.
\end{lemma}
\begin{proof} To see that $\sCong_E$ is a \ccong,  suppose that  $a \nucong a'$ and both $a$ and $a'$ are tangible belonging to ~$E$, then $[a + a'] = [a^\nu] = [a]^\nu$. On the other hand $[a] + [a'] = [a]^\nu$, since $[a] = [a']$, and  hence $\sCong_E$ respects the ghost map $\nu$. The verification that $\sCong_E$ preserves addition and multiplication is routine.
Since $\sCong_E$ unites only tangibles with tangibles from a same \nufiber,  then $\iTcl(\sCong_E)$ is nonempty, containing $\RX$, and thus $\sCong_E$ is   a \qcong.
\end{proof}

The inclusions $\sCong_a \subset \sCong_E \subset  \sCong_\tT $ hold for any $a \subset E \subset \tT$.

\begin{lem}\label{lem:str.2.sml}
The quotient \nusmr\ $R / \sCong_\tT$ of a \nusmr\ $R$ by the \scong\ $\sCong_\tT$ is a faithful  \nusmr \ (Definition \ref{def:nusemiring}).

\end{lem}
\begin{proof} Immediate, since $\nufib\big([a]\big)$ of any $[a] \in R / \sCong_\tT$ contains at most one tangible element.
\end{proof}

\begin{example}\label{exmp:cong.smf.1} All \qcong s on a supertropical semifield $F$ (Definition \ref{def:supertropical})  are  \intrc s. Indeed, the ghost ideal of $F$ is totally ordered, while its \tsset\ is an abelian group; namely,  each tangible element in $F$ is a unit.  Therefore, by Remark \ref{rem:tng.cluster}.(iv), a \qcong\ can only unite elements $a \nucong b$ which are  \nuequivalent.

If, furthermore,  $F$ is faithful (Definition \ref{def:nusemiring}), then the trivial congruence $\diag(F)$ (Definition \ref{def:cong.types}) is the unique  \qcong\ on $F$.
In this case, the only  possible congruences on $F$ are  the trivial \ccong,  the zero congruence, and the congruence defined by   $ a \cng b$ for all nonzero $a,b \in F$.
This case generalizes \cite[Lemma~ 11]{Lic}.
\end{example}

One benefit of \scong s is that $\SCng(R)$ contains the unique maximal \scong\ $\sCong_\tT$, which is determined solely by the \nuequivalence\ $\nucong$, applied only to tangibles. In general two non-tangible $a,a'$ in $\nufib(a)$
 cannot always be unified,
as the congruence properties may be violated. However, one can identify all the non-tangible elements in $\nufib(a)$ with $a^\nu$ to get a congruence, but, this \ccong\ is not necessarily a \qcong.


\begin{defn}\label{def:st.cong}
  The congruence $\cCong_R$ is defined for each $b \in \tG$ by the equivalence
 \begin{equation}\label{eq:str.elm}
 \begin{array}{llr}
 a' \ccng a \quad & \text{ for all tangibles } \  a, a'\in \nufib(b), \\[1mm]
 b' \ccng b \quad & \text{ for all non-tangibles } \ b' \in \nufib(b).
 \end{array}
\end{equation}
We call $\cCong_R$ the \textbf{structure  congruence} of $R$.
\end{defn}

When $\cCong_R$ is a \qcong,   the quotient $R/\cCong_R$ is a faithful definite \nusmr\ (Definition \ref{def:nusemiring}).

\subsection{Congruences vs. \nusmr\ localizations} \sSkip


Recall that $\equiv$ denotes the  underlying  equivalence of a congruence $\Cong$ on a \nusmr\ $R : = (R, \tT, \tG, \nu )$.
Let $\MS \subseteq \tT$ be a tangible multiplicative submonoid of $R $ with $\one \in \MS$, and let $\RMS$ be the tangible localization of $R$ by~ $\MS$ as described in  \S\ref{ssec:tngLocal}.
As $\MS$ is a tangible submonoid, all its elements are \tprs\ (cf. \S\ref{ssec:nusmr}).
Using the canonical injection $\tau_\MS: R \To R_\MS$ in
\eqref{eq:localbyMS}, a congruence ~ $\Cong'$ on~ $R_\MS$ \textbf{restricts} naturally to the  congruence $ \Cong'|_R := \Cong' \cap (\RxR)$ on~ $R$ via  $$ \frac{a}{\one} \equiv' \frac{b}{\one}\Impl a \equiv'_{|_R} b ,$$
for all $a,b \in R.$

\begin{rem}\label{rem:cong.rest}
When $\Cong'$ is a \qcong\ on $R_\MS$, so does its restriction $\Cong'|_R$ to $R$. Indeed, the ghost projection  of $\Cong'|_R$ is obtained as  $\{ b \cnd \frac{b}{c} \in \iGcl(\Cong')  \}$,  since every  $c \in \MS$ is tangible. By the same reason,  $\iTcl(\Cong'|_R) = \{ a \cnd \frac {a}{c} \in \iTcl(\Cong') \} $ is a tangible subset of $R$, containing $\RX$.
Similarly, the restriction $\Cong'|_R$ of an \lcong\ $\Cong'$ on $R_\MS$ is an \lcong\ on $R$.
\end{rem}

Conversely, a congruence $\Cong$ with equivalence $\cng $ on $R$ \textbf{extends} to the congruence $\iMS \Cong$ on $R_\MS$, whose underlying  equivalence ~ $\cng_\MS$ is given by
 \begin{equation}\label{eq:extend.cong}
 \frac{a}{c} \equiv_{\MS} \frac{b}{c'} \dss{\text{iff}} ac'c'' \equiv bcc'' \text{ for some } c'' \in \MS.
\end{equation}
 By this definition, we see that the equivalence $\sim_\MS$ in \eqref{eq:ms:1} implies the equivalence  $\cng_\MS$.
 We set
 $$\iMS \Cong := \bigg\{ \big(\frac a c,\frac {b}{c'} \big) \ds \vert \frac{ a}{ c} \equiv_{\MS} \frac{ b}{c'} \bigg\} \ds \subset R_\MS \times R_\MS,$$
which is a  congruence on $\RMS.$
%
%
The following diagram summarizes the structures we have so far
$$\xymatrix{
R \ar@{->>}[rr]^{\pi_\Cong } \ar@{_{(}->}[d]^{\tau_\MS} && R/\Cong \ar[d]^{\phi} \\
\RMS \ar@{->>}[rr]^{\pi_{\iMS\Cong }} && \RMS/(\iMS \Cong) \\
}
$$
We aim for additional comprehensive relations on these \nusmr s, for the case that $\Cong$ is a  \qcong.
\begin{lem}\label{lem:tcor.1}
     Let $R_\MS$ be the localization of $R $ by a tangible multiplicative submonoid $\MS \subseteq \tT$, and let $\Cong$ be a \qcong\ on $R$.
\begin{enumerate} \eroman

 \item   $\frac{a}{c} \notin \iTcl(\iMS\Cong)$ iff $ac'' \notin \iTcl(\Cong)$ for any $c'' \in \MS$.

  \item  If $\frac{a}{c} \in \iPcl(\iMS\Cong)$, then  $a \in \iPcl(\Cong)$.

\end{enumerate}
\end{lem}

\begin{proof} (i): $(\Rightarrow)$:
Suppose $a c'' \notin \iTcl(\Cong)$ for some $c'' \in \MS$, that is $a \cng b$ for some non-tangible $b$ in $\Cong$. Then, $acc'' \equiv bcc''$ for all $c \in C$, implying that  $\frac{a}{c} \cng_\MS \frac{b}{c}$, which is not tangible, and  hence $\frac{a}{c} \notin \iTcl(\iMS\Cong)$ -- a contradiction.

\noindent
$(\Leftarrow)$: 
Suppose $ac'' \in \iTcl(\Cong)$, then $\frac{a c''}{c''} = \frac{a}{\one} \in \iTcl(\iMS\Cong) $, and hence $a \in \iTcl(\Cong)$, implying that $\frac{a}{c} \in \iTcl(\iMS\Cong)$ -- a contradiction.
\pSkip
(ii): $\iPcl(\iMS\Cong)$ is closed for multiplication and thus $\frac{a}{\one}, \frac{b}{\one}  \in \iPcl(\iMS\Cong) \Rightarrow
 \frac{ab}{\one}   \in \iPcl(\iMS\Cong)$, implying that  $ab \in \iPcl(\Cong)$. On the other hand, if  $a \notin \iPcl(\Cong)$ is \tprs, then there exists $b \in \iPcl(\Cong)$ such that $ab \notin \iPcl(\Cong)$. Hence $\frac{ab}{c} \notin \iPcl(\iMS\Cong)$ for every $c \in \MS$, and thus
$\frac{a}{c} \notin \iPcl(\iMS\Cong)$, since $\frac{b}{\one} \in \iPcl(\iMS\Cong)$.
\end{proof}
Part (ii) of the  lemma shows that the restriction of a tangible submonoid $\iPcl(\iMS\Cong)$ of  $\RMS$ to $R$ is a tangible submonoid of $R$.
\pSkip

Recall that for an \lcong\ $\Cong$ the tangible projection $\iTcl(\Cong)$ is a multiplicative  monoid.
The next proposition  provides the converse of Remark \ref{rem:cong.rest} for \lcong s and  plays a central role in our forthcoming study.
\begin{prop}\label{prop:cong.1}
  Let $R_\MS$ be the tangible localization of $R$ by a tangible  multiplicative submonoid $\MS \subseteq \tT$.
  \begin{enumerate} \eroman
    \item An \lcong\ $\Cong$ on $R$ extends to an \lcong\ $\iMS \Cong$ on $R_\MS$ if and only if $\MS \subseteq  \iTcl(\Cong)$.
    \item The restriction   $\Cong'|_{R} = \Cong '\cap (\RxR)$   of an \lcong\ $\Cong'$ on  $R_\MS$ to $R$ satisfies $ \Cong' = \iMS (\Cong'|_R)$.
  \end{enumerate}
\end{prop}
\begin{proof} Recall that $\one := \rone \in \MS$, since  $\MS \subseteq \tT$ is a tangible multiplicative submonoid. 
\pSkip
(i): Let $\Cong$ be an \lcong\ on $R$ and assume that
 $\MS \nsubseteq  \iTcl(\Cong) $. Namely, there exists a tangible
 $c \in \MS \sm  \iTcl (\Cong) $ such that $c \equiv b$ for some $b \notin \tT$.  By extending $\Cong$ to $\iMS \Cong$ we have  $\frac{c}{\one} \equiv_\MS \frac{b}{\one}$, since $\one \in \MS$, which implies $\frac{\one}{c}\frac{c}{\one} \equiv_\MS \frac{\one}{c} \frac{b}{\one} = \frac{b }{c}$ --  a non-tangible element in $\RMS$.  On the other hand, $\frac{\one}{c}\frac{c}{\one} = \frac{\one}{\one} = \one_{R_\MS}$ is the identity of $R_\MS$, and thus $\one_{R_\MS}$ is non-tangible.
   Hence $(R_\MS)^\times \nsubseteq\iTcl(\iMS\Cong)$, and therefore  $\iMS\Cong$ is not an  \lcong. \pSkip

Conversely, suppose that  $\iMS \Cong$ is a congruence on $\RMS$ which  is not an \lcong. We have the  following two cases: \pSkip
 (a) $\iTcl(\iMS \Cong) = \emptyset$, i.e., $\iMS \Cong$ is not a \qcong, and in particular $\one_{\RMS} \notin \iTcl(\iMS \Cong)$.
As $\Cong$ is an \lcong, we have $\iTcl(\Cong) \neq \emptyset$ with  $\one \in \RX \subseteq \iTcl( \Cong)$.   Thus,
there are
$a \in R \sm \MS$ and $c\in \MS$, where $a$ is non-tangible,    such that $\frac{a}{c} \equiv_{\MS} \frac{\one}{\one} = \one_{\RMS}$ in $R_\MS / (\iMS \Cong)$.
This means that $ac'' \equiv c c''$ for some $c'' \in \MS$, where  $ac''$  is non-tangible (otherwise $\frac{ac''}{\one} \frac{\one}{c''} = \frac{a}{\one}$ would be tangible, cf. Lemma \ref{lem:tcor.1}.(i)).   Therefore, since $ac'' \equiv c c''$, we obtain  that  $cc'' \notin \iTcl(\Cong)$. But, $cc'' \in \MS$, since ~$\MS$ is a monoid, and hence $\MS \nsubseteq  \iTcl(\Cong) $.
\pSkip
(b) $\iTcl(\iMS \Cong) \neq \emptyset$.
If $(R_\MS)^\times \nsubseteq \iTcl(\iMS \Cong)$, then $\one_{\RMS} \notin \iTcl(\iMS \Cong)$,
and we are done by part~ (a). Otherwise, $\iTcl(\iMS \Cong)$  is not closed for multiplication, i.e., $\iMS \Cong$ is a \qcong\ but not an \lcong. Let $\frac {a_1}{c_1}, \frac{a_2}{c_2} \in \iTcl(\iMS \Cong)$, which implies  $a_1, a_2 \in \iTcl( \Cong)$ by Lemma \ref{lem:tcor.1}.(i), and hence $a_1a_2\in \iTcl( \Cong)$, since  $\Cong$ is an \lcong.

Assume  that  $\frac {a_1}{c_1} \frac{a_2}{c_2} \notin \iTcl(\iMS \Cong)$, which gives  $\frac {c_1}{\one} \frac {a_1}{c_1} \frac{a_2}{c_2}  \frac{c_2}{\one} = \frac{a_1 a_2}{\one} \notin \iTcl(\iMS \Cong)$, since $c_1 c_2 \in \MS$.
Namely,  $\frac {a_1a_2}{\one} \cng_\MS \frac{b}{c'}$ for some non-tangible $\frac{b}{c'} \in \RMS$, where $b \notin \tT$ and $c' \in \MS$ is tangible.
This implies that $a_1 a_2 c' c'' \cng b  c'' $ over  $R$, for some $c'' \in \MS$, where $bc''$ is non-tangible in $R$ (cf. Lemma \ref{lem:tcor.1}.(ii)). Hence $a_1 a_2 c' c'' \notin \iTcl(\Cong)$,
while  $a_1 a_2 \in \iTcl(\Cong)$,  implying that $c'c'' \notin  \iTcl(\Cong)$, since $\iTcl(\Cong)$ is a submonoid. Thus,    $\MS \nsubseteq  \iTcl(\Cong)$, since $c c'' \in \MS$.

\pSkip
(ii): To show that $\Cong'\subseteq  \iMS (\Cong'{|_R})$, let $\Cong' \subset \RMS \times \RMS$, and  take  $(\frac{a}{c}, \frac{b}{c'}) \in \Cong'$. Now assume that $(\frac{a}{c}, \frac{b}{c'}) \notin \iMS (\Cong'{|_R}),$ which means that  $ac'c'' \not\equiv'_{|_R} bcc''$ for all $c'' \in \MS$, and in  particular for $c'' = \one$. But then $ac' \not\equiv'_{|_R} bc$,  implying that  $\frac{ac'}{\one} \not\equiv' \frac{bc}{\one}$, and thus $\frac{a}{c} \not\equiv' \frac{b}{c'}$ -- a contradiction.
%
%
\pSkip
The opposite inclusion $\Cong'\supseteq  \iMS (\Cong'{|_R})$ is trivial.
\end{proof}

Let $\Cong$ be an \lcong, and let $\MS_\Cong := \tcor(\pi_\Cong)$ be the \tcore\ (Definition \ref{def:nusmr.hom}) of the canonical \qhom\ $\pi_\Cong: R \Onto R/\Cong$.  Then   $\MS_\Cong \subseteq \tT$ is a tangible multiplicative submonoid, by Corollaries  ~ \ref{cor:rpim} and \ref{cor:lcong.vs.tcls}, equals to  $\iTcl(\Cong) \subseteq \tT$ by \eqref{eq:gker2},  which localizes $R$ as in \eqref{eq:homLocal}.
 \begin{defn}\label{def:cong.localization}
 We define the \textbf{localization of a \nusmr\ $R$ by an \lcong}\ $\Cong$ to be the tangible localization (Definition \ref{def:tangible.localization})
$$R_\Cong := \iMS_\Cong R = (\tcor(\pi_\Cong))^{-1} R, \qquad \text{where } \MS_\Cong = \iTcl(\Cong),
$$
whose elements are fractions $\frac ac$ with tangible $c \in \MS_\Cong$.
\end{defn}
Given an \lcong\ $\Cong$ on~ $R$,  we write
\begin{equation}\label{eq:congR}
 \Cong_R:= (\iTcl(\Cong))^\inv \Cong \quad
\end{equation}
for the extension of $\Cong$ to the localization $R_\Cong$ of $R$ by $\Cong$, cf. \eqref{eq:extend.cong}. Then, by Corollary \ref{cor:lcong.vs.tcls}, $R_\Cong / \Cong_R$ is a \tclsnusmr\ (but not necessarily a \nudom), in which
every tangible element is a unit.
For the localized \nusmr\  $R_\Cong := (R_\Cong, \tT_\Cong, \tG_\Cong, \nu_\Cong)$, this gives the diagram
\begin{equation}\label{eq:loc.diag}
\xymatrix{
R \ \ar@{->}[rrd]_{\phi } \ar@{^{(}->}[rr]^{\tau_{\MS_\Cong}} && R_\Cong \ar[d]^{\pi_{\Cong_R}} \\
 && R_\Cong/\Cong_R \ ,  & \\
}
\end{equation}
which later helps us to construct local \nusmr s.

\subsection{Prime congruences and \nusmr\ localization} \sSkip
Let $\Cong$ be a \qcong\ with underlying  equivalence $\cng$ on a \nusmr\ $R := (R, \tT, \tG, \nu )$.
Recall that we write $$a \cng \ghost \dss{\text{if}} a \cng b \text{ for some } b \in \tG,   $$
and that by Lemma \ref{lem:ghs.eqv} this condition is equivalent to $a \cng a^\nu$. To ease the exposition, we use both notations. We write $(R/\Cong)|_\tng^\PrS$   for the \tprsset\ of the quotient \nusmr\  $R/ \Cong$,
$(R/\Cong)|_\tng$   for its \tsset, and $(R/ \Cong)|_\ghs$ for its ghost ideal.

\pSkip

The following definition is a  key definition in this paper.
\begin{defn}\label{def:prmCng}
A  \textbf{\gprimec} $\pCong$ (alluded for \textbf{ghost prime congruence})  on a \nusmr\ ~$R$ is an \lcong\ whose underlying  equivalence $\pcng$ satisfies  for any $a,b \in R $ the condition
  \begin{equation}\label{eq:prime.0}
    a b \pcng \ghost \ds \Dir a \pcng  \ghost  \text{ or } \ b \pcng \ghost.
     \end{equation}
         That is, $ab \in \iGcl(\pCong) $ implies $a \in \iGcl(\pCong)$ or $b \in \iGcl(\pCong)$, or equivalently $a,b \notin \iGcl(\pCong) $ implies $ab \notin \iGcl(\pCong)$.
     The \textbf{\gprime\ (congruence) spectrum} of $R$ is defined as
  \begin{equation}\label{eq:primeSpec}
    \Spec(R) := \{ \pCong \ds | \pCong \text{ is a \gprimec\  on }  R \  \}.
     \end{equation}

\end{defn}

In supertropical context, being  ``prime'' for a congruence  $\pCong$ essentially means that a \emph{product of two non-ghost elements can never be equivalent to a ghost}, while its tangible projection $\iTcl(\pCong)$ is a (tangible) monoid.
By Lemma \ref{lem:ghs.eqv},  Condition \eqref{eq:prime.0} is stated equivalently as
 \begin{equation}\label{eq:prime}
    a b \pcng (ab)^\nu \ds\Dir a \pcng  a^\nu \text{ or } b \pcng b^\nu  ,
     \end{equation}
     which means that the ghsotification of a product $ab$ by $\pCong$ is obtained by ghostifying at least one of its terms.
Also,   from  \eqref{eq:q.cong} we obtain
\begin{equation*}
  ab \notin \iTcl(\pCong) \Dir a \notin \iTcl(\pCong) \text{ or } b \notin \iTcl(\pCong),
\end{equation*}
since a \gprimec\ is an \lcong,  in which $\iTcl(\pCong)= \iPcl(\pCong)$. (Thereby $a
\notin \iTcl(\pCong)$ for each  $a \notin \tTPS$.)
Accordingly, localizing  by a \gprimec\ $\pCong$ is the same as localizing  by an \lcong 
, written specifically $R_\pCong$.


\begin{remark}\label{rem:mult1} 
 The condition that  if $ab \in \iGcl(\pCong)$ then $a$ or $b$ belongs to $\iGcl(\pCong)$ shows that  the ghost projection $\iGcl(\pCong)$ of a \gprimec\ $\pCong$ is a \gprime\ ideal  of $R$ containing its ghost ideal  ~$\tGz$ (Definition ~\ref{def:nuprime.ideal}). Furthermore, it implies that the complement  $\ciGcl(\pCong) = R  \sm \iGcl(\pCong)$ of $\iGcl(\pCong)$  is a multiplicative submonoid of $R$, consisting of non-ghost elements. In particular, it  contains the \tprsmon\ $\iTcl(\pCong)= \iPcl(\pCong)$ which is employed to execute localization.
 \end{remark}

We derive the following observation, which is of much help later.

\begin{rem}\label{rem:biject.proj} The canonical surjection $\pi_\Cong: R \Onto R/\Cong$ and its inverse $[a] \longmapsto \pi_\Cong^{-1}([a])$ define a one-to-one correspondence between all \gprimec s on $R/\Cong$ and the \gprimec s  on $R$ that contain ~$\Cong$, cf. Remark \ref{rem:induced.cong}.(i).

If $\vrp: R \To R'$ is a \qhom\ and $\pCong'$ is a \gprimec \ on $R'$, then $\vvrp(\pCong')$ is  a
\gprimec\ on $R$. Indeed, first, $\vvrp(\pCong')$ is an \lcong\ on $R$ by Remark \ref{rem:induced.cong}.(ii). Second, for the ghost component, recall that  $\vvrp(\pCong')$ is determined by the equivalence $\cng$ given by   $a \cng b$ iff $\vrp(a) \pcng' \vrp(b)$.
Then, $ab \cng \ghost$ iff $\vrp(ab) = \vrp(a)\vrp(b) \pcng' \ghost $ iff $\vrp(a) \pcng' \ghost$ or $\vrp(b) \pcng' \ghost $ iff $a \cng \ghost$ or $b \cng \ghost$.
Therefore, $\vvrp(\pCong')$ is a \gprimec.

\end{rem}

We start with an easy characterization of  the simplest type of \gprimec s, arising when the clusters of an \lcong\ are the complement of each other.
\begin{lem}\label{lem:p.cong.1}
An \lcong\ $\Cong$  in which  $\Cong = \Tcl{\Cong} \cup \Gcl{\Cong}$  is a \gprimec.
\end{lem}
\begin{proof}
  Assume that $ab \cng \ghost$ where both $a \not \cng
\ghost$
 and $b \not \cng \ghost$, hence  $a,b \in \iTcl(\Cong)$, as $\Cong = \Tcl{\Cong} \cup \Gcl{\Cong}$ is a disjoint union.  But  $\iTcl(\Cong)$ is a multiplicative monoid of $R$, since $\Cong$ is an \lcong,  and thus  $ab \in \iTcl(\Cong)$, implying that  $ab \not \cng
\ghost$ -- a contradiction.
\end{proof}

\begin{example}\label{examp:prime} $ $
\begin{enumerate} \eroman
\item The trivial congruence $\diag(R)$ on a \nudom\ $R$ is a \gprimec.
  \item Any \gprimec\ on a supertropical semifield (Definition \ref{def:supertropical}) is an \intrc,  cf. Example ~\ref{exmp:cong.smf.1}.
  \item  All  \gprimec s
   $\pCong$ on a definite \nusmf\ $F$ are determined by equivalences on  the tangible monoid $\tT$ of $F$, which is an  abelian group (Definition \ref{def:nudomain}). Therefore, $\iPcl(\pCong) = \iTcl(\pCong) = \tT$ for all $\pCong$ on~ $F$, where their equivalence classes are varied.

  \item  Let $R:= \tlF[\lm]$ be the (\tcls) \nusmr\ of polynomial functions over a \ssmf\ ~$F$. For any $E = \{ \lm + a \} $ the \gcong\ $\mfG_E$ (cf. \eqref{eq:ghost.smr}) is a \gprimec\ on~ $R$.  The same holds for any factor as in Theorem \ref{thm:poly.factorization}, with the respective conditions. 

\item The polynomial function
$f  = g_1 g_2 =(\lm_1 + \lm_2 + \one)(\lm_1 + \lm_2 + \lm_1 \lm_2 )                                                                          = (\lm_1 + \one) (\lm_2 + \one)(\lm_1 + \lm_2 ) = h_1 h_2 h_3$  in Example~ \ref{examp:non-unique-factorization} demonstrates a pathological  behavior where  unique factorization fails. Although the factor $g_1$ is irreducible, its ghsotfication congruence  $\gCong_E$, $E= \{g_1\}$, is not a \gprimec, as $f \in \iGcl(\gCong_E)$, but the $h_i$'s  are  not necessarily contained in  $\iGcl(\gCong_E)$.

\end{enumerate}
\end{example}

Definition \ref{def:prmCng} lays  the first supertropical analogy to the familiar connection between prime ideals and  integral domains in ring theory.

\begin{prop}\label{prop:prime2domain} Let $\Cong$ be a congruence on  a \nusmr \ $R$.
  The quotient $R/\Cong$ is a \nudom\ (Definition ~\ref{def:nudomain}) if and only if $\Cong$ is a \gprimec.
\end{prop}
\begin{proof} $(\Rightarrow):$ Let  $R/ \Cong $
  be a \nudom. As $R/ \Cong $ is a \tclsnusmr, $\iTcl(\Cong)$ must be a monoid, so $\Cong$ is an \lcong.
   For $ab \cng \ghost$ in $R$, we  have $[ab] = [a][b] \in (R/\Cong)|_\ghs$,  implying that $ [a] \in (R/\Cong)|_\ghs$ or $[b] \in (R/\Cong)|_\ghs$, since $R/\Cong$ has no ghost divisors. This is equivalent to $a \cng \ghost$  or $b \cng \ghost$, and thus $\Cong $ is \gprime. \pSkip
$(\Leftarrow):$ Let $\Cong = \pCong$ be a \gprimec. $R/\pCong$  is  \tclsnusmr, by Corollary \ref{cor:lcong.vs.tcls}, since $\pCong$ is an \lcong.
Suppose that  $[a] [b] \in (R/\pCong)|_\ghs$,
then $[ab] \in (R/\pCong)|_\ghs$. Namely $ab \pcng \ghost$, implying  that $a \pcng \ghost$ or
$b \pcng \ghost$, since $\pCong$ is \gprime. Hence, $[a] \in (R/\pCong)|_\ghs  $ or $[b] \in (R/\pCong)|_\ghs$,
and thus  $R/\pCong$ has no ghost-divisors and is a \nudom.
\end{proof}

From Proposition \ref{prop:prime2domain} we deduce that a \nusmr\ $R$ is a \nudom\ if and only if the trivial congruence $\diag(R)$ on $R$  is a \gprimec. When $\pCong$ is \gprimec\ as in  Lemma \ref{lem:p.cong.1}, $R/\pCong$ is a \dfnt\ \nudom, cf. Example \ref{exmp:nudom} and Corollary \ref{cor:lcong.vs.tcls}.

\begin{example} Let $\STR(\NetZ)$ be the supertropical domain (Definition \ref{def:supertropical}) over $\NetZ = \Net \cup \{ 0\}$, as constructed in  Example \ref{exmp:grp.fld}, and let ~$\pCong \neq \diag(\NetZ)$ be a nontrivial  \gprimec\  on $\STR(\NetZ)$. Suppose $n \pcng n^\nu$ for some $n \in \NetZ$, then  \eqref{eq:prime} inductively implies that  $1 \pcng  1^\nu$ or $0 \pcng 0^\nu$. But $\one = 0$ is the identity of $\STR(\NetZ)$, and thus in the latter  case $\pCong$ is not a \qcong. Hence $\iTcl(\pCong) = \{ 0 \} $, which shows that all nontrivial \gprimec s on  $\STR(\NetZ)$ have the same tangible cluster.

The ghost classes of $\pCong$ can be varied, induced by the equivalences $n \pcng n^\nu$ of ~ $n \in \Net$. For a fixed $p \in \Net$,  the ghost classes are as follows:
$$ \quad [n^\nu] = \{n^\nu\} \text{ for each } 0 \leq   n < p, \qquad  [p^\nu] = \{  n^\nu \ds | n \geq p  \}. $$
Accordingly, we see that there exists  an injection of $\NetZ$ in the spectrum $\Spec(\STR(\NetZ))$.

By similar considerations, as every tangible in $\STR(\Z)$ is invertible,   the trivial congruence $\diag(\Z)$  is the only \gprimec \ on the supertropical semifield $\STR(\Z)$.
Indeed, if $n \cng n^\nu$ then  $0 = (-n)n \cng (-n)n^\nu = e$, which gives a ghost congruence.
Furthermore, by Remark \ref{rem:cong.order}.(iii), if  $n \cng m$ for $n > m$, then
$n = n + m \cng n^\nu$, which implies again $0 \cng e$.
 The same holds for
the supertropical semifields $\STR(\Q)$ and $\STR(\Real)$.
\end{example}

Recalling that  $\iTcl(\pCong)= \iPcl(\pCong)$ is a \tmon\ for every \gprimec\ $\pCong$, we specialize  Proposition \ref{prop:cong.1} to  \gprimec s.

\begin{prop}\label{prop:cong.2}
  Let $R$ be a \nusmr, and let $R_\MS$ be its  localization  by a tangible multiplicative submonoid $\MS \subset R$.
  \begin{enumerate} \eroman
    \item A \gprimec\ $\pCong$ on $R$ such that $\MS \subseteq  \iTcl(\pCong) $ extends to a \gprimec\ $\iMS \pCong$ on $\RMS$ and satisfies $ (\iMS \pCong)|_R= \pCong$.
    \item The restriction $\pCong'|_{R} := \pCong '\cap (\RxR)$ of a \gprimec\ $\pCong'$ on $\RMS$ to $R$   is a \gprimec\  satisfying $ \pCong'= \iMS(\pCong'|_R)$. In particular $\MS \subseteq \iTcl(\pCong'|_R)$.
  \end{enumerate}

\end{prop}
\begin{proof} (i): First, $\iMS \pCong$ is an  \lcong\ on $\RMS$, by Proposition \ref{prop:cong.1}.(i). Assume that $\pCong$ is a \gprimec\ on $R$. Take  $\frac{a}{c}, \frac{b}{c'} \in \RMS$ such that $\frac{a}{c} \frac{b}{c'} \equiv_\MS \frac{g}{h}$ in $\iMS\pCong$, where $\frac{g}{h}$ is a ghost of $\RMS$, which implies that $g \in \tGz$, since  $h \in \MS$ is tangible. 
  Then, $$ab(h c'') = abh c'' \pcng gcc'c'' =  g(cc'c'') $$ for some $c'' \in \MS$.
  Since $\pCong$ is \gprime\ and $g(cc'c'') \pcng \ghost$, then $ab \pcng \ghost $ or $h c'' \pcng \ghost $.
  But $h c''\in \MS \subseteq \tT$, and thus $ab\pcng \ghost $.  By the same argument, as $\pCong$ is \gprime, this implies $a \pcng \ghost$ or $b \pcng \ghost$, and therefore
  $\frac{a}{c}$ or $\frac{b}{c'}$ is ghost in $\RMS$.  Hence, $\iMS\pCong$ is a \gprimec\ on ~ $\RMS$.

  Clearly $ \pCong \subseteq (\iMS \pCong)|_R $. To verify the opposite inclusion, by the use of Proposition \ref{prop:cong.1}.(i),  we only need to show that this restriction preserves the property of being \gprime.
  Suppose $\frac{a}{\one}\in  \iGcl((\iMS \pCong)|_R)$, that is $\frac{a}{\one} \cng_\MS \ghost$ in $\iMS \pCong$. As $\frac{a}{\one} \in \RMS$, there are $a' \in R$, $c, c' \in \MS$, and $g \in \tGz$ such that
  \begin{equation}\label{eq:str}
  \frac{a}{\one} = \frac{a'}{c} \cng_\MS \frac{g}{c'} \; . \tag{$*$}
\end{equation}
From the right hand side, we have
$a'c' c'' \pcng gcc''$ for some $c'' \in \MS$, where $gcc'' \in \tGz$. This implies $a' \pcng \ghost$, since the product $c'c''$ is tangible and $\pCong$ is \gprime. From the left hand side of \eqref{eq:str}, we  have $ac = a' \pcng \ghost$, and thus $a \pcng \ghost$, as $c \in \MS$ is tangible. Therefore $a \in \iGcl(\pCong)$, and we conclude that $a \in \iGcl(\pCong)$ iff $\frac{a}{\one}\in  \iGcl((\iMS \pCong)|_R)$.

  \pSkip
  (ii): Suppose $\pCong'$ is a \gprimec \ on $\RMS$, then an equivalence  $\frac{a}{c} \frac{b}{c'} \equiv'_p \frac{g}{h}$ to some ghost $\frac{g}{h}$  implies that $\frac{a}{c}$ or $\frac{b}{c'}$ is ghost. In particular, this holds when $c = c' = h = \one$, and thus also  for the restriction~ $\pCong'|_R$. Hence, $\pCong'|_R$ is a \gprimec.
 The same argument as in the proof of Proposition \ref{prop:cong.1} shows that $ \pCong' = \iMS(\pCong'|_R)$.  Finally, $\MS \subseteq \iTcl(\pCong'|_R)$ by Proposition \ref{prop:cong.1}.(i).
\end{proof}

Having the correspondences between \gprimec s on $R$ and \gprimec s on $R_\MS$  settled, we have the following desired result.

\begin{theorem}\label{thm:prime.bij}
The canonical injection  $\tau_\MS: R \Into R_\MS$ of a \nusmr\ $R$ into its
localization $R_\MS$ by a tangible multiplicative submonoid $\MS \subset R$ induces a bijection
  \begin{equation}\label{eq:spec.inj}
    \Spec (R_\MS) \ISOTO \{ \pCong \in \Spec (R) \ds | \MS \subseteq  \iTcl(\pCong) \}, \qquad \pCong_\MS \longmapsto \pCong_\MS|_R,
  \end{equation}
  that, together with its inverse, respects inclusions between \gprimec s on $R$ and \gprimec s on
$R_\MS$.

\end{theorem}
\begin{proof}  Proposition \ref{prop:cong.1} determines the  map of \lcong s,  while  Proposition \ref{prop:cong.2} restricts to the case of \gprimec s.
\end{proof}

\subsection{\Dtrmc s} \sSkip
In classical ring theory,  prime  and maximal ideals  are  main structural ideals, where there is no significant  intermediate ideal structure between them. In supertropical theory, with our analogues congruences, we do have the following special  \lcong s.
\begin{defn}\label{def:dspec}
  A \textbf{\qcong} $\dCong$ is called  \textbf{\dtrm}, if  $\dCong = \Tcl{\dCong} \cup \Gcl{\dCong}$, i.e., it has only tangible and ghost classes. The set of all \dtrmlc s on $R$ is denoted by $$\DSpec(R):= \{ \dCong \cnd \dCong \text{ is a \dtrmc\ } \}, $$  called the \textbf{\dtrm\ spectrum} of $R$.
\end{defn}
For example, every \lcong\ on a supertropical domain $R$ (Definition \ref{def:supertropical}) is \dtrm, as each  element in $R$ is either tangible or ghost. Over an arbitrary   \nusmr\ $R$  \dtrmqc s  are heavily dependent  on the properties of $R$.  For example, if  $R$ has two non-tangibles $a,b$ whose sum $a+b$ is a unit, then $R$ cannot carry a \dtrmqc\ which unites  these elements, cf.  Example ~\ref{exmp:cong}.(i).

\begin{lem}\label{lem:dtrm.cong.prime}
  A \dtrmc\ $\dCong$ is \gprime.
\end{lem}
\begin{proof}
  Follows immediately from  Lemma \ref{lem:p.cong.1}.
\end{proof}
Accordingly, we have the inclusion $\DSpec(R) \subseteq \SpecR$ of spectra (Definition \ref{def:dspec}).

\begin{prop}\label{prop:deter2smf} Let $R$ be a \nusmr, and let $\Cong$ be a \qcong\ on $R$.
  \begin{enumerate} \eroman
    \item  $R/\Cong$ is a \dfnt\ \nusmr\ (Definition \ref{def:nusemiring})  iff $\Cong$ is a \dtrmqc.
    \item
$R/\Cong$ is a \nuddom\ (Definition \ref{def:nudomain})  iff $\Cong$ is a \dtrmlc.

  \end{enumerate}
\end{prop}

\begin{proof} (i):
$(\Rightarrow):$ As  $R/ \Cong $
  is  a \nuddom, i.e., $R/ \Cong  = (R/ \Cong)|_\tng \ds{\dot \cup} (R/ \Cong)|_\ghs$, and $\Cong$ is a \qcong,
   $[ab] = [a][b] \in (R/\Cong)_\tng$  implies $ [a],[b] \in (R/\Cong)_\tng$, and thus   $a,b \in \tT$. Otherwise, $[a][b]$ belongs to the complement $(R/ \Cong)|_\ghs$. Hence, $\Cong$ is a \dtrmqc. \pSkip
$(\Leftarrow):$ Assume that $\Cong = \dCong$ is a \dtrmqc, i.e., $\dCong = \Tcl{\dCong} \cup \Gcl{\dCong}$. Then $R/\dCong$ has only tangible and ghost classes, and thus  $R/\dCong$ is definite.
\pSkip
(ii) $\Cong$ is \gprime\ by Lemma \ref{lem:dtrm.cong.prime}, where
 $R/\Cong$ is a \nudom\  iff $\Cong$ is a \gprimec \ by Proposition~ \ref{prop:prime2domain}.
\end{proof}

\begin{example}\label{examp:deter}
Let $\dCong_\tT$ be the congruence   whose underlying equivalence  $\dcng$ is determined by the relation
$$a \dcng a^\nu \qquad \text{for all } a \notin \tT.$$
In other words,  $\dCong_\tT = \gCong_{R \sm \tT}$ is the \gcong\ that ghostifies all non-tangible elements of $R$.  We call $\dCong_\tT $ the \textbf{ghost determination} of~ $R$. (Note that $\dCong_\tT$ is not necessarily an \lcong.)
\end{example}


\subsection{Maximal \lcong s} \label{ssec:maximal.cong}
\sSkip

In ring theory, a maximal ideal is defined set-theoretically by inclusions, and induces equivalences   on its complement.
  This  characterization is sometimes too restrictive for \nusmr s, especially regarding   tangible clusters whose equivalences are not directly induced by relations on their complements.
One reason for this  is that the tangible cluster  and the ghost cluster in general are not complements of each other. Furthermore, maximality of a congruence  does not implies maximality of its tangible projection, cf. \eqref{eq:inclusion.tcls}.
Therefore, we  need a coarse setup that concerns tangible projections directly.
We begin with the naive definition.


\begin{defn}\label{def:maximalSpec}
An \lcong\ (resp. \qcong) on $R$ is  \textbf{maximal}, if it is a proper congruence and is maximal with respect to inclusion in $\LCng(R)$ (resp. in $\QCng(R)$).
The \textbf{maximal spectrum} of a \nusmr\ $R$ is defined to be
$$ \MSpec(R) := \{ \mCong \ds | \mCong \text{ is a maximal \lcong \ }\}.$$
\end{defn}

The classification of \maximalc s on \nusmr s is rather complicated, involving a careful analysis, as seen from the following constraints.
Remark \ref{rem:tcng.units} shows that in \lcong s, or in any \qcong\ on a \nusmr, units enforce major constraints. For example,
if  $ b + a  =u $ for  $u \in \RX$, we cannot have $a \cng b $, but  it may happen  that $a$ or $b$ is congruent  to a non-tangible.  Moreover, from Example~ \ref{exmp:cong}  we learn that two ordered units cannot be congruent to each other,  and that a unit  might only be congruent to other units. (In addition, tangible projections must contain \tualt\ elements.) The \maximalc s consisting  $\MSpec(R)$ must obey  these constraints.

\begin{remark}\label{rem:max.cong}
%
From Remark \ref{rem:induced.cong}.(i) it infers that $R/\mCong$ carries no nontrivial \lcong s.
\end{remark}
To link maximality to $\nu$-primeness we need an extra coarse  resolution, concerned with projections of \lcong s, more precisely with their tangible projection.

\begin{defn}\label{def:nt.maximalSpec}
  A  \textbf{\tminimalc}, alluded for \textbf{tangibly minimal \lcong}, is an  \lcong\ $\Cong$ whose    projection $\iTcl(\Cong)$  is minimal with respect to inclusion of tangible projections in~ $R$.
The \textbf{\tminimal\ spectrum} of a \nusmr\ ~$R$ is defined to be
$$ \tNSpec(R) := \{ \nCong \ds | \nCong \text{ is a \tminimalc \ }\}.$$
Namely, $\tNSpec(R)$ contains all \lcong s $\Cong \in \LCng(R)$  for which the complement $\ciTcl(\Cong)$ of the tangible projection $\iTcl(\Cong)$ is maximal in $R$.
\end{defn}

In other words, \tminimal ty of $\Cong$ is equivalent to maximality of non-tangible elements in $R/\Cong$ -- the analogy of maximal ideals in ring theory.
As only \lcong s are considered, for which $\RX \subseteq \iTcl(\Cong)$,  the trivial minimality of $\iTcl(\Cong) = \emptyset$  is excluded, and our setup is properly defined.

A \tminimalc\ $\nCong$ need not be maximal in the sense of Definition \ref{def:maximalSpec}, since it might be contained in some $\mCong \in \MSpec(R) $ with
$\iTcl(\nCong) = \iTcl(\mCong)$, but $\Tcl{\nCong} \subset \Tcl{\mCong}$.
On the other hand, recalling from \eqref{eq:inclusion.tcls} that $\Cong_1 \subset \Cong_2$ implies $\Tcl{\Cong_1} \supset \Tcl{\Cong_2}$,  a \maximalc \ is  \tminimal\ and
  $\MSpec(R) \subseteq \tNSpec(R)$.

\begin{example} All \maximalc s on a \ssmf\ $F := (F,\tT, \tG, \nu)$  share  the same tangible  projection $\tT $, while their equivalence  classes are  varied over  $\tT \times \tT$ and $R \sm \tT \times R \sm \tT$. The same holds for every \tminimalc .

The \nusmr\ $\tlF[\lm]$ of polynomial functions (e.g., $\Trop[\lm]$ in Example
\ref{exmp:extended}) carries  many \maximalc s and \tminimalc s. On the other hand, it has the unique maximal ideal $\tlF[\lm]\sm \tT$.
A \maximalc\ $\mCong$ on $\tlF[\lm]$ has tangible equivalence  classes  of the form $[a]_\mfm$ or $[a\lm + b]_\mfm$,  with  $a,b \in \tT$, as follows from Theorem~\ref{thm:poly.factorization}.
\end{example}

 A \dtrmc \ (Definition \ref{def:dspec}) need not be  maximal nor \tminimal, and vice versa, despite  its tangible cluster is the complement of its  ghost cluster.
To deal with $\nu$-primeness,  we mostly employ  \tminimalc s.

\begin{prop}\label{prop:max.cong.1} Every  \lcong\ on a \nusmr\ is contained in a \maximalc. Hence, any \nusmr\ $R$ with $\LCng(R) \neq \emptyset$ carries at least one \maximalc.
\end{prop}

\begin{proof}
  Let $\Cong'$ be an \lcong, and let  $\tA := \{ \Cong \in  \LCng(R) \cnd \Cong' \subset \Cong\}$, which is  nonempty since $\Cong'  \in \tA$.
Let $\tC$ be an inclusion chain in $\tA$, and set $\mfK = \bigcup_{\Cong \in \tC} \Cong$.
Suppose $(a_1, b_1), (a_2,b_2) \in \mfK$,  then there are $\Cong_1,\Cong_2 \in \tC$ such that $(a_i, b_i) \in \Cong_i$,
where either $\Cong_1 \subset \Cong_2$ or $\Cong_2 \subset  \Cong_1$; say the former.
Thus  $(a_1, b_1) \in \Cong_2$, so $(a_1 + a_2, b_1 + b_2) \in \Cong_2 \subset  \mfK$ and
 $x(a_i,b_i) \in \Cong_2$ for each $i$ and every $x \in R$. So $\mfK$ is a congruence. Clearly, $\RX \subseteq \iTcl(\Cong)$ for each $\Cong \in \tC$; thus $\RX \subseteq \iTcl(\mfK)$ and $\mfK$ is a \qcong.
 Furthermore, if $ab \notin \iTcl(\mfK)$ for $a,b \in \iTcl(\mfK)$, then the same holds for some $\Cong \in \tC$ -- contradicting the fact that $\Cong$ is an \lcong.
  Hence, $\mfK$ is an \lcong.

Therefore, $\tA$ is an inclusion poset of \lcong s  containing $\Cong'$, where every chain has as upper bound the \lcong\  that is the union of all the \lcong s in
the chain.
By Zorn's Lemma, we deduce that $\tA$ has a maximal element, say $\mCong$, which is an \lcong\ containing $\Cong'$,  and there exists no \lcong\ on $R$ that
strictly contains $\mCong$.
\end{proof}

In particular, by the proposition, any \tminimalc\ is contained in a \lmaximalc.

\begin{cor}\label{cor:min.cong.1}
  Any \nusmr\ $R$  with $\SpecR \neq \emptyset$ carries at least one (nontrivial) \tminimalc.
\end{cor}

\begin{proof} $\LCng(R) \supset \SpecR \neq \emptyset$ and thus  $\MSpec(R) \neq \emptyset$, by Proposition
\ref{prop:max.cong.1}, implying that $\NSpec(R) \neq \emptyset$, since
$\NSpec(R) \supseteq \MSpec(R)$.
\end{proof}

Maximality and \tminimal ity of \ccong\ can be considered equivalently for \qcong s.
However, we are mostly interested in elements of  $\Spec(R)$  and  restrict to \lcong s.

\begin{prop}
 If   $R/\Cong$ is a  \nusmf, then  $\Cong$ is a \tminimal\ \gprimec.

\end{prop}
\begin{proof}
Assume $R/\Cong$ is a \nusmf, namely $(R/\Cong)|_\tng = (R/\Cong)^\times$, where  every tangible $u \in (R/\Cong)|_\tng$ is a unit.   This means that $\iTcl(\Cong)$ cannot be reduced further by additional  equivalences of type $u \cng b$ with a non-tangible $b$, since otherwise we would get an improper \lcong\ by Remark \ref{rem:tng.cluster}.(v).
Hence, $\Cong$ is \tminimal. As  any \nusmf\ is a \nudom,  Proposition~ \ref{prop:prime2domain} implies that  $\Cong$ is a \gprimec.
%
\end{proof}

To overcome the discrepancy of  maximality in the sense of congruences on \nusmr s, we drive  the next definition which generalizes the classical notion of a ``local ring''.
\begin{defn}\label{def:local.ring}
A \nusmr\  $R$ is called \textbf{local}, if $\iTcl(\nCong) = \iTcl(\nCong')$ for all $\nCong, \nCong' \in \tNSpec(R)$.
 A quotient $R/\nCong$ is called \textbf{residue \nusmr}\ of the local \nusmr\  $R$.
\end{defn}
When $R$ is local, despite $\iTcl(\nCong) = \iTcl(\nCong')$ for all $\nCong, \nCong' \in \tNSpec(R)$, we may have $(R/\nCong)|_\tng \neq (R/\nCong')|_\tng $, since $\nCong $ and $\nCong'$ can have different tangible classes. The same holds for their ghost clusters.
By the above discussion we see that  any \nusmr\ $R$ with $\tT \subseteq  \SST$ is local; for example, when $R$ is a supertropical \smf.

\begin{rem} A \qhom\ $\vrp: R/\nCong  \To R'/\nCong'$ of  residue \nusmr s, where  $R$ and $R'$ are local \nusmr s,  is a local \hom, i.e.,  $\srHom^\inv((R'/\nCong')^\times) = (R/\nCong)^\times$ (Definition \ref{def:mon.hom}).
\end{rem}
Similarly  to \eqref{eq:congR}, we write $\pCong_R $ for $  \iMS \pCong$,  where   $\MS = \iTcl(\pCong)$.

\begin{cor}\label{cor:RmodPrime}
The localization $R_{\pCong}$  of  a \nusmr\ $R$ by a \gprimec\  $\pCong$ (Definition \ref{def:cong.localization})  is a local \nusmr\ with \tminimalc\ $\nCong = \pCong_R$, for which the residue \nusmr\ $R_{\pCong}/\nCong$ is a \nusmf.
\end{cor}

\begin{proof}
Observe that $\iTcl(\pCong_R) = (R_{\pCong})^\times $, that is  the tangible projection of $\pCong_R$
consists of the units in~ $ R_{\pCong}$; thus $\iTcl(\pCong_R)$  cannot be reduced further and is \tminimal.
Hence,  $\iTcl(\nCong) = \iTcl(\pCong_R)$ for every $\nCong \in \tNSpec(R)$,  implying that  $R_{\pCong}$ is local.

The \lcong\ $\pCong_R$ on $R_\pCong$ is \gprime\ by Proposition \ref{prop:cong.2}, and  thus $R_{\pCong}/\pCong_R$ is a \nudom\ by  Proposition \ref{prop:prime2domain}. Furthermore, since $\pCong_R$ unites each $a \notin \iTcl(\pCong_R)$ with a non-tangible element, we obtain that every tangible in $R_{\pCong}/\pCong_R$ is a unit. Therefore
$R_{\pCong}/\pCong_R$ is a \nusmf.
\end{proof}

To indicate that $\pCong_R$ is the \tminimalc\ on $R_{\pCong}$ determined by $\pCong$, we write
\begin{equation}\label{eq:max.prime}
   \nCong_\pCong \ds{:=}  \pCong_R,
\end{equation}
 and call it the \textbf{\ctminimalc} of $R_{\pCong}$. By the proof of Corollary~ \ref{cor:RmodPrime},  every element in $\iTcl(\nCong_\pCong)$ is a unit, and hence $\iTcl(\nCong_{\nCong_\pCong}) = \iTcl(\nCong_\pCong)$.

\begin{notation}\label{notat:res}
Given $f \in R$,  we write  $\brf(\pCong)$ for the equivalence class $[\frac{f}{ \one}]$ in $\nCong_\pCong$ of the localization of ~$f$ by $\pCong$. Namely, $\brf(\pCong)$ is an   element of the residue \nusmr\  $R_\pCong/ \nCong_\pCong$ --  the image of $f$ under the map composition
  $R \overset{\tau}{\Into} R_\pCong \overset{\pi}{\Onto} R_\pCong/ \nCong_\pCong$, cf. \eqref{eq:localbyMS} and \eqref{eq:canonical.qu}, respectively.

  For a subset $E \subseteq \SpecR$, we write $\brf|_E = \ghost $,  if  $\brf(\pCong)$ is ghost in
  $R_\pCong/ \nCong_\pCong$, i.e., $\brf(\pCong) \in \iGcl(\nCong_\pCong)$, for all $\pCong \in E$.
\end{notation}

\subsection{Radical congruences}
\sSkip

To approach  the interplay between ghostified subsets and congruences on \nusmr s we employ several types of radicals, which later are shown to coincide.
(Radical congruences initially need not be \lcong s, but they are \qcong s.)

\subsubsection{Congruence radicals} \sSkip

\begin{defn}\label{def:g.rad.Cng}
A \qcong\ $\rCong$ on $R$ is \textbf{\gradical}, alluded for \textbf{ghost radical}, if  its underlying  equivalence $\rcng$ satisfies   for  any $a \in R $ the condition
  \begin{equation}\label{eq:nurad.cong.0}
    a^k \rcng \ghost  \text{ for some } k \in \Net  \Dir a\rcng \ghost.
     \end{equation}
    We define the \textbf{\gradical\ spectrum} of $R$ to be
       \begin{equation*}\label{eq:radSpec}
    \RSpec(R) := \{ \ \rCong \ds | \rCong \text{ is a \gradicalc \ on }  R  \ \}.
     \end{equation*}
     The \textbf{\congradical}, written \textbf{\cradical},  of a \cqcong\ $\Cong \in \Cng(R)$  is defined as
  \begin{equation}\label{eq:rad}
    \crad(\Cong) :=   \bigcap_{\scriptsize \begin{array}{c}
 \rCong \in \RSpec(R)\\
\rCong\supseteq \Cong
\end{array}} \rCong \ .
     \end{equation}
When $\crad(\Cong) = \Cong$, we say that $\Cong$ is \textbf{\gradical ly closed}.
  \end{defn}

  Law \eqref{eq:nurad.cong.0} can be equivalently stated  by
  Lemma \ref{lem:ghs.eqv} as:
  \begin{equation}\label{eq:nurad.cong}
  \text{$a^k \rcng (a^k)^\nu$ $\Dir$  $a \rcng a^\nu$.}
  \end{equation}
    The \cradical\  is  defined for any \ccong \ $\Cong$, not necessarily a \qcong, and $\Cong$ may not be contained in any \qcong.  In this case we formally set  $\crad(\Cong)$ to be the empty set.
  \begin{lem}\label{lem:qcong.radical} A (nonempty) \cradical\
  $\crad(\Cong)$  is a \qcong\ satisfying property \eqref{eq:nurad.cong}.
 \end{lem}
\noindent Accordingly, a \cradical\  cannot be a ghost congruence, and a \gradical ly closed congruence must be  a \qcong.

\begin{proof} $\crad(\Cong) \neq \emptyset$ is a nonempty intersection of  \qcong s; hence it is a \qcong\ by Remark ~\ref{rem:qcong.intersection}.
   $\Gcl{\Cong} \subseteq \Gcl{\rCong} $  for  every  \gradicalc \ $\rCong$   containing $\Cong$,  thus \eqref {eq:nurad.cong} holds for each  $a^k  \in \iGcl(\Cong)$, with  $k \in \Net$.
\end{proof}

Clearly,  any \gprimec\ (Definition \ref{def:prmCng}) is \gradical, and therefore there are  the  inclusions
\begin{equation}\label{eq:Spec.RSpec} 
\DSpec(R) \ds\subseteq \Spec(R) \ds\subseteq \RSpec(R) .
\end{equation}
By definition,  $\RSpec(R)$ does not contain congruences that are not \qcong s, e.g., ghost congruences. Yet,  \gcong s,   which could be \gstcong s, are involved in our framework and should be considered.

\begin{rem}\label{rem:c.cong}
 In fact, formula  \eqref{eq:rad} can be applied to an arbitrary  subset of $R \times R$, not necessarily to \cqcong s, and in particular to sums $\Cong_1 + \cdots + \Cong_k$ of congruences. In this way we directly obtain   a \cqcong\ which  is the \cradical\ of the closure of a  sum of congruences \eqref{eq:q.cong.closure}.
\end{rem}

\begin{lem}\label{lem:rad.1} Let $\rCong$ be a \gradicalc\ on $R$, then
  \begin{equation*}\label{eq:c.rad.1}
  \rCong = \bigcap_{\scriptsize  \begin{array}{c}
\pCong\in \pcspec(R) \\ \pCong \supseteq \rCong
\end{array}} \pCong.
\end{equation*}
\end{lem}

\begin{proof}
 $(\subseteq)$:
  Immediate by the inclusion \eqref{eq:Spec.RSpec}.
\pSkip
$(\supseteq)$:
Each   $(a,b) \in \rCong$ is contained in every $\pCong \supseteq \rCong$,  and therefore it belongs to their intersection. \end{proof}

From  Lemma \ref{lem:rad.1} we deduce the following. \begin{cor}
\label{cor:rad.1} Let $\Cong$ be a \qcong\ on $R$, then
   \begin{equation}\label{eq:c.rad.2}
  \crad(\Cong) := \bigcap_{\scriptsize  \begin{array}{c}
\pCong\in \pcspec(R) \\ \pCong \supseteq \Cong
\end{array}} \pCong.\end{equation}
\end{cor}

This setup  leads to an abstract form of a Nullstellensatz, analogous to the Hilbert’s Nullstellensatz, now taking place over \nusmr s.

\begin{thm}[Abstract Nullstellensatz]\label{thm:Null} Let $\Cong$ be a \qcong\ on a \nusmr\ $R$, and define  $\VV(\Cong) := \{ \pCong \in \Spec(R) \cnd \pCong \supseteq \Cong \}$.\footnote{This set of congruences will be studied in much details later in \S\ref{sec:5}.}
For any $f \in R$ we have
  \begin{equation}\label{eq:null.abst}
  \brf|_{\VV(\Cong)} = \ghost \dss \Iff  f \in \iGcl(\crad(\Cong)),
     \end{equation}
     cf. Notation \ref{notat:res}.
\end{thm}
\begin{proof}
     Recall that  $\nCong_\pCong = \pCong_R$ is the \ctminimalc\ of $R_\pCong$, cf.  \eqref{eq:max.prime}.
The following hold
\begin{itemize} \dispace
\item[--] $\brf|_{\VV(\Cong)} = \ghost$ iff  \hfill\quad (by notation)
  \item[--]   $\brf(\pCong) = \ghost      $ for all $\pCong \in \VV(\Cong)$ iff  \hfill\quad (by definition)
  \item[--] $[f/\one] \in \iGcl(\nCong_\pCong) \ds \subset R_\pCong$ for all  \gprimec s  $\pCong \supseteq \Cong$  iff  \hfill\quad  (by Proposition \ref{prop:cong.2})
  \item[--] $f \in \iGcl\big(\bigcap_{\pCong \in \VV(\Cong)} \pCong \big ) \ds \subset R$ iff \hfill\quad (by  Corollary \ref{cor:rad.1})

  \item[--] $f \in \iGcl\big(\crad(\Cong) \big)$ .
\end{itemize}
\end{proof}

\subsubsection{Set radicals}\label{ssec:set.radical}\sSkip

We  turn to our second type of radicals,  applied to subsets $E \subseteq R$, possibly empty,  via their ghostifying  congruences  $\gCong_E$, or equivalently throughout ghost clusters.
\begin{defn}\label{def:s.radical}
The \textbf{set-radical  congruence}, written \textbf{\sradical}, of a  subset $E \subseteq R$ is defined as
   \begin{equation}\label{eq:radE}
\srad(E) := \crad(\gCong_E)  = \bigcap_{\scriptsize \begin{array}{c}
\pCong\in \SpecR \\  \iGcl(\pCong) \supseteq E
\end{array}} \pCong  \ .   \end{equation}
 When $E = \{ a\}$, we write $\srad(a)$ for $\srad(\{ a \})$ and say that $\srad(a)$ is a \textbf{\psradical}.

The \textbf{set-radical  closure} $\rcl(E)$ of  $E \subseteq R$ is the subset
$$ \rcl(E) :=  \iGcl(\srad(E)) \ds \subseteq R \; .$$
A subset
$E$ is called \textbf{radically g-closed} if $\rcl(E) = E \cup \tG $.
\end{defn}

 It  may happen  that $E$ is not contained in any ghost projection  $ \iGcl(\pCong)$, e.g. when $ E \cap \RX \neq \emptyset$.    In this case, we formally set  $\srad(E)$ and $\rcl(E)$ to be the empty set, for example $\srad(R) = \emptyset$ and  $\rcl(\RX) = \emptyset$.
Otherwise, an \sradicalc \ $\srad(E) $  is identified with  the \cradical\ of  the \gcong\ $\gCong_E$ of~ $E$, cf. \eqref{eq:G.E}.  By Lemma \ref{lem:qcong.radical}, this implies that $\srad(E)$ is a \qcong\ (and thus is not a ghost congruence) which obeys condition  \eqref{eq:nurad.cong}.
Therefore, the correspondence between \cradical s and  \sradical s is established.

We also see that by definition
\begin{equation}\label{eq:EEG}
 \srad(E)  = \srad(E \cup \tG),
\end{equation}
for every $E \subseteq R$, in particular  $\srad(\emptyset) = \srad(E) = \srad(\tG)$ for all $E \subseteq \tG$.
When $\srad(E)$ is not empty, $\srad(E)$  decomposes $E$ to ghost classes, while its ghost projection $\rcl(E) \subset R$   dismisses  this decomposition  and only care of being ghost or not (Remark~\ref{rem:classfor}).
Moreover, we always have $\rcl(E) \subseteq R \sm \tTX$, and $\rcl(E) \neq R$ for any $E$.
Also $E \subseteq \rcl(E) $, where  $\rcl(E) $ records the membership in the  ghost projection of~ $\gCong_E$,   linking it to a \nusmr\ ideal.

\begin{lem} A nonempty subset
 $\rcl(E)$ is a \gradical\ ideal of $R$  (Definition~ \ref{def:nuprime.ideal}).
\end{lem}
\begin{proof} The subset  $\rcl(E)$ is the ghost kernel of the surjection $\pi: R \Onto R/ \srad(E)$, which is  an ideal by Remark~ \ref{rem:cong.ideal}. If $a^k \cng \ghost$, then $a^k \in \iGcl(\srad(E))$,  implying that  $a \in \iGcl(\srad(E))$, since $a^k$  lies in the intersection of \gprimec s, cf. \eqref{eq:radE}.
\end{proof}

\begin{rem}
  For any \cqcong\ $\Cong$ we have
  $$ \srad(\iGcl(\Cong))  \subseteq \crad(\Cong), $$
since $\iGcl(\Cong)$ dismisses the decomposition into equivalence classes. On the other hand, viewing a subset $E \subseteq R$ as a partial congruence $\diag(E)$ on $R$ (Definition \ref{def:partial.cong}), we have
$$ \srad(E) = \crad(\overline{\diag(E)}) = \crad(\gCong_E).$$
We use both forms to distinguish between the different types of radicals.
\end{rem}

Considering radically g-closed subsets which contain no ghosts, we define
$$ \RSet(R) := \{ E \subseteq R \sm \tGz \ds | E \text{ is radically g-closed } \} ,$$
for which the map
\begin{equation}\label{eq:SetToRadical}
\vartheta:\RSet(R) \ISOTO \RSpec(R), \qquad E \longmapsto \srad(E),
\end{equation}
is bijective.   Indeed, each radically g-closed subset $E \subseteq R \sm \tG$ is uniquely mapped to $\srad(E)$,  since $E \cup \tG= \rcl(E) = \iGcl(\srad(E))$. Conversely, a \gradicalc\ $\rCong \in \RSpec(R)$ is mapped to its ghost projection $\iGcl(\rCong)$.

For a \gprimec\ $\pCong$ and a subset $E \subseteq \iGcl(\pCong)$, we  have  $\srad(E) \subseteq \pCong$ and
\begin{equation}\label{eq:srad.1}
E \subseteq E \cup \tGz \subseteq \rcl(E) \subseteq \iGcl(\pCong).
\end{equation}
This leads to a notion of primeness for subsets of a \nusmr.
\begin{defn}\label{def:prime.set}
  A radically g-closed subset $E \in \RSet(R)$ is called \textbf{\cprime}, if  $\rcl(E) = \iGcl(\pCong)$ for some \gprimec \ $\pCong$.
\end{defn}
The study of  \cprime\ subsets and their role in arithmetic geometry is left for future work.

\subsubsection{Properties of set radicals} \sSkip

Properties of \sradicalc s  (Definition \ref{def:s.radical}) may classify the generators of their ghost clusters, or at least determine dependence relations on these generators. These relations are  useful for the passage from subsets to congruences and vice versa.
We first  specialize \eqref{eq:a.in.igcl} in Remark \ref{rem:ghostification.smr} to tangibles.

\begin{lem}\label{lem:rad.prop.1}
Suppose $a \in R$ is  \tprs, i.e., $a\in \tTPS$,  then  $\srad(a) \subseteq \srad(b)$ if and only if $a^n = bc$ for some $c \in R$ and $n \in \Net$.
\end{lem}
\begin{proof} $(\Rightarrow)$:
Assume that $\srad(a) \subseteq \srad(b)$,  hence $a^n \in  \iGcl(\srad(b))$ for some $n \in \Net$, and  thus $a^n = bc+g$, where $g \in \tG$, by Remark~ \ref{rem:ghostification.smr}. But, $a^n$ is tangible,  whereas  $a$ is \tprs, where $g$ must be inessential by~ \eqref{eq:tan.sum},  as follows from Axiom \PSRb\ in Definition \ref{def:nusemiring}. Thus $a^n = bc$.

\pSkip
$(\Leftarrow)$: The inclusion $b \in \iGcl(\pCong)$ gives  $b \pcng b ^\nu$, by
  Lemma \ref{lem:ghs.eqv}, and thus $a^n \pcng (a^n)^\nu= (bc)^\nu$. Since $\pCong$ is \gprime, we deduce  that  $a \pcng a^\nu$. Taking the intersection of all such \gprimec s, we obtain  $\srad(a) \subseteq \srad(b)$.
\end{proof}

\begin{cor}\label{cor:rad.incl.2}  Let $b \in \tTPS \sm \gDiv(R)$ be a \tprs\ element in a \tame\ \nusmr\ $R$.
\begin{enumerate} \eroman
  \item  If
  $\srad(a) \subseteq \srad(b)$,  then $a$ is \tprs.

  \item If $a \in \iGcl(\srad(b))$,  then $a^n = bc$, for some $c \in R$ and $n \in \Net$, and thus   $a \in \tT$.
  \item  If $a \in \iGcl(\srad(E))$,  then there exists a finite subset $E' \subset E$ such that   $a^n = \sum_{j } e'_j c_j$ for some $e'_j \in E'$,  $c_j \in R$, and $n \in \Net$.
\end{enumerate}
\end{cor}
\begin{proof}
(i): As $a$ is \tprs, also $a^n$ is \tprs.  Since  $\srad(a) \subseteq \srad(b)$, from Lemma~ \ref{lem:rad.prop.1} we obtain that $a^n = bc$ for some $c \in R$. This forces that  $b,c \in \tT$, as $R$ is \tame\ (Definition \ref{def:nusemiring}), and moreover that
$b$ and $c$ are \tprs s, by Lemma~ \ref{lem:gdiv.in.tame}.(iii).

\pSkip
(ii): Follows from (i),   since $a \in \iGcl(\srad(b))$ implies that $\srad(a) \subseteq \srad(b)$.

\pSkip (iii): By (ii), $a \in \iGcl(\srad(E))$  means  that $a^n = b c$ for some $b $ that is generated by a subset $E'$ of $E$ and possibly an extra ghost term, cf. Remark \ref{rem:ghostification.smr}.
But $a^n$ is tangible, so, by \eqref{eq:tan.sum}, no extra  ghost element can be included.
\end{proof}

\subsubsection{Ghostpotent radicals} \sSkip

We turn to our third type  of radical,  emerging  in a global sense.

\begin{defn}\label{def:ghostpotent}
  An element $a \in R$ is called \textbf{ghsotpotent}, if $a^k \in \tG$ for some $k \in \Net.$ The set
  $$\gprad(R):= \{a \in R \ds |  a \text{ is ghsotpotent } \}$$
is called the \textbf{ghsotpotent  ideal of} $R$ (see Remark \ref{rem:ghost.rad.ideal} below).  A \nusmr\ $R$ is said to be  \textbf{ghost reduced}, if $\gprad(R) = \tGz$.

The \textbf{\ggradicalc}  of $R$, written \textbf{\gpradicalc}, is defined to be the \qcong
$$\grad(R) := \srad(\gprad(R)),$$
determined by \eqref{eq:radE}.
\end{defn}
As every ghost element  $a \in \tGz$ is ghostpotent, we see that $\tGz \subseteq \gprad(R)$.
On the other hand, a \tprs\ element cannot be ghostpotent, i.e.,  $\tTPS \cap \gprad(R) = \emptyset$,   but  we may have tangibles which  are  ghostpotents.
 Therefore, we immediately deduce that
 $\grad(R)$ is indeed a \qcong\ on $R$. (Alternatively, it follows from \eqref{eq:radE} and  Lemma \ref{lem:qcong.radical}.)

\begin{rem}\label{rem:ghost.rad.ideal} $\gprad(R)$ is  an ideal of the \nusmr\ $R$ (Definition \ref{def:ideal.smr}). Indeed, if $a,b \in \gprad(R)$, i.e.,  $a^{k_a} \in \tG$ and $b^{k_b} \in \tG$ for some $k_a,k_b \in \Net$, then
 \[(a+b)^{k_a+k_b} = \sum_{i=0}^{k_a+k_b} \binom{k_a + k_b}{i}a^ib^{k_a+k_b-i},\]
 where  either $i \geq k_a$ or $k_a+k_b - i \geq k_b$. So,  each term in the sum is a ghost, and $\gprad(R)$ is closed for addition.
For $c \in R$, we have $(ac)^{k_a} = a^{k_a} c^{k_a} \in \tG$, since $\tG$ is an ideal, showing that  $ac \in \gprad(R)$.

\end{rem}


Clearly, when $R$ is reduced $\grad(R) = \diag(R)$ is the trivial congruence.

\begin{lem}\label{lem:rad-reduced}
A \cqcong\ $\Cong$ is the \gpradical\ of $R$  if and only if $R/ \Cong$ is ghost reduced.
\end{lem}
\begin{proof} $\Cong$   is the  \gpradicalc\ of $R$ iff for every  $a \in R$ such that $a^k \in \iGcl(\Cong)$
also $a \in \iGcl(\Cong)$. Passing to the quotient \nusmr\  $R' := R / \Cong$, this is obviously equivalent to saying that $a^k \in \tG'$ implies $a \in \tG'$, i.e., that $R/ \Cong$ has no ghostpotents except pure ghosts.
\end{proof}

\begin{rem}\label{rem:ghostpotent} An element  $b \in R$ is a ghostpotent if and only if
  $b \in \iGcl(\rCong)$ for every \gradicalc\ $\rCong \in \RSpec(R)$ (Definition \ref{def:g.rad.Cng}). In particular
  $b \in \iGcl(\pCong)$ for every   \gprimec\ ~$\pCong$ on $R$. Therefore, when $b$ is ghostpotent, we always have
  $ \srad(b) \subseteq \srad(a)$
  for any $a \in R$.
\end{rem}

Clearly, any ghostpotent which is not ghost by itself is a ghost divisor (Definition \ref{def:nudomain}), implying  that $\tN(R) \sm \tG \subset \gDiv(R)$.

\begin{lem}\label{lem:R-G.rad}
  $\grad(R) = \srad(\tGz) =  \srad(\emptyset)$.
\end{lem}
\begin{proof} $(\subseteq):$
Suppose that $a \in \gprad(R)$ is non-ghost, then $a^k = \ghost $ for some $k \in \Net$. In any \gprimec, $a^k \pcng \ghost$ implies $a \pcng \ghost$, and thus
$a \in \iGcl(\srad(\tGz))$  by \eqref{eq:radE}.

\pSkip
  $(\supseteq):$ Immediate by \eqref{eq:EEG},  since $\tGz \subseteq \gprad(R)$.

  \pSkip The equality $\srad(\tGz) =  \srad(\emptyset)$ is given by \eqref{eq:EEG}.
\end{proof}

From the above exposition  we derive the following theorem.
\begin{thm}[Krull]\label{thm:krull}
  For any \nusmr\  $R$ we have
   \begin{equation}\label{eq:radRP}
\grad(R) = \bigcap_{\scriptsize \begin{array}{c}
\pCong\in \SpecR
\end{array}} \hskip -13pt \pCong  \ .   \end{equation}
\end{thm}
\begin{proof}
A direct consequence of Lemma \ref{lem:R-G.rad}, in the view of Definition \ref{def:s.radical}.
\end{proof}

\begin{cor} Let $a \in R$, then
  $a \in \iGcl(\grad(R))$ if and only if $a $ is a ghostpotent, i.e., $a \in  \gprad(R)$.
\end{cor}
\begin{proof} $(\Rightarrow)$:
By Theorem \ref{thm:krull}, $a$ belongs to the ghost projection $\iGcl(\pCong)$ of every \gprimec\ $\pCong$  and thus  is a ghostpotent.

  \pSkip
  $(\Leftarrow)$: Clear by definition, since $a \in \gprad(R)$.
\end{proof}

\begin{cor}\label{cor:a.vs.prime}
  For any $a \notin \gprad(R)$ there exists a \gprimec\ $\pCong$ such that $a \notin \iGcl(\pCong)$
\end{cor}
 The corollary is a  strengthening of Lemma \ref{lem:a.vs.qcong}, applied there to \qcong s, and it can be enhanced further for \tprs\ elements.
 \begin{lem}\label{lem:f.prs.in.prime} For each $a \in \tTPS$
 there exists a \gprimec\ $\pCong$ such that $a \in \iTcl(\pCong)$.
\end{lem}
\begin{proof}
The multiplicative  monoid $\MS= \gen{a}$ is tangible, since $a$ is \tprs.
Take the tangible localization~ $R_\MS$  of~ $R$ by $\MS$, in which $\frac{a}{\one}$ is a tangible unit. By Corollary~ \ref{cor:a.vs.prime}  there exists a \gprimec \ $\pCong'$ on~ $R_\MS$ for which $\frac{a}{\one} \notin \iGcl(\pCong')$, and furthermore $\frac{a}{\one} \in \iTcl(\pCong')$, since $\frac{a}{\one}$ is a unit and $\pCong'$ is \gprime.
 Then,   by Proposition ~\ref{prop:cong.2}.(ii), the restriction of $\pCong'$ to $R$ gives a \gprimec\ $\pCong$ with  $\MS \subseteq \iTcl(\pCong)$, and thus  $a\in \iTcl(\pCong)$.
\end{proof}

For the \gpradicalc, Remark \ref{rem:ghostification.smr} can be strengthened as follows.
\begin{rem}\label{rem:grad.prs}  $ $
\begin{enumerate} \eroman
  \item A \tprs\ element $ a \in \tTPS$ cannot be written as a sum which involves a ghostpotent term  ~$b$.
  Indeed, otherwise, if $a= b + c$ such that $b^n \in \tG$, then  $a^n = (b + c)^n = b^n + \sum_{i=1}^{n} \binom{n}{i}b^{n-i}c^i$,  which contradicts Axiom \emph{\PSRb}\ in Definition \ref{def:nusemiring}, since $a^n \in \tT,$ cf. \eqref{eq:tan.sum}.

  \item  If $R$ is a \prsfl\ (resp. \tcls) \nusmr, then $R/\grad(R) $ is also  \prsfl\ (resp. \tcls). Indeed, each member of   $\iTcl(\grad(R))$ is \tprs, since $\tT = \tTPS$ (resp. $\tT = \tTP$), where a \tprs\ element cannot be ghostpotent, implying by ~ (i) that $\iTcl(\grad(R)) = \tT$ is a monoid. Hence,
$R/\grad(R) $ is \prsfl\ (resp. \tcls).

  \item  If $R$ is a \tamesmr, then $R/\grad(R) $ is also  \tame, since  $\grad(R)$ is a \qcong\ which respects the \nusmr\ operations.

\end{enumerate}
\end{rem}

\subsubsection{Jacobson radical} \sSkip

Finally, we reach the last type of radical, defined in terms of \maximalc s (Definition ~ \ref{def:maximalSpec}).
\begin{defn}\label{def:jacbson.rad} The \textbf{Jacobson radical} of a congruence $\Cong$ on a \nusmr\ $R$ is defined as
  $$ \jac(\Cong) := \hskip -10pt \bigcap_{\scriptsize  \begin{array}{c}
\mCong\in \MSpec(R) \\ \mCong \supseteq \Cong
\end{array}} \hskip -13pt \mCong \; .$$

\end{defn}

\begin{lem}
   Let $\pi = \pi_\Cong : R \Onto R/\Cong$ be the
canonical surjective  homomorphism. Then,
\begin{enumerate}\eroman
  \item $\jac(\Cong) = \ipi (\jac(R/\Cong))$;
\item $\crad(\Cong) = \ipi(\grad(R/\Cong))$;

\end{enumerate}
\end{lem}
\begin{proof} (i): Observe that the map $\mCong \Mto \ipi (\mCong)$
defines a bijection between all \maximalc s  of $ R / \Cong$ and the
\maximalc s  on $R$ that contain $\Cong$, cf. Remark \ref{rem:induced.cong}.(i). As
inverse images with respect to $\pi$ commute with intersections, (i) follows
from Definition \ref{def:jacbson.rad} and Theorem \ref{thm:krull}.
\pSkip
(ii): Follows from  the
fact that a power $a^n$ of an element $a \in R$ belongs to $\iGcl(\Cong)$ if and only if $\pi(a)^n \in (R/\Cong)|_\ghs$.
\end{proof}

\subsection{Krull dimension}
\sSkip

The development of dimension theory is left for future work. To give the flavor of its basics,  we bring some  basic definitions.

\begin{defn}\label{def:noetherian.nusmr}
  A \nusmr\ $R$ is called \textbf{noetherian} (resp. \textbf{\qnoetherian}),
if any ascending chain
$$ \Cong_0  \subsetneq \Cong_1 \subsetneq \Cong_2 \subsetneq \cdots  $$
of \ccong s (resp. \qcong s)  on  $R$ stabilizes after finitely many steps, i.e., $\Cong_n = \Cong_{n+1} =  \cdots$ for some $n$.

  $R$ is called \textbf{artinian} (resp. \textbf{\qartinian}),
if any descending chain
$$ \Cong_0  \supsetneq \Cong_1 \supsetneq \Cong_2 \supsetneq \cdots  $$
of \ccong s (resp. \qcong s)   stabilizes after finitely many steps.

\end{defn}

 The use of \gprimec s, supported by the results if this section, allows for a natural definition of dimension of \nusmr s.

\begin{defn}\label{def:krull.dim}
The \textbf{Krull
 dimension} of a \nusmr\ $R$, denoted $\dim(R)$, is defined to be
the supremum of lengths of chains
$$\iTcl(\pCong_0) \varsupsetneq \iTcl(\pCong_{1}) \varsupsetneq \cdots  \varsupsetneq \iTcl(\pCong_{n}),$$
such that  $\pCong_0 \varsubsetneq \pCong_{1} \varsubsetneq \cdots  \varsubsetneq \pCong_{n}$  are \gprimec  s on $R$.
\end{defn}
For example,  any \ssmf \ (and more generally any \nusmf)  $F$ has dimension $0$, while  the  \nusmr\   $\tlF[\lm]$ of polynomial functions over $F$ has dimension $1$, cf. Example \ref{examp:prime}.
The case of general \nusmr s is more subtle.

\begin{defn}\label{def:heigth} The
\textbf{height} $\hgt(\pCong)$  (or \textbf{codimension}) of a \gprimec\  $\pCong$, is the supremum of the lengths of all chains of \gprimec s contained in $\pCong$, meaning that $\pCong_ 0  \subsetneq \pCong_1   \subsetneq \cdots \subsetneq  \pCong_ n   = \pCong$ such that  $\iTcl(\pCong_0) \varsupsetneq \iTcl(\pCong_{1}) \varsupsetneq \cdots  \varsupsetneq \iTcl(\pCong_{n})$.
\end{defn}
   By Theorem \ref{thm:prime.bij}, we see that the height of $\pCong$ is the Krull dimension of the localization $R_\pCong$ of $R$ by ~ $\pCong$.     A \gprimec \ has height zero if and only if it is a minimal \gprimec. The Krull dimension of a \nusmr\ is the supremum of the heights of all \gprimec s that it carries.
\pSkip

   The study of \qnoetherian\ \nusmr s, \qartinian  \ \nusmr s,  and Krull dimension is left for future work.

\section{$\nu$-modules}\label{sec:4}

Modules over \nusmr s are a specialization of modules over semirings (Definition \ref{def:modules}), playing the similar role to that of  modules over rings. As before, \nusmr s in this section are all assumed to be commutative.

\begin{defn}\label{def:numodule} A \textbf{left $R$-\numod}\  $M:= (M,\tH, \mu )$ over a (commutative) \nusmr\ $R := (R, \tT, \tG, \nu)$ is a smeimodule (Definition \ref{def:modules}) having the structure of an additive
 \numon\ (Definition \ref{def:numonoid}) whose ghost map $\mu: M \To \tH$  satisfies for all $a \in R$, $u,v \in M$ the additional axioms:
\begin{description} \dispace
  \item[\NDa:] $\mu(au) = a \mu(u);$

  \item[\NDb:] $\mu(u+v) = \mu(u) + \mu(v).$   \footnote{This  axiom is part of the \numon\ structure that  makes the ghost map a monoid homomorphism. }
\end{description}
$\tH$ is called the \textbf{ghost submodule} of $M$. \pSkip

An
$R$-\textbf{\numod\ congruence} is a congruence on a \numon\ which also respects multiplication by elements of~ $R$, i.e., if $u \equiv v$, then $au \equiv
a v $ for all $a \in R$, $u,v  \in M$. \pSkip

A \textbf{\hom}\ of $R$-\numod s is a \numon\  \hom\  (Definition \ref{def:numon.hom}) $$\vrp: (M,\tH, \mu ) \TO (M',\tH', \mu'),$$ in which
$\vrp(au) = a\vrp(u)$ for any $a \in R $, $u\in M$.  In particular,
$\srHom(\zero_M) = \zero_{M'}$.  The \textbf{g-kernel}  of $\vrp$  is defined as
$$\gker(\vrp) : = \{ u \in M \cnd \vrp(u) \in \tH'\} \subset M. $$

\end{defn}

 When $R$ is clear from the context, to simplify notations,  we write \numod\ for  $R$-\numod. As in the case of \nusmr s,  we  write $u^\mu$ for the ghost image $\mu(u)$ of an element $u \in M$ in the submodule $\tH \subseteq M$.
Since  $M$ has the structure of a \numon, the ghost submodule $\tH$ is partially ordered and it  induces a (partial) $\mu$-ordering on the whole $M$, i.e., $u \mug v $ iff $u^\mu > v^\mu$.
We say that $M$ is a \textbf{ghost \numod}, if~ $\tH =  M$.

By definition, for every $u \in M$,
\begin{equation}\label{eq:e.u}
u^\mu = u + u = (\one + \one ) u = e u
\end{equation}
i.e., $\mu(u) = eu$. Axiom \NDa\  implies that $\mu(e u) = e\mu( u)$, and thus $\mu(a u ) = a^\nu \mu (u) $ for every $a \in R$.
 One observes that the action of~ $R$ on $M$ respects the $\mu$-ordering  \eqref{eq:nu.order} of $M $, i.e.,
\begin{equation}\label{eq:mod.order}
  u \mug v \Dir a u \mug a v  \qquad \text{for all } u, v\in M, \ a \neq \zero \text{ in }  R.
\end{equation}
Indeed, for $u \mug v $ in $M$, by the  \numon\ properties,   $a(u+v) = au = (au+av)$.

For a \hom\ $\vrp:  M \To M'$ of \numod s, we have $\vrp(u^\mu) = \vrp(u)^{\mu'}$ by Lemma \ref{lem:nu.hom}, providing the ghost inclusion  $\vrp(\tH) \subseteq \tH'$.

\begin{defn}\label{def:annihilator} Let $M$ be an $R$-\numod, and let $S \subseteq M$ be a subset.
   The \textbf{ghost annihilator} $\Ann_R(S)$ of  $S $  is the set of all $a \in R$ such that  $a u$ is a ghost for every  $u \in S$, i.e.,
$$  \Ann _{R}(S)=\{a\in R\mid au \in \tH \text{ for all } u \in S \}. $$
The \textbf{Krull dimension} of $M $ is defined as
$$ \dim_R(M) := \dim(R\qq \hskip -2pt  \Ann _{R}(M)),$$
where the quotient $R\qq \hskip -2pt  \Ann _{R}(M)$ is as given in   Definition \ref{def:ghostification}.
\end{defn}
\noindent Clearly,  $\tG \subseteq \Ann _{R}(S) $ for any subset $S \subseteq M$, while  $\Ann _{R}(\tH) = R$.   When $M = R$ is a \nusmr, $\Ann_R(S)$ is a \nusmr\ ideal (Definition \ref{def:ideal.smr}).
\pSkip

The usual verification shows that the direct sum of
\numod s is a \numod. Thus, one can construct the free
\numod\ as a direct sum of copies of $R$.

\begin{defn}\label{def:free.module} An $R$-\numod\ $M$ is  \textbf{free}, if it is isomorphic
to ${R}^{(I)}$ for some index set $I$. $M$ is
 \textbf{projective}, if there is a split epic from $R^{(I)}$ to $M$ for some
$I$ (which can be taken to have order $m$, if~ $M$ is finitely
generated by $m$ elements).
\end{defn}

Ghostifying  congruences $\gCong_N$ of \nusmod s $N$  are naturally defined in terms of \numon s, cf. ~\eqref{eq:G.E}.

\begin{defn}\label{def:factor.mod}
  The \textbf{quotient} of a \numod\ $M$ by a \nusmod\ $N \subseteq M$  is defined as $M \qq N := M / \gCong_N$, where $\gCong_N$ is the ghostifying  congruence of $N$.

\end{defn} \noindent
\noindent In fact, this process of quotienting is applicable  for a general subset $S \subseteq  M$, and not only for \nusmod s.
\pSkip

Given two  $R$-\numod s $M$ and $N$, we denote by $\Hom_R(M,N)$ the set of all \numodh s $\phi : M \To N$.
A routine check shows that $\Hom_R (M,N)$, with $R$  a commutative \nusmr,   is an $R$-\numod\  where  $a \phi$ is defined via $(a \phi)(u) := a\phi(u)$, cf. Proposition \ref{prop:Hom.is.numon}.

 The category of \numod\ over a \nusmr\  $R$ is denoted by $\NMOD$, its morphisms are  \numodh s.
As usual,  we have the following functors.

\begin{defn}\label{def:functors} The \textbf{covariant functor}
$$\Hom_R(M,\udscr\,): \NMOD \TO \NMOD $$ is given by
sending $N $ to $\Hom_R (M,N)$  and sending $\varphi: N_1 \To N_2$
to $\htvrp: \Hom_R(M  ,N_1)\To \Hom_R(M ,N_2)$ by $\htvrp (\ff) =
\varphi \ff$ for $\ff: M \To N_1.$

 The  \textbf{contravariant functor}
$$\Hom(\, \udscr\, ,N): \NMOD \TO \NMOD$$ is given by
sending $M  $ to $\Hom_R (M,N)$ and sending the \numod\ homomorphism
$\varphi: M_1 \To M_2$ to $\htvrp: \Hom_R(M_2  ,N)\To \Hom_R(M_1 ,N )$
by $\htvrp (\ff) = \ff \varphi$ for $\ff: M _2 \To N.$
\end{defn}

\subsection{Exact sequences}
\sSkip

 A sequence of $R$-\numod s  is a chain of morphisms of $R$-\numod s
$$\cdots \Right{3}{\phi_{n-2}}  M_{n-1} \Right{3}{\phi_{n-1}} M_n \Right{3}{\phi_n} M_{n+1} \Right{3}{\phi_{n+1}} \cdots$$
with indices varying over a finite or an infinite part of $\Z$. The  sequence satisfies the \textbf{complex ghost  property} at $M_n$ if  $\phi_n \circ \phi_{n-1} = \tH_{n+1}$ or, in equivalent terms, $\im(\phi_{n-1}) \subseteq  \gker (\phi_n)$, cf. Definition~ \ref{eq:nu.mon.hom}.
The sequence is said to be \textbf{ghost exact} (written \textbf{g-exact}) at $M_n$, if $\im(\phi_{n-1}) = \gker(\phi_n)$.
When the sequence satisfies the complex ghost property at every  $M_n$, it is called  \textbf{g-complex}. Likewise, the sequence is called \textbf{g-exact}, if it is exact at all places.  \textbf{Short g-exact sequences} are sequences of type $$\tH' \TO M' \overset{\phi}{\TO} M \overset{\psi}{\ONTO} M'' \TO \tH'',$$
such that
\begin{enumerate} \ealph
  \item $\im(\phi) = \gker(\psi); $
  \item $\psi$ is  surjective.
\end{enumerate}
Often we also require that $\phi$ is injective. Then, for short g-exact sequences, $M'$  can  be viewed as a \nusmod\ of
$M$  via $\phi$, where $\psi$  induces a homomorphism  $M \qq M' \To M''$. Conversely, for any  \nusmod\
$N \subseteq  M$ the sequence
$$ \tH_N \TO N \INTO M \ONTO M \qq N \TO \tH_{M \qq N} \ (= \tH_M)$$
is a short g-exact sequence.
A further study of g-exact sequences is left for future work.

\subsection{Tensor products}\label{ssec:tensor}
\sSkip

The tensor product of modules over semirings has appeared in the
literature; for example in ~\cite{W}. Since the theory closely
parallels tensor products over algebras, we briefly review it for
the reader's convenience. Patchkoria has
built  an extensive theory of derived functors in ~\cite{Pa2}, but assumes that
the modules are (additively) cancellative, in order to be able to
use   factor modules. Here, since factorization by submodule is performed by means of  ghostification, as indicated  in Definition \ref{def:factor.mod},  we can avoid  cancellativity conditions on \numod s.  

Suppose $M$ and $N$ are $R$-\numod s, and  $\tM := (\tM, \tG, + )$ is an additive   \numon\ (Definition \ref{def:numonoid}). A map $$\psi: \MxN \TO \tM $$ is \textbf{bilinear} (also called \textbf{balanced}), if
\begin{equation}\label{eq:bilinear.map}
 \begin{array}{rcl}
 \psi  (u + u', v) & =&  \psi  (u, v) + \psi  (u', v), \\[1mm]
\psi  (u , v+v')  & =&  \psi  (u, v)+ \psi (u, v'), \\[1mm] \psi  (u a, v) &= & \psi  (u, av), \end{array}\end{equation}
 for all $a\in R$, $ u, u' \in M,$ $v, v' \in N$.

\begin{defn}\label{def:tensor.prod}  A \textbf{tensor product} of $R$-\numod s $M$ and $N$ (over $R$) is  an $R$-\numod\ $
T$, together with an bilinear map $\psi : M \times  N \To T$, such that the universal property holds:
For each bilinear map $\phi: M \times  N \To L$ to some \numod\ $L$, there is
a unique linear map $\vrp: T  \To L $ such that $\phi = \vrp  \circ \psi$, i.e., that renders the diagram
$$\xymatrix{
\MxN \ar@{->}[d]_{\phi } \ar@{->}[rr]^{\psi} && T \ar@{-->}[lld]^{\vrp} \\
  L &&  \\
}
$$
commutative.
\end{defn}

Given  $R$-\numod s $M := (M, \tH_M, \mu_M)$ and $ N := (N, \tH_N, \mu_N)$, a routine  verification shows that $\RMxN$ is a again an $R$-\numod\ with ghost submodule $\tH_{\MxN}:= R(\tH_M \times \tH_N) \cong \tG(\tH_M \times \tH_N) $.
Its   $\mu$-ordering  is induced jointly from the $\mu$-orderings of $M$ and $N$ by coordinate-wise addition.
By the \numod\ properties one can identify $\RMxN$ with $\MxN $. But, to stress the fact that  $\tH_{\MxN}$ is determined as $e \cdot (\MxN)$, we retain the notation $\RMxN$.

Let $\tCong$ be the congruence on $\RMxN$ whose underlying  equivalence $\tsim$ is given by the generating relations
\begin{equation}\label{eq:tensor.1}
 \begin{array}{lrll}
 (i) & (u+u',v) &\tsim &  (u, v) + (u', v),  \\
 (ii) & (u,v+v ')  &\tsim  & (u , v) + (u , v') ,  \\
 (iii) & a \cdot(u, v)  & \tsim &   (au ,v) \;  \ds {\tsim} \;  (u, av),
\end{array}\end{equation}
for  $ a \in R$, $u, u'  \in M $, $v,v' \in N$.
We define the \textbf{tensor product}  of $M$ and $N$ to be
\[
\MoN :=  \RMxN/\tCong,
\]
and set  $\zero_\otimes := \zero_M \otimes \zero_N$, which is the class  $[\zero_\RMxN]$. The equivalence classes $[(u , v)]$ of $\tCong$ are  denoted as customary   by $u \otimes v$.

Observe that $f \tsim f'$ and $g \tsim g'$ in $\tCong$ implies  $f+g \tsim f'+g'$; therefore $\tCong$ induces a binary operation  $[f]+[g] : =  [f+g]$ on $\MoN$. So, it suffices to check that $f+g \tsim f'+g$, which is built from the three relations \eqref{eq:tensor.1}. The operation  $+$ of $\MoN$ is commutative, associative, and  unital with respect to ~ $\zero_\otimes$, as this is the case in $\RMxN$ and $\tCong$ is a  congruence.

The ghost  submodule $\tH_{\MoN}$ of $M \otimes_R N$ is defined as
$$\tH_{\MoN} := \{ u \otimes v \cnd u \in \tH_M \text{ or } v \in \tH_N\}.$$
When no confusion arises, we write $u ^\mu$ and $v^\mu$  for $\mu_M(u)$ and $\mu_N(v)$, respectively.
Recalling from \eqref{eq:e.u} that  $u^\mu = e \cdot u$,  by the third generating relation in \eqref{eq:tensor.1} we have
$$ (u^\mu,v) = (eu,v) \tsim e \cdot (u,v) \tsim  (u,e v) = (u,v^\mu),$$
 moreover
$$ (u^\mu,v) = (eu,v) =  (e^2u,v)\tsim e \cdot (eu,v)   \tsim   (eu,e v)
 = (u^\mu,v^\mu),$$
providing the equivalence
\begin{equation}\label{eq:mu.tens} 
(u^\mu,v^\mu) \tsim
  e \cdot (u,v).
\end{equation}
Therefore $$ \tH_{\MoN} = \{ e \cdot (u,v) \cnd  u \in M, v \in N\}= \tH_M \otimes_R \tH_N.$$

To define the ghost map $\mu_\otimes$ of $\MoN$, first observe by \eqref{eq:mu.tens} that the ghost map of $R(\MxN)$ is given by $  \mu_{R(\MxN)} =  (\mu_M, \mu_N)$, and it  respects the defining relations \eqref{eq:tensor.1} of   $\tCong$. Indeed, writing $\tlmu$ for $\mu_{R(\MxN)}$,  for the  first generating relation we have:
$$\begin{array}{rllll}
\tlmu(u+u',v)   \tsim & e \cdot(u+u',v) \ds{ \tsim}
(u+u',e v) \\[1mm]    \tsim &
  (u,e v)+ (u',e v) \tsim
e  \cdot(u,v)+ e \cdot(u',v) \\[1mm]    \tsim &
\tlmu(u,v) + \tlmu(u',v).
\end{array}$$
The second relation is checked similarly, where for the third relation we have
$$\begin{array}{ll}
\tlmu(a \cdot(u,v))  \tsim ea\cdot (u,v) \tsim
 (ea u,v)   \tsim e \cdot(au,v) \tsim
\tlmu (au,v) .
\end{array}$$
We assign $\MoN$ with the ghost map $$\muo : \MoN \TO \tH_M \otimes_R \tH_N, \qquad  \text{ defined by } \ u \otimes v \longmapsto  u^{\mu} \otimes v^{\mu},$$
for which
\begin{equation}\label{eq:mu.tensor}
 (u \otimes v)^\muo  = [(u , v)^\tlmu] = [(u^\mu, v^\mu)] = [e \cdot (u, v)]= u^{\mu} \otimes v^{\mu}
\end{equation}
 for all $u,u' \in \tH_M$, $v,v' \in \tH_N$.

The partial ordering $<_\otimes$ of $\tH_\MoN$ is induced from  the partial ordering  of $\tH_{R(\MxN)}$, determined as
\begin{equation}\label{eq:tensor.order.1}
 u \otimes v <_\otimes u' \otimes v'  \dss \Iff u < u' \text{ and } v < v', \end{equation}
for  $u,u'\in \tH_M$, $v,v' \in \tH_N$.  Equivalently,  the  partial ordering $<_\otimes$ can be  extracted directly from the additive structure of ghost submonoid $\tH_\MoN$, 
 that is
\begin{equation*}\label{eq:tensor.order.2}
\begin{array}{lll}
  u \otimes v <_\otimes u' \otimes v' & \Iff & u \otimes v + u' \otimes v' =  u' \otimes v'.\\[1mm]
\end{array}
 \end{equation*}
This formulation shows that the ordering $<_\otimes$ does not depend   on the representatives of classes $u \otimes v$ in $M \otimes N$, so we only need to verify that $<_\otimes$ respects the third generating relation in~ \eqref{eq:tensor.1}. But this  is clear by choosing  $ a \cdot(u, v)$ as representatives.

We define a map
\[
(\, \cdot \,) \colon R \times (\MoN) \xto{\qquad (a,[f]) \Mto [a \cdot f]\qquad} \MoN,
\]
which by routine check  seen to be well defined, namely $f \tsim g \Rightarrow a \cdot f \tsim a  \cdot g$, for any $a \in R$.

\begin{prop}\label{prop:tensors}
  $(\MoN , \tH_M \otimes_R \tH_N, \mu_\otimes)$   is an $R$-\numod.
\end{prop}
\begin{proof}
First,  observe that $\MoN$ is a \numon\ (Definition \ref{def:numonoid}) and its ghost map $ \mu_\otimes$ coincides  with the  \numon\ operation.
Indeed, let $ L :=  R(\MxN) $,  write $\tlmu$ for $\mu_L$, and for elements $f,g \in L $ use ~ \eqref{eq:mu.tensor} to obtain    $$[f+g]^{\mu_\otimes} = [(f+g)^\tlmu]=[f^\tlmu+g^\tlmu]=[f^\tlmu]+[g^\tlmu]=[f]^{\mu_\otimes}+[g]^{\mu_\otimes}$$
  and $(\zero_\otimes)^{\mu_\otimes} =[\zero_L^\tlmu]=[\zero_L]=\zero_\otimes$.
Since  $[f]^{\mu_\otimes}=[f^\tlmu]$ and $\tlmu$ is an idempotent map on $L$, then $\mu_\otimes$ is idempotent as well.

Having the partial ordering  $>_\otimes$  on $\tH_M \otimes \tH_N$ defined in \eqref{eq:tensor.order.1}, we verify the axioms of a \numon\ (Definition \ref{def:numonoid}).

\begin{description}\dispace
  \item[\NMa:] Assume $\mu_\otimes([f]) >_\otimes \mu_\otimes([g]) $, that is $[f^\tlmu] > [g^\tlmu]$,  implying that $f^\tlmu > g^\tlmu$. Thus
  $ [f] + [g] = [f+g] = [f] . $
  \item[\NMb:] If $\mu_\otimes(f) = \mu_\otimes(g) $, that is $[f^\tlmu] = [g^\tlmu]$, then  $f^\tlmu = g^\tlmu$ and
  $ [f] + [g] = [f+g] = [f^\tlmu] = [f]^{\mu_\otimes}.$
  \item[\NMc:]  The condition $ [f] + [g] \notin \tH_M \otimes_R \tH_N$ and
  $ [f] + [g]^{\mu_\otimes} \in \tH_M \otimes_R \tH_N$
    are equivalent to $f+g \notin \tH_L$ and $f +g^\tlmu \in \tH_L$, which implies $f+g=f^\tlmu +g$.
   But then $[f]+[g]=[f+g]=f^\tlmu +g= [f^\tlmu] +[g]=[f]^{\mu_\otimes}+[g]$,
   since $[f]^{\mu_\otimes}=[f^\tlmu]$.
\end{description}
Finally, a routine check shows that $\MoN$ is an $R$-\numod\ via (\,$\cdot$\,), since
$L$ by  itself is an $R$-\numod.
\end{proof}

\begin{rem}
  Since $R(\MxN)$ is generated by $\{(u,v) \cnd u \in M, v \in N \}$, we see that
  $\MoN$ is generated as an \numod\  by $\{u \otimes v \cnd  u \in M , v \in N \}$.
  Therefore, the ghost submodule $\tH_{\MoN}$ of $\MoN$ is generated by $\{u^{\mu} \otimes v^{\mu} \cnd u \in M, v \in N\}$, since  $u^{\mu} \otimes v^{\mu} = [e\cdot (u,v)]= e [ (u,v)]$ by \eqref{eq:mu.tensor}. Thus,  if $[f]= \sum_i a_i \cdot (u_i \otimes v_i)$, then
\[
[f]^{\muo} = e \cdot[f] = \sum_i a_i \cdot e \cdot e \cdot [(u_i,v_i)] =
\sum_i a_i \cdot [(e u_i, e v_i)] =
\sum_i a_i \cdot [(u_i^{\mu}, v_i^{\mu})].
\]
Hence, the ghost map $\muo$ is induced by the ghost maps $\mu_M$ and $\mu_N$ of $M$ and $N$, respectively.
\end{rem}

\begin{cor} The tensor product $T = M \otimes_R N$ exists for any $R$-\numod s $M$ and $N$.
\end{cor}
\begin{proof}
  Follows from Proposition \ref{prop:tensors}.
\end{proof}
For two \nusmr s $R$ and $R'$,  an $(R,R')$-$\nu$-\textbf{bimodule} is a \numod\ $M$ such that:
\begin{enumerate}\ealph
  \item  $M$ is a left $R$-\numod\ and also a right $R'$-\numod;
  \item  $(au)b' = a(ub')$ for all $a \in R$, $b' \in R'$ and $u \in M$.
\end{enumerate}
An $(R,R)$-$\nu$-bimodule is shortly termed  $R$-$\nu$-bimodule.

\begin{thm}\label{tenprd} There is a left adjoint functor $\Ten_N(\udscr)$ to the contravariant functor
$\Hom(\udscr \,  ,N)$ in Definition~ \ref{def:functors}. \end{thm}
The theorem is a restatement of the \textbf{adjoint
isomorphism}, writing $\Ten_N(M)$ as $\MoN,$ that is
$$\Hom(W,\Hom(M,N)) = \Hom (W \otimes M, N).$$
 $\Ten_N(\udscr)$ is called the \textbf{tensor product functor} $\udscr \otimes N,$
 where all constructions of  $\Ten_N(\udscr)$ are naturally isomorphic, by the uniqueness of the left adjoint functor.

  \begin{proof}[Proof of Theorem~\ref{tenprd}] The proof follows from the above  construction of the tensor product $\MoN$, based on congruences instead of \nusmod s,  and a verification of the adjoint
 isomorphism.

Given an abelian semigroup $\tS := (\tS,+)$, any
bilinear map $\psi: \MxN \To \tS$ gives rise to a semigroup
homomorphism determined by $u \otimes v \Mto \psi(u,v).$ Having gone
this far, one follows the program spelled out in
\cite[\S3.7]{Jac2}  to show that
if $M$ is an $(R',R)$-$\nu$-bimodule and $N$ is an $(R,R'')$-$\nu$-bimodule,
then $\MoN$ is an $(R',R'')$-$\nu$-bimodule under the natural
operations
$$a'(u\otimes v) = a'u \otimes v, \qquad  (u\otimes v)a'' = u \otimes
va''.$$
The verification of the adjoint isomorphism now is exactly as in
the proof of  \cite[Propositon~3.2]{Jac2}.
 \end{proof}

\begin{rem}\label{rem:tensor}
Let $R, M$, and $N$ be \nusmr s, with  \qhom s $\ff : R \To M$
and $\hh : R \To N$, that make $M$ and $N$ into $R$-\numod s. For every \nusmr\  $A$ and \qhom s\
$\al:M \To A$, $\bt:N \To A$, rendering  a commutative diagram with $\ff$ and $\hh$, there is a unique
 \qhom\ $ \xi: \MoRN \To A$ such that  the whole diagram commutes:
\begin{equation*}\label{eq:tensor.R} {
 \xymatrix{ A      & & \\
& \MoRN \ar@{..>}[ul]_{\xi}   & N \ar@{->}[l]  \ar@/_1pc/@{->}[llu]_\bt \\
& \ar@/^1pc/@{->}[luu]^\al M \ar@{->}[u] &  \ar@{->}[l]^\ff R \ar@{->}[u]_\hh } } \end{equation*}
(the maps $M \To \MoRN$  and $N \To \MoRN$
 are the obvious maps $u \Mto  u \otimes \one$ and $v \Mto \one \otimes
v$). For \qhom s  $\al : M \To A$ and $\bt : N  \To A$, the map $\xi: \MoRN \To A$ is
determined by $ u \otimes v \Mto  \al (u) \cdot  \bt (v)$.
\end{rem}

Once we have the notion of a tensor product at our disposal, we can recover classical structures, especially localization of \numod s.
\begin{defn}\label{def:module.localization}
The \textbf{tangible localization} of an $R$-\numod\ $M$ by a tangible submonoid $\MS \subseteq R$ is defined as  $$M_\MS := R_\MS  \otimes M,$$ where $R_\MS$ is the tangible localization of $R$ by  $\MS$  (Definition~\ref{def:tangible.localization}).
\end{defn}
When $\MS$ is generated by a \tprs\ element  $f \in \tTPS$, we write $M_f$ for $M_\MS$.

\subsection{$F$-\nualg}\label{ssec:F.Alg} \sSkip

Let $F$ be a \nusmf, and let $A$ be an $F$-\numod\ equipped with an additional binary operation  $ (\, \cdot \, ): A \times A \To A$.  $A$ is an algebra over $F$, if the following properties hold for every elements $x, y,z \in A,$  $a,b \in F$:
\begin{enumerate}
   \ealph
  \item  Right distributivity: $(x + y) \cdot z = x \cdot z + y \cdot z$,
  \item  Left distributivity: $x \cdot (y + z) = x \cdot y + x \cdot z$,

  \item  Multiplications by scalars: $(ax) \cdot (by) = (ab) (x \cdot y)$.

\end{enumerate}
In other words,  the binary operation $(\, \cdot \, )$ is bilinear. An algebra over $F$ is called \textbf{$F$-\nualg}, while  ~$F$ is said to be the base \nusmf\ of $A$.

  A \textbf{homomorphism} of $F$-\nualg s  $A,B$ is a homomorphism $\phi: A \To B $ of \numod s (i.e., an $F$-linear map) such that $\phi(xy) = \phi(x) \phi(y)$ for all $x,y \in A$. The set of all $F$-\nualg\ homomorphisms from  $A$ to $B$ is denoted by $\Hom_{F}(A,B)$.

\begin{defn}\label{def:f.g.alg}
  An $F$-\nualg\ is said to be \textbf{finitely generated}, if there
exist $f_1, \dots , f_n \in A$ such that any $f \in A$ can be presented in the form
$$f= \sum^n_{j=1}
a_jf_j, \qquad \text{with } a_1, \dots a_n \in F.$$
The elements $f_1, \dots , f_n$ are called \textbf{generators} of $A$.
\end{defn} Tensor products and
\nualg s lay  a foundation  for developing new types of algebra, parallel to known algebras,  e.g., exterior \nualg, Lie \nualg, or Clifford \nualg.
\bigskip
\part*{Part II: Supertropical Algebraic Geometry}
\bigskip
\section{Varieties}\label{sec:5}

Recall that all our underlying \nusmr s (Definition \ref{def:nusemiring}) are assumed to be commutative. Henceforth, since \nusmr s are referred to also as \nualg s whose elements are functions, unless otherwise is specified,
$\sA:= (\sA, \tT, \tG, \nu)$ denotes a commutative \nusmr.
For a clearer exposition, given a \qcong\ ~$\Cong$ on ~$\sA$,  we explicitly write $(\sA/\Cong)|_\tng^\PrS$ (resp. $(\sA/\Cong)|_\tng^\Pr$) for the \tprsset\ (resp. a  \tprsmon) of the quotient \nusmr\ $\sA/\Cong$, and $(\sA/\Cong)|_\ual$ for the set of its \tualt\ elements.  Similarly, we write  $(\sA/\Cong)|_\tng$ for the \tsset \ of $\sA/\Cong$, and $(\sA/\Cong)|_\ghs$ for its ghost ideal. To unify notations, we write $\sAual$, $\sAptng$, $\sAptngs$,  $\sAtng$, and $\sAghs$ for $\SST$, $\tTP$, $\tTPS$,  $\tT$, and $\tG$, respectively.

 \begin{notation}\label{notat:cong.x}
A \gprimec \ is denoted by $\pCong$, where its underlying  equivalence is  denoted by~ $\pcng$.  We let  $X = \SpecA$ be the \gprime\ spectrum of $\sA$ whose formal elements are \gprimec s (Definition~\ref{def:prmCng}).
  Later, $X$ is realized as a topological space. To designate this view, we denote a point of~ $X$ by $x$, where $\xpCong$ stands for \gprimec\ assigned with $x$. We  write $\sA_x$ for the localization~ $\sA_{\xpCong}$ of ~$\sA$ by  $\xpCong$ (Definition \ref{def:cong.localization}).

Elements of $\sA$ are denoted by the letters $f,g, h$, as from now on they are realized also as functions (as explained below).
    We write
  $\gen{f_1, \dots, f_\ell} $ for the subset generated by elements $f_1, \dots, f_\ell \in \sA$, i.e., all  finite sums of the form
  $\sum_{i}  g_i f_i$ with $g_i \in \sA$.
 \end{notation}

 Recall from Remark \ref{rem:classfor} that
$\iGcl(\udscr)$ and $\iTcl(\udscr)$ provide respectively the class-forgetful maps $\Gcl{\udscr}:\QCng(A) \To A$ and  $\Tcl{\udscr}: \QCng(A)  \To A$  that encode clusters' decomposition.  Recall also that a \gprimec\ $\pCong$ is  an \lcong, and thus its tangible projection $\iTcl(\pCong)$ is a multiplicative monoid, written $\iTcl(\pCong) = \iPcl(\pCong)$.

\begin{comm}\label{comm:E}
A subset $E \subseteq \sA$ can be realized as the trivial partial congruence $\diag(E)$, which set theoretically is contained in $\AxA$  (Definition \ref{def:partial.cong}). We are mainly concerned  with the case that $\diag(E) \subseteq \Gcl{\pCong}$, written equivalently as $E \subseteq \iGcl(\pCong)$, where $\pCong$ is a \gprimec.  When possible, to simplify notations, we use the latter  form.
Equivalently, this setup  is formulated in terms of  \gcong s
as $\gCong_E \subseteq \pCong$, cf. \eqref{eq:G.E}, which reads as $f \pcng f^\nu$ in $\pCong$ for all $f \in E$.
We alternate between these equivalent descriptions, for a clearer exposition.
\end{comm}

Our forthcoming  exposition includes some equivalent definitions,   relying on   different structural views, that later help for a better understanding of the interplay among the involved objects.

\subsection{Varieties over  \nusmr s}\label{ssec:ver}
\sSkip

Let $\sA= (\sA,\tT,\tG, \nu)$ be a \nusmr, and let $\SpecA$ be its \gprime\ spectrum (Definition \ref{def:prmCng}).
The elements of $\sA$ can be interpreted as functions on
$\SpecA$ by defining  $f(\pCong)$ to be the residue class $[f]$ of
~$f$  in $\sA / \pCong$, that is
\begin{equation}\label{eq:f.function}
f: \pCong \longmapsto [f] \in \sA / \pCong.
\end{equation} Thereby, every element  $f \in \sA$ determines a map
$$ \SpecA \ds \TO \coprod_{\pCong \in \SpecA } \sA/\pCong \quad  \subseteq \coprod_{\pCong \in \SpecA } Q(\sA/\pCong), $$
where $f(\pCong) \in (\sA/\pCong)|_\ghs$ if $f \in \iGcl(\pCong)$, and likewise
$f(\pCong) \in (\sA/\pCong)|_\tng$  when $f \in \iTcl(\pCong)$. Namely, on a \gprimec\ $\pCong$ the function $f$ possesses ghost or tangible, respectively; but it could also result in a class which is neither ghost nor  tangible.

\begin{rem} A function $f \notin  \sAtng$ cannot not be congruent to any  $h \in \iTcl(\pCong)$, in each $\pCong \in \SpecA$, but it is not necessarily evaluated as ghost.
When  $f \notin \sAghs$,  it may still be considered as tangibly evaluated  on some $\pCong$.
\end{rem}
  Since $\pCong \in \SpecA$ is a \gprimec, a ghost product $(fg)(\pCong) \in (\sA/\pCong)|_\ghs$ of  functions $f, g \in \sA $   implies  that $f(\pCong) \in  (\sA/\pCong)|_\ghs$ or $g(\pCong) \in (\sA/\pCong)|_\ghs$. Equivalently, this reads as $fg \pcng \ghost $  implies
  $f \pcng \ghost$ or $g \pcng \ghost$, where $\pcng$ is the underlying  equivalence of $\pCong$. Hence,  $f \pcng f^\nu$ or $g \pcng f^\nu$ by Lemma \ref{lem:ghs.eqv}.
We identify $f\in \sA$ with the pair $(f,f) \in \AxA$.
\pSkip

We define the \textbf{ghost locus} of a nonempty subset $E \subseteq \sA$ to be
$$
\begin{array}{lrl}
\VV(E)  := & \{ \pCong \in \SpecA \ds |  f(\pCong) \in (\sA/\pCong)_\ghs \text{ for all } f \in E \} .
\end{array}$$
For $f \in \sA$ we therefore equivalently define  (cf. Comment \ref{comm:E}):
\begin{equation}\label{eq:Var.1}
  \begin{array}{lrl}
\VV(f) & := & \{ \pCong \in \SpecA \ds | f \in  \iGcl(\pCong)  \} \\[1mm]
       & = & \{ \pCong \in \SpecA \ds | (f,f) \in  \Gcl{\pCong}  \}, \\[1mm]
       & = & \{ \pCong \in \SpecA \ds | \text{s.t. } f \pcng f^\nu \text{ in } \pCong \}, \\[2mm]
\DDf& := &  \{ \pCong \in \SpecA \ds | f \notin  \iGcl(\pCong) \} \\[1mm]
 & = &  \{ \pCong \in \SpecA \ds | (f,f) \notin  \Gcl{\pCong} \} = \SpecA \sm \VV(f)
\end{array}
\end{equation}
A set of the form $\DDf$  is called a \textbf{principal subset}.
Using the bijection $\iota: E \Isoto \diag(E)$, a subset~$\VV(E)$ can be written as
\begin{equation}\label{eq:var.2}
\begin{array}{lrl}
\VV(E)  & :=  &  \{ \pCong \in \SpecA \ds| E \subseteq \iGcl(\pCong) \} \\[1mm]
& =  &  \{ \pCong \in \SpecA \ds |  \diag(E) \subseteq \Gcl{\pCong} \}  .\\
\end{array}
\end{equation}
We use these equivalent forms of  $\VV(E)$ and $\DDf$ to ease the exposition.

\begin{comm}
  Note that  $\pCong \in \DD(f)$ does not imply that $f \in \iTcl(\pCong)$, which means that as a function $f$ need not take tangible values on $\pCong$. However, it takes non-ghost values on the entire $\DD(f)$.
\end{comm}

 A \textbf{\nuvar} in $\SpecA$ is a subset of type $\VV(E)$ for some $E \subseteq A$; it is called  \textbf{$\nu$-hypersurface} when $E = \{f \}$ for some ~$f \in A$, in other words $\gCong_E$ is a principal \qcong. We write  $\VV_X(f)$ and $\DD_X(f)$ when we want to keep track of the ambient space $X =  \SpecA$.

Recall   that an element $f \in \sA$ is ghostpotent, belonging to $\tN(\sA)$, if $f^n \in \sAghs$ for some $n \in \Net$ (Definition \ref{def:ghostpotent}), where ~$f$ by itself could be ghost.

\begin{rem}\label{rem:var.1}
  Since $A|_\ghs:= \tGz  \subseteq \iGcl(\pCong)$ for any \gprimec\ $\pCong$, from \eqref{eq:var.2} we obtain  that $\VV(E) = \VV(E \cup \tGz)$ for any $E \subseteq
  \sA$. Therefore, $\VV(E) = \VV(\tGz)$ when $E \subseteq \tGz$, and
  $\VV(g) = \VV(\tGz) $ for any  $g \in A|_\ghs$.
    Clearly, $\VV(\sA) = \emptyset$, since $\AX  \nsubseteq \iGcl(\pCong)$ for every $\pCong \in \SpecA$,  by  Remark \ref{rem:cong.ideal}.
%
On the other hand, for principal subsets, we immediately see that
$\DDf = \emptyset$ for every   $f \in \sAghs$, and more generally for all $f \in \tN(\sA)$, while
 $\DDf = \SpecA$ for any unit $f \in \sA^\times$ and any $f \in \sAual$, in particular $\DD(\one) = \SpecA$.
\end{rem}

Recall that  $ f \in \iGcl(\pCong) $  means that $f \cng  g$ for some $g \in \sAghs$, yet $f$ need not be ghost in $\sA$.

\begin{proposition}\label{prop:var1}  Let $E, E' $ be subsets of $\sA$, and let $(E_{i})_{i \in I} $ be  a family of subsets of $\sA$. Then,
\begin{enumerate} \eroman
  \item $\VV(h) = \SpecA$ for any  $h \in \sAghs $, and for any  ghostpotent $h \in \tN(\sA)$;
  \item $\VV(f) = \emptyset$ for every  unit $f \in \sAX$; 
  \item $E \subset E' \Rightarrow \VV(E) \supset \VV(E')$;
  \item $\VV(\bigcup_{i\in I} E_i)  = \bigcap_{i\in I} \VV(E_i) = \VV(\sum_{i\in I} E_i )$;
  \item $\VV(EE') = \VV(E) \cup \VV(E')$, where $EE' = \{ ff' \cnd f \in E, f' \in E'\}$.
  \end{enumerate}

\end{proposition}

\begin{proof} (i)--(iii) are obvious, since each $\pCong \in \SpecA$ is an  \lcong , where  $h \in \iGcl(\pCong)$ for every  $h \in \tN(\sA)$, cf. Remarks \ref{rem:ghostpotent} and \ref{rem:var.1}. In addition,
 $\sAX \subseteq \ciGcl(\pCong) $ for each $\pCong$, since otherwise $\pCong$ would be a \gstcong, and thus
$\VV(f) = \emptyset$.
\pSkip
(iv):
The set $\VV(\bigcup_{i\in I} E_i)$
 consists of all $\pCong \in \pcspec(A)$  with  $\bigcup_{i\in I} E_i \subseteq \iGcl(\pCong)$, hence  $E_i \subseteq \iGcl(\pCong)$ for every  $i\in I$, and therefore coincides with $ \bigcap_{i\in I} \VV(E_i)$. Given $\pCong \in \pcspec(A)$,  $\bigcup_{i\in I} E_i$ is contained in $\iGcl(\pCong)$  iff the set  $\sum_{i\in I} E_i$
generated by $\bigcup_{i\in I} E_i$ (which is a semiring ideal) is contained in $\iGcl(\pCong)$. Therefore, we see that, in addition, $\VV (\bigcup_{i\in I} E_i)$
coincides with $\VV (\sum_{i\in I} E_i)$.
\pSkip
(v): Take  $\pCong \in \SpecA$  such that  $\pCong \notin \VV(E)$
and $\pCong \notin \VV(E')$. Then $E \nsubseteq \iGcl(\pCong) $ and $E' \nsubseteq \iGcl(\pCong)$, and there are elements $f \in E$ and $f' \in E'$ such that $f,f' \notin \iGcl(\pCong)$. But then,  $ff' \notin \iGcl(\pCong)$ since $\pCong$ is a \gprimec, and thus $\pCong \notin \VV (EE')$ so that $\VV (EE') \subseteq \VV (E) \cup \VV (E')$. Conversely, $\pCong \in \VV (E)$ implies $E \subseteq \iGcl(\pCong)$ and hence $EE' \subseteq \iGcl(\pCong)$ by Remark \ref{rem:mult1}, thus $\pCong \in \VV (EE')$. This shows that $\VV (E) \subseteq \VV (EE')$ and,
likewise,  $\VV (E') \subseteq \VV (EE')$.
\end{proof}

For the set-radical closure $\rcl(E)$ of a subset $E \subseteq \sA$, 
 determined by \sradical \ (Definition \ref{def:s.radical}), we  receive a one-to-one correspondence.

\begin{proposition}\label{prop:var2}
   $\VV(E) =  \VV(\rcl(E))$
for any $ E \subset \sA$.
\end{proposition}
\begin{proof} Clearly,  $E  \subseteq \rcl(E)$, and thus  $\VV(E)  \supseteq \VV(\rcl(E))$ by Proposition \ref{prop:var1}.(iii).
Conversely, by \eqref{eq:srad.1}, $E \subseteq \iGcl(\pCong)$ is equivalent to $\rcl(E) \subseteq  \iGcl(\pCong)$, since $\pCong$ is \gprime. This holds for any $\pCong \in \VV(E)$, and thus $\VV(E) \subseteq \VV(\rcl(E))$.
\end{proof}

With the above setting,  the \gprime\ spectrum $\SpecA$ is endowed with a Zariski type
 topology, defined in the obvious way.

\begin{cor}\label{cor:1.2}
Let $\sA$ be a \nusmr, and let $X = \SpecA$ be  its spectrum. There exists a Zariski
 topology on $X$ whose closed sets are subsets of type $\VV (E) \subseteq X$, where  $E \subseteq \sA$.
Moreover:
\begin{enumerate} \eroman
  \item  The sets of type $\DDf$ with $f \in\sA$ are open  and satisfy $\DDf \cap \DDg = \DD(fg)$ for $f, g \in  \sA$.
\item  Every open subset of $X$ is a union of sets of type $D(f)$, which form a basis of the  topology on $X$.
\end{enumerate}

  \end{cor}

\begin{proof}
The characteristic  properties of closed sets of a topology are given by parts (i), (iii), and (iv) of Proposition \ref{prop:var1}, where (iv) is generalized
by induction to finite unions of sets of type $\VV (E)$. The sets of type $\DDf$ are open, as they are  complements of sets of type $\VV(f)$, and thus
satisfy the intersection property by Proposition \ref{prop:var1}.(iv). Finally, an
arbitrary open subset $U \subseteq  X = \SpecA$ is the complement of a closed set of
type $\VV (E)$ for some subset $E \subseteq \sA$. Therefore,
$\VV (E) =  \bigcap_{f\in E} \VV(f)$  and hence $
U = \bigcup_{f\in E} \DDf,$
which says that every open subset in $X$ is a union of sets of type $\DDf$.
\end{proof}

{Note that, in contrast to the familiar spectra of rings, arbitrary  open sets in the Zariski topology on $\SpecA$  need not be dense.}

\begin{prop}\label{prop:dense.sset}
$f \in \sA\sm (\sAghs \cup \gDiv(\sA))$ iff
  $\DDf$ is dense in $X$ 
  (Definition \ref{def:nudomain}).
\end{prop}
\begin{proof} $\DDf \neq \emptyset$ is dense iff   $\VVg =X$ for every closed set $\VVg$ that contains $\DDf$. The inclusion  $\DDf \subset \VVg$ means that $g(x) \in (\sA/\xpCong)|_\ghs$ for every $x\in X$ for which $f(x) \notin (\sA/\xpCong)|_\ghs$.  This holds iff  $g(x) f(x) \in (\sA/\xpCong)|_\ghs$ for every $x \in X$ iff $gf$ is a ghostpotent (Corollary \ref{cor:a.vs.prime}) iff $g$ is ghostpotent and $f \notin \gDiv(\sA) $ is not ghost. The former condition holds, since $\VVg =X$ for every $g$ such that $\VVg$ contains~ $\DDf$.
\end{proof}

\begin{rem}\label{rem:dense.sset.in.tame}
 If $\sA$ is a \tamesmr, then $\DDf$ is dense iff $f \in \sAtng \sm \gDiv(\sA)$, since otherwise, either,  $f \in \sAghs$ and thus $\DDf = \emptyset$, or $f$  is a ghost divisor by Lemma \ref{lem:gdiv.in.tame}.(i) and  thus $\DDf$ is not dense by Proposition \ref{prop:dense.sset}.
\end{rem}

We strengthen Remark \ref{rem:var.1} and Proposition \ref{prop:var1}.

\begin{lem}\label{lem:D.contain}
 Let $\sA$ be a \nusmr. 
  \begin{enumerate}\eroman

 \item  $\VV(f) = \SpecA$ iff $f$ is a ghostpotent.

  \item If $\VV(f) \subseteq \VV(g) $ where $f$ is a ghostpotent, then $g$  is a ghostpotent.

\item $\DDf = \emptyset$ iff  $f$ is a ghostpotent.

\item $\DDf = \DDg$ iff $\srad(f)= \srad(g)$, cf.  Definition \ref{def:s.radical}.

\item   If $f$ is a  unit (or \tualt) in  $\sA$,  then $\DDf = \SpecA$.

\item $\DD(f^n) = \DDf$ for any $f \in \sA$.

  \item If $\DDf \subseteq  \DDg $ where $f \in \sAptngs$ and $\sA$ is a \tamenusmr, then $g \in \sAptngs.$

\item $E' \subseteq E$ \ iff \ $\DD(E') \subseteq  \DD(E)$ \ iff \ $\VV(E') \supseteq  \VV(E)$ \ iff \ $\srad(E') \subseteq  \srad(E)$.
\end{enumerate}
\end{lem}
\begin{proof}
(i): $\VV(f) = \SpecA$ $\Iff$ $f \in \iGcl(\pCong)$ for every \gprimec\ $\pCong$ on $\sA$ $\Iff$ $f$ is ghostpotent, cf. Remark \ref{rem:ghostpotent}.
\pSkip
(ii): Follows immediately  from  part (i).

\pSkip (iii): $\DDf = \emptyset \Iff \VV(f) = \SpecA$, and apply part (i).

\pSkip (iv): Using set theoretic considerations we see that
\begin{equation}\label{eq:str}
   \DDf \subseteq \DDg \dss\Iff \VV(f) \supseteq \VV(g) \dss\Iff \srad(f) \subseteq \srad(g) \tag{$*$}.  \end{equation}
 By symmetry we conclude that $\srad(f) = \srad(g)$.

\pSkip (v): $(\Rightarrow)$: $\DDf = \SpecA = \DD(g)$ for some unit $g \in \sAX$  $\Rightarrow$  $\srad(g)  \subseteq \srad(f)$ by \eqref{eq:str}  $\Rightarrow$  $g^n = f h$ by Lemma ~ \ref{lem:rad.prop.1}. But $g^n $ is a unit, and thus $fh$ is also a unit,  implying that $f$ is a unit by Remark \ref{rem:prud.unit}.
\\
$(\Leftarrow)$: $f \in \sAX$ is a unit of $\sA$ $ \Rightarrow  \VV(f) = \emptyset \Rightarrow \DDf = \SpecA$, by Lemma \ref{prop:var1}.(ii).

\pSkip (vi): $\VV(f^n) = \VV(f) \cap \cdots \cap \VV(f) = \VV(f)$ by Lemma \ref{prop:var1},   taking complements we get $\DD(f^n) = \DD(f)$.

\pSkip (vii): Follows from Corollary
 \ref{cor:rad.incl.2},  applied to \eqref{eq:str}, as $f$ is \tprs\ and $\sA$ is \tame.  

\pSkip (viii): It is an obvious generalization of (iv), combined with Proposition \ref{prop:var1}.(iii).
\end{proof}

By   Comment \ref{comm:E},  a subset $E$ of a \nusmr\ $\sA$  is identified with $\diag(E)$,  realized as a partial ghost concurrence.
  In addition, $E$ is canonically  associated to the \ccong\ ~$\gCong_E$ by the means of ghostification~ \eqref{eq:numon.var},  which gives the inclusion  $\diag(E) \subseteq \Gcl{\gCong_E}$.
In this  view,
  $\VV(\udscr)$  defines a map $\sA \To \SpecA$ via  \eqref{eq:var.2}, which extends naturally  to the map
  \begin{equation*}\label{eq:Var.Cong}
  \VV(\udscr):  \Cng(\sA)   \TO  \SpecA, \qquad \Cong \mTo \VV(\Cong),
 \end{equation*}
    where $\VV(\Cong)$  is defined as
   \begin{equation}\label{eqvar.3}
     \VV(\Cong) : = \{ \pCong \in \SpecA \ds | \pCong \supseteq \Cong  \}.
   \end{equation}
    Using the same notation, we apply $\VV(\udscr)$  to  both subsets $E$ of $\sA$  and arbitrary  \ccong s on $\sA$, no confusion arises.
    More generally, $\VV(\udscr)$ can be applied to subsets  $S \subset \sA \times \sA$, e.g., to
    $\Cong_1 \cap \Cong_2$ and $\Cong_1 + \Cong_2$. 
   We write $\VV_A$, when want to stress the fact that~$\VV$ is taken over $\sA$.

\begin{rem}\label{rem:E.ghostification}  With this notation, we have
  $\VV(\gCong_E) = \VV(E)$ for any $E \subseteq \sA$.  Indeed,
  $$ \begin{array}{lll}
\VV(\gCong_E) & =  \{ \pCong \in \SpecA \ds | \pCong \supseteq \gCong_E  \}
\\[1mm] &= \{ \pCong \in \SpecA \ds | \pCong \supseteq \bigcap \Cong \text{ such that } E \subseteq \iGcl(\Cong) \}      \\[1mm] &= \{ \pCong \in \SpecA \ds |  E \subseteq \iGcl(\pCong) \}  = \VV(E).     \end{array}
  $$
Note that $\gCong_E$ need not be a \qcong,  e.g., if $E \cap \sAX \neq \emptyset$;
in such case $\VV(\gCong_E) = \emptyset$  by definition.
\end{rem}

We define the converse map  of \eqref{eq:Var.Cong} to be the map
      $$ \II(\udscr) : \SpecA \TO \QCng(\sA) \quad [\subset \Cng(\sA)]$$
that sends a subset $Y \subseteq   \SpecA$ to the \qcong
  \begin{equation}\label{eq:II} \II(Y) := \bigcap_{\pCong \in Y} \pCong .\end{equation}
 $\II(Y)$ is indeed a \qcong, since it is an intersection of \lcong s, cf. Remark \ref{rem:qcong.intersection}.
In particular,
$\II(\{\pCong\}) = \pCong$ for any $\pCong \in \SpecA$, while $\II(\SpecA) = \grad(\sA)$ by Theorem \ref{thm:krull}.

The map
$\II(\udscr)$ determines the restricted map from $\SpecA$ to $A$, given in terms of projections:
\begin{align*}
  \II_\ghs(\udscr) :\SpecA \TO A, & \qquad  Y \Mto \bigcap_{\pCong \in Y} \iGcl(\pCong) =   \iGcl(\II(Y)).   
\end{align*}
For $Y = \VV(\Cong)$, the inclusion $(f,f) \in \II(Y)$ implies $f \in \iGcl(\pCong)$ for each $\pCong \in Y$.  Therefore, $\II_\ghs(Y)$ is the subset of all functions in $\sA$ that take ghost values over the entire $Y$.
\begin{rem}\label{rem:6.6.2} From \eqref{eq:II} it follows immediately that
\begin{enumerate} \eroman
  \item If $Y \subset Y' \subset \SpecA$, then $\II(Y) \supset \II(Y'), $
  \item $\II(\bigcup_{j\in J} Y_j) = \bigcap_{j \in J} \II(Y_j)$ for any family of subsets $(Y_j)_{j \in J}$ of $\SpecA$.
\end{enumerate}

\end{rem}

Applying $\VV(\udscr)$ to \cqcong s, we get the following.
\begin{prop}
\label{prop:6.1.5} Let $\sA$ be a \nusmr, and let $X = \SpecA$.
\begin{enumerate} \eroman

\item  $ \II(\VV (E)) = \srad(E)$ for any subset $E \subseteq \sA$, cf. Definition \ref{def:s.radical}.
In particular, $\II(\VV(E)) =\gCong_E$ when $\gCong_E = \crad(\gCong_E)$, cf. Definition \ref{def:g.rad.Cng}.

\item 
 $\VV (\II(Y ))$ coincides with the closure
$\brY$ of $Y \subset X$ with respect to the Zariski topology on $X$, and
$Y = \VV (\II(Y ))$ for closed subsets $Y \subset X$.
\end{enumerate}
\end{prop}

\begin{proof} (i): Using  Proposition \ref{prop:var2} and Lemma \ref{lem:rad.1}, write
\begin{align*}
 \II(\VV(E)) & \ds = \bigcap_{\pCong\in \VV(E)} \pCong \quad  = \bigcap_{\scriptsize \begin{array}{c}
\pCong\in \SpecA \\   \Gcl{\pCong} \supseteq \diag(E)
\end{array}} \hskip -3mm \pCong \quad =   \bigcap_{\scriptsize \begin{array}{c}
\pCong\in \SpecA \\   \iGcl(\pCong) \supseteq E
\end{array}} \hskip -2mm \pCong \quad  = \srad(E) .
\end{align*}
(This assertion is also obtained from Lemma \ref{rem:rad.cong} by considering $\diag (E)$ as a partial ghost congruence.)
\pSkip
(ii): First
$Y \subseteq  \VV (\II(Y ))$,
since $$
\begin{array}{ll}
\VV(\II(Y)) &  = \{ \pCong \in \SpecA \ds | \pCong \supseteq \II(Y) \} \\[1mm] & =
\bigg\{ \pCong \in \SpecA \ds | \pCong \supseteq \bigcap\limits_{\pCong{ \in Y}} \pCong \bigg\},\end{array}
$$
and clearly $\pCong \in \VV(\II(Y)) $ for  every $\pCong \in Y$.
Moreover,  $\brY \subseteq  \VV (\II(Y ))$, since the closure $\brY$ of $Y$ in $X$  is the
smallest closed subset in $X$ that contains $Y$, i.e., the intersection of all closed subsets
$\VV (E) \subseteq  \pcspec(A)$ such that $Y \subseteq \VV (E)$. To see that  $\brY = \VV (\II(Y))$, it remains
to check that $Y \subseteq \VV (E)$  implies $\VV (\II(Y )) \subseteq \VV (E)$. Indeed, from the inclusion $Y \subseteq \VV (E)$ we
conclude that
 $\diag(E) \subseteq \Gcl{\pCong}$ for all $\pCong \in Y$, hence  $\diag(E) \subseteq \Gcl{\II(Y )}$,
 which is equivalent to $E \subseteq \iGcl(\II(Y )).$ Then,
 $\VV (\II(Y )) \subseteq \VV (E)$ by Proposition ~\ref{prop:var1}.(iii).
\end{proof}

  Proposition  \ref{prop:6.1.5}, together with  Lemma \ref{lem:D.contain}, shows that $\VV (f) = \VV (g)$ is equivalent to $\srad(f) = \srad(g)$,  for $f,g \in \sA$,  and therefore also to $\DDf = \DDg$.

\begin{lem}\label{rem:rad.cong} Let $\Cong \in \Cng(\sA)$ be an arbitrary  \ccong\ on $\sA$, then:
 \begin{enumerate}\eroman
   \item $\VV(\crad(\Cong)) = \VV(\Cong)$,
   \item $\II(\VV(\Cong)) = \crad(\Cong)$.
 \end{enumerate}
 \end{lem}
 \begin{proof} Write explicitly to obtain  the following.
 $$\begin{array}{lll}
 (i): & \VV(\crad(\Cong)) & =  \{ \pCong \in \SpecA \cnd \pCong \supseteq \crad(\Cong)  \} \\ &  & =
  \bigg\{ \pCong \in \SpecA \cnd \pCong \supseteq \hskip -2mm  \bigcap\limits_{\scriptsize \begin{array}{c}
\pCong\in \SpecA  \\
\pCong \supseteq \Cong
\end{array}}\hskip -2mm \pCong \bigg\}
   = \{ \pCong \in \SpecA \cnd \pCong \supseteq \Cong  \}  = \VV(\Cong) ,
\\[2mm]
 (ii): & \II(\VV(\Cong))  & = \hskip -2mm \bigcap\limits_{\pCong\in \VV(\Cong)} \pCong \ds = \bigcap\limits_{\scriptsize \begin{array}{c}
\pCong\in \SpecA \\   \pCong \supseteq \Cong
\end{array}} \hskip -2mm \pCong  = \crad(\Cong).
\end{array}$$ \vskip -6mm
 \end{proof}

Accordingly,  we conclude  that, if $\Cong$ is a \qcong, then also $\II(\VV(\Cong))$ is  a \qcong, unless $\VV(\Cong)$ is empty.
\pSkip

\begin{remark} The
usual straightforward inverse Zariski
correspondence holds for $\Cong_1, \Cong_2 \in \CngA$:

\begin{enumerate} \eroman
\item
$ \VV(\overline{\Cong_{1} \cup \Cong_{2}})= \VV(\Cong_{1}) \cap V(\Cong_{2})$, cf. \eqref{eq:q.cong.closure};

\item If $\Cong_1 \supseteq \Cong_2,$ then $\VV(\Cong_1) \subseteq
\VV(\Cong_2);$

\item
$\VV(\Cong_{1} \cap \Cong_{2}) \supseteq \VV(\Cong_{1}) \cup \VV(\Cong_{2})$.
\end{enumerate}
\end{remark}

By this remark and Lemma \ref{rem:rad.cong},  using Notation \ref{notat:cong.x},  we conclude the following.
\begin{corollary}\label{cor:6} Let $\sA$ be a \nusmr, and let $X = \SpecA$ be its spectrum.
\begin{enumerate} \eroman
  \item For $x \in X$ we have $\overline{\{x\}} = \VV (\xpCong)$ and the closure of $x$ consists of all  $\pCong \supseteq \xpCong$.
  \item A point $x \in X $ is closed  if and only if $\xpCong$ is a \maximalc\  in $\SpecA$.
\end{enumerate}

 \end{corollary} \noindent
(Note that part (ii) does not hold for a  \tminimalc\ (Definition \ref{def:nt.maximalSpec}).)

\begin{corollary}\label{cor:7}
The mappings between $X = \SpecA$ and $\QCng(\sA)$, given by
$$\xymatrix{ { \left \{\begin{array}{l}
                 \text{closed} \\
                 \text{subsets in } X \end{array}\right\} }\ar@<1ex>[rr]^\II  && \ar@<1ex>[ll]^\VV  {\left\{\begin{array}{l} \text{\cradical\  congruences  }  \\ \rCong = \crad(\rCong) \text{ on  } \sA \end{array}\right\},}}
                 \qquad
                 \xymatrix{ {\begin{array}{l} Y \\
                   \VV(\rCong ) \end{array}}  \ar@{|->}@<1ex>[r] & {\begin{array}{l}  \II(Y) \\ \rCong \    \end{array}}                 \ar@{|->}@<1ex>[l] } ,$$
are inclusion-reversing, bijective, and  inverse to each other. Furthermore,
$$ \II\bigg( \bigcup_{j \in J} Y_j\bigg) = \bigcap_{j \in J} \II(Y_j), \qquad
\II\bigg( \bigcap_{j \in J} Y_j\bigg) = \crad \bigg(\sum_{j \in J} \II(Y_j)\bigg), $$
for a family $(Y_j)_{j \in J}$ of subsets  $Y_i \subset X$, where in the latter equation the
$Y_i$'s are closed in $X$.
\end{corollary}
\begin{proof}
The assertions infer from Proposition  \ref{prop:6.1.5}, except
the latter equation. Since
 $Y_j  = \VV (\II(Y_j))$  by Proposition  \ref{prop:6.1.5}.(ii),  from Proposition \ref{prop:var1}.(iv) we obtain
$$ \bigcap_{j \in J} Y_j  = \VV \bigg(\sum_{j \in J} \II(Y_j)\bigg), \dss{\text{
and thus}} \II\bigg( \bigcap_{j \in J} Y_j\bigg) = \crad \bigg(\sum_{j \in J} \II(Y_j)\bigg)$$
by Remark \ref{rem:c.cong} and Proposition  \ref{prop:6.1.5}.(i).
\end{proof}

Quasi compactness occurs for open sets in the Zariski toplology of $\SpecA$.

\begin{prop}\label{prop:6.1.11} Let $A$ be a \nusmr, and let $X = \SpecA$ be  its spectrum. Every
set $\DDg \subset X$,  with $g \in \sA$,  is quasi-compact. In particular,
$X = \DD(\one)$ is quasi-compact.
\end{prop}

\begin{proof}
  Since the sets $\DDf$ form a basis of the Zariski topology on $X$ (Corrollary \ref{cor:1.2}), it is enough to show that every covering of $\DDg$ admits a finite subcover. We may assume that $g \notin \sAghs$ is not a ghospotent, since otherwise $\DDg = \emptyset$ by Lemma \ref{lem:D.contain}.   Let $ \tF :=  (f_i)_{i \in I}$ be a family
of elements in $\sA$ such that $\DDg \subseteq  \bigcup_{i \in I} \DD(f_i)$. Taking complements, this is equivalent to
$$ \VV (g) \ds\supseteq \bigcap_{i \in I} \VV (f_i) = \VV (\gCong_\tF),$$
where $\gCong_\tF$ is the ghostifying  congruence of $\tF$. Then,
$\srad(g) \subset  \srad(\tF) = \crad(\gCong_\tF)$ by Proposition~ \ref{prop:6.1.5}, and there exists  $n  \in \Net$ such that $g^n \in \iGcl(\gCong_\tF)$.

 $\gCong_\tF$ is the minimal congruence determined by the ghost relations on $\tF$. Its ghost projection  $\iGcl(\gCong_\tF)$ is an  ideal  of $\sA$ (Remark~ \ref{rem:cong.ideal}), generated by elements from  $\tF$ and elements of $\sAghs$  (Remark~ \ref{rem:ghostification.smr}).
  Thus, there exists a finite set of indices  ${i_1}, \dots , i_m\in I$ such that  $g^n = \sum^m_{j=1} a_{i_j}g_{i_j}$, with  $g_{i_j} \in \iGcl(\gCong_\tF)$.
 In terms of ideals, this implies
$\gen{ g^n }  \subseteq  \gen{ g_{i_1}, \dots, g_{i_m} }$.

 Since $g^n $ is not a ghost, the set $\{ g_{i_1}, \dots, g_{i_m} \} $ contains a subset $K$ of non-ghost elements $\{ f_{k_1}, \dots, f_{k_\ell} \} $ from $\tF$.
Hence $g^n \in \iGcl(\gCong_K)$,  and $\VV (g) = \VV (g^n) \supseteq \VV ( \{ f_{k_1}, \dots, f_{k_\ell} \} )$. Therefore
$\DDg \subseteq \bigcup ^\ell_{j=1}\DD(f_{k_j})$,
and  $\DDg$ admits a finite subcover.
\end{proof}

Recall from Remark \ref{rem:cong1}.(iii) that a  \qhom\ $\vrp: \sA \To \sB$ of \nusmr s induces the   congruence pull-back map
\begin{equation}\label{eq:specMap}  \avrp: \SpecB \TO \SpecA, \qquad  \pCong' \longmapsto \vvrp(\pCong'),
\end{equation}
given by $(a,b) \in \pCong$ if $(\vrp(a),\vrp(b)) \in \pCong'$. The map $\avrp$ is well defined for \gprimec s, preserving \gprime s as well,  by
Remark ~\ref{rem:biject.proj}.

\begin{proposition}\label{prop:11} Let  $\Cong$ be a \qcong\ on $\sA$, and let $\pi: \sA \To A/ \Cong$ be the
canonical surjection. The map
\begin{equation*}
\api: \pcspec(\sA/\Cong)  \TO \pcspec( \sA), \qquad \pCong' \longmapsto \vpi(\pCong'),
\end{equation*}
induces a homeomorphism of topological spaces $\pcspec( \sA/ \Cong) \Isoto
\VV (\Cong)$, where
$\VV (\Cong)$ is equipped with the subspace topology obtained from the Zariski topology
of $\pcspec( \sA)$.
\end{proposition}

\begin{proof}
The map $\api$ defines a bijection between
$\pcspec (\sA/\Cong)$ and the subset of $\SpecA$ consisting of all \gprimec s  which contain
$\Cong$, cf. Remark \ref{rem:biject.proj}. Hence, it provides a bijection $\pcspec( \sA/ \Cong) \Isoto
\VV (\Cong)$. Considering this map as an identification, for elements $f \in \sA$,
we obtain
$$\SpecA \supseteq \DDf \cap  \VV (\Cong) = \DD(\pi(f)) \subseteq \pcspec( \sA/\Cong).$$
Since $\pi$ is surjective, for
$\olf \in \sA/\Cong$, the sets $\DD(\olf) \subseteq  \pcspec
 (\sA/\Cong)$  correspond bijectively to the restrictions $\VV (\Cong)\cap  \DDf \subseteq \pcspec(\sA)$ with
$f \in \sA$, which proves the assertion.
\end{proof}

\begin{cor}\label{cor:12} The spectrum of a \nusmr\ $\sA$ and the spectrum of its reduction
$\sA/\grad(\sA)$ are canonically homeomorphic.
\end{cor}

\begin{proof}
  Recall that the \gpradicalc\ $\grad(\sA)$  of $\sA$   is defined as the
\sradical\ of the \ggradical\ ideal $\gprad(R) \supseteq \sAghs$ (Definition \ref{def:ghostpotent}). Therefore, $\VV (\grad(\sA)) = \VV (\sAghs) = \SpecA$ by Lemma \ref{lem:R-G.rad}, and
Proposition~ \ref{prop:11} applies.
\end{proof}

\subsection{Irreducible varieties}
\sSkip

To deal directly with irreducibility of \nuvars,  analogous to irreducibility over rings,  in this sub-section we assume that our underlining \nusmr\ $\sA$ is \tame. \pSkip

We recall the following standard definitions.

\begin{defn} \label{def:top.irrd}

Let $X \neq \emptyset$ be an arbitrary topological space.
\begin{enumerate} \eroman
\item $X$ is called \textbf{irreducible}, if any decomposition $X = X_1 \cup X_2$ into closed subsets $X_1,X_2$ implies $X_1 = X$ or $X_2 = X$.
A subset $Y \subset X$ is irreducible, if it
is irreducible under the topology induced from $X$ on $Y$.
  \item  A point $x \in X$ is \textbf{closed}, if the set $\{x\}$ is closed;  $x$ is \textbf{generic}, if $\overline{\{x\}} = X$;   $x$ is a \textbf{generalization} of a point $y \in X $, if $y \in \overline{\{x\}} $.
  \item A point $x \in X$ is called a \textbf{maximal point}, if its closure $\overline{\{x\}}$ is an irreducible component
of $X$.
\end{enumerate}

\end{defn}

Thus, a point $x \in X$ is generic if and only if it is a generalization of every point of $X$.
Since the closure of an irreducible set is again irreducible, the existence of a generic point
implies that $X$ is irreducible.

\begin{prop}\label{prop:15} For the spectrum $X = \SpecA$ of a \tamesmr\ $\sA$ the following
conditions are equivalent:
\begin{enumerate}\eroman
  \item  $X$ is irreducible as a topological space under the Zariski topology,
\item $\sA/\grad(\sA)$ is a \nudom,
\item  $\grad(\sA)$ is a \gprimec.
\end{enumerate}
\end{prop}
\begin{proof}
 $(i) \Rightarrow(ii)$: The \nusmr\ $\sA/\grad(\sA)$ is \tame\ by Remark \ref{rem:grad.prs}.(iii), so  we may replace~ $\sA$ by its reduction $\sA/\grad(\sA)$, cf. Corollary \ref{cor:12},  to get a reduced \tamesmr\ with  $\tGz := \grad(\sA) =  A|_\ghs$.  Let $X$ be irreducible. Suppose there exist  non-ghost elements
 $f, g \in \sA \sm \tGz$ such that $fg \cng \ghost$,  then,  by Remark ~\ref{rem:var.1} and
 Proposition \ref{prop:var1}.(v),
$$X = \VV(\tGz) = \VV (fg) = \VV (f) \cup \VV (g).$$ Namely, $X$ decomposes
into the closed subsets $\VV (f), \VV (g) \subseteq X$,  implying   $X = \VV(f)$ or $X = \VV(g)$, since ~$X$ irreducible.
If $\VV (f) = X \ [ = \VV (\tGz)]$, then we conclude by Proposition \ref{prop:6.1.5} that $\srad(f)$
coincides with $\srad(\tGz) = \grad(\sA)$, cf. Lemma \ref{lem:R-G.rad}, and, hence, that $f \cng \ghost$. Similarly, $\VV(g)= X$ implies that $g \cng \ghost $.
This shows that $\sA$ has no ghost divisors, i.e. $\gDiv(\sA) = \emptyset.$
Moreover, since $\sA$ is \tame, $\sAptngs \sm \gDiv(\sA)$ is a tangible monoid by Lemma \ref{lem:gdiv.in.tame}.(ii), and therefore $\sA$ is a \nudom.

\pSkip
$(ii) \Leftrightarrow (iii)$: Immediate by Proposition \ref{prop:prime2domain}.
\pSkip
$(ii) \Rightarrow (i)$:
Assume that $\sA$ is a \nudom\
and  $X = X_1 \cup X_2$ is  a proper decomposition of $X$ into
closed subsets $X_1, X_2$. Then, for $\tG := A|_\ghs$,  Remark \ref{rem:6.6.2} and  Proposition \ref{prop:6.1.5} give
$$\srad(\tGz) = \II(X) = \II(X_1) \cap \II(X_2),  \qquad \II(X_1) \neq  \srad(\tGz) \neq  \II(X_2).$$
Now, for $f_i \in \iGcl(\II(X_i)) \sm \tGz$, $i = 1, 2$, we get $f_1f_2 \cng \ghost$ -- a contradiction as  $\sA$ was assumed to be a \nudom. Therefore  $X$
must be irreducible.
\end{proof}

\begin{thm}\label{thm:16}
Let $Y \subset X = \SpecA$ be a closed subset, where  $\sA$ a \tamesmr. Then, $Y$ is irreducible if and only if $\II(Y )$ is a \gprimec.
\end{thm}

\begin{proof}
Write $\Cong  = \II(Y )$, then  $Y = \VV (\Cong)$ and $\Cong = \crad(\Cong)$, cf. \eqref{eq:II} and Lemma \ref{rem:rad.cong}.  Thus,
 the homeomorphism $\Spec (\sA/ \Cong) \Isoto Y$ of Proposition \ref{prop:11} is applicable. Accordingly,  $Y$ is irreducible iff $\Spec(\sA/ \Cong)$ is irreducible iff $\sA/\Cong$ is a \nudom\ (by Proposition \ref{prop:15})  iff $\Cong$ is a \gprimec\ (by Proposition ~\ref{prop:prime2domain}).
\end{proof}

\begin{corollary}\label{cor:irreducible}
The  mappings $\II$ and $\VV$ between $X = \SpecA$ and $\QCng(\sA)$, given
in Corollary \ref{cor:7}, with $\sA$ a \tamesmr,   yield mutually inverse and inclusion-reversing bijections
$$\xymatrix{ { \left \{\begin{array}{l}
                 \text{irreducible closed} \\
                 \text{subsets of  } X \end{array}\right\} }\ar@<1ex>[rr]^\II  && \ar@<1ex>[ll]^\VV  {\left\{\begin{array}{l} \text{\gprimec s }  \\  \text{ on } \sA \end{array}\right\},}}
                 \qquad
                 \xymatrix{ {\begin{array}{l} Y \\
                   \overline{\{ y \}} \end{array}}  \ar@{|->}@<1ex>[r] & {\begin{array}{l}  \II(Y) \\ \pCong_y \ \end{array}} .                 \ar@{|->}@<1ex>[l] }$$
\end{corollary}

\begin{proof}
This is an immediate consequence of Corollary  \ref{cor:7} and Theorem
\ref{thm:16}, since  $\VV (\ypCong) = \overline{\{y\}}$ by Corollary ~\ref{cor:6}.
\end{proof}

Let
 $x \in X$ be a point of $X = \SpecA$ corresponding to the \gprimec\ $\xpCong$ on $\sA$.  The topological notions of Definition \ref{def:top.irrd} have the following algebraic meanings.
\begin{enumerate}\ealph
  \item   $x$ is closed iff $\xpCong$ is a \maximalc.
  \item $x$ is a generic point of $X$ iff $\xpCong$  is the unique minimal \gprimec, which  exists iff the \gpradicalc \ $\grad(A)$ of (a \tclsnusmr) $\sA$ is a \gprimec. Thus,
Proposition~ \ref{prop:15} shows that $X$ is irreducible iff  $\grad(A)$ is a \gprimec.
\item $x$ is a generization of a point $y \in X$ iff $\xpCong \subset \ypCong$.
\item $x$ is a maximal point iff $\xpCong$ is a minimal \gprimec.
\end{enumerate}


The next  example recovers the traditional correspondence between irreducible polynomials and irreducible hypersurfaces. This correspondence  is not so evident in standard tropical  geometry, whose objects are polyhedral complexes satisfying certain constraints, cf. \S\ref{ssec:polyhedral.geomtry}.

\begin{example}[Tangible  hypersurfaces]\label{examp:hypersurfaces}
 Let $f$ be a tangible polynomial function in the \tamesmr\ $ \tlF[\Lm]$, with $F$ a \nusmf, cf.~ \S\ref{ssec:ploynomials}.
Write  $Y := \VV(f) = \{ \pCong \in \SpecA \cnd \gCong_f \subseteq \pCong \}$, and let  $\II(Y) = \bigcap_{\pCong \in Y} \pCong$, which is the \sradical\ $\rCong := \srad(f)$ of $f$ (Definition \ref{def:s.radical} and Proposition \ref{prop:6.1.5}).

Recall that an element $f$ is irreducible, if it  cannot be factored as a nontrivial product $gh$ with  $g,h \in \tlF[\Lm]$.
If $f$ factorises into  irreducible polynomial functions $f = g_1^{m_1} \cdots g_s^{m_s}$, then
$\srad(f) = \srad( g_1^{m_1} \cdots g_s^{m_s})$, cf. Lemma~ \ref{lem:rad.prop.1}, which is \gprime\ iff $s =1$, i.e., $f = g_1^{m_1}$, where $f$ is not irreducible if  $m_1 >1$.
Therefore,  when $f$ is irreducible, $\srad(f)$ is a \gprimec, and by Proposition~ \ref{thm:16} we conclude  that $Y$ is an irreducible \nuvar.

\end{example}

The example assumes that the polynomial function  $f$ is tangible, which thereby captures all standard tropical hypersurfaces  via Example ~\ref{exmp:extended}, see also parts (i) and (ii) of Example \ref{examp:varieties}. However, in general,  irreducibility of a  function $f \in \tlF[\Lm]$ does not imply irreducibility of its \nuvar, for example take $f$ as in Example \ref{examp:non-unique-factorization}.(ii).
\begin{example}\label{exmp:affine.line}
The spectrum $\Aff_R = \SpecA$ of the \nusmr\ $\sA= \tlR[\lm]$ of polynomial functions in variable~ $\lm$ over a \nusmr\ $R$ is referred to as the \textbf{affine line} over $R$, viewed as an object over $\Spec(R)$ via the projection $ \Aff_R \To \Spec(R)$ induced by the canonical injection $ R \Into \tlR[\lm]$.
\end{example}

\subsection{Homomorphisms and functorial properties} \sSkip

 Let $\vrp: \sA \To \sB$ be a \qhom\  of \nusemirings0. By Remark~\ref{rem:cong1}.(iii), $\vrp$ determines
 for any \qcong\ $\bCong$ on $\sB$
a \qhom\ $\sA/\vvrp(\bCong)  \Into \sB/\bCong$. In the case of a \gprimec\ $\pCong_y \in \SpecB$ on a \nusmr\ $\sB$,  the factor \nusmr \ $\sB/\pCong_y$ is a \nudom\ (Proposition ~ \ref{prop:prime2domain}). Moreover, since $\vrp$ is a \qhom, $\vrp$ induces the map $\vvrp$ from \qcong s on $\sB$ to  \qcong s on $\sA$ which preserves intersection, i.e.,
\begin{equation}\label{eq:cong.inter}
  \vvrp(\bCong' \cap \bCong'') = \vvrp(\bCong') \cap \vvrp(\bCong'').
\end{equation}
Recall that $\vvrp$ defines the map   \eqref{eq:specMap} of spectra, written shortly  as \begin{equation}\label{eq:specMap.2} \avrp: Y  \TO
X, \qquad \pCong_y \longmapsto \vvrp(\pCong_y),  \end{equation}
where  $X= \SpecA$ and $Y=\SpecB$.

\begin{lem}\label{lem:2.2} The following diagram commutes for any $y \in \SpecB$
$$\xymatrix{
\sA \ar[r]^-{\vrp } \ar[d]^{\pi} & \sB \ar[d]^{\pi_{\ypCong}} \\
\sA/ \vvrp(\ypCong) \ar@{^{(}->}[r]^{\vrp_x} \ar@{^{(}->}[d] & \sB/\ypCong \ar@{^{(}->}[d]\\
Q(\sA/ \vvrp(\ypCong)) \ar@{^{(}->}[r]^{\vrp_x} & Q(\sB/\ypCong))
}$$
 In particular, for $f \in \sA$,  we have
$$ \vrp_x(f(\avrp(y))) = \vrp(f)(y),$$ i.e.,
$f \circ \, \avrp = \vrp(f).$
\end{lem}
Recall that $f \in \sA$ is viewed as a map $f: \pCong \mTo [f]$ on $\SpecA$, cf. \eqref{eq:f.function}.
\begin{proof}
 The map $\avrp: \SpecB \To  \SpecA$ in \eqref{eq:specMap.2} gives the upper part of the diagram, whereas the lower part
is the canonical extension to  residue \nusmr, cf. the beginning of \S\ref{ssec:ver}.
\end{proof}

\begin{prop}\label{prop:2.3}  Let $\vrp: \sA \To \sB$ be a \qhom, and let $\avrp: \SpecB  \To \SpecA$ be the associated map of spectra. Then,
\begin{enumerate} \eroman \dispace
  \item  $\iavrp(\VV(E)) = \VV(\vrp(E))$ for any subset $E \subset \sA,$
  \item $\overline{\avrp(\VV(\bCong))} = \VV(\vvrp(\bCong))$ for any \cqcong \ $\bCong$ on $\sB$.
\end{enumerate}
\end{prop}
\begin{proof} (i): The relation $y \in \iavrp(\VV(E))$  for
$y \in \SpecB$ is equivalent to $\avrp(y)  \in \VV(E)$, and hence
to $E \subseteq \iGcl(\avrp(y))$. Thus  $f \in \iGcl(\avrp(y))$ for all $f \in E$.
 In functional interpretation, cf. the beginning of \S\ref{ssec:ver}, this means that
 $f(\avrp(y)) \in (\sA/\avrp(y))|_\ghs$ for all
$f \in E$, which by Lemma  \ref{lem:2.2}  is equivalent
to $\vrp(f)(y) \in (\sA/\avrp(y))|_\ghs$. Thus, $y \in \iavrp(\VV(E))$ is equivalent to $y \in \VV(\vrp(E))$.
\pSkip
(ii): By Proposition \ref{prop:6.1.5}.(ii), $\overline{\avrp(\VV(\bCong))}$ can be written as
$\VV(\II(\avrp(\VV(\bCong))))$. Using \eqref{eq:cong.inter},  we have
$$ \II(\avrp(\VV(\bCong))) = \hSkip  \bigcap_{
\scriptsize \begin{array}{c}
\ypCong \in  \VV(\bCong) \end{array}} \hSkip  \vvrp(\ypCong) = \vvrp\bigg(
\bigcap_{
\scriptsize \begin{array}{c}
\ypCong \in  \VV(\bCong) \end{array} } \hskip -2mm \ypCong \bigg) = \vvrp(\crad(\bCong)) = \crad(\vvrp(\bCong)),
$$
and, in view of Lemma \ref{rem:rad.cong}, the claim follows by applying $\VV(\udscr)$.
%
\end{proof}

\begin{cor}\label{cor:2.5} The map $\avrp: \SpecB \To \SpecA$   associated to a \qhom\
$\vrp: \sA \To \sB$ is continuous with respect to Zariski topologies on $X= \SpecA$
and $Y= \SpecB$, i.e., the preimage of any open (resp. closed)
subset in $X$ is open (resp. closed) in $Y$ with
$$\iavrp(\DDf) = \DD(\vrp(f)), \qquad \iavrp(\VV(f)) = \VV(\vrp(f)),$$
for every $f \in \sA.$
\end{cor}
\begin{proof} The former holds  as the  preimages of $\avrp$ is compatible with passing to complements.
 The latter follows  from Proposition \ref{prop:2.3}.
\end{proof}

Next,  we focus on a \qhom\ $\vrp: \sA \To \sB$ whose associated  map $\avrp$
is injective,  and therefore defines an isomorphism   between $\SpecB$ and $\im(\avrp)$.

\begin{prop}\label{prop:2.6}
Let $\vrp: \sA \To \sB$ be a \qhom\ such that every $f' \in \sB$ can be written as  $f' = \vrp(f)h'$ for some $f\in \sA$ and a unit $h' \in {\sB}^\times$. Then, $\avrp: \SpecB\To \SpecA$ is injective and defines a \hom\ $\SpecB \Isoto \im(\avrp) \subset \SpecA$, where $\im(\avrp)$ is equipped with the subspace topology induced from
the Zariski topology on $\SpecA$.
\end{prop}

\begin{proof} 
Write $X = \SpecA$, $Y = \SpecB$.
Assume that $y', y'' \in Y$ satisfy $\avrp(y') = \avrp(y'')$, namely   $\vvrp(\pCong_{y'}) = \vvrp(\pCong_{y''})$. We claim that  $\pCong_{y'} = \pCong_{y''}$, that is $y' = y''.$ Indeed, given $(f_1', f'_2) \in \pypCong$, there exist $f_1, f_2  \in \sA$ and $h_1', h'_2 \in  \sB^\times$ such that $(f'_1, f'_2) = (\vrp(f_1)   h'_1, \vrp(f_2)   h'_2)$. Then,  $(\vrp(f_1), \vrp(f_2)) = (f'_1  {h'_1}^\inv, f'_2  {h'_2}^\inv) \in \pypCong$, and hence $(f_1, f_2) \in \vvrp(\pypCong) = \vvrp(\pCong_{y''}).$
This implies $(\vrp(f_1), \vrp(f_2)) \in  \pCong_{y''}$ and therefore $(f'_1, f'_2) = (\vrp(f_1)h'_1, \vrp(f_2)h'_2) \in \pCong_{y''}$, so that $\pypCong \subseteq \pCong_{y''}$. By symmetry $\pCong_{y'} \supseteq \pCong_{y''}$, implying that $\pypCong = \pCong_{y''}$. Hence $y' = y''$, and thus $\avrp$ is injective.

Recall that a subset $U \subseteq \im(\avrp)$ is
closed (resp. open) with respect to the subspace topology of $\im(\avrp)$ iff there is a closed (resp. open) set $\htU \subseteq  X$ such that $U = \htU \cap \; \im(\avrp)$.
Since  $\avrp: Y \To  X$  is continuous by Corollary \ref{cor:2.5}, clearly also the induced bijection $\atlvrp: Y \To \im(\avrp)$ is continuous. We next verify
that $\atlvrp$ is a homeomorphism, i.e., that
for any closed subset $V \subset Y$ there is a closed subset $U \subset X$ such
that $V = (\avrp)^\inv (U )$.  If $V = \VV (E')$ for
some subset $E' \subseteq \sB$, adjusting the elements of $E'$ by suitable units in~
${\sB}^\times$, we  may assume that  $E' \subseteq  \vrp(\sA)$.
Then,  taking $E' = \vrp(E) $ for some  $E \subseteq  \sA$,  Proposition \ref{prop:2.3} gives
$$V = \VV (E') = \VV (\vrp(E)) = (\avrp)^\inv (\VV (E)) = (\avrp)^\inv(U ),$$
with  $U = \VV (E)$, as required.
\end{proof}

\begin{cor}\label{cor:2.7} Let $\Cong$ be a \qcong\  on  $\sA$, and let $\pi:  \sA \To \sA/\Cong$ be  the canonical projection. The map
$\api: \Spec(\sA/\Cong) \To   \SpecA$
defines a \textbf{closed immersion} of spectra, i.e., a homeomorphism
$$ \Spec( \sA/\Cong) \ISOTO \VV (\Cong) \subset \Spec( A).$$
\end{cor}
\begin{proof} Proved in Proposition \ref{prop:11}.  Alternatively,  $
\im(\api) = \VV (\Cong)$ as follows from Proposition \ref{prop:2.6}.
\end{proof}

Recall that the tangible cluster and the ghost cluster of a \ccong\ on $\sA$ are disjoint, but need not be the complement of each other, and so does their projections. Recall also that
an open set $\DDf$ in $\SpecA$ may contain points over which $f$ is not tangibly evaluated.
To cope with this discrepancy,  we designate  several special subsets within open sets of  $\SpecA$, as described below.

The  \textbf{tangible support} of an element $f \in \sA$ is defined as
\begin{equation}\label{eq:CC.f}
 \EEf : = \big \{ \pCong \in \SpecA   \cnd f \in  \iTcl(\pCong) \big \} \quad  \subseteq \DDf.
\end{equation}
That is,  $\EEf$ is the subset of $\DDf$ on which, as a function \eqref{eq:f.function}, $f$ takes  tangible values, cf. \S\ref{ssec:ver}.
$\EEf$ is nonempty for every  \tprs\ $f \in \sAptngs$, by Lemma ~\ref{lem:f.prs.in.prime}, where  $\EE(f)= \SpecA$ for every $f \in \AX$. 
The latter also holds  for all \tualt\ elements $f \in \sAual$ (Definition \ref{def:t.strict}), for which  $\DD(f) = \EE(f)$.
On the other hand, by definition,  $\EEf= \emptyset$  for any $f \notin \sAptngs$, since every $\pCong$ is an \lcong.
 Furthermore, $\DDf \cap \DDg = \DD(fg)$ by Corollary ~\ref{cor:1.2}, and hence
  \begin{equation}\label{eq:EE.intersection}
  \EEf \cap \EEg = \EE(fg)
  \end{equation}
for any $f,g \in \sA$.
Tangible supports are
   generalized  to (nonempty) subsets $E \subseteq  \sA$ by defining
\begin{equation}\label{eq:CC.E} \EE(E) :=  \big \{ \pCong \in \SpecA   \cnd E \subseteq   \iTcl(\pCong)  \big \}
   = \bigcap_{h \in E} \EE(h).
\end{equation}
For these tangible supports, Corollary \ref{cor:2.5} specializes further.
\begin{prop}\label{prop:2.3.b}  Let $\vrp: \sA \To \sB$ be a \qhom, and let  $\avrp: \SpecB  \To \SpecA$ be the  associated map of spectra. Then,
  $\iavrp(\EE(E)) = \EE(\vrp(E))$, for any subset $E \subseteq \sA,$
\end{prop}
\begin{proof}  The relation $y \in \iavrp(\EE(E))$  for
$y \in \SpecB$  is equivalent to $\avrp(y)  \in \EE(E)$, and hence
to $E \subseteq \iTcl(\avrp(y))$. Therefore,  $f \in \iTcl(\avrp(y))$ for all $f \in E$. This   means that
 $f(\avrp(y)) \in (\sA/\avrp(y))|_\tng$ for all
$f \in E$, which by Lemma  \ref{lem:2.2}  is equivalent
to $\vrp(f)(y) \in (\sA/\avrp(y))|_\tng$. Thus, $y \in \iavrp(\EE(E))$ is equivalent to $y \in \EE(\vrp(E))$, since $\vrp$ is a \qhom.
\end{proof}

 Given an open set $\DDf$ and a tangible multiplicative monoid $\MS \subseteq \sAtng$  of $\sA$ (accordingly  $\MS$ must consist of \tprs\ elements and thus $\MS \subseteq \sAptngs$),  we define the \textbf{\rsset}
\begin{equation}\label{eq:DD.C}
  \DD(\MS , f) := \big \{ \pCong \in \DDf  \cnd \MS \ssetq  \iTcl(\pCong) \big  \} \quad  \subseteq \DDf.
\end{equation}
 In other words,  $\DD(\MS , f)$ can be written in terms of 
 \eqref{eq:CC.E} as
\begin{equation}\label{eq:DD.C.2}
 \DD(\MS , f) =\DDf \cap \EE(\MS) = \DDf \cap \bigcap_{h \in \MS} \EE(h), \end{equation}
i.e.,  the restriction of $\DDf$ to the tangible support of $\MS$, which could be empty.
In general, $f$ need not be contained in $\MS$ nor in the tangible projections  $\iTcl(\pCong)$ of   $ \pCong \in \DD(\MS , f)$.
\begin{proper}\label{proper:CC} For a \rsset\  $\DD(\MS,f)$ we have the following properties.
\begin{enumerate}\ealph
\item $\DD(\MS,f)  = \DDf$ for $\MS  \subset \sAual$.

\item $\DD(\MS,f)  = \EE(\MS)$ for  any $\MS$ and $f \in \sAual $.

\item $\DD(\MS,f)  = \EE(\MS) = \DDf = \SpecA$ for  $f\in \sAual$ and  $\MS \subset \sAual$.

  \item If  $f \in \MS$, then  $\DD(\MS,f) = \EE(\MS)$.



  \item $\DD(\MS, f) = \DD(\MS',f)$ if the generating sets of $\MS$ and $\MS'$ are different by units.

\end{enumerate}\end{proper}

Although, in general, $f$ is independent on $\MS$, we still have the following observation.

\begin{lem}\label{lem:D(C,f)} $\DD(\MS , f)$ is nonempty for any non-ghostpotent  $f\notin \gprad(\sA)$ and $\MS \subseteq \sAptngs$.
\end{lem}
\begin{proof}
Let $\sA_\MS$ be the tangible localization of $\sA$ by $\MS$, in which $\frac{f}{\one}$ is not a ghost. By Corollary~ \ref{cor:a.vs.prime}  there exists a \gprimec \ $\pCong'$ on $\sA_\MS$ for  $\frac{f}{\one} \notin \iGcl(\pCong')$.
 Then,   by Proposition ~\ref{prop:cong.2}.(ii), the restriction of~ $\pCong'$ to $\sA$ gives a \gprimec\ $\pCong$ with  $\MS \subseteq \iTcl(\pCong)$ and $f \notin \iGcl(\pCong)$. Hence, $\pCong$ belongs to $\DD(\MS,f)$ by \eqref{eq:DD.C}, and $\DD(\MS,f)$ is nonempty.
\end{proof}

  Using the \rsset s \eqref{eq:DD.C}, we have the next analogy to Corollary \ref{cor:2.7}, now referring to   localization.
\begin{cor}\label{cor:2.8} Let $\MS \subseteq \sAptngs$ be a tangible multiplicative submonoid of a \nusmr\ $\sA$. The canonical \qhom\ $\tau : \sA \To \sA_\MS$ induces a \hom
\begin{equation}\label{eq:open.imer}
 \atau: \Spec(\sA_\MS) \ISOTO  \bigcap_{
 h \in \MS }\DD(\MS,h) = \bigcap_{
 h \in \MS }\EEh = \EE(\MS).
\end{equation}
 Thus,  $\atau: \Spec(\sA_\MS) \Into  \bigcap_{
 h \in \MS }\DDh$ is an injective homeomorphism.\footnote{This is the place where our theory slightly deviates from classical theory in which this homeomorphism is bijective. However, later we are mainly interested in the converse homomorphism, and show that it is  a surjective homeomorphism. }
  \end{cor}

\begin{proof} Using Proposition \ref{prop:2.6}, we only need to determine  the image $\im( \atau)$ of $\atau$. Recall from Propositions~ \ref{prop:cong.1} and \ref{prop:cong.2} that taking inverse images with respect to
$\tau : \sA \To \sA_\MS$ yields a bijective correspondence between all \gprimec s on
$\sA_\MS$ and the \gprimec s $\pCong$ on  $\sA$ satisfying $\MS \subseteq  \iTcl(\pCong) $. But, for a point $x \in \SpecA$, we
have $\MS \subseteq  \iTcl(\xpCong) $  iff   $x \in \DD(\MS,h)$ for all $h \in \MS$,  so that
$\im (\atau) =   \bigcap_{
 h \in \MS }\DD(\MS,h) $. The latter equalities follow from \eqref{eq:CC.E} and \eqref{eq:DD.C.2}.
\end{proof}

If $\im(\atau)$ is open in $\SpecA$, i.e., when $\im(\atau) = \bigcap_{h \in \MS} \DDh$, then $\atau$ is called an \textbf{open immersion} of spectra.
For example, $\atau$ is an open immersion
for a finitely generated tangible monoid  $\MS \subseteq \sAual$  of \tualt\ elements, say by   $h_1, \dots , h_\ell \in  \sAX,$
since  $\bigcap_{ h \in \MS} \DD(\MS,h) = \bigcap_{ h \in \MS} \DDh = \DD(h_1 \cdots h_\ell)$.

 \begin{rem}
 The occurrence of open immersions of \gprime\ spectra might be rare, as it appears
 when $ \DD(\MS,h) = \DDh$ for all $h \in \MS$, e.g., for units of  $\sA$, cf. Properties \ref{proper:CC}.(b).   An open subset $\DDh$ may contain \gprime s over which $h$ does  not possesses  tangibles, even if $h$ is \tprs.
 In contrary,  localization is  performed only upon \tprs s.
 As a consequence of that
 a \gprimec \ in $\bigcap_{
  h\in \MS }\DDh$ which includes equivalences of $h \in \MS$ to non-tangibles (or non-\tprs s) does not necessarily have a pre-image in $\Spec(\sA_\MS)$ under the map $\atau$ in  \eqref{eq:open.imer}.

 From
 Remark \ref{rem:tcng.units} it follows that \eqref{eq:open.imer} is an open immersion for any subgroup $\MS \subseteq \sAX \subseteq  \sAual$  of units, since any \gprimec\ $\pCong$ is an \lcong,  in which \tualt s are congruent only to \tprs\ elements;  in particular $ \sAual \subseteq \iTcl(\pCong)$.
 \end{rem}

In the extreme  case that $\MS = \tT$,  where  $\tT := \sAtng$ is a monoid, then the \nusmr\ $\sA$ is  \tcls, $\sA_\tT = Q(\sA)$, and
$$\atau: \Spec(Q(\sA)) \ISOTO \bigcap_{
 h  \in \tT }\DD(\tT, h) = \big \{ \pCong \in \SpecA \cnd \iTcl(\pCong) = \sAtng \big \}. $$
We denote the subset $\big \{ \pCong \in \SpecA \cnd \iTcl(\pCong) = \sAtng \big \}$ by $\TSpec(\sA)$.


\section{Sheaves}\label{sec:6} We recall the classical setup of sheaves over a topological space~$X$, applied here to \nusmr s.  Later, some of the below objects are slightly generalized, making them applicable for our  framework. In this framework standard objects may have different interpretations. Yet, our forthcoming  abstraction well-captures the familiar notions, as described next.
\begin{defn}\label{def:sheaf}
A \textbf{presheaf} $\scF$ of \nusmr s on a topological space $X$ consists of the
data ($U,V,W \subseteq X$ are open sets):
\begin{enumerate}\ealph
  \item A \nusmr\ $\scF(U)$ for every $U$;
  \item   A  restriction map $\res^V_U:  \scF(V ) \To \scF(U)$
   for each pair $U \subset  V$, such that $\res^U_U = \id_{\scF(U)}$ for any $U$, and  $\res^W_U = \res^V_U
\circ  \res^W_V$  for any  $U \subseteq V \subseteq W$.
\end{enumerate}
The elements $\scn$ of $\scF(U)$ are called the  \textbf{sections} of $\scF$ over $U$.

A presheaf $\scF$ is a \textbf{sheaf}, if for any
coverings $U = \bigcup_iU_i$ by open sets $U_i \subset X$ the following hold:
\begin{enumerate}\ealph
  \item[(c)]
If  $\scn, \scn' \in \scF(U)$ with $\scn|_{U_i} = \scn' |_{U_i}$ for all $i$,  then $\scn = \scn'$.

\item[(d)]   If  $\scn_i|_{U_i \cap U_j} = \scn_j |_{U_i \cap U_j}$ for any  $\scn_i \in \scF(U_i)$, $\scn_j \in \scF(U_j)$, then  there
exists  $\scn \in \scF(U)$ such that $\scn|_{U_i} = \scn_i$ (which  is unique by (c)).
\end{enumerate}
 A \textbf{morphism} of presheaves $\phi: \scF  \To  \scG$ is a
family of maps $\phi_U :\scF(U) \To  \scG(U)$, for all $U \subseteq X$ open, such that
$ \res^V_U \circ \phi_V = \phi_U \circ \res^V_U$
for all pairs $U \subset  V$ from  $X$.
\end{defn}

The  \textbf{stalk} of  a sheaf $\scF$ at a point $x \in X$  is the inductive limit (i.e., colimit)
\begin{equation}\label{eq:stalk}\scF_x := \ilim{U \ni x} \scF(U).
\end{equation}
That is, $\scF_x$ is the set of equivalence classes of pairs $(U, \scn)$, where $U$ is an
open neighborhood of $x$ and $\scn \in \scF(U)$, such that $(U_1, \scn _1)$ and $(U_2, \scn_2)$ are
equivalent if $\scn_1|_V = \scn_2|_V$ for some open neighborhood $V \subseteq U_1 \cap  U_2 $ of $x$.
For each open neighborhood $U$ of $x$ there is the  canonical map
 $$\scF(U) \TO \scF_x, \qquad \scn \longmapsto  \scn_x,$$
sending  $\scn \in \scF(U)$ to the class $\scn_x$ of $(U, \scn)$ in $\scF_x$, called the \textbf{germ} of $\scn$  at  $x$.  A standard proof shows that
 the map
$$ \scF(U) \TO \prod_{x \in U} \scF_x, \qquad f \Mto (f_x)_{x\in U}, $$
is injective for any open set $U \subset  X$.

A morphism  $\vrp : \scF \To  \scG$ of sheaves on $X$ induces a map of the stalks at  $x$
$$\vrp_x := \ilim{U \ni x} \vrp_U: \scF_x \TO \scG_x,$$
providing  a functor $\scF \To \scF_x$ from the category of sheaves on $X$
to the category of sets.

Given a continuous map $\ff : X \To  Y$ of topological spaces and a sheaf $\scF$ on
 $X$,  the \textbf{direct image} of~$\scF$ is the sheaf $\dff \scF$ on $Y$, defined for open subsets $V \subset Y$ by
$$(\dff \scF)(V ) = \scF(\iff(V )), 
$$
whose restriction maps are inherited from $\scF$.  For a morphism  $\psi: \scF \To \scG$ of sheaves, the family of maps
$\dff(\psi)_V := \psi_{\iff(V )}$,  with  $V \subset Y$ open,  is a morphism $\dff(\psi) : \dff \scF \To \dff \scG$. Therefore, $\dff$ is
a functor from the category of sheaves on $X$ to the category of sheaves on $Y$ that admits composition $\psi_*(\dff \scF) = (\psi \circ \ff)_* \scF$ for a continuous map $\psi: Y \To Z$.

\pSkip

Henceforth, unless otherwise is indicated,  \semph{we assume that $\sA$ is a \tame\ \nusmr}, i.e.,
every $f \in \sA \sm (\sAtng \cup \sAghs)$ can be written as $f = p + eq$ for some $p,q \in \sAtng$ (Definition \ref{def:nusemiring}). This is a very mild assumption,  holds in all of our examples, especially in Examples~ \ref{exmp:grp.fld} and ~ \ref{exmp:extended}.
Recall from Lemma ~\ref{lem:gdiv.in.tame} that, since $\sA$ is \tame, $\sAptngsD$ is a monoid of \tprs\ elements. We denote this monoid by $\sAptngM$, for short. Recall from \S\ref{ssec:nudomain} that $\sAX \subseteq \sAptngM$.

\subsection{Functoriality  towards sheaves}\label{ssec:towards.sheaves}
\sSkip

Sheaves have characteristic functorial properties which  are   shown later to fit well  our constructions.
We begin with a \tprs\ element $f \in \sAptngs$,  and  write
\begin{equation}\label{eq:MSf}
  \MSf : = \{ \one, f, f^2, \dots \}
\end{equation} for  the tangible monoid generated by $f$.
For an element $f \notin \sAptngs$, i.e., is not \tprs, we formally set $\MSf = \{ \one \}$.
 For this  monoid we have the inclusion
$\MSf \subseteq \iTcl(\pCong)$ in every $ \pCong \in  \EEf$ where  $\EEf = \DD(\MSf, f)$ if $f \in\MSf $, cf. \eqref{eq:CC.f} and  \eqref{eq:DD.C}.
We write $\sA_f$ for the tangible localization $\sA_{\MSf}$ of~ $\sA$ by
~$\MSf$, cf. \S\ref{ssec:tngLocal}.

\begin{rem}\label{rem:Dfg.tng} Given $f \in \sAptngs$, we  assign the set $\DDf \sset \SpecA$ with the localization $\sA_f$ of $\sA$ by~ $\MSf$.
Then, for $h \in \sAptngs$,  the inclusion $\DDf \subseteq  \DDh$  gives rise to a \nusmr\ \qhom\
$\tau^h_f: \sA_h \To ~\sA_f$. Indeed, $\DDf \subseteq  \DDh$ is equivalent to $\VV (f) \supseteq \VV (h)$  and to $\srad(f) \subseteq \srad(h)$ by  Lemma~\ref{lem:D.contain}.(iv).
Furthermore, $h \in \sAptngs$, and hence,  by Lemma ~ \ref{lem:rad.prop.1},  $f^n = hg$ for some $n \in \Net$ and $g \in \sA$.


The canonical map $\tau_f : \sA \To \sA_f$ shows that $\tau_f (f)^n = \tau_f (h)\tau_f (g)$, and hence $\tau_f (h)$ is a unit in~ $\sA_f$.  Thus, by the universal property of  localization (Proposition \ref{prop:local.univ}), the \qhom\ $\tau_f: \sA \To \sA_f$ factorizes  uniquely as
$\tau_f =  \tau_h \circ \tau^h_f$ through  the \qhom s   $\tau_h : \sA \To  \sA_h$  and $\tau^h_f : \sA_h \To \sA_f$, i.e., the diagram
$$\xymatrix{
\sA \ar@{->}[d]_{\tau_h } \ar@{->}[rr]^{\tau_f} && \sA_f  \\
 \sA_h \ar@{-->}[rru]_{\tau_f^h} &&  \\
}
$$ commutes.
Then, in functorial sense,  $\tau^h_f : \sA_h \To \sA_f$  is assigned
to the inclusion $\DDf \subseteq
\DDh$.
Also, for every  $f \in \sAptngs$  the map  $\tau^f_f : \sA_f \To \sA_f$ is the
identity map, and for any inclusions $\DD(f) \subset  \DD(g) \subset  \DD(h)$,
the map composition
 $\xymatrix{
\sA_h \ar@{->}[r]^{\tau^h_g } & \sA_g \ar@{->}[r]^{\tau_f^g } & \sA_f   }
$
coincides with $\tau^h_f : \sA_h \To \sA_f$.
\end{rem}

Besides basic open sets $\DDf$ with $f$ \tprs, our topological space also includes sets $\DDf$ which are determined by elements $f \notin \sAptngs$. Each such $\DDf$  should be allocated with a  \nusmr. To cope  with this type of open sets, in a way that is compatible with the case of  $f \in \sAptngs$, for every $f \in \sA$ we define the subset
\begin{equation}\label{eq:S(f).str}
\STs(f) :=  \{h \in \sAptngM  \cnd \DDf \subseteq  \DD(h)\} \quad  \subseteq \sAptngM,
\end{equation}
which consists  of \tprs\ elements which  are not ghost divisors. We then let
\begin{equation}\label{eq:S(f)}
\STf := \gen{\MSf, \STs(f) } \quad  \subseteq \sAptngs,
\end{equation}
be the set multiplicatively  generated by $\MSf$ and  $\STs(f)$.
Accordingly, $\STsf = \STf$  when $f \notin \sAptngs$, or when $f \in \sAptngM$.

\begin{lem}\label{lem:STf.mon} For any $f \in \sA$, the subset
  $\STf$ is a tangible multiplicative submonoid of $\sA$
\end{lem}
\begin{proof}
  $\STf$ is nonempty, as $\one \in \sAptngM$ and $\DDf \subseteq \DD(\one)$ for every $f \in \sA$.
  If $\DD(f) \subseteq \DD(h')$ and $\DD(f) \subseteq \DD(h'')$ for $h',h'' \in \STf$, then $\DD(f) \subseteq \DD(h') \cap \DD (h'') =  \DD(h' h'')$ by Corollary \ref{cor:1.2}.(i).

  If $f \notin  \sAptngs$, then both $h',h'' \in \STsf$. But  $\sAptngM$  is a tangible monoid by Lemma \ref{lem:gdiv.in.tame}.(ii),   since $\sA$ is a \tamesmr, and thus $h'h'' \in \STsf \subseteq \sAptngM$, implying that $\STsf = \STf$ is a multiplicative submonoid of $\sAptngM \subseteq \sAptngs$.
  (The same holds  for $f \in \sAptngM$.)

   If $f \in \sAptngs$, i.e., $f$ is \tprs, then $h' h'' \in \sAptngs$ by Lemma \ref{lem:D.contain}.(vii), since   $\DD(f) \subseteq   \DD(h' h'')$. Thus, $\STf$ is a multiplicative submonoid of $\sAptngs$.
\end{proof}

\begin{proper}\label{proper:Sf} Let $f$ be an element of $\sA$.
 \begin{enumerate} \ealph \dispace
 \item All units $h \in \sAX  $ are contained in $ \STf$,  in particular $\one \in \STf$.

       \item
$ \STf = \sAptngM$ for any ghostpotent  $f \in \gprad(\sA)$, since $\DDf = \emptyset$ (see   Lemma \ref{lem:gdiv.in.tame}.(ii)).

\item $ \STf = \sAX$ for every  unit  $f \in \sAX \cap \sAptngs$, since $\DDf = \SpecA$ (cf. Lemma \ref{lem:D.contain}.(v)).

   \item When  $f \in \sAptngs$ is \tprs, $\MSf \subseteq \STf$, and $\STf$ contains all powers of $f$.

   \item If $f \notin \sAptngs$, then $\STf$ does not contain $f$, even if $f \in \sAtng$. But, it contains every   $h \in \sAptngM$ such that $h \nucong f$, if exist.


\item $\STh \subseteq \STf$ for every $h \in \STf$.
 \end{enumerate}
 \end{proper}
Note that $\DDf \subseteq \DDg$ is equivalent to $\STf \supseteq \STg$, and hence to
\begin{equation}\label{eq:DfDg}
  g \pcng \ghost \Dir f \pcng \ghost \ ,
\end{equation}
 in every  $\pCong \in \SpecA$.

\begin{rem}\label{rem:SfSg.intersection}
$\STf \cap \STg \subseteq  \ST(f g)$ for any $f,g \in \sA$. Indeed, $h \in \STf \cap \STg $ means that  $\DDh \supseteq \DDf$ and $\DDh \supseteq \DDg$, and thus  $\DDh \supseteq \DDf \cup \DDg \supseteq \DDf \cap \DDg = \DDfg$, by Corollary \ref{cor:1.2}.(i).
 Moreover, $\STfg \supseteq \STf$, since $\DD(fg)\subseteq \DDf$, and, by symmetry, also $\STfg \supseteq \STg$, hence $\STfg \supseteq \STf \cup \STg $.
\end{rem}

The inclusion $\MSf \subseteq \STf$ for \tprs\ elements from Properties \ref{proper:Sf}.(d) gives rise to an isomorphism of localized \nusmr s.

\begin{lem}\label{lem:fg.1} For $f \in \sAptngs $, the map $\iota_f:\sA_f \To \sA_{\STf}$ is an isomorphism.
\end{lem}
\begin{proof} Let $\tau_f : \sA \To \sA_f$ be the  canonical injection, and let $h \in \STf$.  Then, $\tau_f (h) \in \sA_f^\times$ is a unit, since $\DDf \subseteq  \DDh$ by Remark \ref{rem:Dfg.tng}. Thus, $\tau_f$ maps $\STf$ into the group of units in $\sA_f$, and hence  $\iota_f: \sA_f \To \sA_{\STf}$ is an isomorphism by the universal property of  localization (Proposition \ref{prop:local.univ}).
\end{proof}

\begin{rem}\label{rem:f.vs.Sf} Given $f \in \iTcl(\pCong)$, and hence $f \in \sAptngs$, we have
the  injection $\sA_f \Into \sA_\pCong$. Then  Lemma~ \ref{lem:fg.1} gives the injection $\sA_\STf \Into \sA_\pCong$, which implies that $\STf \subseteq \iTcl(\pCong)$.
\end{rem}

We can now extend Remark \ref{rem:Dfg.tng} to all elements in $\sA$, to obtain functoriality of arbitrary   open sets ~$\DDf$.

\begin{lem}\label{lem:SfSg.map}
  The inclusion $\DDf \subseteq \DDh $ 
  gives rise to a \qhom
  $$\tltau_f^h : \sA_{\STh} \TO \sA_{\STf}.$$
  Furthermore, if $f,h \in \sAptngs$, then  $\tltau_f^h = \iota_f \circ \tau_f^h .$
\end{lem}
\begin{proof}
$\DDf \subseteq \DDh$ implies $\STh \subseteq \STf$, whose elements are all \tprs s. Thereby,  the canonical \qhom\ $\tltau_f: \sA \To \sA_\STf$ maps $\STh$ to units of $\sA_\STf$,  i.e., $\tltau_f (\STh) \subseteq (\sA_\STf)^\times $. Then, by  the universal property of  localization (Proposition \ref{prop:local.univ}), $\tltau_f$ factorizes  uniquely as
$\tltau_{f} =  \tltau_{h} \circ \tltau^{h}_{f}$ through  the \qhom s  $\tltau_{h} : \sA \To  \sA_{\STh}$  and $\tltau^{h}_{f} : \sA_{\ST(h)} \To \sA_{\STf}$, rendering the diagram
$$\xymatrix{
\sA \ar@{->}[d]_{\tltau_h } \ar@{->}[rr]^{\tltau_f} && \sA_\STf  \\
 \sA_\STh \ar@{-->}[rru]_{\tltau_f^h} &&  \\
}
$$ commutative.
The relation $\tltau_f^h = \iota_f \circ \tau_f^h $ is obtained from  Remark \ref{rem:Dfg.tng} and Lemma \ref{lem:fg.1}.
\end{proof}

 We see that  for every  $f \in \sA$  the map  $\tltau^f_f : \sA_\STf \To \sA_\STf$ is the
identity map, and for any inclusions $\DD(f) \subset  \DD(g) \subset  \DD(h)$,
the map compositions
 $$\xymatrix{
\sA_\STh \ar@{->}[r]^{\tltau^h_g } & \sA_\STg \ar@{->}[r]^{\tltau_f^g } & \sA_\STf   }
$$
coincides with $\tltau^h_f : \sA_\STh \To \sA_\STf$.

For a ghostpotent $f \in \gprad(\sA)$ we have $\DDf = \emptyset$ and thus $ \STf = \sAptngM$, cf. Properties ~\ref{proper:Sf}.(b). Thereby, the localized \nusmr\ $\sA_\STf$ is the same for all $f \in \gprad(\sA) $, and for any $h \in \sA$ the \qhom\ $\tltau^{h}_{f} : \sA_{\STh} \To ~ \sA_{\STf}$ can be taken to be the identity map. However, we are mostly interested in elements $f \notin \gprad(\sA)$, for which $\DDf \neq \emptyset$. \pSkip

Let $A$ be a  \tamenusmr, and let $X = \SpecA$ be its \gprime\ spectrum (Definition ~\ref{def:prmCng}). Denote by  $\bfDD(X)$ the category whose objects are open subsets $\DDf \subseteq  X$, with $f \in \sA$,  and its morphisms are inclusions $\DDf \subseteq  \DDh$.
By the above construction (Lemma \ref{lem:SfSg.map}) we obtain a well defined functor
$$\OX : \bfDD(X)  \TO  \NSMR $$
from $\bfDD(X)$ to the category $\NSMR$ of \nusmr s (Definition \ref{def:nusmr.cag}), given by
\begin{equation}\label{eq:OX}
\begin{array}{rcl}
\DDf & \longmapsto & \sA_{\STf}, \\[1mm]
\DDf \subseteq  \DDh & \longmapsto &   \sA_{\STh} \To
\sA_{\STf}.
\end{array}
\end{equation}
The image of $\OX$ restricts to the subcategory of \tamenusmr s in \NSMR.
 From our preceding discussion we conclude the following.

\begin{cor}\label{cor:presheaf}
  $\OX$ is a presheaf of (\tame) \nusmr s on $\bfDD(X)$.
\end{cor}

Given an element $f \in \sA$, assumed not to be ghostpotent,  with $\STf$ as defined in \eqref{eq:S(f)},  we write
$$ X_f := \Spec\big(\sA_\STf\big). $$
  We define the \textbf{(tangible) cover set} of $\DDf$ to be
   \begin{equation}\label{eq:tng.cover}
   \CCf := \bigcap_{h \in \STf} \DDh,
 \end{equation}
  for which we  have $\DDf \subseteq \CCf.$
 When $f\in \sAptngs$ is \tprs, $f\in \STf$ and thus $ \CCf = \DDf, $  where   the map   $\iota_f:\sA_f \To \sA_{\STf}$ is isomorphism by Lemma ~ \ref{lem:fg.1}. If $f \in \gprad(\sA)$, then    $\CCf = \bigcap_{h \in \sAptngM} \DDh$, cf. Properties \ref{proper:Sf}.(b).

\pSkip

  Recall  from \eqref{eq:DD.C} that  $\DD(\MS,f )$ denotes the restriction of $\DDf$ to those \gprime s $\pCong\in \SpecA$ satisfying $\MS\subseteq \iTcl(\pCong)$, where  $\MS$ is a multiplicative tangible submonoid of $\sA$. Obviously,
  $\DDf \subseteq \DDh$ implies   $\DD(\MS, f) \subseteq \DD(\MS, h)$. Furthermore, by Lemma ~ \ref{lem:D(C,f)},  the restricted set $\DD(\MS,f)$ is nonempty for $\MS = \STf$ when   $f \notin \gprad(\sA)$, cf.
   Corollary~ \ref{cor:a.vs.prime}.

\begin{rem}\label{rem:Xf2DSS}
Letting $\ST:= \STf$ with $f \in \sA$,  Corollary ~\ref{cor:2.8} asserts  that  the canonical \qhom \
$\tau : \sA \To \sA_\STf$ induces an isomorphism
$$\atau: X_f \ISOTO \CC(\ST,f) := \bigcap_{h \in \ST} \DD(\ST, h) = \EE(\ST), \qquad \xpCong' \Mto \xpCong'|_\sA,$$
 by restricting \gprimec s  on $\sA_\STf$ to  \gprimec s on $\sA$,
cf. Propositions~ \ref{prop:cong.1} and ~ \ref{prop:cong.2}.
(This correspondence relies  on the bijection  between \gprimec s on $\sA_{\STf}$ and those \gprimec s on $\sA$ with
$\STf \subseteq \iTcl(\pCong)$.)
The latter equality to $\EE(\ST)$ --the tangible support of $\STf$-- is by definition, see  \eqref{eq:CC.E}.
\end{rem}

Note that when ~$f \notin \sAptngs$,  $\DDf$  does not necessarily  contain the entire  restriction $\CC(\ST,f)$ of the cover set~ $\CCf$ to $\ST := \STf$, but  their intersection $\CC(\ST,f) \cap \DDf$ is nonempty, again by Lemma \ref{lem:D(C,f)}.
 Also, we always have $\EE(f) \subseteq \CC(\ST,f)$, cf.~ \eqref{eq:CC.f}.
Otherwise, if $f \in \sAptngs$, then $\CC(\ST,f) \subseteq \DD(\ST,f) \subseteq  \DDf$, since $f \in \ST := \STf$.

\pSkip

  We denote by $\DD_f(h)$ the open set $\DD(h) \subseteq  \Spec(\sA_\STf)$ associated to  $h \in \sA_{\STf}$.
To summarize,  in this setting, for the case of \tprs\ elements $f \in \sAptngs$ we have
the following.
\begin{lem}\label{lem:6.6.1} Let $\sA$ be a (\tame) \nusmr\ and let  $f \in \sAptngs$ be  a \tprs\ element.
  \begin{enumerate} \eroman
    \item
    The canonical \qhom\
$\tau : A \To A_\STf$ induces an injective homeomorphism $$\atau : X_f \INTO  \DDf \subset X$$
and, conversely,  a surjective homeomorphism
$$\btau : \DDf \ONTO X_f.$$
\item $\iatau (\DDg) = \DD_f (\tau (g))$ for any $g \in   \sAptngs $ with $\DDg \subseteq \DDf$. Furthermore, $\iatau$ induces a bijective correspondence between sets  $\DDg \subseteq  \DDf$  and the open sets  $\DD_f (h)$, where $h \in A_\STf$.
\item The restriction of the functor $\OX : \bfDD(X) \To \NSMR$ to the subcategory
induced on $\DDf$ is equivalent to the functor $\OXf : \bfDD(X_f ) \To \NSMR$.
    \end{enumerate}
\end{lem}

\begin{proof} (i): Corollary \ref{cor:2.8} gives  $\atau : X_f \Isoto \CC(\ST,f) = \DD(\ST,f)\subseteq \DDf$, for $S = \STf$,  which  proves the first part and also induces  the converse map $\btau$, showing that $\btau$ is onto.


\pSkip
(ii): Corollary \ref{cor:2.5} gives the first part.
The map   $\iota_f:\sA_f \To \sA_{\STf}$ is isomorphism by Lemma ~ \ref{lem:fg.1}, so that ~ $\sA_{\STf}$ can be replaced by $\sA_f$. Then, the second part is clear, as  $h = \frac{g} {f^k} \in  \sA_f$  yields
$$\DD_f (h) = \DD_f \bigg( \frac{fg}{f^{k+1}}\bigg) = \DD_f \big(\tau (fg)\big) = \iatau\big(\DD(fg)\big), $$
where  $\DD(fg) \subseteq  \DDf$.

\pSkip
(iii): Obtained by the canonical isomorphism
$\sA_{\ST(fg)} \Isoto \big(\sA_\STf \big)_{\tau(\STg)}.$
\end{proof}

Lemma \ref{lem:6.6.1} applies only to  \tprs\ elements $f \in \sAptngs$, while Remark
\ref{rem:Xf2DSS} refers to the tangible support $\EE(\STf) = \CC(\STf,f)$ 
of the tangible cover $\CC(f)$ of $\DDf$,
rather than to their (nonempty) intersection.  This type of intersection is dealt next.

\begin{definition}\label{def:focal.zone} Let $f \in \sA$, and take  $\STf$ from \eqref{eq:S(f)}  to be
 the multiplicative monoid $\MS$  in \eqref{eq:DD.C}. The subset
\begin{equation}\label{eq:focal.zone}
\fcDDf := \DD(\STf,f) =  \{ \pCong \in \DDf  \cnd \STf  \ssetq  \iTcl(\pCong)  \} \ \subseteq \DDf
\end{equation}
is called the \textbf{\fzone} of ~$\DDf$.
  An element $f \in \sA$ is said to be \textbf{\strict}, if $\fcDDf = \DDf$.

  \end{definition}

 Clearly, every ghostpotent $f \in \gprad(\sA)$ is \strict,  whereas $\fcDDf = \DDf =  \emptyset$.   On the other edge, every \tualt\ element  $f \in \sAual$, and in particular  every unit,  is \strict \ with $\fcDDf = \DDf =  \SpecA$.

\begin{example}
  Monomials having \tprs\ coefficients in the polynomial \nusmr\ $\sA= F[\Lm]$  over a \nusmf\ $F$ are \strict,  cf. Example \ref{examp:pre.smr}.(i) and \S\ref{ssec:ploynomials}.
\end{example}



As $\STf \subseteq \sAptngs$ is a tangible monoid  for any $f \in \sA$ (Lemma \ref{lem:STf.mon}),  the \fzone\ $\fcDDf \subseteq \DDf $  is canonically defined for every $f \in \sA$, and by Lemma ~ \ref{lem:D(C,f)} is nonempty whenever $f \notin \gprad(\sA)$.
Moreover,  $\EEf \subseteq \fcDDf $,  since  $\MSf \subseteq \STf$ by Properties \ref{proper:Sf}.(d),
where $\EEf$ could be empty, e.g., when $f \notin \sAptngs$.

The \fzone\ $\fcDDf$  can be written in terms of tangible support \eqref{eq:CC.E}  as
\begin{equation}\label{eq:focal.zone.2}
\fcDDf := \DD(\STf,f) =   \DDf  \cap \bigcap_{h \in \STf} \EE(h) = \DDf  \cap  \EE(\STf) \ \subseteq \DDf,
\end{equation}
providing a useful form.

\begin{remark}\label{rem:focal.zone} \Fzone s respect intersections of open sets in the sense that
  $$\fcDD(fg) \subseteq  \fcDD(f) \cap \fcDD(g) \subseteq \DDf \cap \DDg = \DD(fg)$$
  for any $f,g \in \sA$,
 since $\STf \cup \ST(g) \subseteq \ST(fg)$ by  Remark \ref{rem:SfSg.intersection}. This also shows that if
 $\DDf \cap \DDg  \neq \emptyset$, then $\fcDDf \cap \fcDDg  \neq \emptyset$, whereas $\fcDD(fg) \neq \emptyset$ by Lemma \ref{lem:D(C,f)}, since $\ST(fg)$ is a tangible monoid.

 Also, the inclusion $\DDf \subset \DDh$ implies $\fcDDf \subset \fcDDh$, since $\STh \subset \STf$ for every $h \in \STf$, cf. Properties \ref{proper:Sf}.(f).
\end{remark}
This  remark can be strengthened further for \tprs\ elements, specializing  Corollary
\ref{cor:1.2} to \fzone s.

\begin{lemma}\label{lem:fuclZintersection} $\fcDD(fg) =   \fcDD(f) \cap \fcDD(g)$, when $f,g \in \sAptngs$.
\end{lemma}
\begin{proof} $\fcDD(fg) \subseteq  \fcDD(f) \cap \fcDD(g)$ by Remark \ref{rem:focal.zone}.
  Suppose $\fcDD(fg) \subsetneq  \fcDD(f) \cap \fcDD(g)$,  and let  $\pCong \in (\fcDD(f) \cap \fcDD(g)) \setminus \fcDD(fg)$, i.e., $\ST(fg) \nsubseteq \iTcl(\pCong)$, which means that there exits $h \in \STfg$ with
  $h \notin \iTcl(\pCong)$. On the other hand,
  $\DD(fg) \subseteq \DDh$ iff $  \VV(fg) \supseteq \VV(h) $ iff
$ \srad(fg) \subseteq  \srad(h)$ by Lemma \ref{lem:D.contain}.(iv), and thus
  $(fg)^n = h b + c$ with $b \in \sA$, $c \in \tG$, by  Remark~ \ref{rem:ghostification.smr}.  Then, $(fg)^n \notin \iTcl(\pCong)$, implying that
  $(fg) \notin \iTcl(\pCong)$, and furthermore that $f \notin \iTcl(\pCong)$ or $g \notin \iTcl(\pCong)$, since $ \iTcl(\pCong)$ is  a  monoid. Say $f \notin \iTcl(\pCong)$.  But, $f \in \STf$, since $f \in \sAptngs$, and thus $\STf \nsubseteq \iTcl(\pCong) $. Hence,  $\pCong \notin \fcDD(f)$, so $\pCong \notin \fcDD(f) \cap \fcDD(g)$.
\end{proof}

Note that by definition $\fcDDf \subseteq \CC(\STf,f)$ for any $f \in \sA$, cf. \eqref{eq:tng.cover}, while $\fcDDf = \CC(\STf,f)$ when $f \in \sAptngs$. In this view Corollary \ref{cor:2.5} specializes to \fzone s.
\begin{lem}\label{lem:fc.map}
  Let $\vrp: \sA \To \sB$ a be  \qhom, and let $\avrp : \SpecB \To \SpecA$ be the induced map ~ \eqref{eq:specMap.2}.  Then, $\fcV = \iavrp(\fcU)$, where  $V := \iavrp(U)$, for any open set $U \subset \SpecA$.
\end{lem}
\begin{proof} Let $U = \DDf$.  Then, $U \subset \DDh$ for each $h \in \STf$, and by Corollary \ref{cor:2.5} we get
$$V = \DD(\vrp(f)) = \iavrp(\DDf) \subset \iavrp(\DDh) = \DD(\vrp(h)).$$
Hence $\vrp(h) \in \ST(\vrp(f))$, which gives $ \vrp(\STf) \subseteq  \ST(\vrp(f))$. By   \eqref{eq:focal.zone.2} and the specialization in  Proposition ~\ref{prop:2.3.b}, i.e., $\iavrp(\EE(\STf)) = \EE(\vrp(\STf))$,  we obtain the inclusion  $\fcV \subseteq  \iavrp(\fcU)$. \pSkip

On the other hand, if $\fcV \subsetneq  \iavrp(\fcU)$, then there exists $h' \in \ST(\vrp(f)) \subset \sB$ such that $\iavrp(\fcU) \nsubseteq \EE(h') \cap \EE(\vrp(h))$ for some $h \in \STf \subset \sA$.  Then,  $W := \DD(h') \cap \DD(\vrp(h))$ is an open subset of $\SpecB$, which is  contained   in $\DD(\vrp(h))$ and  contains $\DD(\vrp(f))$ as well, by the first part.  Therefore, $\avrp(W)$ is an open set of $\SpecA$ which is contained in $\DD(h)$, so there exists $g \in \sAptngs$ such that  $\avrp(W)= \DD(g)$. As $\DD(\vrp(f)) \subseteq W $,  we get that $\DD(f) \subseteq \DD(g)$, and hence $g \in \STf$ with  $\DD(\vrp(g)) = W$. But, this shows that  $\iavrp(\fcU) \subseteq \EE(h') \cap \EE(\vrp(h))$,  and hence  $\fcV =  \iavrp(\fcU)$.
\end{proof}

The setup \eqref{eq:focal.zone.2} allocates each open set $\DDf$ in the Zariski topology on $X := \SpecA$ with a distinguished (nonempty) subset -- its  \fzone\ $\fcDDf$. By the lemma, \fzone s respect contiguity of maps between spectra, induced by \qhom s of \nusmr s.

\pSkip

Having gone
this far with functorial properties, one can prove that, restricting to basic open coverings $\DDf = \bigcup_{i \in I }\DD(f_i)$, subject to certain  constrainers concerning their  \fzone s,  the functor $\OX : \bfDD(X) \To \NSMR$ as defined in~ \eqref{eq:OX} is a sheaf of (\tame) \nusmr s.
It extends to the entire spectrum
$X = \SpecA$ of $\sA$ by letting
$$\cOX(U) = \plim{\DDf\subset U} \OX(\DDf)=  \plim{f\in \sA, \\ \DDf\subset U}
\sA_\STf$$
for open subsets $U \subset  X$, where the first projective limit runs over all open
subsets $\DDf \subset X$  that are contained in $U$, and the second over
all $f \in \sA$ such that $\DDf \subset  U$. By the universal property
of projective limits $\cOX$ is a functor on the category of open subsets in $X$.
Moreover, if $U = \DDg$ is an open set,  for some $g \in \sA$, then $\cOX(U)$
 is canonically isomorphic to $\OX(\DDg) = \sA_\STg$
and  $\cOX$ restricts to a functor on  $\bfDD(X)$, isomorphic to $\OX$, yielding
$\cOX$ as a sheaf of \nusmr s on $X$.

Nevertheless,   the above  level of generality is not required for the purpose of the present paper, and below we provide an explicit construction of a structure sheave, obtained directly in terms of sections which are  coincident with stalk structures, subject to \fzone s.

Let $M$ be an $\sA$-\numod\ (Definition \ref{def:numodule}), and let   $X = \SpecA$. Using $\OX$ one can  define the map
$$\scF:\DDf  \TO M_{\STf}, $$
 where $M_{\STf} = M \otimes_\sA \sA_{\STf}$ is the localization of $M$ by the
multiplicative system $\STf$ determined by~  $f$ in~ $\sA$ (Definition \ref{def:module.localization}). Then, $\scF$ is a functor from the category $\bfDD(X)$ of all  open
subsets  $\DDf \subset X$ to the category of  $\sA$-\numod s, and it extends to a sheaf of
$\sA$-\numod s  on all open subsets of $X$.
 Since $M_\STf$ is an $\sA_\STf$-\numod \ for every $f \in \sA$, the \numod\
structure on $\scF$  extends canonically to a $\OX$-\numod\ structure, called the
 \textbf{$\OX$-\numod\ associated to
$M$}.


\subsection{Structure sheaves} \sSkip

We begin with an explicit construction of local sections which are  patched together to a sheaf $\OX$ of (\tame) \nusmr s on $X:=\SpecA$, and are compatible with functoriality as described previously in ~\S\ref{ssec:towards.sheaves}.
Recall that ~$\sA_x$ denotes the localization of~ $\sA$ by the \gprimec\ ~ $\xpCong$,  corresponding to a point~ $x$ in $X$ (Notation~ \ref{notat:cong.x}). By Definition \ref{def:cong.localization}, such  localization is defined for $\xpCong$  by letting  $\sA_{\xpCong}:=  \sA_\MS$, where $\MS = \iTcl(\xpCong)$.
Recall also from  Definition ~ \ref{def:focal.zone} that
 $\fcDDf$ denotes the \fzone\ of $\DDf \subseteq X$. With these notations, by \eqref{eq:focal.zone} we can write
 \begin{equation}\label{eq:x.focal.zone}
   \STf \subseteq  \iTcl(\xpCong) \dss\Iff x \in \fcDDf,
 \end{equation}
which easily restates our correspondence.

In classical algebraic geometry over rings, which employs ideals,  the inclusion of a prime ideal $\mfp_x$ in a basic open set $\DDf$
 automatically implies that $f \in \sA \sm \mfp_x$, where  $\sA \sm \mfp_x$ is a multiplicative system. The same holds for a collection of elements $f_1, \dots, f_n \in \sA \sm \mfp_x$,  where now $\mfp_x$ belongs to the intersection of the corresponding ~ $\DD(f_i)$. These  inclusions in a multiplicative system  are  curtail for localization, as these elements become units in $\sA_x$. Furthermore, they establish the correspondence  $\sA_{f} \To \sA_x $
  having the required universal property that coincides with  inductive limits.

  In the supertropical setting, since the complement of a ghost projection $\iGcl(\pCong)$ in $\sA$ may contain elements which are not \tprs\ and thus cannot determine units,  maps of type  $\sA_\STf \To \sA_x$ are not always properly accessible for every $x \in \DDf$.
  Therefore,  a direct supertropical analogy to classical approach  does not suit for this setting and our construction of sections basically relies  on those points of $\DDf$  that admit the right behaviour, i.e. $\STf \subseteq \iTcl(\pCong)$.  These points are the points of the \fzone\ $\fcDDf$ of $\DDf$, defined for every $f \in \sA$ which is not ghostpotnet.  As we shall see, this specialization answers all our needs,   especially concerning computability  with  inductive limits.

\pSkip

Given an open set $U = \DDf$, where $f \in \sA$ is assumed not to be ghostpotent, we write $\fcU = \fcDDf$ for its (nonempty) \fzone\ \eqref{eq:focal.zone}.

\begin{defn}\label{def:structure.sheaf}  Let  $U \subseteq X$ be  an open set.
 We say that a map
 \begin{equation}\label{eq:fc.sections}
 \sig: U \TO \coprod_{x \in\chU} \sAx, \qquad \scn = (\scn_{x})_{x \in \chU}, \ \scn_{x} \in \sAx,
\end{equation}
  is \textbf{locally quotient of $\sA$ on $U$}, if
 for every $x \in \fcU$ there is a neighborhood $V$ in  $U$ and elements  $f ,g \in \sA$
with $g \in \iTcl(\pCong_y)$ such that  $\sig_{y} = \frac fg$ in $\sA_y$ for all
$y \in \fcV$. Such maps are the  \textbf{(local) sections} of the \textbf{structure \nushf} $\OX$  of $X$, defined via
$$\OX(U) := \big\{ \sig \cnd \sig  \text{ is locally quotient of $\sA$ on $U$} \big\}$$
for all open sets $U$ in  $X$.
\end{defn}

The conditions in this definition imply that the elements of $\OX(U)$ are local, and that these sections indeed form a sheaf $\OX$ on $X$ (Definition ~   \ref{def:sheaf}) due to the following structure over all open sets $U,V$ of~ $X$. The restriction maps $\OX(U) \To \OX(V)$ are induced from the inclusions $\io: U \Into V$ by sending $\sig \To \sig \circ \io$. Given sections $\sig_1, \sig_2 \in \OX(U)$, their  addition $\sig_1 + \sig_2$  is the section that sends ~ $x$  to $\sig_1(x)+  \sig_2(x)$ in $\sAx$ for every $x \in \fcU$. Similarly,  their multiplication $\sig_1 \cdot \sig_2$ sends $x$  to the product  $\sig_1(x) \sig_2(x)$ in $\sAx$ for every $x \in \fcU$.  Associativity and distributivity of these operations follow  from the point-wise operations of~  $\sAx$.
The \tprs\  (resp. tangible, ghost) elements of $\OX(U)$ are the sections with
$\sig(x) \in \sAxptngs$
 (resp. $\sig(x) \in \sAx|_\tng$,  $\sig(x) \in \sAx|_\ghs$) for all  $x \in \fcU$. Accordingly, $\OX(U)$ is endowed with a structure of a (\tame) \nusmr.
Note that for any $f \notin \tN(\sA)$ the local sections \eqref{eq:fc.sections} on $U = \DDf$ are subject to elements within  its (nonempty) \fzone\ $\fcU = \fcDDf$. Therefore, having Lemmas \ref{lem:fuclZintersection} and  ~ \ref{lem:fc.map}, these sections coincide  with the functor \eqref{eq:OX} given by $\OX: U \mTo \sA_\STf$, since $\STf \subseteq \iTcl(\xpCong)$ for each~ $x \in \fcU$.

\pSkip

To customize  the stalk $\scF_x$  at a point $x \in X$ (cf. \eqref{eq:stalk}) to our setup,  we  introduce the   family
\begin{equation}\label{eq:DDx}
\begin{array}{lll}
 \MDx & := \big\{ \DDh \cnd   h \in \sAptngs, \  \fcDDh \ni x \big \} \\[1mm]
 & \ =  \big\{ \DDh \cnd h \in \sAptngs, \  \DDh \ni  x \  \text{ with } \ \STh \subseteq \iTcl(\xpCong) \big \} .  \end{array}
\end{equation}
of open sets of $X$, determining by \tprs\ elements $h \in \sAptngs $, that contain~ $x$.
 In particular,  $X = \DD(\one) \in \MDx $ for every $x \in X$, and  thus
$\MDx$ is nonempty.  An open neighborhood $U$ of $x$  which belongs to ~$\MDx$ is denoted by $U_x$.

\begin{rem}\label{rem:DxInredecation}  $\MDx$ is closed for intersection, as follows from Lemma \ref{lem:fuclZintersection}. Namely, $$\DDf, \DDg \in  \MDx \dss\Rightarrow \DDf \cap \DDg \in \MDx,$$ where $\DDf \cap  \DDg = \DDfg$ by Corollary \ref{cor:1.2}.(i).
\end{rem}

For \tprs\ elements $f \in  \sAptngs$ we have the correspondence
\begin{equation}\label{eq:tpsr.cor}
  x \in \fcDDf \dss\Iff \STf \subseteq \iTcl(\xpCong) \dss\Iff  \DDf \in \MDx,
\end{equation}
cf. \eqref{eq:focal.zone} and \eqref{eq:DDx}, respectively.

\begin{defn}\label{def:st.stalk}

The  \textbf{\nustalk} $\OXx$ of $\OX$ at a point $x \in X$  is defined to be the inductive limit
\begin{equation}\label{eq:nu.stalk}
\OXx := \ilim{U_x \in \MDx} \OX(U_x)
\end{equation}
of  sections $\scn \in \OX(U_x)$ over  $\MDx$.
\end{defn}
For each open neighborhood $U_x \in \MDx$ of $x$ we have the well-defined canonical map
 $$\OX(U_x) \TO \OXx, \qquad \scn \longmapsto  \scn_{x},$$
sending  $\scn \in \OX(U_x)$ to the germ $\scn_x$  in the \nustalk\ $\OXx$.

\begin{thm}\label{prop:stalks}
Let $\sA$ be a (\tame) \nusmr\ and let $X = \SpecA$ be its spectrum.
\begin{enumerate}\eroman
  \item For any $x\in  X$  the \nustalk \ $\OXx$ of the \nushf\ $\OX$ is isomorphic to the local \nusmr\  $\sAx$.
  \item  The  \nusmr\ $\OX (\DDf )$, with \strict\ $f \in \sAptngs$,  is isomorphic to the localized \nusmr\ $\sA_\STf$.

   \item   In particular, $\OX (X) \cong \sA$.
\end{enumerate}
\end{thm}
\begin{proof}
  (i):
  The map sending  a local section $\scn$ in a neighborhood $U_x \in \MDx$ of $x$ to $\scn_{x} \in \sAx$
provides  a well-defined  \qhom \ of \nusmr s
$$\vrp : \OXx \TO \sAx, \qquad (U_x, \scn) \mTo  \scn_{x} \ (\in \sAx),$$
 which we claim is a bijection. 

\pSkip \underline{Surjectivity of $\vrp$}:
Each element of  $\sAx$ has the form $\frac fg$
with $f ,g \in \sA$ where   $g \in  \iTcl(\xpCong)$ is \tprs,  thus  $\STg \subseteq \iTcl(\xpCong)$ (cf. Remark \ref{rem:f.vs.Sf}) and $\DDg \in \MDx $.   Hence, the fraction
$\frac fg$ is well-defined on the \fzone \ $\fcDDg \subseteq \DDg$,  and  $\big(\DD(g), \frac fg \big)$
defines an element in $\OXx$ (cf. Lemma \ref{lem:fuclZintersection})  that is mapped by  $\vrp$ to the
given element.

\pSkip \underline{Injectivity of $\vrp$}:
Let $\scn_1,\scn_2 \in \OX (U_x)$ for some neighborhood $U_x \in \MDx$ of $x$, and assume that $\scn_1$ and $\scn_2$ have the same value at $x$, namely
$(\scn_1)_x = (\scn_2)_x$. We will show that $\scn_1$ and $\scn_1$ coincide over $\fcV_x$ in a neighborhood $V_x \in \mfD _x$ of $x$, so that they define the same element in $\OXx$.
Shrinking $U_x$ if necessary, we may assume that
$\scn_i =  \frac{f_i}{g_i}$ on $\fcU_x$,  for $i = 1, 2$, where $f_i,g_i \in \sA$ with $g_i \in \iTcl(\xpCong)$. As $\scn_1$ and~ $\scn_1 $ have the same image in $\sAx$, it follows that $hf_1g_2 = hf_2g_1$ in $\sA$ for some $h \in \iTcl(\xpCong)$. Therefore, we also have
$\frac{f_1}{g_1} = \frac{f_2}{g_2}$
in every local \nusmr \ $\sA_y$ such that $g_1,g_2,h \in \iTcl(\ypCong)$. But, 
 the set of such $y$'s (cf. \eqref{eq:x.focal.zone}) is
the set  $\fcDD(g_1) \cap \fcDD(g_2) \cap  \fcDD(h)$, laying in
the open set
$\DD(g_1) \cap \DD(g_2) \cap  \DD(h)$, and belongs to  ~$\MDx$, since $g_1, g_2, h$ are \tprs s (Remark \ref{rem:DxInredecation}). Hence, $\scn_1 = \scn_2$ on ~$\fcV_x$ for some neighborhood $V_x \in \MDx$ of $x$, and thus on the entire set $V_x$, as required. Therefore, $\vrp$ is injective.

\pSkip

(ii): Recall from Lemma \ref{lem:fg.1} that, for a  \tprs\ element $f \in \sAptngs$,  the \nusmr\ $\sA_\STf$ is isomorphic to~ $\sA_f$, where $f\in \STf$.
%
 We define the   \qhom \ of \nusmr s
$$\psi  : \sA_\STf \TO \OX (\DDf ), \qquad   \ \frac{g}{h^\ell}  \mTo  \frac{g}{h^\ell},
\quad h \in \STf, $$
given by sending $\frac{g} {h^\ell}$
%
to the section of $\OX (
\DDf )$ that assigns to any $x \in \fcDDf$ the image of $\frac{g} {h^\ell}$ in $\sAx$.
(The elements $\frac{g} {h^\ell}$ are well defined in $\sAx$, since $\STf \subseteq \iTcl(\xpCong)$ for $x \in \fcDDf$.)

\pSkip \underline{Injectivity of $\psi$}:
Assume that $\psi \big(\frac{g_1} {h_1^{\ell_1}}\big) = \psi \big(\frac{g_2} {h_2^{\ell_2}}\big) $, with $h_1,h_2
\in \STf$.   Then, for every $x \in \fcDDf$ there is a \tprs\ element
$h_x \in \iTcl(\xpCong)$  such that $h_x g_1 h_2^{\ell_2} = h_x g_2 h_1^{\ell_1}$.
Namely  $h_x$ is contained in the equaliser $E := \Eql(g_1 h_2^{\ell_2} ,  g_2 h_1^{\ell_1})$
 of $g_1 h_1^{\ell_1}$ and $ g_2 h_2^{\ell_2}$  (Definition~ \ref{def:equaliser}) -- a \nusmr\ ideal. Let~ $\gCong_E$ be its  ghostifying congruence~\eqref{eq:G.E}.
 Then, for any $x \in \fcDD(f)$, we have  $E \nsubseteq  \iGcl(\xpCong)$, whereas $h_x \in E$ with     $h_x
 \in \iTcl(\xpCong)$.
  As this  holds for any $x \in \DDf$, where $\DDf=\fcDD(f)$ since  $f$ is \strict,    we have $ \VV(\gCong_E) \cap \DDf = \emptyset$,
or in other words $\VV(E) = \VV(\gCong_E)
 \subseteq  \VV( f )$.  Then, $\srad(f) \subseteq \srad(E)$ by Lemma~  \ref{lem:D.contain}.(viii),
implying  that $f^m  \in E$ for some~ $m$ (Corollary \ref{cor:rad.incl.2}.(iii)). Therefore,
$f^m g_1 h_2^ {\ell_2} = f^m g_2 h_1^{\ell_1}$ in $\sA$, and   hence $\frac {g_1}{h_1^{\ell_1}} = \frac {g_2}{h_2 ^{\ell_2}}$ in~ $\sA_\STf$.

\pSkip \underline{Surjectivity of $\psi$}:
Let $\scn \in \OX (U )$, where  $U = \DDf$. By definition, $U$ can be  covered by open sets $U_i$  on
which $\scn$ is represented as  a quotient $\frac{g_i}{f_i}$, with $f_i  \in  \iTcl(\xpCong)$ for all $x \in  \fcU_i$, i.e., $U_i \subseteq \DD(f_i)$. To emphasize, $f_i \in \sAptngs$, namely it is \tprs.   As the open
sets of type  $\DD(h_i)$ form a base for the topology of $X$, we may assume that $U_i = \DD(h_i)$ for some~ $h_i$.

We may also assume that $f_i = h_i$, and therefore $h_i \in \sAptngs$ is \tprs\ as well.  Indeed, since  $\DD(h_i) \subseteq  \DD(f_i)$,  by taking
complements we obtain $\VV( f_i) \subseteq  \VV(h_i)$,  and thus $\srad(h_i) \subseteq \srad(f_i)$ by Lemma~  \ref{lem:D.contain}.(viii). Therefore $h_i \in \iGcl(\srad(f_i))$  and, since $\sA$ is \tame,   $h^\ell_i =c f_i$ for some $c \in \sA$ by
Corollary \ref{cor:rad.incl.2}.(ii). Thus $\frac{g_i}{f_i} = \frac{cg_i}{h_i^\ell}$. Replacing~ $h_i$ by $h^\ell_i$
(since  $\DD(h_i) = \DD(h^\ell_i)$ by  Lemma \ref{lem:D.contain}.(vi)) and $g_i$ by~ $c g_i$, we can assume that
$\DD(f)$ is covered by open sets of type  $\DD(h_i)$, and that $\scn$ is represented by $\frac{g_i}{h_i}$ on $\fcDD(h_i) \subset \DD(h_i)$.

More precisely, $\DDf$ can  be covered by finitely many such sets $\DD(h_i)$. Indeed, $\DDf \subseteq \bigcap_i \DD(h_i)$ iff $\VV(f) \supseteq \bigcap_i \VV(h_i) = \VV(\sum_i \gen{h_i})$, where $\gen{h_i}$ is the ideal generated by $h_i$, cf. Proposition \ref{prop:var1}.(iv). By Lemma~  \ref{lem:D.contain}.(viii) and Corollary \ref{cor:rad.incl.2}.(iii) this is equivalent
to  $f^n  \in \sum_i \gen{ h_i}$ for some~ $n$, which  means that~ $f$ can be written as a finite sum $f ^n  = \sum_i b_ih_i$ with  $b_i \in \sA$. Hence, we may  assume that only finitely many~ $h_i$ are involved.

Over the intersection $\DD(h_i) \cap \DD(h_j) = \DD(h_i h_j)$ we have two elements $\frac{g_i} {h_i}$
and $\frac{g_j} {h_j}$
representing $\scn$, subject to $\fcDD(h_i) \cap \fcDD(h_j) = \fcDD(h_i h_j)$ (Lemma \ref{lem:fuclZintersection}), which contains the \fzone\ $\fcDD(h_i h_j)$ of $\DD(h_i h_j)$.
Then, by the injectivity
proven above, it follows that $\frac{g_i} {h_i} = \frac{g_j} {h_j}$
in $\sA_{\ST(h_ih_j)}$. But the product $h_i h_j$ is \tprs,  i.e., $h_1 h_2 \sAptngs$, implying that   $\sA_{\ST(h_ih_j)} $ is isomorphic to $\sA_{h_ih_j}$  by Lemma \ref{lem:fg.1}, and  hence $(h_ih_j)^m g_ih_j = (h_ih_j)^m g_jh_i$ for some~ $m$.
As we have only finitely many $h_i$, we may choose one $m$ which holds for all $i, j$. Replacing ~$g_i$
by ~$g_ih^m_i$ and~ $h_i$ by $h_i^{m+1}$ for all $i$, we remain with $\scn$   represented by $\frac{g_i}{h_i}$
on~ $\DD(h_i)$, subject to $\fcDD(h_i)$,  and furthermore
$g_i h_j = g_jh_i$ for any $i, j$.

Write $f^n =\sum_i b_ih_i$ as above, which is possible since the $\DD(h_i)$'s cover $\DDf$, and let $g= \sum_i b_i g_i$, with $b_i \in \sA$.
Then, for every $h_j$ we have
$$ gh_j = \sum_i b_ig_ih_j = \sum_i
b_ih_ig _j = f^n g_j,$$
thus  $\frac fg = \frac{h_j}{g_j}$
on $\fcDD(h_j) \subseteq \DD(h_j)$. Hence, $\scn$ is represented on $\DDf$, subject to $\fcDDf$, by  $\frac{g}{f ^n} \in \sA_{\STf}$,  and therefore $\psi:\sA_\STf \To \OX (\DDf )$ is surjective.

\pSkip
(iii):   Every unit $f \in \sAX$,  in particular $\one$, is \strict, and thus  $\sA = \sA_f \cong \OX (X) = \OX(\DDf)$ by part~ (ii).
%
\end{proof}

\subsection{Locally \nusmrsp s} \sSkip

Let $A$ and $B$  be \nusmr s, having the  spectra $X = \SpecA$ and  $Y = \SpecB$, respectively.
Recall that a \qhom\ $\vrp: A \To B$ induces the pull-back  map $\avrp: Y \To X$  of congruences,   given by sending  $\pCong_b \in \SpecB$  to  $\vvrp(\pCong_b)$ in $\SpecA$, cf. \eqref{eq:specMap} and Remark \ref{rem:cong1}.(iii).

\begin{defn}
A \textbf{\nusmrsp}  $\XOX$ is  a topological space $X$ together  with
 a  \textbf{structure \nushf}  $\OX$ of (commutative) \nusmr s. 
 A  \textbf{morphism} of \nusmrsp s is a pair of maps  $$\pff: \XOX \TO \YOY,$$ where $\ff: X \To Y$ is a continuous map of topological spaces and $\sff: \OY  \To \dff(\OX)$
is a  morphism of \nushv\ on $Y$ (Definition \ref{def:sheaf}).
\end{defn}
The \nushf\ $\dff( \OX)$ on $Y$ is defined by $V \longmapsto \OX(\iff(V))$ for open subsets $V \subset Y$, and $\sff$ is a system of \nusmr\ \qhom s
\begin{equation}\label{eq:pull.back}
 \sff(V): \OY(V) \TO \OX(\iff(V)), \qquad V \subset Y \text{ open},
\end{equation}
which agree with restriction morphisms. In addition, by  Lemma \ref{lem:fc.map}, $\iff(V)$ respects  \fzone s, that is $\fcU = \iff(\fcV)$ for $U = \iff(V)$.

The map $\ff: X \To Y$ induces for any $x \in X$ a map
\begin{equation}\label{eq:DxMap}
  \tlff: \MDx \TO \MD_{\ff(x)}, \qquad U_x \mTo V_{\ff(x)},
\end{equation}
of sets of the form \eqref{eq:DDx}.  This is seen by \eqref{eq:tpsr.cor}, as
a membership of $\DDf$ in $\MDx$ corresponds to containment of $x$ in  \fzone s $\fcDDf$, while $\ff$ respects \fzone s  by Lemma \ref{lem:fc.map}.

Compositions of morphisms of  \nusmrsp s are  defined
in the  natural way.  A \nusmrsp\  $\XOX$ restricts to an open subset  $U \subset X$, yielding the \nusmrsp\ $(U,\OX|_U)$.
The injection $(U,\OX|_U) \Into \XOX$ is then  a morphism of \nusmrsp s, called \textbf{open immersion} of \nusmrsp s.

\begin{rem}\label{rem:6.6.7}
 A morphism $\pff: \XOX \To \YOY$ of \nusmr ed spaces
canonically induces for each   $x\in X$ a \nusmr\ \qhom \
$\sffx : \OYfx \To \OXx$.
 Indeed, for open subsets $V_y \subset  Y$ with $V_y \in \MD_{\ff(x)}$, i.e., $y = \ff(x)$, the
compositions
$$\OY (V_y ) \Right{5}{\sff(V_y)} \OX(\iff(V_y )) \TO \OXx $$
 agree with the restriction morphisms of $\OY$, respecting \fzone s as well,  and therefore  induce a \qhom\
$\sffx : \OYfx \To \OXx$.
\end{rem}

Recall that $\tNSpec(\sA)$ denotes the set of all \tminimalc  s (Definition \ref{def:nt.maximalSpec}) on a \nusmr \  ~$\sA$, and that  $\sA$ is local,  if
 all \tminimalc s $\nCong \in \tNSpec(\sA) $ have the same tangible projection~
$\iTcl(\nCong) $ (Definition \ref{def:local.ring}). (For example, any \nusmf\ is local.) In particular, by  Corollary \ref{cor:RmodPrime}, the localization~ $\sA_\pCong$ of ~ $\sA$ by a \gprimec\  $\pCong$ is a local  \nusmr\ with \ctminimalc\ ~$\nCong_\pCong$, cf.~ \eqref{eq:max.prime}.

As $\OXx$ is isomorphic to the local \nusmr\ $\sAx := \sA_{\pCong_x}$  (Theorem  \ref{prop:stalks}), to allocate   the  \ctminimalc\ of $\OXx$  we identify $\OXx$    with its isomorphic  image, and take the \lcong\ ~$\nCong_{\pCong_x}$ on $\sAx$, which we denote by~ $\nCong_x$.
 A ~\qhom\ $\vrp: \sA \To  \sB$ of local \nusmr s   is
called \textbf{\lqhom}, if
for any
 $\nCong_\sB \in \tNSpec(\sB)$ there exists   $\nCong_\sA \in \tNSpec(\sA)$ such that $\vvrp(\nCong_\sB) = \nCong_\sA$. Namely, $\avrp: \SpecB \To \SpecA$ maps \tminimalc s to \tminimalc s.
 For example, given a \qhom\ $\vrp: \sA \To  \sB$
and a \gprimec\ $\pCong_b$ on  $\sB$, the congruence  $\pCong_a = \vvrp(\pCong_b)$ is a \gprimec\ on $\sA$ (Remark \ref{rem:biject.proj}), then  it
follows that the induced \qhom\ $\sA_{\pCong_a} \To  \sB_{\pCong_b}$ is local.

\begin{defn}
A  \nusmr ed space $(X, \OX)$ is a \textbf{\lnusmrsp}, if all its \nustalk s ~$\scO_{X,x}$ are local \nusmr s.
A \textbf{morphism} $\pff: \XOX \To \YOY$ of locally \nusmr ed spaces is a morphism of \nusmrsp s such that  for all  $x \in X$ the morphisms $$\sff_
x : (\iff \OY)_x = \OYfx  \TO \OXx$$ of \nustalk s are  local \qhom s.
\end{defn}

We write $\Hom(\sB,\sA)$ for the set of all \nusmrhom s $\sB \To \sA$, and $\Hom(\XOX, \YOY)$   for the set of all morphisms of locally \nusmrsp s $\XOX \To \YOY$.

\begin{thm} \label{prop:6.6.9}
Let $\sA$ and $\sB$ be \nusmr s, with spectra   $X = \SpecA$, $Y = \SpecB$  and  structure \nushv\ $\OX$, $\OY$, respectively.
\begin{enumerate} \eroman

  \item
 $\XOX$ is a locally \nusmr ed space with \nustalk\ $\OXx \cong \sA_x$ at every  $x \in X$.

\item  The canonical map
$$\Upsilon : \Hom(\XOX, \YOY) \TO  \Hom(\sB,\sA), \qquad  \pff \longmapsto \sff(Y),$$
 is a bijection.

 \item  For any morphism  $\pff: \XOX \To  \YOY $
as in (i), and a point $x \in X$, the associated map of \nustalk s
$$\sffx: \scO_{Y,\phi(x)} \TO \OXx$$
coincides canonically with the map $\sB_{\ff(x)} \To \sA_x$ obtained  from
$\sff(Y) : \sB \To \sA$ by localization.
\end{enumerate}

\end{thm}
\begin{proof}

(i): Follows form Theorem \ref{prop:stalks}.(i).

\pSkip
(ii):
 Let $\vrp : \sB \To \sA$ be a
\qhom\  of \nusmr s, and let
\begin{equation*}
\ff = \avrp : X  \TO Y , \qquad \pCong \longmapsto \vvrp(\pCong), 
\end{equation*}
be
 the induced  map of spectra,  given by $\ff(x) = \vvrp(x)$.
  First, $\phi: X \To Y$ is continuous by Corollary ~\ref{cor:2.5}. Given $x \in X$, we can localize $\vrp$ to obtain   the \lqhom\ $\vrp_x : \sB_{\vvrp(x)} \To \sA_x $ of local \nusmr s.
   Then, for any open set
   $V  \subset Y$ we obtain the \qhom\ of \nusmr s
$$\sff(V) : \OY (V)= \coprod_{\iff(x) \in \fcV} \sB_{\iff(x)} \TO  \coprod_{x \in \fcU } \sA_x =  \OX(\iphi(V)), \qquad U= \iphi(V),$$
which gives a morphism of \nushv\ $ \sff : \OY \To \dff(\OX)$. The restriction of $\sff$ to   \fzone s   $\fcU$ and $\fcV$ coincides with $\ff$,  since
$\fcU = \iff(\fcV)$ by Lemma  \ref{lem:fc.map}.
By compatibility of sections with the continuity  of~ $\phi$, sending  $x \in X$ to $\phi(x) \in Y$, Theorem \ref{prop:stalks} shows that the morphism of \nustalk s $\sffx : \scO_{Y,\phi(x)} \To \OXx$
coincides with the canonical map $\sB_{\ff(x)} = \sB_{\vvrp(x)}  \To \sA_x$, which is local. Hence, $\pff$ is
a morphism of locally \nusmrsp s.

Conversely,
let $(\ff, \sff): \OXX \To \OYY$ be a morphism of \lnusmrsp s.
On global sections, by Theorem \ref{prop:stalks}.(iii),  the map $\sff$ induces a \qhom\ of \nusmr s
$$ \vrp: \sB \cong \OY(Y) \TO \dff \OX(X) = \OX(X) \cong \sA. $$
Furthermore, for  any $x \in X$ there is the induced map of \nustalk s  $\sffx:  \OYfx \To  \OXx$, which must coincide with $\vrp$ on global sections, and thus rendering the diagram
\begin{equation*}
 \begin{gathered}
\xymatrix{
\sB \ar@{->}[d] \ar@{->}[rr]^{\vrp = \Upsilon((\ff,\sff)) = \sff(Y)} && \sA \ar@{->}[d] \\
\sB_{\ff(x)} \cong \OYfx \ar@{->}[rr]^{\sffx} && \sA_x \cong \OXx \\
} 
\end{gathered}
\end{equation*}
commutative. (The \hom \ $\sffx$ is local, by assumption.)  Since $\ff(x) = (\sffx)^\inv(x) $, it follows that~ $\sffx$ is the localization of $\vrp$, which shows that $\ff$ coincides with $\avrp$. Hence, $\sff$ is induced from $\vrp$, and the morphism $(\ff,\sff)$ of \lnusmrsp s  derives from  $\vrp$.

\pSkip (iii): As  $\pff$ is
a morphism of locally \nusmrsp s, by part (i) it  coincides with the local map $\sB_{\ff(x)} \To \sA_x$.
\end{proof}

\subsection{Local \nusmr s}\sSkip

The \nustalk\ $\OXx$ (Definition \ref{def:st.stalk}) of a \lnusmrsp \ $\XOX$ is called the \textbf{local \nusmr}\ of $\XOX$ at  $x$, cf. Definition ~\ref{def:local.ring}.  Its \ctminimalc\ ~ \eqref{eq:max.prime} is obtained from  ~$\sA_x$ as
$$ \xnCong := \iTcl(\xpCong) \xpCong,$$
 and $\rsmf(x) := \OXx/\xnCong$ is the corresponding residue
\nusmf, cf. Corollary \ref{cor:RmodPrime}. 
The tangible (resp. ghost) cluster of the   \ctminimalc\  $\xnCong$ on the \nusmr\  $\OXx$ is the set of all sections that
possess tangible (resp. ghost) values at the point $x \in X$.

If $\Ux \in \MDx$ is an open neighborhood of $x \in X$, cf. \eqref{eq:DDx}, and $f \in \OX(\Ux)$,  we write  $f(x) \in \rsmf(x)$ for  the
image of $f$ under the composition of the canonical \qhom s $\OX(\Ux) \To \OXx \To \rsmf(x)$.

Writing $X$ for a \lnusmrsp\ $\XOX$, the following definitions of
geometric notions  are now directly accessible in terms of local \nusmr s.

\begin{enumerate}\ealph
  \item  The \textbf{local dimension} $\dim_x(X)$ of $X$ at a point $x \in X$ is defined to be the Krull dimension (Definition \ref{def:krull.dim}) of the local \nusmr\  $\OXx$.
The \textbf{dimension} $\dim(X)$ of the whole $X$ is the supremum of
local dimensions $\dim_x(X)$ over all $x \in X$.

  \item  The \textbf{Zariski cotangent space} to $X$ at a point $x$ is defined as $\xnCong/\xnCong^2$, realized as
a \numod\ over the residue \nusmf\ $\rsmf(x) = \OXx / \xnCong$ (cf. Definition \ref{def:cong.fractor} and \eqref{eq:cong.chain.2}), whose dual is  called the \textbf{Zariski tangent space} to  $X$ at $x$.

  \item $X$ is called  \textbf{nonsingular}  at a point  $x \in X$, if the Zariski tangent
space to $X$ at $x$ has dimension equal to $\dim_x(X)$; otherwise,  $X$ is said to be \textbf{singular}
at $x$ (i.e., the dimension of the Zariski tangent space is larger).

\end{enumerate}
The study of these notions requires a further development of dimension theory, this is left for future work.

\section{\nusch s}\label{sec:7}

\subsection{Affine \nusch s}\label{ssec:aff.schemes}
\sSkip

Having the structure of \lnusmrsp s settled, we employ  these objects as
prototypes for the so-called schemes, introduced by A. Grothendieck.

\begin{defn}\label{def:nuscheme}
An \textbf{affine \nusch}\ is a \lnusmrsp\ $\XOX$   which  is isomorphic  to a  \lnusmrsp\ over a \nusmr, i.e.,  $\XOX \cong (\SpecA,\OSpecA)$
for some \nusmr\ ~$\sA$.
\pSkip
A \textbf{\nusch}\ is a \lnusmrsp\ $\XOX$ that has an open
covering by affine \nusch s  $(U_i,\OX \vert U_i)_{i \in I}$.
A \textbf{morphism of \nusch s} is a (local) morphism of \lnusmrsp s.\end{defn}

Often, for short,  we  write $X$ for the \nusch\ $\XOX$, and $\ff : X \To  Y$ for  a morphism  $(\ff, \sff) : \XOX \To  \YOY$ of \nusch s.
 The sheaf $\OX$ is called the \textbf{structure \nushf} of $X$.
 $\OX(U)$ is called  the \textbf{\nusmr\ of
sections} of $\OX$ over $U$, with $U \subseteq X$ open, and is sometimes denoted $\Gm(U, \OX)$. The notation $\Gm(X)$ stands for $\Gm(X,\OX)$, where $\Gm(\sA)$ denotes $\Gm(X)$ with $X = \SpecA$.

\begin{example}[Affine spaces]\label{exmp:affine.spc}
  Let $\sA = \tlR[\lm_1, \dots, \lm_n]$ be the \nusmr \ of polynomial functions over a (\tame) \nusmr\ $R$, cf. \S\ref{ssec:ploynomials}.  Then,
  $\Aff_R^n := \SpecA$ is the affine space of relative dimension ~$n$ over $R$.
\end{example}

The restriction of a morphism of \nusch s $\ff : X \To Y$ to a subset
 $U \subset X$ gives the morphism  $\ff|_U : U \To Y $ of \nusch s, obtained by composing the open immersion $(U,\OX|_U) \Into \XOX$ with the morphism $\ff : \XOX \To \YOY$. Then, an open subset $V \subset  Y$ determines the  open subset $\iff(V ) \subset  X$, together with a unique morphism of \nusch s $\ff' : \iff(V ) \To V$ that renders the diagram
$$\xymatrix{
X \ar[r]^-{\ff } \ & Y  \\
\iff(V) \ar@{^{(}->}[u] \ar[r]^{\ff'}  & V \ar@{^{(}->}[u]
}$$
commutative.
All together,  \nusch s and their morphisms (Definition \ref{def:nuscheme}) establish the category $\Sch$,
containing the full subcategory $\ASch$ of affine \nusch s  (the morphisms between affine \nusch s are the same in $\ASch$ and in $\Sch$).

A morphism $\phi : X \To Y$ of \nusch s  is  injective (resp. surjective, open, closed. homeomorphism),  if the continuous map $X \To Y$ of the underlying topological
spaces has this property.

Recall from \eqref{eq:specMap.2} that for a \qhom\ $\vrp: \sA \To \sB$ we have the induced map
\begin{equation}\label{eq:specMap.3} \avrp: \SpecB  \TO
\SpecA, \qquad \pCong' \longmapsto \vvrp(\pCong'). \end{equation}
For a second \qhom \ $\psi: \sB \To \sC$, we then have $^a(\psi \circ \vrp) =
\avrp \circ \apsi . $ With this view, $\Spec$ can be viewed as a contravariant functor
$$\Spec: \NSMR \TO \ASch $$
from the category of \nusmr s to the category of affine \nusch s,
assigning  to a \nusmr\  $\sA$ the corresponding affine \nusch\
$(\SpecA,\OSpecA)$ and to a  \qhom\ $B \To A$ the corresponding
morphism $(\SpecA,\OSpecA) \To (\SpecB,\OSpecB)$ as characterized in Theorem \ref{prop:6.6.9}.(i).

On the other hand, if $\ff : X \To  Y$ is a morphism of \nusmrsp s, using the notation $\Gm(U, \OX) = \OX(U)$,  we obtain  a  \qhom\ of \nusmr s
\begin{equation}\label{eq:2.11.1}
\Gm(\ff) := \sff_Y: \Gm \YOY  = \OY (Y ) \TO  \Gm \XOX = (\dff \OX)(Y ) = \OX(X).
\end{equation}
In this way, $\Gm$ sets up a contravariant functor
$$\begin{array}{rll}
X & \longmapsto &  \OX(X), \\[1mm]
X \To Y & \longmapsto &  \sff(Y ) : \OY(Y ) \To \OX(X),
\end{array} $$
from the category of \nusmrsp s to the category $\NSMR$ of \nusmr s,  
which restricts  to  a contravariant
functor
$$\Gm: \ASch \TO \NSMR.$$
on the category of affine \nusch s.

Recall that by Theorem \ref{prop:6.6.9} morphisms of affine
\nusch s correspond bijectively to \qhom s of \nusmr s, and thus  we
can state the following.

\begin{prop}\label{prop:6.6.12} The category of affine \nusch s  $\ASch$ is equivalent to the opposite
of the category $\NSMR$ of \nusmr s.
\end{prop}
%
\begin{proof}
  The functor $\Spec:\NSMR \To \ASch$ is  surjective by definition, and  $\Gm \circ \Spec$ is clearly
isomorphic to $\id$ on \nusmr s. It suffices to show that for any two \nusmr s $\sA$ and $\sB$ the
maps
$$
\xymatrix{  \Hom(A,B) \ar@<0.5ex>[rr]^{\Spec \qquad}
& & \Hom(\SpecB, \SpecA) \ar@<0.5ex>[ll]^{\Gm \qquad} }$$
are mutually inverse bijections. But, $\Gm \circ  \Spec =
\id$ by \eqref{eq:2.11.1}, while  it follows from   Theorem \ref{prop:6.6.9} that $\Spec \circ  \Gm =
\id$.
\end{proof}

\subsection{Open \nussch s}
\sSkip

The restriction  $ (U,\OX|_U)$ of  a \nusch \ $\XOX$ to an open set   $U \subset X$, called  \textbf{open \nussch}\ of $X $, is by itself is a \nusch  \ by  Lemma \ref{lem:6.6.1}.(iii).
If $U$ is an affine \nusch,
then $U$ is an \textbf{affine open \nussch}.

\begin{prop}\label{prop:u.3.2}
Let $X$ be a \nusch, and let $U \subset X$ be an open set.
\begin{enumerate}\eroman
  \item    The \lnusmrsp\
$(U,\OX|U)$ is a \nusch.  \item The open subsets that  give rise to affine open \nussch s are a basis of the topology.
\end{enumerate}
\end{prop}
\begin{proof} By definition, the \lnusmrsp\ $X$ can be covered by affine \nusch s.  By Proposition ~ \ref{prop:6.1.11} each of these affine \nusch s has a basis of its topology which consists of
affine \nusch s. This yields both parts of the proposition.
\end{proof}

Let $U \subseteq  X$ be an open subset, considered  as an open
\nusch\ of~ $X$, with  the inclusion $\io : U \Into X$. For  $V \subseteq  X$ open, the restriction map of the structure \nushf\ $\OX$ gives
a \qhom\ of \nusmr s
$$\Gm(V,\OX) \TO \Gm(V \cap U, \OX) = \Gm(\io^\inv( V ),\OXU) = \Gm(V, \io_*\OXU).$$
These maps determine a  morphism $\io^\#
 : \OX \To \io_* \OXU$ of \nushv\ of \nusmr s
and, hence,  by the inclusion $U \subseteq X$,  a morphism $U \To X$ of \nusch s.
An \textbf{affine open covering} of a \nusch\ $X$ is an open covering $X = \bigcup_i U_i$ by affine open \nussch s $U_i $ of $X$.

\begin{lem}
Let $X$ be a \nusch, and let $U, V$ be affine open \nussch s of $X$.
 For every  $x \in U \cap V$ there exists an open \nussch\ $W \subseteq  U \cap V$ containing $x$, such that $W$ is a principal open,  cf. \eqref{eq:Var.1}, both in $U$  and  in $V$.
\end{lem}
\begin{proof}
Replacing $V$ by a principal open set of $V$ containing $x$, we may
assume that $V \subseteq  U$.  Choose \tprs \ $f \in \Gm(U,\OX)$ such that $x \in \fcDDf$, where $\DDf  \subseteq V$, i.e. an element $f \in \iTcl(\xpCong)$.
Let $f|_V$
be  the restriction of the image of $f$ under  the  \qhom\ $\Gm(U,\OX) \To \Gm(V,\OX)$. Then,
$\DD_U(f) = \DD_V (f|_V )$,  which  also implies that $\Gm(U,\OX)_f
=
\Gm(V,\OX)_{f|_V}$, by the  sheaf axioms (Definition ~\ref{def:sheaf}).
\end{proof}

\subsection{Gluing \nusch s }
\sSkip

With the notion of morphisms of \nusch s at our disposal (Definition \ref{def:nuscheme}), the gluing procedure of \nusch  s becomes applicable  in the usual  way. That is, we start with a given collection of \nusch s $(X_i)_{i \in I}$, and an open set $U_{ij} \subseteq X_i$ for each $i \neq j$, together with a family of
isomorphisms of \nusch s
$$\psi_{ij} : U_{ij} \ISOTO U_{ji}$$
such that  for all $i,j,k \in I$
\begin{enumerate} \ealph
  \item  $\psi_{ji} = \psi_{ij}^\inv$;
  \item $\psi_{ij}(U_{ij} \cap U_{ik}) = U_{ji} \cap U_{jk}$;
  \item $(\psi_{jk} \circ \psi_{ij})|_{U_{ij} \cap U_{ik}} = \psi_{ik}|_{X_{ij} \cap U_{ik}}$ (compatibility condition).
\end{enumerate}
With this data, we define the \nusch\ $X$ by gluing the $(X_i)_{i\in I}$
along the $\psi_{ij}$ in the obvious way, i.e., the unique \nusch\ $X$ covered by open \nussch s isomorphic to the $X_i$ whose identity maps on  $X_i \cap  X_j \subseteq X$ correspond to
the isomorphisms $\psi_{ij}$.
When the $X_i$ are affine \nusch s with structure \nushv\  $\scO_{X_i}$,
$\scO_{X_i}(X_i \cap X_j)$ is naturally identified with $\scO_{X_j}(X_i \cap X_j)$.
Accordingly, as any \nusch\ ~$X$ admits a covering by affine open
\nussch s $(X_i)_{ i \in I}$,  $X$  can be viewed as a gluing of the
affine \nusch s ~$X_i$ along their intersections $X_i \cap X_j$.

\begin{example}\label{exmp:proj.space}
Let $R$ be a \nusmr. The \textbf{projective space} $\Proj^n_
R$ over $R$ is obtained by gluing $n+1$ copies $U_i = \Aff_R^n$, $i = 0, \dots, n$,  of affine
space $\Aff^n_R$ (cf. Example \ref{exmp:affine.line}).
Each $U_i = \Spec(\sA_i)$ is the \gprime\ spectrum of a  \nusmr\ $\sA_i$  of polynomial functions   in $n$ indeterminates over $R$, cf. \S\ref{ssec:ploynomials}, written as
$$ \sA_i = \tlR \bigg[ \frac{\lm_0}{\lm_i}, \cdots, \frac{\widecheck{\lm}_i}{\lm_i}, \cdots, \frac{\lm_n}{\lm_i}\bigg],$$
where $\widecheck{\lm_i}$ means that $\lm_i$ is to be discarded. The \nusmr s $\sA_i$ are
viewed  as \nussmr s of the Laurent polynomials   $\tlR[\lm_0, \dots, \lm_n, \lm_0^\inv, \dots, \lm_n^\inv]$.

We define a gluing datum with index set $\{0, \dots , n\}$ as follows: for $0 \leq  i, j \leq n$,  let
$$U_{ij} = \left\{  \begin{array}{lll}
                      \DD_{U_i} (\frac{\lm_j}{\lm_i}) \subseteq  U_i  &  & \text{if } \  i \neq j,  \\
                      U_i & & \text{if } \  i = j.
                    \end{array} \right.$$  Furthermore, let $\vrp_{ii} = \id_{U_i}$, and for $i \neq j$ let
$$\vrp_{ji} : U_{ij} \TO U_{ji}$$
be the isomorphism defined by the equality (as \nussmr  s of $\tlR[\lm_0, \dots, \lm_n, \lm_0^\inv, \dots, \lm_n^\inv]$)
$$ \tlR \bigg[ \frac{\lm_0}{\lm_i}, \cdots, \frac{\widecheck{\lm}_i}{\lm_i}, \cdots \frac{\lm_n}{\lm_i}\bigg]_{\frac{\lm_i}{\lm_j}} \TO
\tlR \bigg[ \frac{\lm_0}{\lm_i}, \cdots, \frac{\widecheck{\lm}_j}{\lm_j}, \cdots \frac{\lm_n}{\lm_i}\bigg]_{\frac{\lm_j}{\lm_i}}$$
 of the affine \nusch s $U_{ij}$ and $U_{ji}$. Since the isomorphisms $\vrp_{ij}$ are defined by equalities, the cocycle condition holds trivially, and we obtain
a gluing datum. This gives a \nusch, called the \textbf{projective
space} $\Proj^n_R$ of relative dimension $n$ over $R$. The \nusch s
$U_i$ are considered as open \nussch s of $\Proj^n_R$.
\end{example}

A morphism from a glued \nusch\ $X$ to another \nusch\ $Y$ can be determined by a given collection of  morphisms $X_i \To Y$ that coincide on the overlaps in the obvious sense.
In this gluing view,   Theorem ~\ref{prop:6.6.9} generalizes  as follows.

\begin{prop}\label{prop:8.7}
Let $X = \SpecA$ be any \nusch\ and let $Y = \SpecB$ be an affine \nusch. There is a one-to-one correspondence between morphisms $X \To Y$ and \nusmr\ \qhom s
$\sB \cong \Gm(\sB) =  \OY (Y) \To \OX (X) = \Gm(\sA) \cong\sA$.
\end{prop}
\begin{proof}
 Let $\{U_i\}$ be an open affine cover of $X$, and let $\{U_{i, j,k}\}$ be an open affine cover
of $U_i \cap  U_j$. Giving a morphism $\ff : X  \To Y$ is the same as giving
morphisms $\ff_i :U_i  \To Y$ such that $\ff_i$ and $\ff_j$ agree on $U_i \cap U_j$, i.e., such that $\ff_i|_{U_{i, j,k}} = \ff_j|_{U_{i, j,k}}$
for all $i, j,k$. Since  the $U_i$ and $U_{i, j,k}$ are affine, by Theorem~ \ref{prop:6.6.9}, the morphisms $\ff_i$
and $\ff_i|_{U_{i, j,k}}$  correspond exactly to \qhom s of \nusmr s $\OY (Y) \To \O_{U_i} (U_i) = \OX (U_i)$ and $\OY (Y) \To \O_{U_{i, j,k}} (U_{i, j,k})=\OX (U_{i, j,k})$, respectively. Hence, a morphism $\ff : X \To Y$  is the same as a collection of \nusmr\ \qhom s $\sff_i : \OY (Y) \To \OX (U_i)$ such that the compositions
$\rho_{U_i,U_{i, j,k}} \circ  \sff_i: \OY (Y) \To \OX (U_{i, j,k})$ and $\rho_{U_j,U_{i, j,k}} \circ  \sff_j: \OY (Y) \To \OX (U_{i, j,k})$ agree for all
$i, j,k$. By the sheaf axiom for $\OX$, this is exactly the data of a  \qhom\
$\OY (Y)\To \OX (X)$ of \nusmr s.
\end{proof}

\subsection{Properties of \nusch s }
\sSkip

In this subsection, unless otherwise is specified, we assume that all \nusch s are built over   \tamesmr s (Definition \ref{def:nusemiring}).

\subsubsection{Generic points}
\sSkip

Let $X$ be a \nusch, and let $Z \subseteq X$ be a
subset.  A point $z \in Z$ is a \textbf{generic point} of $Z$, if the set~$\{z\}$ is dense in $Z$ (Definition
\ref{def:top.irrd}). Topologically, if $Z$
 admits a generic point, then $Z$ is irreducible. In the case of underlying topological spaces of \nusch s, the converse relation holds:

\begin{prop}
The map
$$X \TO \{Z \subseteq  X \cnd Z \text{ closed, irreducible }\}, \qquad x\To  \overline{\{x\}},$$
is a bijection, i.e., every irreducible closed subset contains a unique generic point.
\end{prop}
\begin{proof} The correspondence holds for affine \nusch s by Corollary \ref{cor:irreducible}, as $X = \SpecA$ where $\sA$ is \tame. Let $Z \subseteq  X$ be closed irreducible, and let $U \subseteq X$, $Z \cap  U \neq \emptyset$,  be an affine open subset. The  closure of $Z \cap  U$ in $X$ is $Z$, since  $Z$ is irreducible, and $Z \cap U$ is irreducible. Hence, the generic point in $Z \cap U$ is a generic point of $Z$.
A generic point $z \in Z$ is contained in every open subset of $X$ that meets~
$Z$, and thus in every $U$. The uniqueness of generic points is then derived from the affine uniqueness.
\end{proof}

 Any point $x$ in a \nusch \ $X$ as above has a generalization by a maximal point $\gpt$, that is, the generic point of an irreducible component of $X$ such that $x \in \overline{\{\gpt\}}$. On the other hand, specializations to closed points are more subtle,
 as a nonempty \nusch\ $X$ may have no closed points, even
if it is irreducible. But, when $X$ is affine, this cannot happen (as any \gprimec\ is contained in a \maximalc \ by Proposition \ref{prop:max.cong.1}), which implies that in a quasi-compact \nusch\ $X$ the  closure $\overline{\{x\}}$ of any $x \in X$ contains a closed point.

\begin{prop}
 Let $\ff : X \To Y$ be an open morphism of \nusch s, where $Y$ is
irreducible with generic point $\gpt$. Then, $X$ is irreducible if and only if the fiber $\iff(\gpt)$
is irreducible.
\end{prop}
\begin{proof}
 $\ff$ is open, and thus  $\overline{\iff(\gpt)} = \iff(\overline{\{\gpt\}}) = \iff(Y ) = X$. Then the  claim follows from the topological property that a subspace is irreducible iff its closure is irreducible.
\end{proof}
Recall that a \nudom\ is a \tcls\ \nusmr\
 that has no ghost divisors (Definition \ref{def:nudomain}).

\begin{remark}\label{rem:u.2.37}
  Let $\sA$ be a \nudom\ with spectrum $X = \SpecA$,  and let $Q(A)$ be its \nusmf\ of
fractions.  The trivial congruence $\diag(\sA)$ of $\sA$ is a \gprimec, corresponding to the point $q \in X$, where $X$ is the closure of $\{q \}$. Therefore, $q$ is contained in every
nonempty open set of $X$, i.e., $q$ is a generic point of $X$. The \lnusmr\ $\O_{X,q}$
is the localization of $\sA$ by $\diag(\sA)$, and thus, by Theorem~ \ref{prop:stalks}.(i),  $$\O_{X,q} \cong Q(\sA).$$
For all tangible multiplicative monoids $T \subseteq \MS $  of $\sA$ the canonical \qhom\
$T^{-1}\sA \To \iMS A$ is injective, and the localizations $T^{-1} \sA$ are considered as \nussmr s of $Q(\sA)$.

For every $f \notin \tN(\sA)$, we have $\OX(\DDf) = \coprod_{x \in \fcU} \sA_x$, $U = \DDf$,  by  definition \eqref{eq:fc.sections} of  structure \nushv.
If $V \subseteq X$ is an arbitrary open subset, then  $\OX(V) = \bigcap_U \OX(U)$ where $U = \DDf$ runs through the open sets  $U \subseteq V$. Equivalently, $\OX(V) = \bigcap_f \OX(U)$ with $U = \DDf \subseteq V$. By Theorem \ref{prop:6.6.9} $\sAx \cong \OXx$ for every
$x \in X$, and thus, for any nonempty open subset $V \subseteq X$ we have
$$\OX(V) \cong \bigcap_{x\in \fcV} \OXx,$$  where  $\fcV = \bigcap_{U \subseteq V} \fcU .$
\end{remark}

\subsubsection{Reduced and integral schemes} \sSkip

We generalize the notion of being reduced from \nusmr s
to \nusch s.
\begin{defn}
 A \nusch\  $X$ is called \textbf{reduced}, if all its  local \nusmr s $\OXx$ are ghost reduced \nusmr s (Definition ~ \ref{def:ghostpotent}).
   $X$ is \textbf{integral}, if it is reduced and irreducible.
\end{defn}

\begin{prop}
Let $X$ be a \nusch.
\begin{enumerate} \eroman
\item  $X$ is reduced if and only if for every open subset $U \subseteq X$ the \nusmr\ $\Gm(U,\OX)$
is reduced.
\item  $X$ is integral if and only if for every open subset $ \emptyset \neq  U \subseteq X$ the \nusmr\
$\Gm(U,\OX)$ is a \nudom.

\item  If $X$ is an integral \nusch, then for each  $x \in X$ the local \nusmr\ $\OXx$ is a \nudom.\footnote{The converse does not hold.}
\end{enumerate}
\end{prop}

\begin{proof}
(i): Suppose $X$ is reduced and $U \subseteq X$ is open.  Assume $f \in \Gm(U,\OX)$ such that
$f^n = \ghost$ for some $n$. If we had $f \neq \ghost $, then there would exist $x \in U$ with $f_x \neq \ghost$ in $\OXx$, but $f^n_x = \ghost$.
Conversely, given a ghostpotent $\olf \in \OXx$  (Definition \ref{def:ghostpotent}), there exists an
open $U \subseteq X$ and a lift $f \in \Gm(U,\OX)$ of $\olf$. Shrinking $U$ if necessary, we may assume
that $f$ is ghostpotent, and hence is ghost.

\pSkip
(ii): Let $X$ be integral. All open \nussch s of $X$ are integral, so it is
enough to show that $\Gm(X,\OX)$ is a \nudom. Taking $f, g \in \Gm(X,\OX)$ such that $fg = \ghost $,
we have $X = \VV (f) \cup \VV (g)$, and by irreducibility, say, $X = \VV (f)$. Checking this locally on $X$, which we may assume
is affine,  we claim that~ $f$ must be ghost. Indeed, $f$ lies in the intersection of the ghost projections  of all \gprimec s, i.e., in the \ggradical \ ideal    of the
affine \nusmr\ of $X$. Since $X$ is reduced, by (i), the ghost projection of its \ggradicalc\ is the ghost ideal of $\Gm(X,\OX)$. Namely, $\Gm(X,\OX)$ is a \nudom, cf. Lemma \ref{lem:rad-reduced} and Lemma \ref{prop:15}.

Conversely, if all $\Gm(U,\OX)$ are \nudom s, then  $X$ is reduced by (i). For nonempty affine open subsets $U_1, U_2 \subseteq X$ with empty intersection, if exists, the
sheaf axioms imply that
$$\Gm(U_1 \cup U_2,\OX) = \Gm(U_1,\OX) \times  \Gm(U_2,\OX).$$
Obviously,  the product on the right contains ghost divisors.

\pSkip
(iii): Follows from (ii), since any tangible localization of a \nudom\ is a \nudom.
\end{proof}
An affine \nusch\ over  $X = \SpecA$ is integral if and only if the corresponding \nusmr ~ $\sA$ is a \nudom. Then, the  generic point
$q$ of $X$ corresponds to the trivial congruence $\diag(\sA)$ of $\sA$, and the local \nusmr \  $\scO_{X,q}$ is the localization of
$\sA$ by $\diag(\sA)$, i.e., by $\sAptng = \sAtng$, which is a monoid as $\sA$ is a \nudom.  But, this localization is just the \nusmf\ of fractions $Q(\sA)$ of $\sA$ (Definition  \ref{def:tangible.localization}). This also shows that the local \nusmr\ at the
generic point of an arbitrary integral \nusch\ is a \nusmf.

\begin{defn}
 Let $\gpt \in X$ be the generic point of an integral \nusch\ $X$. The local \nusmr\ $\scO_{X,q}$ is a \nusmf,  denoted by $K(X)$ and called the \textbf{function \nusmf} of $X$.
\end{defn}
For an integral \nusch\ all ``\nusmr s of functions'' are contained in its function \nusmf.
\begin{prop}

Let $X$ be an integral \nusch\ with generic point $\gpt$, and let $K(X)$ be
its function \nusmf.

\begin{enumerate} \eroman
  \item  If $U = \SpecA$ is a nonempty open affine \nussch\ of $X$, then $K(X) = Q(\sA)$. Furthermore,
$Q(\OXx) = K(X)$ for $x \in X$.

\item For  nonempty open subsets $U \ssetq V \ssetq X$, the maps
$$\Gm(V,\OX) \Right{5}{\rho^V_U} \Gm(U,\OX) \Right{5}{f \longmapsto f_\gpt} K(X)$$
are injective.
\item
 For every nonempty open subset $U \ssetq X$ and for every open covering $U =\bigcup_i U_i$ the following holds
$$ \Gm(U,\OX) = \bigcap_i \Gm(U_i,\OX) = \bigcap_{x \in U} \OXx,$$
where the intersection occurs  in $K(X)$.
\end{enumerate}

\end{prop}
\begin{proof} (i):
For $x \in U = \SpecA \ssetq X$ we have $\gpt \in U$, where $\gpt$ corresponds to the trivial congruence  on  the \nudom\ $\sA$. Since $\OXx \cong \sAx$, we have $K(X) = \scO_{U,\gpt} = Q(\sA) = Q(\sAx )$.

\pSkip
(ii): Given  $ \emptyset \neq U \ssetq X$, it suffices to prove that if  $f \in \Gm(U,\OX)$ with $f_\gpt = \ghost$ in $K(X)$,   then  ~$f$ is a ghost.
Since $f = \ghost$ is equivalent to $f|_V = \ghost$ for all open nonempty affine \nussch s
$V \ssetq U$, we may assume that $U = \SpecA$ is affine. Then, the map $\Gm(U,\OX)  \To K(X)$ is just the
canonical inclusion $A \Into Q(\sA) = K(X)$.

\pSkip
(iii): The injectivity of the restriction maps $\rho^U_{U_i}: U \To U_i$, and the fact
that $\OX$ is a sheaf, implies left equality.  The analogous assertion for affine
integral \nusch s in Remark ~\ref{rem:u.2.37} gives the right equality.
\end{proof}

\subsection{Fiber products}%
\sSkip

Let $Y$ be a \nusch. A \textbf{\nusch\ over} $Y$ is a \nusch\ $X$  with a
morphism $\ff: X  \To Y$. A~ morphism of \nusch s $X, Z$ over $Y$ is a morphism of \nusch s
$X \To Z$  that renders the diagram
$$\xymatrix{
X \ar@{->}[rd] \ar@{->}[rr] && Z \ar@{->}[ld] \\
 &Y  &  \\
}
$$
commutative. A \nusch\ over $Y = \SpecB$ is termed  a  \nusch\ over $\sB$, for short.
A \nusch\ $X$ over a \nusmf\ $F$  is of \textbf{finite type}, if $X$ has a finite cover by open affine subsets
$U_i = \Spec(\sA_i)$, where each $\sA_i$ is a finitely generated $F$-\nualg, cf. \S\ref{ssec:F.Alg}. We then have the following observations for an affine \nusch\ $X = \SpecA$.

\begin{enumerate}  \ealph
\item $X$ is a \nusch\ over $F$ if and only if there exists a morphism $F \To \sA$, i.e.,  if $\sA$
is an $F$-\nualg.
\item
A morphism $X \To Y = \SpecB$ is a morphism of \nusch s
over $F$ if and only if the corresponding \qhom\ $\sB \To \sA$ of \nusmr s is a morphism
of $F$-\nualg s.
\item $X$ is of finite type over $F$ if and only if $F$ is a finitely generated $F$-\nualg\ (Definition \ref{def:f.g.alg}).
\item $X$ is reduced and irreducible if and only if $f  g = \ghost$ in $\sA$ implies $f = \ghost$ or
$g=\ghost$, i.e., if and only if $\sA$ is a \nudom. Indeed, suppose that $f  g= \ghost$ where $f \neq \ghost$
and $g \neq \ghost$. If $f = g^n $ or $g = f^m$ for some $m, n \in \Net$, then $\sA$ has a ghostpotent, namely $X$ is not reduced. Otherwise, $X$ decomposes into two
proper closed subsets $\VV( f )$ and $\VV(g)$, and thus $X$ is not irreducible.

\end{enumerate}

\begin{defn}\label{fiber.prd}
Let $\ff : X  \To S$ and $\hh : Y \To S$  be morphisms of \nusch s. Their
\textbf{fiber product} $\XxSY$  is defined  to be a \nusch \ together with ``projection''  morphisms $\pi_X : \XxSY \To
X$ and $\pi_Y : \XxSY \To Y$ such that the square in \eqref{eq:fiber.1} below commutes, and
such that for any \nusch\ $Z$ with morphisms $Z \To X$  and $Z \To Y$ rendering  \eqref{eq:fiber.1} commutative with $\ff$  and $\hh$ there is a unique morphism $\xi: Z \To X \times_S Y$ which renders the whole diagram
\begin{equation}\label{eq:fiber.1}\begin{gathered}
\xymatrix{ Z  \ar@{..>}[dr]^{\xi}  \ar@/^1pc/@{->}[rrd]  \ar@/_1pc/@{->}[rdd] & & \\
& \XxSY \ar@{->}[d]^{\pi_X}  \ar@{->}[r]_{\pi_Y} & Y \ar@{->}[d]^\hh \\
& X \ar@{->}[r]_\ff & S }\end{gathered}
\end{equation} commutative.
\end{defn}

A routine proof shows that the fiber product is uniquely determined by its property.
\begin{lem}\label{lem:g.5.4.2} If the fiber product $\XxSY$ exists, then it is unique. (Namely, if two fiber products satisfy the above  property, then they are canonically isomorphic.)
\end{lem}
\begin{proof}
Let $F_1$ and $F_2$ be two fiber products satisfying the  property of Definition \ref{fiber.prd}, i.e., each $F_i$ is assigned with morphisms to $X$ and $Y$. Since $F_i$ is a fiber product, the morphism $\xi_{ij} : F_i \To F_j$, $i,j \in \{ 1,2\} $, renders the diagram
$$ \xymatrix{ F_i  \ar@{..>}[dr]^{\xi_{ij}}  \ar@/^1pc/@{->}[rrd]  \ar@/_1pc/@{->}[rdd] & & \\
& F_j \ar@{->}[d]^{\pi_X}  \ar@{->}[r]_{\pi_Y} & Y \ar@{->}[d]^\hh \\
& X \ar@{->}[r]_\ff & S } $$
commutative.
Composing together,  by the uniqueness part of  Definition \ref{fiber.prd}, it follows that $\xi_{ij} \circ \xi_{ji} = \id$, and thus $F_1$ and $F_2$ are canonically isomorphic.
\end{proof}

\begin{rem}\label{rem:g.5.4.3}
From  Definition \ref{fiber.prd} we obtain the following properties.
\begin{enumerate} \eroman
  \item $\XxSY = \XxTY$ for any open subset $S  \subset T$ (morphisms from any $Z$ to $X$
and to $Y$ that commute with $\ff$ and $\hh$ are the same, independently if the base
\nusch\ is $S$ or $T$).
\item Given open subsets $U \subseteq  X$ and $V \subseteq Y$, the fiber product
$$U \times_S V = \ipi_X (U) \cap \ipi_Y (V) \subseteq  \XxSY $$
is an open subset of the  fiber product $\XxSY$.
\end{enumerate}
\end{rem}

We next show that fiber products of \nusch s  always exist, where in the affine case they should correspond to tensor products in commutative \nualg. Using the tensor product of \numod s, cf. ~ \S\ref{ssec:tensor}, we can  construct the fiber product of \nusch s.
\begin{prop}\label{lem:g.5.4.7} Let $\ff : X  \To S$ and $\hh : Y  \To S$ be morphisms of \nusch s. Then, the fiber product $\XxSY$ exists. \end{prop}
\begin{proof} Assume first  that  $X = \SpecA$, $Y = \SpecB$, and
$S = \SpecR$ are affine \nusch s.  The morphisms $X \To S$ and $Y \To S$ provide $\sA$ and $\sB$ as $R$-\numod s by Theorem \ref{prop:6.6.9},
yielding the tensor product $\AoRB$.
We claim that $\Spec(\AoRB)$ is the fiber product $\XxSY$. Indeed,  a
morphism $Z \To \Spec(\AoRB)$ corresponds to   a \qhom\
$\AoRB \To \scO_Z(Z)$ by
Proposition \ref{prop:8.7}, which by Remark \ref{rem:tensor} is the same as \qhom s  $\sA \To \scO_Z(Z)$
and $\sB \To \scO_Z(Z)$. The latter induce the same \qhom\ on $R$, which again by Proposition
\ref{prop:8.7} corresponds to  morphisms $Z \To X$ and $Z \To Y$, which in turn give rise to the same morphism from $Z \To S$. Therefore,  $\Spec(\AoRB)$ is the desired product.

Assume that $X$, $Y$ and $S$ are general \nusch s, and take coverings by open affine \nusch s, first of $S$, and then of $\iff(S_i)$, $\ihh(S_i)$  by $X_{i, j}$, $Y_{i,k}$, respectively.
 The fiber products $X_{i, j} \times_{S_i} Y_{i,k}$ exist by the construction of tensor products, which are also fiber products over $S$ by Remark \ref{rem:g.5.4.3}.(i).  If we had another such product $X_{i', j'} \times_S
Y_{i',k'}$, both of them would contain the (unique) fiber product $(X_{i, j } \cap X_{i', j'} )\times_S (Y_{i,k} \cap  Y_{i',k'} )$ as
an open subset by Remark \ref{rem:g.5.4.3}.(ii), hence they can be glued along these isomorphic open
subsets. Thus, the \nusch\ $\XxSY$ obtained by glueing these patches satisfies
the fiber product property.
\end{proof}

\end{document}